\documentclass[a4paper,12pt]{article}
\usepackage{amsthm,amsfonts,amssymb,euscript}
\usepackage{latexsym, multicol, fancybox}
\usepackage{graphicx}
\usepackage{color}
\usepackage{amsmath, amsthm, amssymb, bm}
\usepackage{epstopdf}
\usepackage{caption}
\usepackage{psfrag}

\DeclareFontFamily{U}{mathx}{\hyphenchar\font45}
\DeclareFontShape{U}{mathx}{m}{n}{
      <5> <6> <7> <8> <9> <10>
      <10.95> <12> <14.4> <17.28> <20.74> <24.88>
      mathx10
      }{}
\DeclareSymbolFont{mathx}{U}{mathx}{m}{n}
\DeclareFontSubstitution{U}{mathx}{m}{n}
\DeclareMathAccent{\widecheck}{0}{mathx}{"71}

\def\ombc{\widecheck{\omb}}

\newtheorem{theorem}{Theorem}[section]
\newtheorem{lemma}[theorem]{Lemma}
\newtheorem{proposition}[theorem]{Proposition}
\newtheorem{corollary}[theorem]{Corollary}
\newtheorem{definition}[theorem]{Definition}
\newtheorem{remark}[theorem]{Remark}

\numberwithin{equation}{section}

\newcommand{\bea}{\begin{eqnarray}}
\newcommand{\eea}{\end{eqnarray}}
\def\beaa{\begin{eqnarray*}}
\def\eeaa{\end{eqnarray*}}
\def\ba{\begin{array}}
\def\ea{\end{array}}
\def\be#1{\begin{equation} \label{#1}}
\def \eeq{\end{equation}}
\def\bsplit{\begin{split}}
\newcommand{\nn}{\nonumber}

\def\lab{\label}
\def\les{\lesssim}

\def\c{\cdot}
\def\lot{\mbox{l.o.t.}}
\def\dual{{\,^\star \mkern-2mu}}

\def\tr{\mbox{tr}}
\def\atr{\,^{(a)}\mbox{tr}}

\def\ov{\overline}

\renewcommand{\div}{\mbox{div}\,}
\newcommand{\curl}{\mbox{curl}\,}
\def\nab{\nabla}
\def\lap{\Delta}
\def\pr{\partial}

\def\dkb{ \, \mathfrak{d}     \mkern-9mu /}
\def\dk{\mathfrak{d}}

\def\dddS{\ddd^{\,\S}}

\def\dddSdiez{\ddd^{\,\S, \#}}
\def\ddsSdiez{\big( \ddd^{\,\S, \#}\big)^\star}

\def\dds{ \, d  \hspace{-1.5pt}\dual    \mkern-14mu /\,\,}
\def\ddd{ \,  d \hspace{-2.4pt}    \mkern-6mu /\,}

\def\ddsS{{\ddd^{\,\S,\star}}}

\def\ddsSone{{\ddd_1^{\,\S,\star}}}
\def\ddsStwo{{\ddd_2^{\,\S,\star}}}

\def\curlSn{\curl^{\S(n)}}
\def\divSn{\div^{\S(n)}}

\def\rhod{ \,^\star  \hspace{-2.2pt} \rho}

\def\ovu{\overset{\circ}{ u}}

\def\ovs{\overset{\circ}{ s}}

\def\ovr{\overset{\circ}{ r}}
\def\ovla{\protect\overset{\circ}{ \la}\,}
\def\ovb{\protect\overset{\circ}{b}\,}

\def\ovS{\overset{\circ}{ \S}}

\def\epg{\protect\overset{\circ}{\ep}}  

\def\ug{\overset{\circ}{u}}   
\def\sg{\overset{\circ}{s}}               
\def\rg{\overset{\circ}{r}}   
 
\def\mg{\overset{\circ}{m}}   

\def\ovu{\overset{\circ}{ u}}

\def\ovs{\overset{\circ}{ s}}

\def\ovr{{\overset{\circ}{ r}\,}}
\def\ovm{{\overset{\circ}{ m}\,}}
\def\ovS{\overset{\circ}{ S}}
\def\ovga{\overset{\circ}{\ga}\,}

\def\ovg{\overset{\circ}{ g}}

\def\ovGa{\overset{\circ}{ \Ga}}

   \def\Unn{U^{(n+1)}}
  \def\Un{U^{(n)}}

                 \def\lapSn{\lap^{\S(n)}}

             \def\hb{\underline{h}}

\def\atrchS{{\atrch^\S}}

\def\atrchbS{{\atrch^\S}}

\def\a{\alpha}
\def\b{\beta}
\def\ga{\gamma}
\def\Ga{\Gamma}
\def\de{\delta}

\def\ep{\epsilon}
\def\la{\lambda}
\def\La{\Lambda}
\def\si{\sigma}
\def\Si{\Sigma}
\def\om{\omega}

\def\rhod{\dual \rho}
\def\th{{\theta}}

\def\ka{\kappa}
\def\ze{\zeta}
\def\Up{\Upsilon}

\def\atrch{\atr\chi}
\def\atrchb{\atr\chib}

\def\vphi{{\varphi}}

\def\vsi{\varsigma}

\renewcommand{\aa}{\protect\underline{\a}}
\newcommand{\bb}{\protect\underline{\b}}
\def\omb{\protect\underline{\om}}
\def\Lb{{\underline{L}}}

\def\Omb{\underline{\Omega}}
\def\ub{{\underline{u}} }
\newcommand{\chib}{\protect\underline{\chi}}
\newcommand{\xib}{\protect\underline{\xi}}
\newcommand{\etab}{\protect\underline{\eta}}

\def\kab{\protect\underline{\kappa}}

\def\AA{{\mathcal A}}

\def\CC{{\mathcal C}}
\def\DD{{\mathcal D}}

\def\GG{{\mathcal G}}
\def\HH{{\mathcal H}}
\def\II{{\mathcal I}}

\def\MM{{\mathcal M}}
\def\NN{{\mathcal N}}

\def\QQ{{\mathcal Q}}
\def\RR{{\mathcal R}}
\def\SS{{\mathcal S}}
\def\TT{{\mathcal T}}
\def\UU{{\mathcal U}}
\def\YY{{\mathcal Y}}

\def\D{{\bf D}}

\def\O{{\bf O}}

\def\R{{\bf R}}

\def\S{{\bf S}}

\def\U{{\bf U}}

\def\Z{{\bf Z}}

\def\g{{\bf g}}

\def\Bb{\underline{B}}

\def\fb{\protect\underline{f}}
\def\CCb{\underline{\CC}}
\def\Cb{{\underline{C}}}

\def\NNN{{\Bbb N}}

\def\RRR{{\Bbb R}}
\def\SSS{{\Bbb S}}

\def\hk{\mathfrak{h}}

\def\Gac{\check \Gamma}

\def\ombc{\underline{\check \omega}}

\def\chih{\widehat{\chi}}
\def\chibh{\widehat{\chib}}
\def\trch{\tr \chi}
\def\trchb{\tr \chib}

\def\hot{\widehat{\otimes}}
\def\c{\cdot}
\def \f12{\frac 1 2 }
\def\ov{\overline}
\def\err{\mbox{Err}}

\def\gS{g}

\def\err{\mbox{Err}}

\def\Lab{\protect\underline{\La}}

\def\aka{\,^{(a)} \ka}
\def\akab{\,^{(a)} \kab}

\def\epg{\protect\overset{\circ}{\ep}}  
\def\dg{\overset{\circ}{\de}}

\def\undB{\underline{B}}

  \def\ovJq{J^{(q)}}

\def\Un{U^{(n)}}
\def\Sn{{S^{(n)}}}
\def\fn{f^{(n)}}
\def\fbn{\fb^{(n)}}
\def\fnn{f^{(n+1)}}
\def\fbnn{\fb^{(n+1)}}
\def\ovlan{\ovla^{(n)}}
\def\ovlann{\ovla^{(n+1)}}

\def\curln{\curl^{(n)}}

\def\divn{\div^{(n)}}
\def\lapn{\lap^{(n)}}
\def\Sn{ S^{(n)}}

\def\Snn{ S^{(n+1)}}

\def\gn{g^{\S(n)}}

\def\Sinfty{ S^{(\infty)}}
\def\Uinfty{ U^{(\infty)}}

\def\finfty{f^{(\infty)}}
\def\fbinfty{\fb^{(\infty)}}
\def\ovlainfty{\ovla^\infty}

\def\divzero{{\overset{\circ}{ \div}}}

\def\lapzero{{\overset{\circ}{ \lap}}}
\def\nabzero{{\overset{\circ}{ \nab}}}

\def\jp{j^{(p)}}
\def\Jp{J^{(p)}}
\def\JpS{{J^{(\S, p)}}} 
\def\Cbp{\Cb^{(p)}} 
\def\Mp{M^{(p)}}

\def\CbpS{\Cb^{(\S, p)}} 
\def\MpS{M^{(\S, p)}}

\def\Uinfty{U^{(\infty)}}
\def\Sinfty{S^{(\infty)}}
\def\finfty{f^{(\infty)}}
\def\fbinfty{\fb^{(\infty)}}
\def\ovlainfty{\ovla^{(\infty)}}
\def\einfty{e^{(\infty)}}

\def\curlSinfty{\curl^{\S^{(\infty)}}}
\def\divSinfty{\div^{\S^{(\infty)}}}
\def\kainfty{\ka^{(\infty)}}
\def\kabinfty{\kab^{(\infty)}}
\def\muinfty{\mu^{(\infty)}}
\def\lapSinfty{\lap^{\S(\infty)}}

\def\Minfty{M^{(\infty)}}
\def\Cbinfty{\Cb^{(\infty)}}
\def\Cbpinfty{\Cb^{(\infty),p}}
\def\Mpinfty{M^{(\infty),p}}
\def\Jpinfty{J^{(\infty),p}}

\def\kadot{\dot{\ka}}
\def\kabdot{\dot{\kab}}
\def\mudot{\dot{\mu}}

\def\Cbdot{\dot{\Cb}}
   \def\Mdot{\dot{M}}
   \def\Cbpdot{{\dot{\Cb}^{(p)}}}
   \def\Mpdot{{\dot{M}^{(p)}}}

   \def\Cbnn{\Cb^{(n+1)}}
   \def\Cbpnn{\Cb^{(n+1), p}}
   \def\Mpnn{M^{(n+1), p}}
   \def\Mnn{M^{(n+1)} }
   
     \def\Cbnndot{\Cbdot^{(n+1)}}
   \def\Cbpnndot{\Cbdot^{(n+1), p}}
   \def\Mpnndot{\Mdot^{(n+1), p}}
   \def\Mnndot{\Mdot^{(n+1)} }

   \def\Cbn{\Cb^{(n)}}
   \def\Cbpn{\Cb^{(n), p}}
   \def\Mpn{M^{(n), p}}
   \def\Mn{M^{(n)} }
   
     \def\Cbndot{\Cbdot^{(n)}}
   \def\Cbpndot{\Cbdot^{(n), p}}
   \def\Mpndot{\Mdot^{(n), p}}
   \def\Mndot{\Mdot^{(n)} }
   
   \def\fnndot{(\de f)^{(n+1)}}
    \def\fbnndot{ (\de\fb)^{(n+1)}}
   \def\ovlanndot{ (\de\ovla)^{(n+1)}}
   
    \def\fndot{(\de f)^{(n)}}
    \def\fbndot{ (\de\fb)^{(n)}}
   \def\ovlandot{ (\de\ovla)^{(n)}}

\def\Cbnnde{\, \de\Cbnn}
  \def\Cbpnnde{\, \de\Cbpnn}
 \def\Mnnde{\, \de \mkern-3mu\Mnn}
  \def\Mpnnde{\, \de \mkern-3mu\Mpnn}

\def\ovlanndiez{\ovla^{n+1,\#}}
\def\fnndiez{f^{n+1,\#}}
\def\fbnndiez{\fb^{n+1,\#}}
\def\ovlandiez{\ovla^{n,\#}}
\def\fndiez{f^{n,\#}}
\def\fbndiez{\fb^{n,\#}}
\def\ovlanndiez{\ovla^{n+1,\#}}
\def\fnndiez{f^{n+1,\#}}
\def\fbnndiez{\fb^{n+1,\#}}

\def\curln{\curl^{(n)}}
\def\divn{\div^{(n)}}
\def\lapn{\lap^{(n)}}

\def\curlnmin{\curl^{(n-1)}}
\def\divnmin{\div^{(n-1)}}
\def\lapnmin{\lap^{(n-1)}}

\def\gn{g^{(n)}}
\def\Fndiez{ F^{(n), \#}}

\def\JpSn{J^{\S(n), p}}

\def\curlnmin{\curl^{(n-1)}}
\def\divnmin{\div^{(n-1)}}
\def\lapnmin{\lap^{(n-1)}}

\newcommand{\Mext}{\,{}^{(ext)}\mathcal{M}}

\newcommand{\Mint}{\,{}^{(int)}\mathcal{M}}

\begin{document}

\title{Construction of GCM spheres in perturbations of Kerr}
\author{Sergiu Klainerman, J\'er\'emie Szeftel}

\maketitle

{\bf Abstract.} \textit{This the first in a series of papers   whose  ultimate goal is to establish the full nonlinear stability  of the Kerr family for $|a|\ll m$.    The paper  builds on the strategy  laid out in  \cite{KS} in the context of  the nonlinear stability of Schwarzschild for axially symmetric polarized  perturbations.   In fact  the central idea of    \cite{KS}  was  the  introduction and construction  of  generally covariant modulated  (GCM) hypersurfaces on which specific  geometric quantities take Schwarzschildian values.  This was made possible by taking into account the full  general covariance of the Einstein vacuum equations.  The goal of   this paper  is to  get rid  of   the symmetry  restriction  in the construction of GCM spheres  and thus  remove   an  essential obstruction  in extending the  result of \cite{KS}  to a full stability proof of the   Kerr family. }

 \tableofcontents

%%%%%%%%%%%%%%%%%%%%%%%%%%%%%
    
  \section{Introduction}

%%%%%%%%%%%%%%%%%%%%%%%%%%%%%

%%%%%%%%%%%%%%%%%%%%%%%%%%%%%

\subsection{Stability of Kerr conjecture}

%%%%%%%%%%%%%%%%%%%%%%%%%%%%%

 This the first in a series of papers   whose  ultimate goal is to establish the full nonlinear stability  of the Kerr family for $|a|\ll m$. 
  
     {\bf Conjecture} (Stability of Kerr conjecture).\,\,{\it  Vacuum initial data sets, 
   sufficiently close to Kerr initial data, have a maximal development with complete
   future null infinity\footnote{This means, roughly, that observers  which are  far away  
    from the black hole  may live forever.   } and with   domain of outer communication which
   approaches  (globally)  a nearby Kerr solution.}
   
For an in depth introduction to  the conjecture, see our introduction in \cite{KS}
 as well as the survey article \cite{Daf} and the lecture notes \cite{DafRod}.

%%%%%%%%%%%%%%%%%%%%%%%%%%%%%%%%
   
   \subsection{Stability of Schwarzschild  in the polarized case} 
   
%%%%%%%%%%%%%%%%%%%%%%%%%%%%%%%%   

%%%%%%%%%%%%%%%%%%%%%%%%%%%%%%%%   
   
   \subsubsection{GCM admissible spacetimes  in \cite{KS}}      
   
%%%%%%%%%%%%%%%%%%%%%%%%%%%%%%%%   

   In \cite{KS} we were able to prove the nonlinear stability of  the Schwarzschild space under axially symmetric polarized perturbations.  These are spacetimes possessing\footnote{Condition which, if  imposed on the initial  data, is preserved  by   evolution.}  a   spacelike, axial,   hypersurface orthogonal      Killing vectorfield  $\Z$.

   The  final spacetime  in \cite{KS}  was constructed    as the limit    of a continuous   family  of finite  GCM admissible  spacetimes as   represented  in Figure \ref{fig1-introd} below,
   whose  future boundaries   consist of the union $\AA\cup \CCb_* \cup \CC_* \cup \Si_*$ 
 where $\AA$  and  $ \Si_*$  are spacelike,  $   \CCb_*$ is incoming null, and $\CC_*$ outgoing null.   The boundary $\AA$ is chosen so  that,  in the limit when   $\MM$ converges to the final state, it  is included in  the  perturbed    black hole. The spacetime $\MM$ also contains a timelike hypersurface  $\TT$  which divides  $\MM$ into  an exterior region we call $\Mext$ and an  interior one $\Mint$.       Both $\Mext$ and $\Mint$  are foliated by $2$ surfaces as follows.
\begin{figure}[h!]
\centering
\includegraphics[scale=0.5]{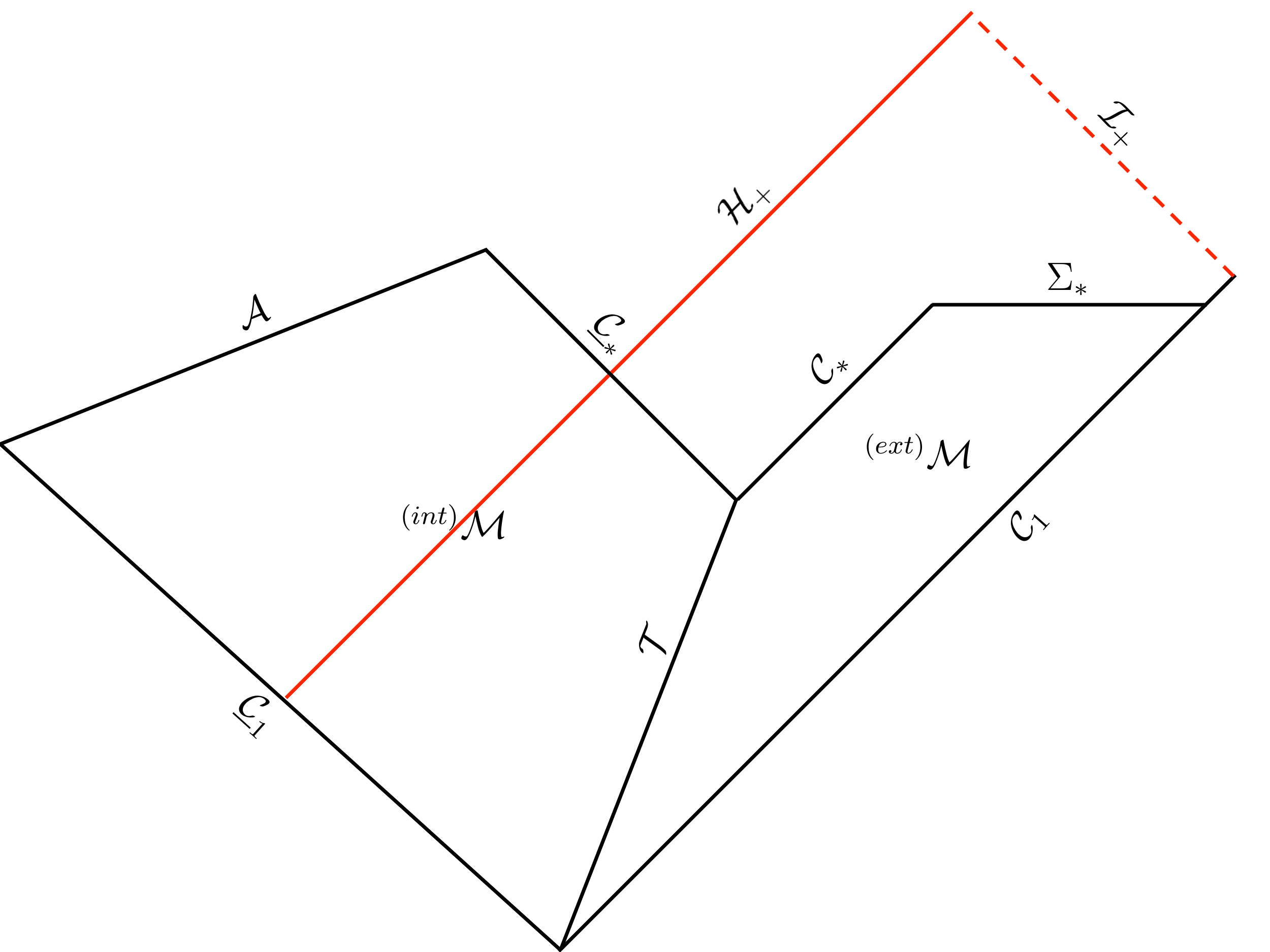}
\caption{The GCM admissible space-time $\mathcal{M}$}
\lab{fig1-introd}
\end{figure} 
 \begin{enumerate}
 \item[(i)]  The  far region  $\Mext  $ is foliated by a geodesic foliation $S(u, s)$        induced by   an outgoing  optical function $u$  initialized on  $\Si_*$ with $s$ the affine parameter along the null geodesic generators of $\Mext$.  We denote by $r=r(u, s)$ the area radius of $S(u,s)$. On the boundary $\Si_*$ of $\Mext  $ we also assume that $r$ is sufficiently large.
  
\item[(ii)]  The  near region $\Mint $ is foliated by a geodesic foliation induced by   an incoming   optical function $\ub$ 
initialized at $\TT$ such that its level sets on $\TT$ coincide with  those of $u$.  
\end{enumerate}

To prove convergence to   the final state we had to establish   precise decay estimates  for  all Ricci   and curvature coefficients    decomposed relative  to the    null geodesic  frames associated  to the foliations in $\Mext$ and $\Mint$. We note  that the estimates for $\Mint$   are  relatively simple once  the estimates in $\Mext $  have been derived;   most difficulties had to do  with this latter region. In fact    the    decay properties of both  Ricci and curvature  coefficients in $\Mext$  depend heavily on   the choice of   the boundary $\Si_*$  as well  as  on   the choice of  the cuts of the optical function   $u$ on it.             As such, the central idea of    \cite{KS}  was  the  introduction and construction  of  generally covariant modulated  (GCM) hypersurfaces on which specific  geometric quantities take Schwarzschildian values.  This was made possible by taking into account 
the full  general covariance of the Einstein vacuum equations.   More precisely,     the GCM spacelike boundary  $\Si_*$ are  foliated by spheres   $S$   on which three key geometric quantities  are set to have  the same values as  in the case of canonical spheres in  Schwarzschild.
 To make sense of this, we recall that  the  Schwarzschild metric  in  outgoing Eddington-Finkelstein coordinates has the form\footnote{Here $u=t-r_*$,    $\frac{dr_*} {dr}= \Up^{-1}$ and $r=s$.  Recall also that  in standard  spherical coordinates, we have   $g_{m} =-\Up  dt^2 +\Up^{-1} dr^2 + r^2  d\si^2$.}
\bea
\g_{m} &=&- 2 du ds -\Up  du^2 + r^2  d\si, \qquad   \Up= 1-\frac{2m}{r},
\eea
where  $r=r(u, s)$ denotes the area radius   of the spheres $S(u, s)$ of constant $u$ and $s$,  and  $d\si$ denotes the  standard metric on $\SSS^2$.   For a      given canonical  sphere $S(u, s)$,  the   expansions\footnote{See section \ref{sec:defnullpairandhorizontalstructurre} and \eqref{def:massaspectfunctions.general} for  the precise definition of these quantities.} 
 $\ka=\trch$ and  $\kab=\trchb$,  and the mass aspect function $\mu$       are given by 
 \bea
 \lab{Introd:GCM spheres}
 \ka=\frac 2 r , \qquad \kab=-\frac{2\Up}{r}, \qquad \mu=\frac{2m}{r^3}.
 \eea
Thus a sphere $S$ on the above mentioned foliation of $\Si_*$ is said to be a GCM sphere  if, relative  to the   canonical frame of $\Mext$,  the conditions\footnote{In reality \eqref{Introd:GCM spheres} had to  be slightly modified on  the $ \ell =0, 1$ modes of $\kab$ and $\mu$, see more explanations below.    } 
  \eqref{Introd:GCM spheres}  are verified.  Note that the three  exact conditions  in \eqref{Introd:GCM spheres} are  matched by the number of degree of  freedoms 
  of gauge transformations  which preserve the polarization condition. Another  way to express this is by noticing that a sphere $S$   in  a given spacetime  
   can be specified by  two  scalar functions  while   a  future  null  pair\footnote{A null pair adapted to $S$ is a pair of null vectors such that $e_3$ and $e_4$ are orthogonal to the tangent space of $S$ and $\g(e_3, e_4)=-2$, see section \ref{sec:defnullpairandhorizontalstructurre}.}  $(e_3, e_4)$   adapted to $S$   is uniquely determined   by one scalar  function\footnote{A null pair adapted to $S$ is uniquely determined up to the transformation $(e_4, e_3)\to (\la e_4, \la^{-1}e_3)$ for any scalar function $\la>0$.}.

%%%%%%%%%%%%%%%%%%%%%%%%%%%%%%%%%%   

  \subsubsection{The role played by GCM admissible spacetimes}    
  
%%%%%%%%%%%%%%%%%%%%%%%%%%%%%%%%%%  
  
  As mentioned  above the  final spacetime was constructed    as the limit    of a continuous   family  of finite  GCM admissible  spacetimes. At every stage one assumes that all Ricci and curvature coefficients of a fixed  GCM admissible spacetime $\MM$   verify precise    bootstrap assumptions.  One   makes use of the GCM admissibility  properties of $\Si_*$ and the smallness of the initial conditions  to show that  all  the   bounds of the  Ricci and curvature coefficients of $\MM$ depend only on the size of the initial data  and thus, in particular, improve the bootstrap assumptions.  This  
  allows us to extend the spacetime to a larger one  $\MM'$ in which  the bootstrap assumptions are still valid.  Note that  the exact  conditions  \eqref {Introd:GCM spheres}  cannot be maintained in the extended spacetime $\MM'$ but we can control the size of  the quantities 
  \bea
  \lab{Introd:GCM spheres-small}
 \ka-\frac 2 r , \qquad \kab+\frac{2\Up}{r}, \qquad \mu- \frac{2m}{r^3},
  \eea 
  defined relative to the   geodesic foliation of $\MM'$, extended from that of $\MM$. 
       To  make sure that the  extended spacetime is admissible, one has to construct  a new  GCM    hypersurface $\widetilde{\Si}_* $ in $\MM'\setminus\MM$  and  use it to  define a  new extended GCM admissible spacetime $\widetilde{\MM}$.  It is at this stage  that we have to prove the existence  of GCM spheres  in  $\MM'\setminus \MM$.  More precisely, using the bounds on  the Ricci and curvature coefficients  on $\MM'$, we have to construct  GCM spheres $S$  in  $\MM'\setminus \MM$ as building blocks for  $\widetilde{\Si}_* $.  This was done   in \cite{KS}  by   a  deformation argument  in which  the polarization assumption  seemed to play  an important role, as it will be explained below.

%%%%%%%%%%%%%%%%%%%%%%%%%%%%%%%%%%%

 \subsection{Construction of  GCM  spheres in perturbations of Kerr}  
 
%%%%%%%%%%%%%%%%%%%%%%%%%%%%%%%%%%% 
 
 The goal of   this paper  is to  get rid  of   the polarization  restriction  in the construction of GCM spheres  and thus  remove   an  essential obstruction  in extending the  result of \cite{KS}  to a full stability proof of the   Kerr family.   The construction of GCM spheres and GCM hypersurfaces in  perturbations of Kerr  are meant to  play a  role similar  to that  discussed above,  i.e.  their construction is needed in  spacetime regions\footnote{Corresponding  to $\MM'\setminus\MM$.} $\RR$ where $r$ is sufficiently large and  where  we  already have  complete 
  control of the Ricci and curvature  components, denoted $\Ga$ and $R$,   relative to a prescribed  outgoing geodesic foliation $S(u, s)$ and adapted null frames $(e_1, e_2, e_3, e_4)$  with
  $e_1, e_2$  tangent to  the spheres $S$.    The size of  the quantities in \eqref{Introd:GCM spheres-small} is assumed  to be  controlled by  a small constant\footnote{This depends on the size of the extension mentioned above.}  $\dg>0$ while the size of all other  linearized  Ricci and curvature coefficients  is controlled  by  a second small constant\footnote{Depending on the size of the initial data.} $\epg>0$ with   $\dg\leq \epg$.  We also    control  the coefficients  
    of  the  spacetime metric  in adapted  coordinate charts\footnote{See Lemma \ref{Lemma:geodesic-coordinates} for details.}  $(u, s, y^1, y^2) $.
    
  Given a sphere  $ \ovS=S(\ovu, \ovs)$ of this background foliation of $\RR$, we look for  a   $O(\dg)$  deformation of  it, i.e a map $\Psi:\ovS\longrightarrow \S $   of the form
   \bea\lab{eq:definitionofthedeformationmapPsiinintroduction}
 \Psi(\ovu, \ovs,  y^1, y^2)=\left( \ovu+ U(y^1, y^2 ), \, \ovs+S(y^1, y^2 ), y^1, y^2  \right)
 \eea
  with $(U, S)$ smooth functions on $\ovS$, vanishing at a fixed point of $\ovS$,  of size  proportional to the small constant  $\dg$. The   goal  is then to show that there exist   spheres $\S$, described by the  functions $(U, S) $,  and   adapted   null pairs $(e_3^\S, e_4^\S)$    such that\footnote{It needs recalling that in reality  we  only impose these conditions  for the $\ell\ge 2$ modes of $\kab$ and $\mu$.}$^{,}$\footnote{While \eqref{Introd:GCM spheres-S} corresponds to prescribing the Schwarzschild values, note that such spheres also exists in Kerr for a sufficiently large $r$, see Corollary \ref{cor:ExistenceGCMS1inKerr}.}
\bea
 \lab{Introd:GCM spheres-S}
 \ka^\S=\frac{2}{r^\S}  , \qquad \kab^\S=-\frac{2\Up^\S}{r^\S}, \qquad \mu^\S=\frac{2m^\S}{(r^\S)^3},
 \eea
where $r^\S$ is the  area radius of $\S$,  $m^\S$ is the Hawing mass of $\S$ and  $\Up^\S= 1-\frac{2m^\S}{r^\S}$.  Note that, given such a deformation, at any point on $\S$ we  have two  different null frames: the null frame $( e_3, e_4, e_1, e_2)$ of the background foliation  of $\RR$ and the  null frame  $( e^\S_3, e^\S_4, e^\S_1, e^\S_2)$.  In general, two  null frames   $(e_3, e_4, e_1, e_2)$ and $(e_3', e_4', e_1', e_2') $ are related by  a  frame transformation of the form, see Lemma \ref{Lemma:Generalframetransf}, 
 \bea
 \lab{eq:Generalframetransf-intro}
 \bsplit
  e_4'&=\la\left(e_4 + f^b  e_b +\frac 1 4 |f|^2  e_3\right),\\
  e_a'&= \left(\de_{ab} +\frac{1}{2}\fb_af_b\right) e_b +\frac 1 2  \fb_a  e_4 +\left(\frac 1 2 f_a +\frac{1}{8}|f|^2\fb_a\right)   e_3,\\
 e_3'&=\la^{-1}\left( \left(1+\frac{1}{2}f\c\fb  +\frac{1}{16} |f|^2  |\fb|^2\right) e_3 + \left(\fb^b+\frac 1 4 |\fb|^2f^b\right) e_b  + \frac 1 4 |\fb|^2 e_4 \right),
 \end{split}
 \eea
  where the scalar $\la$    and the 1-forms $f$ and $\fb$      are    called the transition coefficients of the transformation\footnote{The dot product and magnitude  $|\c |$ are taken with respect to the standard euclidian norm of $\RRR^2$.}. One can then  relate all  Ricci and curvature  coefficients  of the primed frame in terms of  the Ricci and curvature coefficients of the  un-primed one, see Proposition \ref{Prop:transformation-formulas-generalcasewithoutassumptions}.    In particular, the GCM  conditions \eqref{Introd:GCM spheres-S}   can be expressed in terms of differential  conditions for    the transition coefficients $(f, \fb, \la) $.    The condition that  the horizontal  part of the frame $(e'_1, e_2')$  is tangent to $\S$  also leads to a relation  between the gradients of $U, S$, defined in \eqref{eq:definitionofthedeformationmapPsiinintroduction},  and $(f, \fb) $.  Roughly we thus expect   to derive a  coupled system of the form
   \bea
 \lab{Compatibility-Deformation2-intro}
 \bsplit
 \pr_{y^a} S&=  \Big( \big(\SS(f, \fb, \Ga)\big)^\#_a \Big), \quad a=1,2,\\
 \pr_{y^a}  U&=\Big(\big(\UU(f, \fb, \Ga)\big)^\#_a\Big),\quad a=1,2,\\
  \DD^\S(f, \fb, \ovla) &=\GG(\Ga)+\HH(f, \fb, \ovla, \Ga),
 \end{split}
   \eea
  where   the terms  $\SS, \UU, \HH, \GG, \DD^\S$ have the following meaning.
   \begin{enumerate}
   \item  The expressions        $\SS(f, \fb, \Ga), \,  \UU(f, \fb, \Ga)$ are  1-forms depending on $f, \fb$ and $\Ga$,  with $\Ga$ denoting the Ricci coefficients  of the background foliation of $\RR$ and   with  $\#$ denoting  the pull back  by the map $\Psi$ defined in \eqref{eq:definitionofthedeformationmapPsiinintroduction}.
   \item    The expression $\HH$ refers to a system of scalar functions  on $\S$   depending on $( f, \fb, \ovla) $  and $\Ga$, where $ \ovla=\la-1$.
   \item The expressions    $(\UU, \SS)$  and   $\HH$ satisfy,  schematically, the following.
 \beaa
 \big|\SS, \UU  \big| \les  \big|(f, \fb) \big| +  \big|(f, \fb) \big|^2, \qquad \big |\HH \big|\les \big(r^{-1}+  \epg \big) \big|(f, \fb, \ovla)  \big| +   \big|(f, \fb, \ovla ) \big| ^2.
  \eeaa
  \item The  expression $\DD^\S$  denotes   a linear    differential  operator  on $\S$.
  \item The term  $ \GG(\Ga)$  denotes a system of scalars  involving the GCM  quantities for the $\RR$-foliation appearing in \eqref{Introd:GCM spheres-small}.
   \end{enumerate}

The construction of  a GCM sphere can thus be reduced to the problem of finding   solutions $ (U, S, f, \fb, \ovla)$  to the system \eqref{Compatibility-Deformation2-intro}   of size $\dg$.   There are however  various difficulties in solving  \eqref{Compatibility-Deformation2-intro}  which we emphasize below.

%%%%%%%%%%%%%%%% 
  
\subsubsection{Integrability}    

%%%%%%%%%%%%%%%%

Note  that  the transition coefficients  have in fact  five degrees of freedom while   \eqref{Introd:GCM spheres-S} provides us with only three scalar constraints. The additional degrees of freedom of the triplet $(f, \fb, \la)$  have to be constrained by integrability conditions, that is 
    integrability in the sense of Frobenius. 
     Indeed,   since  the  vectorfields   $(e'_1, e_2')$   have to be tangent to the sphere $\S$,   the distribution generated  by them  has to be integrable\footnote{Recall that a distribution generated by linearly independent  vectorfields  $X, Y$ is  integrable if  the commutator $[X,Y]$  belongs to the distribution.}, see  a more detailed  discussion in section \ref{sec:defnullpairandhorizontalstructurre}.   Given an arbitrary frame  $(e_1', e_2', e_3', e_4')$, related to the background  frame  $(e_1, e_2, e_3, e_4)$   by  the formula  \eqref{eq:Generalframetransf-intro}, the lack of     of integrability of the  distribution generated by $(e_1', e_2') $  translates into  lack  of symmetry  for  the  null second fundamental forms,
      \beaa
   \chi'_{ab} &=& \g(\nab_{e_a'} e'_4, e'_b), \qquad \chib'_{ab} = \g(\nab_{e_a'} e'_3, e_b'),
   \eeaa
   which     can be measured by  the  scalar functions\footnote{See  precise definitions   in  section \ref{sec:defnullpairandhorizontalstructurre}.},
   \beaa
   \atrch'=\in^{ab}   \chi'_{ab}, \qquad  \atrchb'=\in^{ab}   \chib'_{ab}.
   \eeaa
      We note that in  the axial polarized   situation of \cite{KS}, we can always   choose the primed  frame  $(e_3', e_4', e_1', e_2') $ such that  $e_2'$ is  collinear to the axially symmetric Killing vectorfield   $\Z$  and 
    all other elements of the frame commute  with $\Z$. This automatically  ensures the  integrability of the frame without any additional conditions.

    To deal with the issue of integrability,  in the  general  case,  we are    led to add two more  conditions  to \eqref{Introd:GCM spheres-S}
   \bea
   \atrch^\S=\atrchb^\S=0,
   \eea
translating  into two  additional  differential  relations  for $f, \fb$  which  can be incorporated in the definition of  $\DD^\S$ above.  This provides us with the  correct  number of  equations  in the last row of \eqref{Compatibility-Deformation2-intro}, but,  as we discuss below, it  does  not ensure  that  the kernel of $\DD^\S$  is trivial which would be a necessary condition for solvability.

%%%%%%%%%%%%%%%%%%%%%%%%% 
  
\subsubsection{Non-triviality  of $\ker \DD^\S$}  

%%%%%%%%%%%%%%%%%%%%%%%%%

Upon inspection, the  linear operator $\DD^\S$, though elliptic, has a non-trivial kernel.  To circumvent this difficulty we need to modify the  conditions \eqref{Introd:GCM spheres-S} by  requiring instead that  only the $\ell\ge 2$ modes\footnote{We refer here to a generalization of the   spherical harmonics   of the   standard sphere $\SSS^2$.  This   is itself an additional  difficulty one has to overcome, i.e. to  define a suitable  generalization of modes  for  deformed spheres.} of  $\trchb^\S +\frac{2\Up^\S}{r^\S} $ and    $\mu^\S-\frac{2m^\S}{(r^\S)^3} $  are set to vanish.   As a consequence,  we have the freedom to fix the $\ell=1$ modes of $f, \fb$. These modifications allow us to assume that $\DD^\S$ is both elliptic and  coercive.

%%%%%%%%%%%%%%%% 
  
\subsubsection{Solvability}   

%%%%%%%%%%%%%%%%

Note that   the first two equations   in  \eqref{Compatibility-Deformation2-intro} require a compatibility condition i.e. 
   $$ \pr_{y^b} \Big( \SS(f, \fb, \Ga)\Big)^\#_a = \pr_{y^a} \Big( \SS(f, \fb, \Ga) \Big)^\#_b.$$  
   In the  axial polarized case,   this can be  avoided
    by   a simple symmetry  reduction argument, but in the general  case, this  becomes an issue.  We deal with it by modifying the first two equations  in   \eqref{Compatibility-Deformation2-intro}, i.e. we consider instead the system\footnote{Note that the equations for $(U,S)$ in \eqref{Compatibility-Deformation3-intro} do not imply the ones in \eqref{Compatibility-Deformation2-intro}. It is thus a priori not clear that solving \eqref{Compatibility-Deformation3-intro} will lead to a GCM sphere. The fact that it does is discussed in section \ref{sec:haveweproducedasphere-intro}.}
     \bea
 \lab{Compatibility-Deformation3-intro}
 \bsplit
 \lap^{\ovS} S&= \div^{\ovS} \Big( \big(\SS(f, \fb, \Ga)\big)^\# \Big),\\
  \lap^{\ovS} U&=  \div^{\ovS}\Big(\big(\UU(f, \fb, \Ga)\big)^\#\Big),\\
  \DD^\S(f, \fb, \ovla) &=\GG(\Ga)+\HH(f, \fb, \ovla, \Ga).
 \end{split}
   \eea
   We also fix the values  of  $U, S$ to be zero  at  a given point of $\ovS$  to ensure uniqueness.

%%%%%%%%%%%%%%%%%%%%%%%%%%%%%%%%%%%%%%%%%%%% 
  
\subsubsection{Nonlinear implicit  nature of    \eqref{Compatibility-Deformation3-intro}}   

%%%%%%%%%%%%%%%%%%%%%%%%%%%%%%%%%%%%%%%%%%%%%

To disentangle 
     the highly nonlinear and implicit  nature of  \eqref{Compatibility-Deformation3-intro}, we  proceed 
     by an iterative procedure which starts with   the  trivial quintet
  \beaa
  \QQ^{(0)}:=(U^{(0)}, S^{(0)}, \ovla^{(0)},  f^{(0)},  \fb^{(0)})=\big(0,0,0,0,0\big), 
  \eeaa
    corresponding to the un-deformed sphere $\ovS$,   and, making us of the $n$-th iterate  $\QQ^{(n)}$,    produces
\beaa
 \QQ^{(n+1)}= \Big(U^{(n+1)}, S^{(n+1)}, \ovla^{(n+1)},  f^{(n+1)},  \fb^{(n+1)}\Big)
 \eeaa
   as follows.
  \begin{itemize}
  \item  The pair  $(U^{(n)}, S^{(n)})$ defines the deformation  sphere $\S(n)$ and the corresponding pull back map  $\#_n$ given by the map $\Psi^{(n)} :\ovS\longrightarrow \S(n)$,
  \beaa
  (\ovu, \ovs,y^1, y^2) \longrightarrow (\ovu+U^{(n)}(y^1, y^2), \ovs+S^{(n)}(y^1, y^2),  y^1, y^2).
  \eeaa
\item  We  define the triplet $(\fnn, \fbnn, \ovlann)$ as the  solution of the following  linear  system
\beaa
 \DD^{\S(n)} (\fnn, \fbnn, \ovlann)&=\GG(\Ga)+\HH(\fn, \fbn, \ovlan, \Ga).
 \eeaa
 Note that   $  \DD^{\S(n)}$  is  defined  with respect to the geometric structure of $\S(n)$.
 \item   We use  the new pair     $( f^{(n+1)}, \fb^{(n+1)})$   to solve the equations on $\ovS$,
\bea
\lab{systemUU-SS-derivedn+1}
\bsplit
\Delta^{\ovS} \Unn&=\div^{\ovS} \Big(\big(\UU(\fnn, \fbnn , \Ga)\big)^{\#_{n}}\Big),\\
\Delta^{\ovS} \Snn&=\div^{\ovS}\Big(\big(\SS(\fnn, \fbnn , \Ga)\big)^{\#_{n}}\Big),
\end{split}
\eea
with $ \Unn, \Snn$ vanishing at the same given point of $\ovS$ and 
where the pull back  $\#_{n}$ is defined with respect to the  map $\Psi^{(n)} :\ovS\longrightarrow \S(n)$.  The new  pair  $( U^{(n+1)} ,  S^{(n+1)}  )$ defines  the  new sphere $\S(n+1)$ and we can proceed with the next step  of the iteration. 
  \end{itemize}

%%%%%%%%%%%%%%%%%%%%%%%%%%%% 
  
\subsubsection{Have we produced a GCM sphere?}
\lab{sec:haveweproducedasphere-intro}

%%%%%%%%%%%%%%%%%%%%%%%%%%%%

    If  $ \epg$  is sufficiently small one  can show  that  the iterative  procedure  mentioned above  leads to a solution 
     $\big(U^{(\infty)}, S^{(\infty)}, \ovla^{(\infty)}, f^{(\infty)}, \fb^{(\infty)}\big)$   verifying  the system
     
    \bea
\lab{systemUU-SS-derivedinfty-intro}
\bsplit
\lapzero \Uinfty&=\divzero \Big(\big(\UU(\finfty, \fbinfty , \Ga)\big)^{\#_{\infty}}\Big),\\
\lapzero \Sinfty&=\divzero \Big(\big(\SS(\finfty, \fbinfty , \Ga)\big)^{\#_{\infty}}\Big),\\
\DD^{\infty} (\finfty, \fbinfty ,\ovlainfty)&=  \GG(\Ga)+\HH(\finfty, \fbinfty ,\ovlainfty,  \Ga),
\end{split}
\eea
where the elliptic operator $\DD^{\infty} $ is defined on the sphere  $\S(\infty)$, i.e.    the deformation of $\ovS$ induced by 
  $( \Uinfty,  \Sinfty)$. Is  $\S(\infty)$  the desired   solution to the problem, i.e. is it a GCM  sphere in the sense  discussed above? This is a priori not clear as the equations for $(\Uinfty, \Sinfty)$ in \eqref{systemUU-SS-derivedinfty-intro} do not imply those in \eqref{Compatibility-Deformation2-intro}. As a result,  we have potentially two different   frames  associated to $\S=\S(\infty)$.
\begin{itemize}
\item The frame  $\big(\einfty_1, \einfty_2, \einfty_3, \einfty_4\big)$ induced by
 the transition functions $(\ovlainfty, \finfty, \fbinfty )$,  with  the quintet  $\big(\Uinfty, \Sinfty, \ovlainfty, \finfty, \fbinfty \big)$  verifying the limiting  system  \eqref{systemUU-SS-derivedinfty-intro}.
  
\item  The geometric frame\footnote{With a proper normalization for  the null pair $e_3 ^\S, e_4^\S$, in fact  the one corresponding  to $\la=\la^{(\infty)}$.} $(e_1^\S, e_2^\S, e_3^\S, e_4^\S)$,  adapted to $\S$.
\end{itemize}
    The  main remaining  hurdle is to show that these two null frames coincide, see section \ref{sec:wheretheproofofthemaintheoremisfinallyconcluded},  so that $\S$ is indeed the desired GCM deformation.

%%%%%%%%%%%%%%%%%%%%%%%%%%%%%
    
    \subsection{First version of the main theorem}

%%%%%%%%%%%%%%%%%%%%%%%%%%%%%
    
    We give below a  bare bones version of our main theorem, see Theorem \ref{Theorem:ExistenceGCMS1} for the precise version.
  \begin{theorem}[Existence of GCM spheres, version 1]
\lab{Theorem:ExistenceGCMS1-intro}
Let  $\RR$   be  fixed  spacetime region, endowed with an  outgoing geodesic   foliation  $S(u, s)$,   verifying 
 specific asymptotic assumptions  expressed  in terms  of two parameters $0<\dg\leq \epg$.  In particular we assume that
 the  GCM quantities\footnote{This   requires  a careful definition of modes, i.e. analogues of the spherical harmonics.}
   \bea
  \lab{Introd:GCM spheres-small-again}
 \ka-\frac 2 r , \qquad \left( \kab+\frac{2\Up}{r}\right)_{\ell\ge 2 }, \qquad \left( \mu- \frac{2m}{r^3} \right)_{\ell\ge 2},
  \eea 
  are small with respect to the parameter $\dg$.
 Let  $\ovS=S(\ovu, \ovs)$   be  a fixed    sphere of the foliation with  $\rg$ and $\mg$ denoting respectively its area radius and  Hawking mass, with $\rg$ sufficiently large. 
 Then, 
for any fixed triplets   $\La, \Lab \in \RRR^3$  verifying
\bea
|\La|,\,  |\Lab|  &\les & \dg,
\eea
 there 
exists a unique  GCM sphere $\S=\S^{(\La, \Lab)}$, which is a deformation of $\ovS$, 
such that 
 \bea
  \lab{Introd:GCM spheres-small-again}
 \ka^\S -\frac{2}{r^\S} =0, \qquad \left( \kab^\S+\frac{2\Up^\S}{r^\S}\right)_{\ell\ge 2 }=0, \qquad \left( \mu^\S - \frac{2m^\S}{(r^\S)^3} \right)_{\ell\ge 2}=0,
  \eea 
  and
  \bea
  \lab{Introd:GCM spheres-LaLab}
  (\div^\S f)_{\ell=1}=\La, \qquad (\div^\S\fb)_{\ell=1}=\Lab,
  \eea
  where $(f, \fb, \la)$ denote  the transition coefficients of the transformation \eqref{eq:Generalframetransf-intro}  from the background frame of $\RR$ to the frame adapted to $\S$.  
    \end{theorem}
 
 \begin{remark}
 We emphasize again that, unlike the GCM construction in \cite{KS}, Theorem \ref{Theorem:ExistenceGCMS1-intro} does not rely on any symmetry assumption and can thus be used in a general setting.
\end{remark}

\begin{remark}
Note that there is an obvious ambiguity in the statements  \eqref{Introd:GCM spheres-small-again} \eqref{Introd:GCM spheres-LaLab} of Theorem \ref{Theorem:ExistenceGCMS1-intro} due to the   arbitrariness in the choice of the $\ell=1$ modes on $\S$. We will remove this  ambiguity in \cite{KS:Kerr2} where we  show  that the results  of  Theorem  \ref{Theorem:ExistenceGCMS1-intro} hold true for a canonical  basis of ${\ell=1} $ modes on $\S$ based on an effective version of the classical  uniformization theorem. 
\end{remark}
\begin{remark}
The assumptions on the spacetime region $\RR$ in Theorem \ref{Theorem:ExistenceGCMS1-intro} are in particular satisfied in Kerr for $r$ sufficiency large, see Lemma \ref{lemma:controlfarspacetimeregionKerrassumptionRR}. We can thus apply Theorem \ref{Theorem:ExistenceGCMS1-intro} in that context, and obtain the existence of GCM spheres $\S^{(\La, \Lab)}_{Kerr}$ in Kerr for $r$ sufficiency large, see Corollary \ref{cor:ExistenceGCMS1inKerr}. The GCM spheres  $\S^{(\La, \Lab)}$ of Theorem \ref{Theorem:ExistenceGCMS1-intro} thus correspond to the analog of $\S^{(\La, \Lab)}_{Kerr}$ in perturbations of Kerr for $r$ sufficiency large.
 \end{remark}
    
\begin{remark}
 We note  that a  related notion  of preferred spheres\footnote{The spheres in   \cite{HuYau}  have  constant mean curvature}  in an asymptotically euclidean Riemannian 3-manifold  has been  introduced in \cite{HuYau}. In contrast with our work  here    the spheres  in   \cite{HuYau}   have  codimension 1,  while  ours  have  codimension 2 in a 4 dimensional Lorentzian manifold.   
\end{remark}

%%%%%%%%%%%%%%%%%%%%%
    
\subsection{Structure of the paper}

%%%%%%%%%%%%%%%%%%%%%

The structure of the paper is as follows 
\begin{itemize}
\item In section \ref{sec:geometricsetup}, we introduce the geometric set-up and provide our main assumptions for the background foliation of the spacetime region $\RR$. 

\item In section \ref{sec:frametransformation}, we introduce general frame transformations, including the frame transformations for the main GCM quantities. 

\item In section \ref{section:GCMspheres}, we provide the definition of GCM spheres. In particular, we derive the elliptic system for the transition coefficients $(f, \fb, \la)$, and we analyze the corresponding linearized system.

\item In section \ref{sec:deformationofsurfaces}, we study deformations of the background spheres of $\RR$, and derive in particular the equations for the scalar functions $(U,S)$ defining the deformation. 

\item In section \ref{sec:exitstenceofGCMspheres}, we prove Theorem \ref{Theorem:ExistenceGCMS1-intro} on the existence of GCM spheres by relying on an iterative scheme. We also obtain the existence of GCM spheres in Kerr as a by-product. 
\end{itemize}

%%%%%%%%%%%%%%%%%%%%%%%%%%%%%%%%%%

\section{Geometric set up}\lab{sec:geometricsetup}

%%%%%%%%%%%%%%%%%%%%%%%%%%%%%%%%%%

%%%%%%%%%%%%%%%%%%%%%%%%%%%%%%%%%%%%%%%%%%%%%

\subsection{General formalism}

%%%%%%%%%%%%%%%%%%%%%%%%%%%%%%%%%%%%%%%%%%%%%

We  review the part relevant for this paper of the general formalism we have introduced in \cite{GeKSz}. 

%%%%%%%%%%%%%%%%%%%%%%%%%%

\subsubsection{Null pairs and horizontal structures}
\lab{sec:defnullpairandhorizontalstructurre}

%%%%%%%%%%%%%%%%%%%%%%%%%%

Let $(\MM, \g)$ a Lorentzian space-time. Consider a fixed    null pair $e_3, e_4$, i.e. 
\beaa
\g(e_3, e_3)=\g(e_4, e_4)=0, \qquad  \g(e_3, e_4)=-2, 
\eeaa
and denote  by  $\O(\MM)$ the vector space  of horizontal vectorfields $X$  on $\MM$, i.e.  $\g(e_3, X)= \g(e_4, X)=0$.
  Given a fixed   orientation  on $\MM$,  with corresponding  volume form  $\in$,  we define  the induced 
 volume form on   $\O(\MM)$ by,  
 \beaa
  \in(X, Y):=\frac 1 2\in(X, Y, e_3, e_4). 
  \eeaa
 A null  frame on $\MM$ consists of a choice of horizontal vectorfields  $e_1, e_2$, such that
 \beaa
 \g(e_a, e_b)=\de_{ab}\qquad  a, b=1,2.
 \eeaa  
 The commutator $[X,Y]$ of two horizontal vectorfields
may fail however to be horizontal. We say that the pair $(e_3, e_4 )$ is integrable if   $\O(\MM)$  forms an integrable distribution,
i.e. $X, Y\in\O(\MM) $ implies that $[X,Y]\in\O(\MM)$. As  it is well-known,  the  principal null pair in Kerr fails to be integrable, 
see also Remark \ref{rem:nonintegrabilityandatrchatrchb}. Given an arbitrary vectorfield $X$ we denote by $^{(h)}X$
its  horizontal projection, 
$$^{(h)}X=X+ \frac 1 2 g(X,e_3)e_4+ \frac 1 2   g(X,4) e_3.$$ 
A  $k$-covariant tensor-field $U$ is said to be horizontal,  $U\in \O_k(\MM)$,
if  for any $X_1,\ldots X_k$ we have $U(X_1,\ldots X_k)=U( ^{(h)} X_1,\ldots  ^{(h)} X_k)$.

\begin{definition}
We denote by $\SS_0=\SS_0(\MM)$ the set of scalar functions on $\MM$, $\SS_1=\SS_1(\MM)$ the  set of horizontal $1$-forms  on $\MM$, and by $\SS_2=\SS_2(\MM)$
  the set of symmetric traceless   horizontal $2$-forms on $\MM$.
  \end{definition}

For any $ X, Y\in \O(\MM)$ we define  the induced metric $g(X, Y)=\g(X, Y)$ and the null second fundamental forms
\bea
\chi(X,Y)=g(\D_Xe_4 ,Y), \qquad \chib(X,Y)=g(\D_X\Lb,Y).
\eea
Observe that    $\chi$
 and $\chib$  are  symmetric if and 
 only if   the horizontal structure is 
 integrable. Indeed this follows easily from
 the formulas,
 \beaa
 \chi(X,Y)-\chi(Y,X)&=&\g(\D_X e_4, Y)-\g(\D_Ye_4,X)=-\g(e_4, [X,Y]),\\
 \chib(X,Y)-\chib(Y,X)&=&\g(\D_X e_3, Y)-\g(\D_Ye_3,X)=-\g(e_3, [X,Y]).
\eeaa
  Note  that  we  can view  $\chi$ and $\chib$ as horizontal 2-covariant tensor-fields
 by extending their definition to arbitrary vectorfields  $X, Y$  by setting  $\chi(X, Y)= \chi( ^{(h)}X, ^{(h)}Y)$,  $\chib(X, Y)= \chib( ^{(h)}X, ^{(h)}Y)$. We define their trace $\trch$, $\trchb$,  and anti-trace $\atrch$, $\atrchb$ as follows
\beaa
\trch:=\de^{ab}\chi_{ab}, \qquad \trchb:=\de^{ab}\chib_{ab}, \qquad \atrch:=\in^{ab}\chi_{ab}, \qquad \atrchb:=\in^{ab}\chib_{ab}.
\eeaa
Accordingly we  decompose $\chi, \chib$ as follows,
\beaa
\chi_{ab}&=&\chih_{ab} +\frac 1 2 \de_{ab} \trch+\frac 1 2 \in_{ab}\atrch,\\
\chib_{ab}&=&\chibh_{ab} +\frac 1 2 \de_{ab} \trchb+\frac 1 2 \in_{ab}\atrchb.
\eeaa

\begin{remark}\lab{rem:nonintegrabilityandatrchatrchb}
The non integrability of $(e_3, e_4)$  corresponds to non trivial $\atrch$ and $\atrchb$. A celebrated example of  a non integrable null frame  is the principal null frame of Kerr for which $\atrch$ and $\atrchb$ are indeed non trivial.
\end{remark}

We define the horizontal covariant operator $\nab$ as follows:
 \bea
 \nab_X Y&:=&^{(h)}(\D_XY)=\D_XY- \frac 1 2 \chib(X,Y)e_4 -  \frac 1 2 \chi(X,Y) e_3, \quad X, Y\in \O(\MM).
 \eea
Note that,
\beaa
 \nab_X Y-\nab_Y X   &=&[X, Y]-\frac 1 2 (\atrchb\,  L+\atrch \, \Lb)\in(X, Y).
\eeaa
In particular,
 \bea
 [X, Y]^\perp&=&\frac 1 2(\atrchb\,  L+\atrch \, \Lb)\in(X, Y).
 \eea
 Also, for  all  $X,Y, Z\in \O(\MM)$,
 \beaa
 Z g (X,Y)=g(\nab_Z X, Y)+ g(X, \nab_ZY).
 \eeaa

 \begin{remark} 
 In the integrable case, $\nab$ coincides with the Levi-Civita connection
 of the metric induced on the integral surfaces of   $\O(\MM)$.  
 Given $X$ horizontal, $\D_4X$ and $\D_3 X$ are in general not horizontal.
 We define $\nab_4 X$ and $\nab_3 X$  to be the horizontal projections
 of the former.  More precisely,
 \beaa
 \nab_4 X&:=&^{(h)}(\D_4 X)=\D_4 X- \frac 1 2 g(X, \D_4 e_3 ) e_4- \frac 1 2  g(X, \D_4 e_4)  e_3 ,\\
 \nab_3 X&:=&^{(h)}(\D_3 X)=\D_3 X-   \frac 1 2 g(X, \D_3e_3) e_3 - \frac 1 2   g(X, \D_3 e_4 ) e_3. 
 \eeaa
The definition can be easily extended to arbitrary  $  \O_k(\MM) $ tensor-fields  $U$ 
\beaa
 \nab_4U(X_1,\ldots, X_k)&=&e_4 U(X_1,\ldots, X_k))- \sum_i U( X_1,\ldots, \nab_4 X_i, \ldots X_k),\\
  \nab_3 U(X_1,\ldots, X_k)&=&e_3 (U(X_1,\ldots, X_k)) -\sum_i U( X_1,\ldots, \nab_3 X_i, \ldots X_k).
 \eeaa 
 \end{remark}

%%%%%%%%%%%%%%%%%%%%%%%%%%%%%%%%%%%%%%%%%%%%%

\subsubsection{Ricci and curvature  coefficients}

%%%%%%%%%%%%%%%%%%%%%%%%%%%%%%%%%%%%%%%%%%%%%

Given a null frame $e_1, e_2, e_3, e_4$ we define  the     connection coefficients as follows 
 \bea
 \begin{split}
\chib_{ab}&=g(\D_ae_3, e_b),\qquad \,\,\chi_{ab}=g(\D_ae_4, e_b),\\
\xib_a&=\frac 1 2 \g(\D_3 e_3 , e_a),\qquad \xi_a=\frac 1 2 \ g(\D_4 e_4, e_a),\\
\omb&=\frac 1 4 g(\D_3e_3 , e_4),\qquad\,\, \om=\frac 1 4 g(\D_4 e_4, e_3),\qquad \\
\etab_a&=\frac 1 2 (\D_4 e_3, e_a),\qquad \quad \eta_a=\frac 1 2 g(\D_3 e_4, e_a),\qquad\\
 \ze_a&=\frac 1 2 g(\D_{e_a}e_4,  e_3).
 \end{split}
\eea

We have,
\bea
\D_a e_b&=&\nab_a e_b+\frac 1 2 \chi_{ab} e_3+\frac 1 2  \chib_{ab}e_4,\nn\\
\D_a e_4&=&\chi_{ab}e_b -\ze_a e_4,\nn\\
\D_a e_3&=&\chib_{ab} e_b +\ze_ae_3,\nn\\
\D_3 e_a&=&\nab_3 e_a +\eta_a e_3+\xib_a e_4,\nn\\
\D_3 e_3&=& -2\omb e_3+ 2 \xib_b e_b,\label{ricci}\\
\D_3 e_4&=&2\omb e_4+2\eta_b e_b,\nn\\
\D_4 e_a&=&\nab_4 e_a +\etab_a e_4 +\xi_a e_3,\nn\\
\D_4 e_4&=&-2 \om e_4 +2\xi_b e_b,\nn\\
\D_4 e_3&=&2 \om e_3+2\etab_b e_b.\nn
\eea 

For a given horizontal   $1$ -form $\xi$,
we  define the frame independent   operators\footnote{Note that the definition of $\nab\hot $ differs from the given in \cite{GeKSz} by a factor $1/2$.},   
\beaa
\div\xi&=&\de^{ab}\nab_b\xi_a,\qquad 
\curl\xi=\in^{ab}\nab_a\xi_b,\qquad 
(\nab\hot \xi)_{ba}=\nab_b\xi_a+\nab_a  \xi_b-\de_{ab}( \div \xi).
\eeaa
We also define the usual  curvature components, see \cite{Ch-Kl}, 
\beaa
&&\a_{ab}=\R_{a4b4},\qquad \b_a=\frac 12 \R_{a434}, \qquad   \bb_a=\frac 1 2 \R_{a334},  \qquad \aa_{ab}=\R_{a3b3},\\
&&\rho=\frac 1 4 \R_{3434} , \qquad \rhod=\frac 1 4 \R_{3434}.
\eeaa

%%%%%%%%%%%%%%%%%%%%%%%%%%%%%%%%%%%%%%%

\subsection{Outgoing geodesic foliations}
\lab{sec:defoutgoinggeodesicfoliation}

%%%%%%%%%%%%%%%%%%%%%%%%%%%%%%%%%%%%%%%

%%%%%%%%%%%%%%%%%%%%%%%%%%%%%%%%%%%%%%%

\subsubsection{Definition of an outgoing geodesic foliation}

%%%%%%%%%%%%%%%%%%%%%%%%%%%%%%%%%%%%%%%

Assume  given an outgoing  optical function $u$, i.e.   a solution of the equation,
\beaa
\g^{\a\b}\pr_\a u\pr_\b u  =0
\eeaa
and let $L=-g^{ab} \pr_b u \pr_a  $ its  null geodesic generator. We choose $e_4$ such that,
\bea
e_4=\vsi L, \qquad  L(\vsi)=0.
\eea

 We then choose $s$    such that 
 \bea
 e_4 (s)=1.
 \eea
The scalar functions $(u, s)$ generate  what is called an  outgoing  geodesic foliation.  Let $S(u,s)$    be       the    $2$-surfaces  of intersection    between the level surfaces of $u$ and $s$. We  choose  $e_3 $  the unique    null vectorfield orthogonal to  $S(u,s)$   and such that $g(e_3, e_4)=-2$. We  then  let $(e_1, e_2)$  an orthogonal basis of the tangent space of $S(u, s)$.    We also  introduce
 \bea 
 \underline{\Omega} := e_3(s).
 \eea
 
 \begin{lemma}
 \label{le:outgoinggeodesicfoliation}
We have
\beaa
\atrch=\atrchb=0, \qquad \om=\xi=0,  \qquad  \etab = -\ze, \qquad  \vsi=\frac{2}{e_3(u)}.
\eeaa
\end{lemma}

\begin{proof}
Since $(e_1, e_2)$ is a basis of the tangent space of $S(u, s)$, it is integrable, and hence 
\beaa
\atrch=\atrchb=0.
\eeaa 

Next, recall that $L$ is geodesic, $e_4=\vsi L$ and $L(\vsi)=0$. This immediately implies that $e_4$  is geodesic, and hence we have 
\beaa
\om=\xi=0.
\eeaa
Also, applying the vectorfield  
\beaa
[e_4, e_a]=(\etab_a+\ze_a) e_4+\xi_a e_3 -\chi_{ab} e_b
\eeaa
 to $s$, and since $e_4(s)=1$ and $e_a(s)=0$, we derive,
 \beaa
 \etab+\ze=0.
 \eeaa 
 
Finally, note that
\beaa
e_3(u)=g(e_3, -L)=-\vsi^{-1}  g(e_3,e_4)=\frac{2}{\vsi}
\eeaa
and hence
\beaa
\vsi &=& \frac{2}{e_3(u)}
\eeaa
 as desired. 
\end{proof}

We  define    the  area radius $r(u,s)$ of $S(u,s)$   by the formula
\bea
|S|=4\pi r^2
\eea
where $|S|$ is the volume of the surface $S$. Also, the Hawking  mass  $m=m(u,s)$  of $S(u,s)$  is    defined by the formula, 
  \bea
\frac{2m}{r}=1+\frac{1}{16\pi}\int_{S_{}}\trch \trchb.
\eea
The Gauss curvature of $S$ is denoted by $K$ and satisfies from the Gauss equation,
\bea\lab{eq:Gaussequation}
K=-\rho -\frac{1}{4} \trch \trchb +\frac{1}{2}\chih\c\chibh.
\eea
Finally, we define     the mass aspect function $\mu$ as follows
\bea
\label{def:massaspectfunctions.general}
\mu:&=& - \div \ze -\rho+\frac 1 2  \chih\c \chibh.
\eea

%%%%%%%%%%%%%%%%%%%%%%%%%%%%%%%%%%%%

\subsubsection{Coordinates adapted to an outgoing geodesic foliation}

%%%%%%%%%%%%%%%%%%%%%%%%%%%%%%%%%%%%%

\begin{definition}
 A coordinate system $(u, s, y^1, y^2)$ is  said to be adapted to  an outgoing geodesic foliation  on $\MM$  as above if
 \bea
 e_4(y^1)= e_4(y^2)=0.
 \eea
\end{definition}

\begin{lemma}
\lab{Lemma:geodesic-coordinates}
Given a  coordinates system $(u, s, y^1, y^2)$ adapted to a geodesic foliation as above  the following hold true.
\begin{enumerate}
\item The spacetime metric takes the form
\bea
\lab{spacetimemetric-y-coordinates}
\g &=& - 2\vsi du ds + \vsi^2\Omb  du^2 +g_{ab}\big( dy^a- \vsi \undB^a du\big) \big( dy^b-\vsi \undB^b du\big),
\eea
where
\bea
\Omb=e_3(s), \qquad \undB^a =\frac{1}{2} e_3(y^a), \qquad  g_{ab}=\g(\pr_{y^a}, \pr_{y^b}).
\eea

\item  The null pair $(e_3, e_4) $  take the form
\bea
e_4=\pr_s, \qquad \pr_u = \vsi\left(\frac{1}{2}e_3-\frac{1}{2}\Omb e_4-\undB^a\pr_{y^a}\right).
\eea
Moreover 
\bea
\pr_{y^a}= \sum_{c=1,2} Y_{(a)}^c e_c, \qquad   a=1,2,
\eea
with coefficients  $ Y_{(a)}^b $ verifying
\bea
g_{ab}=\sum_{c=1, 2} Y_{(a)}^c Y_{(b)}^c.
\eea
We also write, 
\bea
\bsplit
\pr_s &= e_4,\\
\pr_u &= \vsi \left(\frac{1}{2}e_3-\frac{1}{2}\Omb e_4 -  \sum_{c=1,2}  Z^c  e_c\right),\\
\pr_{y^a}&= \sum_{c=1,2} Y_{(a)}^c e_c, \qquad   a=1,2,
\end{split}
\eea
where 
\bea
Z^c:=\Bb^a  Y_{(a)}^c.
\eea

\item We have 
\bea\lab{eq:propagationequationforundBandgab}
e_4(\undB^a) = -(\eta+\ze)\c\nab(y^a), \qquad \pr_s  g_{ab} &=& 2\chi\left(\pr_{y^a}, \pr_{y^b}\right).
\eea 
\end{enumerate}
\end{lemma}

\begin{proof}
Since $u$ is an optical function, we  deduce
\beaa
0&=&\g^{uu}\pr_u u \pr_u u+\g^{ui} \pr_u u \pr_{y^i}  u+\g^{ij} \pr_{y^i} u  \pr_{y^j} u= \g^{uu}.
\eeaa
 Thus 
\beaa
L=-   \g^{us} \pr_s -   \g^{ua} \pr_{y^a}, \qquad e_4= -   \vsi\g^{us} \pr_s -   \vsi\g^{ua} \pr_{y^a}.
\eeaa
Since $e_4(y^1)=e_4(y^2)=0$ we deduce
\beaa
\g^{u1}=\g^{u2}=0.
\eeaa
Thus $e_4=  -\vsi\g^{us} \pr_s $
and since $e_4(s)=1$ we  deduce,
\beaa
\g^{us}=-\frac{1}{\vsi}, \qquad e_4=\pr_s. 
\eeaa
Since,
\beaa
0&=&\g^{u u} \g_{us}+ \g^{us}\g_{ss}+\g^{u 1} \g_{1s}+ \g^{u2} \g_{2s}= \g^{us}\g_{ss},\\
0&=&\g^{u u} \g_{ua}+ \g^{us}\g_{sa}+\g^{u 1} \g_{1a}+ \g^{u2} \g_{2a}= \g^{us}\g_{sa},\\
1&=&\g^{u u} \g_{uu}+ \g^{us}\g_{su}+\g^{u 1} \g_{1u}+ \g^{u2} \g_{2u}= \g^{us}\g_{su},
\eeaa
we deduce
\beaa
\g_{ss}=0, \qquad \g_{s1}=\g_{s2}=0, \qquad \g_{us}=-\vsi.
\eeaa
Thus the metric $\g$ can be expressed in the form,
\beaa
\g&=& - 2\vsi du ds + \g_{uu}  du^2 + 2  \g_{u a } du dy^a +\g_{ab} dy^a dy^b.
\eeaa
We introduce  $\undB^a $ by the condition
\beaa
\g_{ua} =-\g_{ab}\vsi \undB^b.
\eeaa
Therefore
\beaa
\g&=& - 2\vsi du ds + \g_{uu}  du^2 - 2 \g_{ab}\vsi \undB^b du dy^a +\g_{ab} dy^a dy^b\\
&=& - 2\vsi du ds + \g_{uu}  du^2 +\g_{ab}\big( dy^a- \undB^a\vsi du\big) \big( dy^b-\undB^b\vsi du\big)-\g_{ab} \undB^a \undB^b\vsi^2  du^2.
\eeaa
Thus the metric  takes the form
\beaa
\g &=& - 2\vsi du ds + \big( \g_{uu} -\g_{ab} \undB^a \undB^b\vsi^2\big)  du^2 +g_{ab}\big( dy^a- \undB^a\vsi du\big) \big( dy^b-\undB^b \vsi du\big)
\eeaa
where $g_{ab}=\g_{ab}=\g(\pr_{y^a}, \pr_{y^b}) $.

Also, note that we have, since $e_4(u)=0$, $e_4(s)=1$, $e_3(u)=2/\vsi$, and $e_3(s)=\Omb$, 
\beaa
\pr_u &=& \frac{\vsi}{2}\Big(e_3-\Omb e_4-e_3(y^1)\pr_{y^1}-e_3(y^2)\pr_{y^2}\Big).
\eeaa
Since  $\pr_{y^1}, \pr_{y^2}$ span the tangent space  to $S(u, s)$  and are thus    perpendicular to  $e_3, e_4$, we  deduce
\beaa
\g_{au}=\g(\pr_u, \pr_{y^a} )=-\frac{\vsi}{2} e_3(y^b)\g_{ab}
\eeaa
and hence
\beaa
\undB^a &=& \frac{1}{2}e_3(y^a). 
\eeaa
In the same vein
\beaa
\g_{uu} &=&\g(\pr_u, \pr_u)= \vsi^2\Omb+\g_{ab}\undB^a\undB^b\vsi^2,
\eeaa
and hence
\beaa
\g_{uu} -\g_{ab} \undB^a \undB^b\vsi^2 &=& \vsi^2\Omb. 
\eeaa
We deduce, as stated,
\beaa
\g &=& - 2\vsi du ds + \vsi^2\Omb  du^2 +g_{ab}\big( dy^a- \undB^a\vsi du\big) \big( dy^b-\undB^b\vsi du\big).
\eeaa
Also, as we have seen $e_4=\pr_s$ and
\beaa
\pr_u &=& \frac{\vsi}{2}\Big(e_3-\Omb e_4-e_3(y^1)\pr_{y^1}-e_3(y^2)\pr_{y^2}\Big)= \vsi\left(\frac{1}{2}e_3-\frac{1}{2}\Omb e_4-\undB^a\pr_{y^a}\right).
\eeaa
On the other hand, since $ \pr_{y^1}, \pr_{y^2}$ span the same space as $e_1, e_2$, we can write
\beaa
\pr_{y^a}= \sum_{c=1,2} Y_{(a)}^c e_c, \qquad   a=1,2.
\eeaa
Since $\g(e_a, e_b)=\de_{ab}$ we deduce,
\beaa
g_{ab}=\g(\pr_{y^a}, \pr_{y^b})=\g\left(  \sum_{c=1,2} Y_{(a)}^c e_c,  \sum_{c=1,2} Y_{(b)}^d e_d\right)=\sum_{c=1, 2} Y_{(a)}^c Y_{(b)}^c
\eeaa
as stated.

Finally, since we have $\undB^a = e_3(y^a)/2$ and $e_4(y^a)=0$, we infer
\beaa
e_4(\undB^a) &=& \frac{1}{2}[e_4, e_3]y^a=\frac{1}{2}\Big( -2\omb e_4 +2(-\eta_b+\etab_b)e_b\Big)(y^a)= -(\eta+\ze)\c\nab(y^a).
\eeaa
Moreover, since $e_4=\pr_s$, we have 
\beaa
\pr_s\g(\pr_{y^a}, \pr_{y^b})&=&\g(\D_{\pr_s}\pr_{y^a}, \pr_{y^b})+\g(\pr_{y^a}, \D_{\pr_s}\pr_{y^b})= \g(\D_{\pr_{y^a}}\pr_s, \pr_{y^b})+\g(\pr_{y^a}, \D_{\pr_{y^b}}\pr_s)\\
&=& 2 \chi(\pr_a, \pr_b).
\eeaa
This concludes the proof of the lemma.
\end{proof}

%%%%%%%%%%%%%%%%%%%%%%%%%%%%%%%%%%%%%%%

\subsubsection{Linearized connection coefficients for geodesic foliations}
\lab{subsection:R-Ga_g-Ga_b}

%%%%%%%%%%%%%%%%%%%%%%%%%%%%%%%%%%%%%%%

We  recall that for an outgoing geodesic foliation  we have, 
\beaa
\atrch=\atrchb=0, \qquad \xi=\om=0, \qquad \etab=-\ze.
\eeaa
We define the following renormalized quantities
\beaa
&&\widecheck{\trch} := \trch -\frac{2}{r}, \qquad
\widecheck{\trchb} := \trchb +\frac{2\Up}{r},\qquad 
\widecheck{\omb} := \omb -\frac{m}{r^2},\\
&&\widecheck{K} := K -\frac{1}{r^2},\qquad\,\,\,\,\widecheck{\rho} := \rho +\frac{2m}{r^3},\qquad \quad\,\,\,
\widecheck{\mu} := \mu -\frac{2m}{r^3}, \\
&&\widecheck{\Omb} :=\Omb+\Up, \qquad\quad 
\widecheck{\varsigma} := \varsigma-1,
\eeaa
where 
\beaa
\Up :=1-\frac{2m}{r}.
\eeaa

We  define  the sets 
\bea
\lab{definition:Ga_gGa_b}
\bsplit
\Ga_g &:= \Bigg\{\widecheck{\trch},\,\, \chih, \,\, \ze, \,\, \widecheck{\trchb},\,\,  r\widecheck{\mu} ,\,\,  r\widecheck{\rho}, \,\, r\dual\rho, \,\, r\b, \,\, r\a, \,\, r\widecheck{K}, \,\, r^{-1} \big(e_4(r)-1\big), \,\, r^{-1}e_4(m)\Bigg\},\\
\Ga_b &:= \Bigg\{\eta, \,\,\chibh, \,\, \ombc, \,\, \xib,\,\,  r\bb, \,\, \aa, \,\, r^{-1}\widecheck{\Omb}, \,\,r^{-1}\widecheck{\varsigma}, \,\, r^{-1}(e_3(r)+\Up\big), \,\, r^{-1}e_3(m)  \Bigg\}.
\end{split}
\eea

%%%%%%%%%%%%%%%%%%%%%%%%%%%%%%%%%%%%%%%

\subsubsection{Norms on 2-spheres and Hodge operators}

%%%%%%%%%%%%%%%%%%%%%%%%%%%%%%%%%%%%%%%

Given a 2-sphere $S(u,s)$ and   $f\in \SS_p(S)$, $p=0,1,2$,    we consider   the following   norms, 
  \bea
  \label{Norms-spacetimefoliation-GSMS}
  \bsplit
  \| f\|_{\infty} :&=\| f\|_{L^\infty(S)}, \qquad  \| f\|_{2} :=\| f\|_{L^2(S)}, \\
  \|f\|_{\infty,k} &= \sum_{i=0}^k \|\dk^i f\|_{\infty },  \qquad 
\|f\|_{2,k}=\sum_{i=0}^k \|\dk^i f\|_{2},
\end{split}
  \eea
  where $\dk^i$ stands for any   combination  of length $i$ of operators  of the form 
   $e_3, r e_4,  r\nab $.

We consider the following Hodge operators acting on  $2$ surface $S$:
\begin{enumerate}
\item The operator $\ddd_1$ takes any $1$-form $f$ into the pairs of
functions $(\div  f  \,,\, \curl f  )$.
\item The operator $\ddd_2$ takes any  2-covariant $S$ tangent symmetric, traceless  tensor $v$ into the $S$ tangent 1-form $\div v$.
\item The operator $\dds_1$ takes the pair of scalar functions 
$(\la, \dual\la)$ into the $S$-tangent 1-form $-\nab\la+\dual\nab\dual\la$.
\item The operator $\dds_2$ takes 1-forms $f$  on $S$ into  the 2-covariant, symmetric,
traceless tensors $-\frac{1}{2}  \nab\hot f$.
\end{enumerate}
Observe that $\dds_1$, resp. $\dds_2$ are  the $L^2$ adjoints of
 $\ddd_1$, respectively $\ddd_2$. 
   
The standard Hodge operators $\ddd_1, \ddd_2$ and their a formal adjoints  $\dds_1, \dds_2$ verify the following identities (see \cite{Ch-Kl} page 38).
\bea
\begin{split}
\label{eq:dcalident}
\dds_1\c\ddd_1&=-\Delta_1+K,\qquad\,\,\,\, \ddd_1 \c\dds_1=-\Delta,\\
\dds_2 \c\ddd_2 &=-\f12\Delta_2+K,\qquad \ddd_2\c\dds_2=-\f12(\Delta_1+K).
\end{split}
\eea

%%%%%%%%%%%%%%%%%%%%%%%%%%%%%%

\subsection{The   far spacetime region  $\RR$}
\lab{sec:defintionspacetimeregionRR}

%%%%%%%%%%%%%%%%%%%%%%%%%%%%%%

In this paper  we consider a spacetime region $\RR$ foliated by two functions $(u,s)$ such that 
\begin{enumerate}
\item On $\RR$, $(u, s)$  is a geodesic foliation  of lapse $\vsi$ as in section \ref{sec:defoutgoinggeodesicfoliation}.

\item We denote by $(e_4, e_3, e_1, e_2)$ the null frame adapted to the outgoing geodesic foliation $(u,s)$ on $\RR$.

\item Let $(\ug, \sg)$ to real numbers. Let
\bea
\ovS &:=& S(\ug, \sg),
\eea
$\rg$ the area radius of $\ovS$, and $\mg$ the Hawking mass of $\ovS$, where $S(u,s)$ denote the 2-spheres of the outgoing geodesic foliation $(u,s)$ on $\RR$.

\item $\RR$ is covered by two coordinates charts $\RR=\RR_N\cup \RR_S$ such that
\begin{enumerate}
\item The North coordinate chart   $\RR_N$ is given by the coordinates
$(u, s, y_{N}^1, y_{N}^2)$ with    $(y^1_{N})^2+(y^2_{N})^2<2$. 

\item The South coordinate chart  $\RR_S$ is  given by the coordinates
$(u, s, y_{S}^1, y_{S}^2)$  with $(y^1_{S})^2+(y^2_{S})^2<2$. 

 \item   The two coordinate charts   intersect in  the  open equatorial region
 $\RR_{Eq}:=\RR_N\cap \RR_S$ in which both coordinate systems are defined.
 
 \item  In $\RR_{Eq} $   the transition functions  between the two coordinate  systems are given by  the smooth  functions $ \varphi_{SN}$ and $\varphi_{NS}= \varphi_{SN}^{-1} $. 
 \end{enumerate}
 
 \item The metric coefficients for the two coordinate systems   are given by (see Lemma \ref{Lemma:geodesic-coordinates}) 
 \beaa
\g &=& - 2\vsi du ds + \vsi^2\Omb  du^2 +g^{N}_{ab}\big( dy_N^a- \vsi \undB_{N}^a du\big) \big( dy_N^b-\vsi \undB_N^b du\big),\\
\g &=& - 2\vsi du ds + \vsi^2\Omb  du^2 +g^{S}_{ab}\big( dy_S^a- \vsi \undB_{S}^a du\big) \big( dy_S^b-\vsi \undB_S^b du\big),
\eeaa
where
\beaa
\Omb=e_3(s), \qquad \undB_N^a =\frac{1}{2} e_3(y_N^a), \qquad \undB_S^a =\frac{1}{2} e_3(y_S^a).
\eeaa
\end{enumerate}

 \begin{definition} 
\label{defintion:regionRRovr}
Let $m_0>0$ a constant.   Let $\epg>0$   a sufficiently   small   constant, and let  $(\ug, \sg, \rg)$ three real numbers with $\rg$ sufficiently large so that
\bea\lab{eq:rangeofrgandepsilon}
\epg\ll m_0, \qquad\qquad  \rg\gg m_0.
\eea
We define  $\RR$  to be the region
\bea
\lab{definition:RR(dg,epg)}
\RR:=\left\{|u-\ug|\leq\epg,\quad |s-\sg|\leq  \epg \right\},
\eea
such that  assumptions {\bf A1-A3} below  with constant $\epg$  on  the background foliation of $\RR$,   are verified. 
\end{definition}

%%%%%%%%%%%%%%%%%%%%%%%%%%%%%%

\subsubsection{Assumptions for the   far region  $\RR$}
\lab{sec:defintionspacetimeregionRR:mainassumptions}

%%%%%%%%%%%%%%%%%%%%%%%%%%%%%%

Given an integer $s_{max}\geq 3$, we assume the following.
\begin{enumerate}
\item[\bf A1.]
For  $k\le s_{max}$
\bea\lab{eq:assumtioninRRforGagandGabofbackgroundfoliation}
\bsplit
\| \Ga_g\|_{k, \infty}&\leq  \epg  r^{-2},\\
\| \Ga_b\|_{k, \infty}&\leq  \epg  r^{-1}.
\end{split}
\eea

\item[\bf A2.]  The Hawking mass $m=m(u,s)$ of  $S(u, s)$ verifies 
\bea\lab{eq:assumtionsonthegivenusfoliationforGCMprocedure:Hawkingmass} 
\sup_{\RR}\left|\frac{m}{m_0}-1\right| &\leq& \epg.
\eea

\item[\bf A3.] 
In the  region of their respective validity\footnote{That is  the quantities on the left verify the  same estimates as those for $\Ga_b$, respectively $\Ga_g$.}   we have
\bea
 \undB_N^a,\,\, \undB_S^a \in r^{-1}\Ga_b, \qquad  Z_N^a,\,\, Z_S^a \in \Ga_b,
 \eea
 and,
 \bea
 r^{-2} \widecheck{g}^{N}_{ab},  \,\, r^{-2} \widecheck{g}^{S}_{ab} \in r\Ga_g,
 \eea
 where
 \beaa
 \widecheck{g} ^{N}\!_{ab} &=&   g^N_{ab}   -   \frac{4r^2}{1+(y^1_{N})^2+(y^2_{N})^2) }\de_{ab},\\
 \widecheck{g}^{S}\!_{ab} &=&   g^S_{ab}   -   \frac{4r^2}{(1+(y^1_{S})^2+(y^2_{S})^2) } \de_{ab}.
 \eeaa
\end{enumerate}

\begin{remark}
\lab{remark:GagcupGab}
In  view of \eqref{eq:assumtioninRRforGagandGabofbackgroundfoliation}, we will often replace $\Ga_g$ by $r^{-1} \Ga_b$.
\end{remark}

%%%%%%%%%%%%%%%%%%%%%%%%%%%%%%%%%%%%%%

\subsubsection{Basis of $\ell=1$ modes for the $\RR$-foliation}

%%%%%%%%%%%%%%%%%%%%%%%%%%%%%%%%%%%%%%

{\bf A4.} We assume  the existence of a   smooth family of  scalar functions $\Jp:\RR\longrightarrow\RRR$, for $p=0,+,-$,   verifying the following properties, for all surfaces $S$ of the background foliation.
 \bea
 \lab{eq:Jpsphericalharmonics}
\bsplit
  \Big(r^2\lap+2\Big) \Jp  &= O(\epg),\qquad p=0,+,-,\\
\frac{1}{|S|} \int_{S}  \Jp J^{(q)} &=  \frac{1}{3}\de_{pq} +O(\epg),\qquad p,q=0,+,-,\\
\frac{1}{|S|}  \int_{S}   \Jp   &=O(\epg),\qquad p=0,+,-,
\end{split}
\eea
where $S$ is any sphere of the background foliation of $\RR$.

\begin{remark}
The property of the scalar functions $\Jp$ above is motivated by the fact that the $\ell=1$ spherical harmonics  on the standard sphere  $\SSS^2$, which  are given by
\beaa
J^{(0)}=\cos\th, \qquad J^{(+)}=\sin\th\cos\vphi, \qquad J^{(-)}=\sin\th\sin\vphi,
\eeaa
satisfy\footnote{Note in particular that the following holds true on the standard unit  sphere $\SSS^2$
\beaa
\int_{\mathbb{S}^2}(\cos\th)^2=\int_{\mathbb{S}^2}(\sin\th\cos\vphi)^2=\int_{\mathbb{S}^2}(\sin\th\sin\vphi)^2=\frac{4\pi}{3}, \qquad |\SSS^2|=4\pi.
\eeaa}  \eqref{eq:Jpsphericalharmonics} with $\epg=0$.  
\end{remark}

%%%%%%%%%%%%%%%%%%%%%%%%%%%%%%%%%%%

\subsubsection{Coordinate vectorfields in $\RR$}

%%%%%%%%%%%%%%%%%%%%%%%%%%%%%%%%%%%

Recall that we have,
\beaa
\pr_s &=& e_4,\\
\pr_u &=& \vsi\left(\frac{1}{2}e_3-\frac{1}{2}\Omb e_4-\undB^a\pr_{y^a}\right),\\
\pr_{y^a}&=& \sum_{c=1,2} Y_{(a)}^c e_c, \qquad   a=1,2,
\eeaa
with coefficients  $ Y_{(a)}^c $ verifying
\beaa
g_{ab}=\sum_{c=1, 2} Y_{(a)}^c Y_{(b)}^c.
\eeaa
To simplify we  can choose  $ e_1$  in the direction of $\pr_{y^1}$  so that $ Y_{(1)}^2=0$. In that case 
\beaa
 Y_{(1)}^1 = \sqrt{g_{11}}, \qquad   Y_{(2)}^1=\frac{g_{12} } { \sqrt{g_{11}}}, \qquad  Y^2_{(2)}&=&\sqrt{g_{22}- \frac{g^2_{12} } { g_{11}}}.
\eeaa 
We deduce,
\bea
\lab{eq:assumptionY_a^b}
\bsplit
 Y_{(1)}^1 &=\frac{2r}{(1+|y|^2)^{\frac{1}{2}} } + r^2 \Ga_g,\\
 Y^1_{(2)} &= r^2 \Ga_g,\\
 Y^2_{(2)}&=\frac{2r}{(1+|y|^2)^{\frac{1}{2}} } + r^2 \Ga_g.
 \end{split}
\eea

%%%%%%%%%%%%%%%%%%%%%%%%%%%%%%%%%%

\subsubsection{Far spacetime region in Kerr}
\lab{section:farspacetimeregionKerr}

%%%%%%%%%%%%%%%%%%%%%%%%%%%%%%%%%

We  denote by $(t_0, r_0, \th_0, \vphi_0)$ the standard Boyer-Lindquist coordinates of a Kerr metric $\g_{a_0, m_0}$ with $|a_0|\le m_0$. 
 It is easy to check from the  explicit form of the Kerr metric that for   large $r$, the following asymptotic expansion holds
\bea\lab{eq:asymptoticexpansionfortheKerrmetric}
\g_{a_0,m_0} = \g_{m_0}+O\left(\frac{a_0m_0}{(r_0)^2}\right)\Bigg(    (dt_0)^2+ (dr_0)^2 +r^2 _0\Big(  (d\th_0)^2 + \sin^2 \th_0  (d\vphi_0)^2 \Big) \Bigg),
\eea
where $\g_{m_0}$ denotes the Schwarzschild metric of mass $m_0$.

The following lemma shows that the assumptions on $\RR$ are true in Kerr for sufficiently large $r_0$.
\begin{lemma}\lab{lemma:controlfarspacetimeregionKerrassumptionRR}
Let $\g_{a_0, m_0}$, with $|a_0|\le m_0$, denote a member of the Kerr family of metrics.  Let $u$ a canonical optical function for $\g_{a_0, m_0}$ normalized on the standard foliation of  $\II^+$ by round spheres. Let $S(u, s)$ 
  be the   spheres       of the induced  geodesic foliation, with $s$ the affine parameter, and  $r$ the area radius, normalized such that $\frac{s}{r} =1$ on $ \II^+$.     Define also the  corresponding angular coordinates $\th, \vphi$,  properly normalized  at infinity,  and the corresponding   $\Jp$  defined by  them.         Then,  for $r\geq r_0$ with $r_0=r_0(m_0)$ sufficiently large, the  region $\RR=\{r\geq r_0\}$ satisfies the  assumptions ${\bf A1}$-${\bf A4}$ with the smallness constants 
\beaa
\epg=\frac{a_0m_0}{r_0}, \qquad \dg=\frac{a_0m_0}{r_0}.
\eeaa
\end{lemma}

\begin{proof}
 Let $u=u_{a_0, m_0} $ be   the desired  optical function  for the metric   $\g_{a_0, m_0}$. Also,   let $u_{m_0}:= t_0-r_0-2m_0\log(r_0-2m_0)$ the corresponding  canonical Schwarzschild optical function. Then, in  view  of the  asymptotic  expansion \eqref{eq:asymptoticexpansionfortheKerrmetric} of $\g_{a_0,m_0}$,  we deduce,
 \beaa
 u=u_{m_0}+O\left(\frac{a_0m_0}{r_0}\right).
 \eeaa
 The corresponding null  geodesic  gradient of $u$ is given by
\beaa
e_4 &=&-\g_{a_0,m_0}^{\a\b}\pr_\a u\pr_\b=      \frac{1}{1-\frac{2m_0}{r_0}}\pr_{t_0}+\pr_{r_0}+O\left(\frac{a_0m_0}{(r_0)^2}\right)\left(\pr_{t_0}, \pr_{r_0}, \frac{1}{r_0}\pr_{\th_0}, \frac{1}{r_0}\pr_{\vphi_0}\right)
\eeaa
from which we  easily calculate the affine parameter $s$, $e_4(s)=1$, the   area radius $r$  of the spheres $S(u, s)$  and the coordinates $\th, \vphi$  for which  $e_4(\th)= e_4(\vphi)=0$,
\beaa
&& s = r_0+O\left(\frac{a_0m_0}{r_0}\right),\qquad\qquad\, r = r_0+O\left(\frac{a_0m_0}{r_0}\right),\\
&&\th = \th_0+O\left(\frac{a_0m_0}{(r_0)^2}\right),\qquad \qquad \varphi = \varphi_0+O\left(\frac{a_0m_0}{(r_0)^2}\right).
\eeaa
The   frame adapted to the spheres $S(u, s)$ is given by
\beaa
e_4 &=& \frac{1}{1-\frac{2m_0}{r_0}}\pr_{t_0}+\pr_{r_0}+O\left(\frac{a_0m_0}{(r_0)^2}\right)\left(\pr_{t_0}, \pr_{r_0}, \frac{1}{r_0}\pr_{\th_0}, \frac{1}{r_0}\pr_{\vphi_0}\right),\\
e_3 &=& \pr_{t_0}-\left(1-\frac{2m_0}{r_0}\right)\pr_{r_0}+O\left(\frac{a_0m_0}{(r_0)^2}\right)\left(\pr_{t_0}, \pr_{r_0}, \frac{1}{r_0}\pr_{\th_0}, \frac{1}{r_0}\pr_{\vphi_0}\right),\\
e_1 &=& \frac{1}{r_0}\pr_{\th_0}+O\left(\frac{a_0m_0}{(r_0)^2}\right)\left(\pr_{t_0}, \pr_{r_0}, \frac{1}{r_0}\pr_{\th_0}, \frac{1}{r_0}\pr_{\vphi_0}\right),\\
e_2 &=& \frac{1}{r_0\sin(\th_0)}\pr_{\vphi_0}+O\left(\frac{a_0m_0}{(r_0)^2}\right)\left(\pr_{t_0}, \pr_{r_0}, \frac{1}{r_0}\pr_{\th_0}, \frac{1}{r_0}\pr_{\vphi_0}\right).
\eeaa
This immediately yields for the Ricci coefficients associated to the frame $(e_4, e_3, e_1, e_2)$
\beaa
\Ga &=& \Ga_{m_0}+O\left(\frac{a_0m_0}{(r_0)^3}\right),
\eeaa
where $\Ga_{m_0}$ denotes the corresponding value of the Ricci coefficients for the Schwarzschild metric $g_{m_0}$. We have a similar statement for the curvature components, so that  the  assumptions ${\bf A1}$ and ${\bf A2}$ are indeed verified in the region $\RR=\{r\geq r_0\}$ with the smallness constants 
\beaa
\epg=\frac{a_0m_0}{r_0}, \qquad \dg=\frac{a_0m_0}{r_0}.
\eeaa
The statement for ${\bf A4}$ follows from the definition of $\Jp$, and the above asymptotics for $\th$ and $\vphi$. Finally, one can easily define two coordinate systems $(y^1_N, y^2_N)$ and $(y^1_S, y^2_S)$, initialized by stereographic coordinates on $\II_+$ and  transported  in the interior by $e_4$, so that  ${\bf A3}$ holds as well.
\end{proof}

%%%%%%%%%%%%%%%%%%%%%%%%%%%%%%%%%%%%%%%

\subsection{$O(\epg)$-spheres}
\lab{section:epg-spheres}

%%%%%%%%%%%%%%%%%%%%%%%%%%%%%%%%%%%%%%%

\begin{definition}
Given a compact    $2$-surface $\S\subset \RR$, not necessarily   a leaf  $S(u,s)$ of the  background geodesic foliation of $\RR$, we  denote
\begin{itemize}
\item by $\chi^\S$, $\chib^\S$, $\ze^\S$,..., the corresponding Ricci coefficients, 

\item by $\a^\S$, $\b^\S$, $\rho^\S$, ..., the corresponding curvature coefficients, 

\item by $r^\S$, $m^\S$, $K^\S$ and $\mu^\S$ respectively the corresponding area radius, Hawking mass, Gauss curvature  and mass aspect function,

\item by $\dddS_1, \dddS_2$,
   $\ddsSone, \ddsStwo$  the corresponding Hodge operators and by  $\nab^\S$ the corresponding  covariant derivative.
\end{itemize}
\end{definition}

\begin{remark}\lab{rem:intrinsicdefinitionofrSKSmS}
Note that the quantities $r^\S,  \chi^\S, \chib^\S, \ze^\S, \a^\S, \b^\S, \rho^\S, \rhod^\S, \bb^\S, \aa^\S, \mu^\S, m^\S$   are well defined on $\S$
 and, in addition,  $m^\S$, $K^\S$, $\rho^\S$ and $\mu^\S$ are   invariant with respect to change of scale transformations    $\la\to (\la e_4^\S, \la^{-1}e_3^\S)$, where $(e_4^\S, e_3^\S, e_1^\S, e_2^\S)$ is a null frame adapted to $\S$. See also Remark \ref{remark:welldefinedGa}.
\end{remark}

\begin{definition}
Given a scalar $h$ on $\S$, we denote  by $ \ov{h}^\S $  and  $\widecheck{h}^\S$ the average and average
 free part\footnote{Note that  the operation $\,\widecheck{}\, $ has a different meaning  here than  the one we used 
 earlier   in the definition of  $\Ga_g, \Ga_b$.   To avoid confusion we will always  use  $\widecheck{}\,^\S$   to refer to  the    average free part of a  scalar function on $\S$.}  of  $h$, i.e. 
 \beaa
 \ov{h}^\S=\frac{1}{|\S|} \int_\S h, \qquad \widecheck{h}^\S=h- \ov{h}^\S.
 \eeaa
\end{definition}

\begin{definition}
We  will work with  the following weighted Sobolev norms  on $\S$
   \bea
   \lab{definition:spaceH^k(boldS)}
\| f\|_{\hk_s(\S)} &:=& \sum_{i=0}^s \|( \dkb^\S )^i f\|_{L^2(\S)}, \qquad   \dkb^\S =r^\S \nab^\S.
\eea
\end{definition}

The goal of this paper is to construct  new spheres $\S\subset \RR$   which  verify  special properties we call GCM conditions. In particular these  spheres are  close  to being a round sphere in the sense of the definition  below.

\begin{definition}\lab{def:defintionofOofepgspheres}
A compact surface $\S\subset \RR$ of area radius $r^\S$ is called a $O(\epg)$-sphere  provided that the Gauss curvature $K^\S$ of $\S$ verifies
\bea
K ^\S&=& \frac{1+O(\epg)}{(r ^\S)^2}
\eea
as well as
\bea
\left\|K^\S -\frac{1}{(r^\S)^2}\right\|_{\hk_{s_{max}-1}(\S)} &\les& (r^\S)^{-1}\epg,
\eea
and the area radius $r^\S$ verifies 
\bea
\sup_\S |r^\S-r| &\les& \epg.
\eea
\end{definition}

\begin{remark}
Note that the  spheres $S(u,s)$ of the background foliation of $\RR$ are $O(\epg)$-spheres.
\end{remark}

%%%%%%%%%%%%%%%%%%%%%%%%%%%%%%%%%%%

\subsubsection{Definition of  $\ell=1$ modes on $O(\epg)$-spheres}

%%%%%%%%%%%%%%%%%%%%%%%%%%%%%%%%%%%

We give below a general definition of $\ell=1$ modes   on any $O(\epg)$-sphere  $\S\subset \RR$.

\begin{definition}
\lab{definition-spherical harmonicsS}  Let $\S\subset\RR $   be a  $O(\epg)$ sphere as defined   above. 
  We say that  a triplet  $\JpS$,  $p\in\big\{ -, 0, +\} $,  of  smooth functions on $\S$  is a basis of    $\ell=1$   modes  on $\S$ if  the following are verified 
\bea\lab{eq:basicpropertiesJpSonOepsphere}
\bsplit
\Big((r^\S)^2 \Delta^\S +2\Big)\JpS  &=  O(\epg), \qquad p=0,+,-,\\
\frac{1}{|\S|} \int_{\S}  \JpS J^{( \S, q )} &=  \frac{1}{3}\de_{pq} +O(\epg), \qquad p,q=0,+,-,\\
\frac{1}{|\S|}  \int_S\JpS   &=O(\epg),\qquad p=0,+,-.
\end{split}
\eea
\end{definition}

 \begin{definition}
 We define the $\ell=1$ modes  of scalars and 1-forms on an $O(\epg)$-sphere $\S$  as follows.
  \lab{definition:ell=1modes} 
  \begin{enumerate}
\item    If $\la$ is  a scalar function  on $\S$, we define the triplet 
\bea\lab{eq:defell=1forscalarsonOofepgspheres}
\la_{\ell=1}= \left\{ \int_\S \JpS \la, \qquad p\in\big\{ -, 0, +\big\} \right\}
\eea
  and set
  \beaa
  |(\la)_{\ell=1}| &=&   \sum_{p=0,+,-}\left|\int_\S \JpS\la\right|.
  \eeaa
  
 \item  If   $f$ is a 1-form  on $\S$, we define the  sextet\footnote{Recall  that  $ \ddd_1f=(\div f, \curl f)$.} 
  \beaa
 ( f)_{\ell=1} :&=&  \left\{ \int_\S \JpS \ddd^\S_1f, \quad p\in\big\{ -, 0, +\big\} \right\}
 \eeaa
 and set
  \beaa
  |(f)_{\ell=1}| &=&   \sum_{p=0,+,-}\left|\int_\S \JpS\ddd^\S_1f\right|.
  \eeaa
 \end{enumerate}
    \end{definition}

\begin{lemma}
\lab{Lemma-spherical harmonicsS}
Assume $\S\subset \RR$ is a  sphere endowed with a basis of  $\ell=1$  modes as in  Definition \ref{definition-spherical harmonicsS}  above.
Then, provided that  $\epg>0$ is chosen small enough, the following Poincar\'e inequality holds for any 1-form $f$ on $\S$
\bea
\lab{Hodge:ddd1-dds2}
\int_\S|\ddd^\S_1f|^2 &\les& \int_\S|\ddsStwo  f |^2+r^{-2}  |(f)_{\ell=1}|^2.
\eea
 Note also the obvious inequality
  \bea
  \lab{estimate:badmodeless}
  |(f)_{\ell=1}| &\les & r\|\ddd^\S_1f\|_{L^2(\S)}.
  \eea
\end{lemma}

\begin{proof}
There exists a pair of scalar functions $(h, \dual h)$ on $S$ such that 
\beaa
f=\dds^\S_1(h,\dual h), \qquad \int_\S h=\int_\S \dual h=0, \qquad (f)_{\ell=1}=((\Delta^\S h)_{\ell=1}, (\Delta^\S\dual h)_{\ell=1}).
\eeaa
In particular, we have
\beaa
\ddd^\S_1f = (-\Delta^\S h, \Delta^\S\dual h). 
\eeaa
We infer
\beaa
\int_\S|\dds^\S_2f|^2 &=& \int_\S f\ddd^\S_2\dds^\S_2f = \int_\S f(\dds^\S_1\ddd^\S_1-2K)f= \int_\S|\ddd^\S_1f|^2 -\frac{2+O(\epg)}{(r^\S)^2}\int_\S|f|^2\\ 
&=& \int_\S|\Delta^\S h|^2 -\frac{2+O(\epg)}{(r^\S)^2}\int_\S|\nabla^\S h|^2 + \int_\S|\Delta^\S\dual h|^2 -\frac{2+O(\epg)}{(r^\S)^2}\int_\S|\nabla^\S\dual h|^2.
\eeaa

Now, comparing $(r^\S)^2\Delta^\S$ with the Laplace-Beltrami\footnote{Recall that the two first non zero eigenvalues of $-\Delta_{\SSS^2}$ are given respectively by 2 and 6.} $\Delta_{\SSS^2}$ on the standard sphere $\SSS^2$, we infer by a we have by a standard perturbation argument that
\beaa
\int_\S|\Delta^\S h|^2 -\frac{2+O(\epg)}{(r^\S)^2}\int_\S|\nabla^\S h|^2 &\geq& \int_\S|\Delta^\S h|^2 -r^{-2}\left((\Delta^\S h)_{\ell=1}\right)^2,\\
\int_\S|\Delta^\S\dual h|^2 -\frac{2+O(\epg)}{(r^\S)^2}\int_\S|\nabla^\S\dual h|^2 &\geq& \int_\S|\Delta^\S \dual h|^2 -r^{-2}\left((\Delta^\S \dual h)_{\ell=1}\right)^2.
\eeaa
We infer
\beaa
\int_\S|\dds^\S_2f|^2 &=& \int_\S|\Delta^\S h|^2 -\frac{2+O(\epg)}{(r^\S)^2}\int_\S|\nabla^\S h|^2 + \int_\S|\Delta^\S\dual h|^2 -\frac{2+O(\epg)}{(r^\S)^2}\int_\S|\nabla^\S\dual h|^2\\
&\geq& \int_\S\big(|\Delta^\S h|^2+|\Delta^\S\dual h|^2\big) - r^{-2}\left((\Delta^\S h)_{\ell=1}\right)^2 - r^{-2}\left((\Delta^\S \dual h)_{\ell=1}\right)^2
\eeaa
and hence
\beaa
\int_\S|\dds^\S_2f|^2 &\geq&  \int_\S|\dddS_1  f |^2 -r^{-2}((f)_{\ell=1})^2
\eeaa
as desired.
\end{proof}

%%%%%%%%%%%%%%%%%%%%%%%%%%%%%%%%%%% 
  
\subsubsection{Elliptic lemma for Hodge systems}

%%%%%%%%%%%%%%%%%%%%%%%%%%%%%%%%%%%

 \begin{lemma} 
\label{prop:2D-Hodge}
Let $\S\subset \RR $ be a $O(\epg)$-sphere endowed  with a basis of $\ell=1$ modes  as in Definition  \ref{definition-spherical harmonicsS}. Then, for all $k\le s_{max}$,
\begin{enumerate}
\item  If   $  f  \in \SS_1(\S)$
\bea
 \|  f\|_{\hk_{k+1}   (\S)}    \les r   \|\dddS_1   f  \|_{\hk_k(\S)}.
\eea

\item If $v\in \SS_2(\S)$
\bea
 \|  v\|_{\hk_{k+1}   (\S)}    \les r   \|\dddS_2   v  \|_{\hk_k(\S)}.
\eea

\item  If  $\la, \mu\in \SS_0(\S)$
\bea
  \| (\widecheck{\la}, \widecheck{\mu})\|_{\hk_{k+1} (\S) } \les r  \|\ddsSone\, (\la, \mu)\|_{\hk_{k}(\S)}.
\eea

\item   If  $   f  \in \SS_1(\S) $
\bea
  \|  f\|_{\hk_{k+1}  (\S)}&\les&  r  \|\ddsStwo\, f\|_{\hk_k(\S)}+\big| (f)_{\ell=1}\big|.
  \eea

\end{enumerate}
\end{lemma}
\begin{remark}
Note that, in view of  our  ${\bf A1, A3}$ assumptions  the results of Lemma \ref{prop:2D-Hodge}  hold true for the spheres  $S$ of the background foliation.
\end{remark}

\begin{proof}
The case $k=0$  for the first three estimates can be found in \cite{Ch-Kl}, and concerning the last estimate, it follows from Lemma \ref{Lemma-spherical harmonicsS}. The, case $1\leq k\leq s_{max}$ follows by standard elliptic regularity and the control of $K^\S$ for an  $O(\epg)$-sphere $\S$.
\end{proof}

%%%%%%%%%%%%%%%%%%%%%%%%%%%%%%%%%%%%%%%

\subsubsection{A lemma concerning the solvability of $\Delta^\S+2/(r^\S)^2$}

%%%%%%%%%%%%%%%%%%%%%%%%%%%%%%%%%%%%%%%

\begin{lemma}\lab{lemma:sovabilityoftheoperatorDeltaSplus2overrsquare}
    Let $\S\subset\RR $   be a  $O(\epg)$-sphere  endowed with a triplet   $\JpS$ of  $\ell=1$ modes as in
    Definition \ref {definition-spherical harmonicsS}. The following hold true.
\begin{enumerate}
\item The operator $\Delta^\S +2/(r^\S)^2$ admits three eigenvalues $\nu_p$ with corresponding eigenfunction on $S$  $\jp$, $p=0,+,-$, verifying
\beaa
\nu_p=O\left(\frac{\epg}{r^2}\right), \qquad \jp=\JpS+O\left(\epg\right), \quad p=0,+,-.
\eeaa

\item Any other eigenvalue $\nu$ of  $ \Delta^\S+2/(r^\S)^2$ satisfies $|\nu|\geq r^{-2}$. 

\item Consider the equation 
\beaa
\left(\Delta^\S+\frac{2}{(r^\S)^2}\right)\la &=& h+\sum_pc_pJ^{(p,\S)},
\eeaa
where $\la$ and $h$ are scalar functions and $c_p$ are constants. Then, given three constants $\la_p$, $p=0,+,-$, there exists unique constants $c_p$, and a unique scalar function $(\la)^\perp$ such that the solution $\la$ is given by
\beaa
\la=(\la)^\perp+\sum_p\la_p j^{(p)}, \qquad \int_\S(\la)^\perp  j^{(p)}=0,
\eeaa
with $c_p$ and $(\la)^\perp$ verifying 
\beaa
\sum_p|c_p| + r^{-3}\|\widecheck{(\la)^\perp}^\S\|_{\hk_2(\S)} &\les& r^{-1}\|\widecheck{h}^\S\|_{L^2(\S)}+\frac{\epg}{r^2}\sum_p|\la_p|,
\eeaa 
\beaa
|\ov{(\la)^\perp}^\S| &\les& r^2|\ov{h}^\S|+\epg \sum_p|\la_p|,
\eeaa
 and for $0\leq s\leq s_{max}-1$,
\beaa
r^{-3}\|\widecheck{(\la)^\perp}^\S\|_{\hk_{s+2}(\S)} &\les& r^{-1}\|\widecheck{h}^\S\|_{\hk_s(\S)}+\frac{\epg}{r^2}\sum_p|\la_p|,
\eeaa
where 
\beaa
\widecheck{(\la)^\perp}^\S = (\la)^\perp - \ov{(\la)^\perp}^\S, \qquad \widecheck{h}^\S = h - \ov{h}^\S.
\eeaa
\end{enumerate}
\end{lemma}

\begin{remark}
Since we have $(\Delta^\S+2/(r^\S)^2-\nu_p)j^{(p)} = 0$, we infer, after integrating on $\S$, and since $2/(r^\S)^2-\nu_p\neq 0$, 
\bea\lab{eq:theeigenvectorsjphavezeroaverageonS}
\int_\S j^{(p)} &=& 0, \quad p=0,+,-.
\eea
\end{remark}

\begin{proof}
The first two statements follow from comparing $(r^\S)^2\Delta^\S+2$ with the operator  $\Delta_{\SSS^2}+2$  and using a standard perturbation argument. 

Next, we focus on the third statement. We plug the decomposition 
\beaa
\la=(\la)^\perp+\sum_p\la_p j^{(p)}, \qquad \int_\S(\la)^\perp  j^{(p)}=0
\eeaa
in the equation for $\la$ and find 
\beaa
\left(\Delta^\S+\frac{2}{(r^\S)^2}\right)(\la)^\perp &=& h-\sum_p\la_p\left(\Delta^\S+\frac{2}{(r^\S)^2}\right)j^{(p)}+\sum_pc_pJ^{(p,\S)}\\
&=& h-\sum_p\la_p\nu_pj^{(p)}+\sum_pc_pJ^{(p,\S)}.
\eeaa
We then choose $c_p$ such that 
\beaa
\int_\S\left(h-\sum_p\la_p\nu_pj^{(p)}+\sum_pc_pJ^{(p,\S)}\right)j^{(q)} &=& 0, \qquad q=0,+,-,
\eeaa
i.e. 
\beaa
\sum_pc_p\left(\int_\S J^{(p,\S)}j^{(q)} \right)&=& -\int_\S\left(\widecheck{h}^\S -\sum_p\la_p\nu_pj^{(p)}\right)j^{(q)}, \qquad q=0,+,-,
\eeaa
where we used in particular \eqref{eq:theeigenvectorsjphavezeroaverageonS}. In view of the properties of $j^{(q)}$, and the assumptions for $J^{(p,\S)}$, we have
\beaa
\frac{1}{|\S|}\int_\S J^{(p,\S)}j^{(q)}  &=& \frac{1}{3}\de_{pq}+O(\epg), \quad p, q=0,+,-,
\eeaa
so that the above formula uniquely defines the constants $c_p$, $p=0,+,-$, and yields
\beaa
\sum_p|c_p| &\les& r^{-1}\|\widecheck{h}^\S\|_{L^2(\S)}+\sum_p|\la_p||\nu_p|\\
&\les& r^{-1}\|\widecheck{h}^\S\|_{L^2(\S)}+\frac{\epg}{r^2}\sum_p|\la_p|.
\eeaa

The above choice of the  constants $c_p$ yields the existence of a unique $(\la)^\perp$. To estimate $(\la)^\perp$, we take the average and the average free part of its equation and find, using in particular \eqref{eq:theeigenvectorsjphavezeroaverageonS}, 
\beaa
\frac{2}{(r^\S)^2}\ov{(\la)^\perp}^\S &=& \ov{h}^\S + \sum_pc_p\ov{J^{(p,\S)}}^\S,\\
\left(\Delta^\S+\frac{2}{(r^\S)^2}\right)\widecheck{(\la)^\perp}^\S &=& \widecheck{h}^\S-\sum_p\la_p\nu_pj^{(p)}+\sum_pc_p\Big(J^{(p,\S)} - \ov{J^{(p,\S)}}^\S\Big).
\eeaa
In view of  the above choice of the constants $c_p$, and using the fact that the eigenvalues $\nu$ of $\Delta^\S+2/(r^\S)^2$ with $\nu\neq \nu_p$ satisfy $|\nu|\geq 1$, we infer
\beaa
|\ov{(\la)^\perp}^\S| &\les& r^2|\ov{h}^\S|+r^2\sum_p|c_p|\\
&\les& r^2|\ov{h}^\S|+\epg \sum_p|\la_p|
\eeaa
and
\beaa
\|\widecheck{(\la)^\perp}^\S\|_{\hk_2(\S)} &\les& r^2\|\widecheck{h}^\S\|_{L^2(S)}+r^3\sum_p|\la_p||\nu_p|+r^3\sum_p|c_p|\\
&\les& r^2\|\widecheck{h}^\S\|_{L^2(S)}+\epg r\sum_p|\la_p|.
\eeaa 
Finally, higher order estimates for $\widecheck{(\la)^\perp}^\S$ follow from standard elliptic regularity.
\end{proof}

\begin{remark}
In the generic case where $(\nu_0, \nu_+, \nu_-)\neq (0, 0, 0)$, $(\la)^\perp$ actually depends on $\la_p$ through the term
\beaa
-\sum_p\la_p\nu_pj^{(p)}
\eeaa
appearing on the right-hand side of the equation for $(\la)^\perp$ in the proof above.
\end{remark}

%%%%%%%%%%%%%%%%%%%%%%

\section{Frame transformations}
\lab{sec:frametransformation}

%%%%%%%%%%%%%%%%%%%%%%

%%%%%%%%%%%%%%%%%%%%%%%%%%%%%%%%%%

\subsection{General null frame transformations}

%%%%%%%%%%%%%%%%%%%%%%%%%%%%%%%%%%

\begin{lemma}
\lab{Lemma:Generalframetransf}
Given a null frame $(e_3, e_4, e_1, e_2)$, a general null transformation     from  the null frame  $(e_3, e_4, e_1, e_2)$  to another null frame   $(e_3', e_4', e_1', e_2')$         can be written in   the form,
 \bea
 \lab{eq:Generalframetransf}
 \bsplit
  e_4'&=\la\left(e_4 + f^b  e_b +\frac 1 4 |f|^2  e_3\right),\\
  e_a'&= \left(\de_{ab} +\frac{1}{2}\fb_af_b\right) e_b +\frac 1 2  \fb_a  e_4 +\left(\frac 1 2 f_a +\frac{1}{8}|f|^2\fb_a\right)   e_3,\qquad a=1,2,\\
 e_3'&=\la^{-1}\left( \left(1+\frac{1}{2}f\c\fb  +\frac{1}{16} |f|^2  |\fb|^2\right) e_3 + \left(\fb^b+\frac 1 4 |\fb|^2f^b\right) e_b  + \frac 1 4 |\fb|^2 e_4 \right),
 \end{split}
 \eea
  where $\la$ is a scalar, $f$ and $\fb$ are horizontal 1-forms. The dot product and magnitude  $|\c |$ are taken with respect to the standard euclidian norm of $\RRR^2$.  We call $(f, \fb, \la)$  the transition coefficients of the change of frame.  
  \end{lemma}

 \begin{remark}
 Note that we have in particular the following identities 
 \beaa
   e_a'&=& e_a +\frac 1 2  \fb_a \la^{-1} e_4'  +\frac 1 2 f_a e_3,\\
    e_3'&=&\la^{-1}\left(  e_3 +  \fb^ae_a' -\frac 1 4 |\fb|^2\la^{-1} e_4'\right). 
 \eeaa
 \end{remark}

\begin{proof}
 Clearly  $e_4'$ is null. Also, we have
\beaa
\la^{-1}\g(e_4', e_a') &=& \g\left(e_4 + f^b  e_b +\frac 1 4 |f|^2  e_3, \left(\de_a^c +\frac{1}{2}\fb_af^c\right) e_c +\frac 1 2  \fb_a  e_4+\left(\frac 1 2 f_a +\frac{1}{8}|f|^2\fb_a\right)   e_3\right)\\
&=&  f^b  \left(\de_a^c +\frac{1}{2}\fb_af^c\right)\de_{bc} -2\left(\frac 1 2 f_a +\frac{1}{8}|f|^2\fb_a\right)-\frac 1 4 |f|^2\fb_a\\
&=& f_a +\frac{1}{2}|f|^2\fb_a - f_a -\frac{1}{4}|f|^2\fb_a -\frac 1 4 |f|^2\fb_a = 0.
\eeaa
Similarly,
\beaa
\g(e_a', e_b') 
&=& \left(\de_a^c +\frac{1}{2}\fb_af^c\right)\left(\de_b^d +\frac{1}{2}\fb_bf^d\right)\de_{cd} -  \fb_a\left(\frac 1 2 f_b +\frac{1}{8}|f|^2\fb_b\right) \\
&&-\left(\frac 1 2 f_a +\frac{1}{8}|f|^2\fb_a\right)\fb_b = \de_{ab}
\eeaa
and
\beaa
\g(e_3', e_4') &=&
  \left(\fb^b+\frac 1 4 |\fb|^2f^b\right)f_b -2\left(1+\frac{1}{2}f\c\fb  +\frac{1}{16} |f|^2  |\fb|^2\right) -\frac 1 8 |\fb|^2 |f|^2 = -2.
\eeaa

Also, we have
\beaa
 \la\g(e_3', e_a') &=& \left(\fb^b+\frac 1 4 |\fb|^2f^b\right)\left(\de_a^c +\frac{1}{2}\fb_af^c\right)\de_{bc} -\left(1+\frac{1}{2}f\c\fb  +\frac{1}{16} |f|^2  |\fb|^2\right) \fb_a \\
&& -\frac 1 2 |\fb|^2\left(\frac 1 2 f_a +\frac{1}{8}|f|^2\fb_a\right)\\
&=& \fb_a+\frac 1 4 |\fb|^2f_a+\left(f\cdot\fb+\frac 1 4 |\fb|^2|f|^2\right)\frac{1}{2}\fb_a\\
&& -\left(1+\frac{1}{2}f\c\fb  +\frac{1}{16} |f|^2  |\fb|^2\right) \fb_a  -\frac 1 2 |\fb|^2\left(\frac 1 2 f_a +\frac{1}{8}|f|^2\fb_a\right) = 0.
\eeaa
Finally 
\beaa
 \la^2\g(e_3', e_3') &=& \left|\fb+\frac 1 4 |\fb|^2f\right|^2 -  |\fb|^2\left(1+\frac{1}{2}f\c\fb  +\frac{1}{16} |f|^2  |\fb|^2\right)\\
 &=& |\fb|^2+\frac{1}{2}|\fb|^2f\cdot\fb+\frac{1}{16} |\fb|^4|f|^2 -  |\fb|^2\left(1+\frac{1}{2}f\c\fb  +\frac{1}{16} |f|^2  |\fb|^2\right) = 0.
\eeaa
This concludes the proof of the lemma.
\end{proof}

%%%%%%%%%%%%%%%%%%%%%%%%%%%%%%%%%%%%%%%%%%%%%

\subsection{Transformation formulas for Ricci and Curvature  coefficients}

%%%%%%%%%%%%%%%%%%%%%%%%%%%%%%%%%%%%%%%%%%%%%  

While  we only need the transformation formulas for $\chi$, $\chib$, $\ze$ and $\rho$ for this paper, we nevertheless derive below the transformation formulas for all connection coefficients and curvature components for completeness.

\begin{proposition}
\lab{Prop:transformation-formulas-generalcasewithoutassumptions}
Under a general transformation of type \eqref{eq:Generalframetransf}, the Ricci coefficients  transform as follows:
\begin{itemize}
\item The transformation formula for $\xi$ is given by 
\bea
\bsplit
\la^{-2}\xi' &= \xi +\frac{1}{2}\nab_{\la^{-1}e_4'}f+\frac{1}{4}(\trch f -\atrch\dual f)+\om f +\err(\xi,\xi'),\\
\err(\xi,\xi') &= \frac{1}{2}f\c\chih+\frac{1}{4}|f|^2\eta+\frac{1}{2}(f\c \ze)\,f -\frac{1}{4}|f|^2\etab \\
&+ \la^{-2}\left( \frac{1}{2}(f\c\xi')\,\fb+ \frac{1}{2}(f\c\fb)\,\xi'   \right)  +\lot
       \end{split}
\eea

\item The transformation formula for $\xib$ is given by 
\bea
\bsplit
\la^2\xib' &= \xib + \frac{1}{2}\la\nab_3'\fb' +    \omb\,\fb + \frac{1}{4}\trchb\,\fb - \frac{1}{4}\atrchb\dual\fb +\err(\xib, \xib'),\\
\err(\xib, \xib') &=   \frac{1}{2}\fb\c\chibh - \frac{1}{2}(\fb\c\ze)\fb +  \frac 1 4 |\fb|^2\etab  -\frac 1 4 |\fb|^2\eta'+\lot
       \end{split}
\eea

\item The transformation formulas for $\chi $ are  given by 
\bea
\bsplit
\la^{-1}\trch' &= \trch  +  \div'f + f\c\eta + f\c\ze+\err(\trch,\trch')\\
\err(\trch,\trch') &= \fb\c\xi+\frac{1}{4}\fb\c\left(f\trch -\dual f\atrch\right) +\om (f\c\fb)  -\omb |f|^2 \\
& -\frac{1}{4}|f|^2\trchb -  \frac 1 4 ( f\c\fb) \la^{-1}\trch' +\frac 1 4  (\fb\wedge f) \la^{-1}\atrch'+\lot,
\end{split}
\eea
\bea
\bsplit
\la^{-1}\atrch' &= \atrch  +  \curl'f + f\wedge\eta + f\wedge\ze +\err(\atrch,\atrch'),\\
\err(\atrch,\atrch') &= \fb\wedge\xi+\frac{1}{4}\left(\fb\wedge f\trch +(f\c\fb)\atrch\right) +\om f\wedge\fb   \\
& -\frac{1}{4}|f|^2\atrchb -  \frac 1 4 ( f\c\fb) \la^{-1}\atrch' +\frac 1 4   \la^{-1}(f\wedge \fb)\trch'+\lot,
\end{split}
\eea
\bea
\bsplit
\la^{-1}\chih' &= \chih  +  \nab'\hot f + f\hot\eta + f\hot\ze+\err(\chih,\chih'),\\
\err(\chih,\chih') &=\fb\hot\xi+\frac{1}{4}\fb\hot\left(f\trch -\dual f\atrch\right) +\om f\hot\fb  -\omb f\hot f -\frac{1}{4}|f|^2\atrchb\\
&  +\frac 1 4  (f\hot\fb) \la^{-1}\trch' +\frac 1 4  (\dual f\hot\fb) \la^{-1}\atrch' +\frac 1 2  \fb\hot (f\c\la^{-1}\chih')+\lot
\end{split}
\eea

\item The transformation formulas for $\chib $ are  given by 
\bea
\bsplit
\la\trchb' &= \trchb +\div'\fb +\fb\c\etab  -  \fb\c\ze +\err(\trchb, \trchb'),\\
\err(\trchb, \trchb') &= \frac{1}{2}(f\c\fb)\trchb+f\c\xib -|\fb|^2\om + (f\c\fb)\omb   -\frac 1 4 |\fb|^2\la^{-1}\trch'+\lot,
\end{split}
\eea
\bea
\bsplit
\la\atrchb' &= \atrchb +\curl'\fb +\fb\wedge\etab  -  \ze\wedge\fb+\err(\atrchb, \atrchb'),\\
\err(\atrchb, \atrchb') &= \frac{1}{2}(f\c\fb)\atrchb+f\wedge\xib  + (f\wedge\fb)\omb   -\frac 1 4 |\fb|^2\la^{-1}\atrch'+\lot,
\end{split}
\eea
\bea
\bsplit
\la\chibh' &= \chibh +\nab'\hot\fb +\fb\hot\etab  -  \fb\hot\ze +\err(\chibh, \chibh'),\\
\err(\chibh, \chibh') &= \frac{1}{2}(f\hot\fb)\trchb  +f\hot\xib -(\fb\hot\fb)\om + (f\hot\fb)\omb   -\frac 1 4 |\fb|^2\la^{-1}\chih'+\lot
\end{split}
\eea

\item  The transformation formula for $\ze$ is given by 
\bea
\bsplit
\ze' &= \ze -\nab'(\log\la)  -\frac{1}{4}\trchb f +\frac{1}{4}\atrchb \dual f +\om\fb -\omb f +\frac{1}{4}\fb\trch\\
&+\frac{1}{4}\dual\fb\atrch+\err(\ze, \ze'),\\
\err(\ze, \ze') &= -\frac{1}{2}\chibh\c f + \frac{1}{2}(f\c\ze)\fb -  \frac{1}{2}(f\c\etab)\fb +\frac{1}{4}\fb(f\c\eta) + \frac{1}{4}\fb(f\c\ze)  \\
& + \frac{1}{4}\dual\fb(f\wedge\eta) + \frac{1}{4}\dual\fb(f\wedge\ze) +  \frac{1}{4}\fb\div'f  + \frac{1}{4}\dual\fb \curl'f +\frac{1}{2}\la^{-1}\fb\c\chih' \\
&  -\frac{1}{16}(f\c\fb)\fb\la^{-1}\trch' +\frac{1}{16}  (\fb\wedge f) \fb\la^{-1}\atrch'  -  \frac{1}{16}\dual\fb ( f\c\fb) \la^{-1}\atrch'\\
& +\frac{1}{16}\dual\fb \la^{-1}(f\wedge \fb)\trch' +\lot
\end{split}
\eea

\item   The transformation formula for $\eta$ is given by 
\bea
\bsplit
\eta' &= \eta +\frac{1}{2}\la \nab_3'f  +\frac{1}{4}\fb\trch -\frac{1}{4}\dual\fb\atrch -\omb\, f +\err(\eta, \eta'),\\
\err(\eta, \eta') &= \frac{1}{2}(f\c\fb)\eta +\frac{1}{2}\fb\c\chih
+\frac{1}{2}f(\fb\c\ze)  -  (\fb\c f)\eta'+ \frac{1}{2}\fb (f\c\eta') +\lot
\end{split}
\eea
\item   The transformation formula for $\etab$ is given by 
\bea
\bsplit
\etab' &= \etab +\frac{1}{2}\nab_{\la^{-1}e_4'}\fb +\frac{1}{4}\trchb f - \frac{1}{4}\atrchb\dual f -\om\fb +\err(\etab, \etab'),\\
\err(\etab, \etab') &=  \frac{1}{2}f\c\chibh + \frac{1}{2}(f\c\eta)\fb-\frac 1 4  (f\c\ze)\fb  -\frac 1 4 |\fb|^2\la^{-2}\xi'+\lot
\end{split}
\eea

\item   The transformation formula for $\om$ is given by
\bea
\bsplit
\la^{-1}\om' &=  \om -\frac{1}{2}\la^{-1}e_4'(\log\la)+\frac{1}{2}f\c(\ze-\etab) +\err(\om, \om'),\\
\err(\om, \om') &=   -\frac{1}{4}|f|^2\omb - \frac{1}{8}\trchb |f|^2+\frac{1}{2}\la^{-2}\fb\c\xi' +\lot
\end{split}
\eea

\item   The transformation formula for $\omb$ is given by
\bea
\bsplit
\la\omb' &= \omb+\frac{1}{2}\la e_3'(\log\la)  -\frac{1}{2}\fb\c\ze -\frac{1}{2}\fb\c\eta +\err(\omb,\omb'),\\
\err(\omb,\omb') &= f\c\fb\,\omb-\frac{1}{4} |\fb|^2\om  +\frac{1}{2}f\c\xib + \frac{1}{8}(f\c\fb)\trchb + \frac{1}{8}(\fb\wedge f)\atrchb \\
& -\frac{1}{8}|\fb|^2\trch  -\frac{1}{4}\la \fb\c\nab_3'f    +\frac{1}{2}  (\fb\c f)(\fb\c\eta')- \frac{1}{4}|\fb|^2 (f\c\eta')+\lot
\end{split}
\eea
\end{itemize}
where, for the transformation formulas of the Ricci coefficients above, $\lot$ denote expressions of the type
\beaa
         \lot&=O((f,\fb)^3)\Ga +O((f,\fb)^2) \Gac
\eeaa
containing no derivatives of $f$, $\fb$, $\Ga$ and $\Gac$. 

Also, the curvature components transform as follows
\begin{itemize}
\item The transformation formula for $\a, \aa $  are  given by
\bea
\bsplit
\la^{-2} \a'&=\a +\err(\a, \a'),\\
\err(\a, \a')&=  \big(  f\hot \b  -\dual f \hot \dual  \b )+ \left( f\hot f-\frac 1 2  \dual f \hot   \dual f \right) \rho
+  \frac 3  2 \big(  f \hot  \dual  f\big) \rhod +\lot,
\end{split}
\eea
\bea
\bsplit
\la^2\aa'&=\aa +\err(\a, \a'),\\
\err(\aa, \aa')&=  -\big(  \fb \hot \bb  -\dual \fb \hot \dual  \bb )+ \big( \fb \hot \fb-\frac 1 2  \dual \fb \hot   \dual \fb \big) \rho
+  \frac 3  2 \big(  \fb \hot  \dual  \fb\big) \rhod +\lot
\end{split}
\eea

\item   The transformation formula for $\b , \bb $  are  given by
   \bea
  \bsplit
\la^{-1}   \b'&=\b +\frac 3 2\big(  f \rho+\dual  f  \rhod\big)+\err(\b, \b'), \\
  \err(\b, \b')&= \frac 1 2 \a\c\fb+\lot,
  \end{split}
  \eea
  \bea
  \bsplit
  \la\bb'&=\bb -\frac 3 2\big(  \fb \rho+\dual  \fb  \rhod\big)+\err(\bb, \bb'), \\
  \err(\bb, \bb')&= -\frac 1 2  \aa\c f +\lot 
  \end{split}
  \eea
  \item The transformation formula for $\rho$ and $\rhod $  are  given by
  \bea
  \bsplit
 \rho' &= \rho +\err(\rho, \rho'),\\
\err(\rho, \rho') &= \fb\c\b - f\c\bb +\frac{3}{2}\rho(f\c\fb) -\frac{3}{2}\rhod (f\wedge\fb) +\lot
   \end{split}
  \eea
  \bea
  \bsplit
  \rhod' &= \rhod +\err(\rhod, \rhod'),\\
\err(\rhod, \rhod') &= -\fb\c\dual\b - f\c\dual\bb +\frac{3}{2}\rhod(f\c\fb) +\frac{3}{2}\rho (f\wedge\fb) +\lot
   \end{split}
  \eea  
\end{itemize}
where, for the transformation formulas of the curvature components above, $\lot$ denote expressions of the type
\beaa
         \lot&=O((f,\fb)^3)(\rho, \rhod) +O((f,\fb)^2)(\a,\b,\aa, \b)
\eeaa
containing no derivatives of $f$, $\fb$, $\a$, $\b$, $(\rho, \rhod)$, $\bb$, and $\aa$. 
\end{proposition}

\begin{proof}
 See Appendix \ref{sec:proofofProp:transformation-formulas-generalcasewithoutassumptions}.
\end{proof}

%%%%%%%%%%%%%%%%%%%%%%%%%%%%%%%%%%%%%%%

\subsection{Null frame transformations on $\RR$}

%%%%%%%%%%%%%%%%%%%%%%%%%%%%%%%%%%%%%%%

%%%%%%%%%%%%%%%%%%%%%%%%%%%%%%%%%%%%%%%%%%%%%%

\subsubsection{Transformation formulas in a particular case}

%%%%%%%%%%%%%%%%%%%%%%%%%%%%%%%%%%%%%%%%%%%%%% 

In what follows we revisit some of   the  transformation formulas of Proposition \ref{Prop:transformation-formulas-generalcasewithoutassumptions} in the particular case where   the frame    $(e_3, e_4, e_1, e_2)$ is attached  to the geodesic foliation of $\RR$, while  $(e'_3, e'_4, e'_1, e'_2)$ is  an arbitrary    frame.  Recall that since  the unprimed frame is attached to the  geodesic foliation  we have
\beaa
\atrch=\atrchb=0, \qquad \xi=\om=0, \qquad \etab+\ze=0.
\eeaa
{\bf Notation.}  In the proposition below we write the error terms $\err$ schematically according to the following convention.
\begin{itemize}
\item We introduce the notation
\bea
F:= \{ f, \fb,\ovla\}, \qquad\qquad  \ovla:= \la-1.
\eea
$F^k$ denotes  an  arbitrary homogeneous  polynomial  of degree $k$ in the variables $F$. 
\item  $F^k \c \Ga$ denotes an arbitrary linear combination of  elements of $\Ga$ with coefficients in $F^k $. 
\item  Since the  components  of $F$ are supposed  to be small in all  our  applications here  we  ignore  $F^{k+1} \c \Ga$ if  $F^{k} \c \Ga$ appear among the error  terms.
\end{itemize}

\begin{proposition}
\lab{Prop:transformation formulas-integrtogeneral}
Under a  transformation of type \eqref{eq:Generalframetransf}, in the particular case where the frame $(e_3, e_4, e_1, e_2)$ is the one attached  to the geodesic foliation of $\RR$, and under the assumption 
\bea
|F|  \ll 1, 
\eea
the Ricci coefficients $\chi$, $\chib$ and  $\ze$, and the curvature component $\rho$,  transform as follows:
\begin{itemize}
\item The transformation formulas for $\chi $ are  given by 
\bea
\bsplit
\trch' &= \la\trch  +  \div'f  +\err(\trch,\trch'),\\
\err(\trch,\trch') &=   F\c\Ga_b+r^{-1}F^2+F\c\nab' F,
\end{split}
\eea
\bea
\bsplit
\atrch' &=   \la\curl'f  +\err(\atrch,\atrch'),\\
\err(\atrch,\atrch') &= F\c\Ga_b+r^{-1}F^2+F\c\nab' F,  
\end{split}
\eea
\bea
\bsplit
\chih' &= \chih  +  \nab'\hot f +\err(\chih,\chih'),\\
\err(\chih,\chih') &= F\c\Ga_b+r^{-1}F^2+F\c\nab' F.
\end{split}
\eea

\item The transformation formulas for $\chib $ are  given by 
\bea
\bsplit
\trchb' &= \la^{-1}\trchb +\div'\fb   +\err(\trchb, \trchb'),\\
\err(\trchb, \trchb') &= F\c\Ga_b+r^{-1}F^2+F\c\nab' F,
\end{split}
\eea
\bea
\bsplit
\atrchb' &= \la^{-1}\curl'\fb +\err(\atrchb, \atrchb'),\\
\err(\atrchb, \atrchb') &= F\c\Ga_b+r^{-1}F^2+F\c\nab' F,   
\end{split}
\eea
\bea
\bsplit
\chibh' &= \chibh +\nab'\hot\fb   +\err(\chibh, \chibh'),\\
\err(\chibh, \chibh') &= F\c\Ga_b+r^{-1}F^2+F\c\nab' F.
\end{split}
\eea

\item  The transformation formula for $\ze$ is given by 
\bea
\bsplit
\ze' &= \ze -\nab'\la  -\frac{1}{4}\trchb f  -\omb f +\frac{1}{4}\fb\trch +\err(\ze, \ze'),\\
\err(\ze, \ze') &= F\c\Ga_b+r^{-1}F^2+F\c\nab' F.
\end{split}
\eea

\item The transformation formula for $\rho$ is given by
  \bea
  \bsplit
 \rho' &= \rho +\err(\rho, \rho'),\\
\err(\rho, \rho') &= r^{-1}F\c\Ga_b +r^{-3}F^2.
   \end{split}
  \eea
\end{itemize}
\end{proposition}

\begin{proof}
Since     $(e_3, e_4, e_1, e_2)$ denotes the frame attached  to the geodesic foliation of $\RR$,   we have
\beaa
\atrch=\atrchb=\xi=\om=0, \qquad \etab+\ze=0.
\eeaa
The proposition then immediately follows from plugging these relations in Proposition \ref{Prop:transformation-formulas-generalcasewithoutassumptions} and getting rid of the $\Ga'$ on the RHS thanks to the assumption $|F|  \ll 1$. 
\end{proof}

\begin{remark}
For convenience in what follow  we will use the following notation
\bea
\aka=\atrch, \quad  \akab= \atrchb, \quad   \ka=\trch,  \quad  \kab=\trchb.
\eea
\end{remark}

%%%%%%%%%%%%%%%%%%%%%%%%%%%%%%%%%%%

\subsubsection{Schematic presentation for higher order error  terms}

%%%%%%%%%%%%%%%%%%%%%%%%%%%%%%%%%%%

We  introduce the following  schematic  presentation of the  error terms  which appear  in various calculations below.
\begin{definition}
\lab{Definition:errorterms-prime}
We denote by $\err_k$, $k=1,2$,   error terms\footnote{Note  however that  the precise error terms differ in each particular case and that we only emphasize here  their general structure.} which can be  written schematically in the form,
\bea
\bsplit
r\err_1&= F\c  (r\Ga_b)+   F^2   +   F\c  (r \nab') F= F\c  (r\Ga_b)+ F\c  (r \nab')^{\le 1} F,\\
r^2 \err_2 &= ( r\nab' )^{\le 1}( r\err_1)+( F+\Ga_b )\c  r \dk \Ga_b.
\end{split}
\eea
\end{definition}

%%%%%%%%%%%%%%%%%%%%%%%%%%%%%%%%%%%%

\subsubsection{Transformation formula for the mass aspect function}

%%%%%%%%%%%%%%%%%%%%%%%%%%%%%%%%%%%%%

We start with the following
\begin{lemma}
\lab{Le:transf-mumu'}
The mass aspect function  $\mu=-\div \ze -\rho+\frac{1}{2}\chih\c \chibh $ verifies the transformation  formula.
\bea
\bsplit
\mu'&= \mu -\div' \left(- \nab' \la
  -\left(\omb +\frac 1 4 \kab \right) f +\left(\om +\frac 1 4 \ka \right) \fb\right) +\err_2(\mu, \mu'),
  \end{split}
\eea
with  $\err_2(\mu, \mu')$ an $\err_2$ type  term  as in Definition \ref{Definition:errorterms-prime}.
\end{lemma}
\begin{proof}
 Using the above transformation formulas for $\ze, \rho, \chih, \chibh$  we easily derive
\beaa
\div' \ze'&=&\div'\Big(\ze- \nab'\la
  - \frac 1 4\trchb f  + \frac 1 4 \trch \fb + \om\fb -\omb f +\err(\ze, \ze')\Big)\\
  &=&\div\ze +\div' \Big(- \nab' \la 
  - \frac 1 4\trchb f  + \frac 1 4 \trch \fb + \om\fb -\omb f\Big) 
  \\
  &+& \big(\div'-\div \big)\ze     +   \div' \err(\ze, \ze'),  \\
  \rho'&=&\rho+\err(\rho, \rho'),\\
  \chih'\c \chibh'&=&\chih\c\chibh +  \nab' F\c \nab' F +(\nab' F+ r^{-1} F +F^2 )\c \Ga_b. 
    \eeaa
    Note also that, using the equations for $\nab_3\ze, \nab_4\ze$,
    \beaa
     \big(\div'-\div \big)\ze&=&  f\c \nab_3\ze+\fb \nab_4\ze +\lot  = r^{-1} F\c \dk \Ga_b.
     \eeaa
    
   We deduce,
   \beaa
   \mu'&=& \mu -\div' \Big(- \nab'\la
  - \frac 1 4\trchb f  + \frac 1 4 \trch \fb + \om\fb -\omb f\Big) +\err(\mu, \mu')
   \eeaa
   with
   \beaa
   \err(\mu, \mu')&=&-\div' \, \err(\ze, \ze') -\err(\rho, \rho') +\Ga_g\c \Ga_b +  \nab' F\c \nab' F\\
   & +&(\nab' F+ r^{-1} F +F^2 )\c \Ga_b +           r^{-1} F\c \dk \Ga_b.
   \eeaa
   Thus,  taking into account the  structure of the terms in          $\err(\ze, \ze')$ and  $ \err(\rho, \rho') $   we  can write  schematically,
   \beaa
   \err(\mu, \mu')&=&\nab' \, \err_1  + r^{-1}  \err_1  + r^{-2} \Ga_b \c (r\Ga_b)+  r^{-2}  (  r  \nab' F)\c  (r\nab' F) \\
      &+&r^{-2}( F+\Ga_b )\c  (r\dk \Ga_b).
   \eeaa
   Hence\footnote{Note that $ \nab'(r)=\frac 1 2 f e_3(r)+\frac 1 2 \fb e_4(r)$ and hence  the term  $\nab'(r)  (r\err_1)$ is a lower order  term.},
   \beaa
   r^2 \err(\mu, \mu')&=&  r^2  \nab' \, \err_1+ r\err_1+ ( F+\Ga_b )\c  r \dk \Ga_b +(  r  \nab' F)\c  (r\nab' F)\\
   &=& r  \nab' \,( r \err_1)   +\nab'(r)  (r\err_1)  + r\err_1+ ( F+\Ga_b )\c  r \dk \Ga_b +(  r  \nab' F)\c  (r\nab' F)\\
   &=&  r  \nab' \,( r \err_1) + r\err_1 + ( F+\Ga_b )\c  r \dk \Ga_b +(  r  \nab' F)\c  (r\nab' F).
   \eeaa
   We simplify the expression by  including the terms  $(  r  \nab' F)\c  (r\nab' F)$ in the expression   $ 
     ( r\nab' )^{\le 1}( r\err_1)$. Hence,
   \beaa
   r^2 \err(\mu, \mu')&=& ( r\nab' )^{\le 1}( r\err_1)+  ( F+\Ga_b )\c  r \dk \Ga_b
   \eeaa
   as stated. 
\end{proof}

%%%%%%%%%%%%%%%%%%%%%%%%%%%%%%%%%%%%
  
\subsubsection{Transformation formulas for the main GCM quantities}

%%%%%%%%%%%%%%%%%%%%%%%%%%%%%%%%%%%%

We consider below the equations on $f, \fb, \la$ induced by the transformation formulas  for $\ka, \kab, \aka,  \akab, \mu$. Those will play a fundamental role in the  definition of GCM spheres.

\begin{lemma}
The following relations hold true for any frame $(e_1', e_2', e_3', e_4')$ connected to the reference frame $(e_1, e_2, e_3, e_4)$ by the transition coefficients $(f, \fb, \la)$.
\lab{Lemma: GCMS'}
 \bea
     \lab{GCMS'-1}
     \bsplit
\curl ' f  &=\aka'-\err_1(\aka, \aka'),\\
\curl ' \fb  &=\akab'-\err_1(\akab, \akab'),\\
\div' f + \ka \ovla &= \ka' -\ka -\err_1(\ka,\ka'),\\
\div' \fb - \kab \ovla &= \kab' -\kab -\err_1(\kab,\kab'),
\end{split}
\eea
and
\bea
 \lab{GCMS'-2}
  -\div' \Big(- \nab' \ovla
  -\left(\omb +\frac 1 4 \kab \right) f +\big(\om +\frac 1 4 \ka) \fb\Big)&=&\mu'-\mu  -\err_2(\mu, \mu'),
\eea
with error $ \err_1, \err_2$ error terms  as in Definition \ref{Definition:errorterms-prime}, and where we recall $\ovla=\la-1$.
\end{lemma}

\begin{proof}
The proof follows immediately from Proposition \ref{Prop:transformation formulas-integrtogeneral}, Lemma \ref{Le:transf-mumu'} and Definition \ref{Definition:errorterms-prime}.
\end{proof}

%%%%%%%%%%%%%%%%%%%%%%%%%%%%%%%%%%%%%%%%%%

\section{GCM spheres}
\lab{section:GCMspheres}

%%%%%%%%%%%%%%%%%%%%%%%%%%%%%%%%%%

%%%%%%%%%%%%%%%%%%%%%%%%%%%%%%%%%%

\subsection{Particular case of adapted spheres}

%%%%%%%%%%%%%%%%%%%%%%%%%%%%%%%%%%

We apply the results of Lemma \ref{Lemma: GCMS'}   to the  case   when the  prime frame  is adapted to  
  an  $O(\epg) $-sphere  $\S\subset\RR$ as defined in  section \ref{section:epg-spheres}, i.e. the primed frame coincides with $(e_1^\S, e_2^\S,  e_3^\S, e_4^\S)$  where $e_1^\S, e_2^\S$ are tangent to $\S$. Moreover we assume that
   $\S$ is endowed  with a basis $\JpS$  of  $\ell=1$ modes.     We  denote   by $r^\S$ the area radius of $\S$ and by    and by  $\nab^\S, \div^\S, \curl^\S, \lap^\S$  the standard differential  operators on $\S$.   We also  denote by  $\Ga^\S, R^\S$      the corresponding Ricci coefficients  and curvature components, by $\mu^\S$ the corresponding mass aspect function,  and by $m^\S$ the corresponding Hawking mass.

   \begin{remark}
   \lab{remark:welldefinedGa}
    Note  that  while  the Ricci coefficients $\ka^\S, \kab^\S,  \chih^\S, \chibh^\S, \ze^\S$ as well as all curvature  components $\a^\S, \b^\S, \rho^\S, \rhod^\S, \bb^\S, \aa^\S$ and   mass aspect function $\mu^\S$     are well defined on $\S$, this in not the case  of $\eta^\S, \etab^\S, \xi^\S, \xib^\S, \om^\S, \omb^\S$ which require  the  derivatives of the frame in the $e_3^\S$ and $e_4^\S$ directions.  Taking this observation into account, the GCM construction will only involve the quantities well defined on $\S$.
    \end{remark}

 We rewrite the equations \eqref{GCMS'-1} \eqref{GCMS'-2} in the following form,
 \bea
     \lab{GCMS-1S}
     \bsplit
\curl ^\S f  &=\aka^\S-\err_1(\aka, \aka^\S),\\
\curl^\S \fb  &=\akab^\S-\err_1(\akab, \akab^\S),\\
\div^\S f + \ka \ovla &= \ka^\S -\ka -\err_1(\ka,\ka^\S),\\
\div^\S\fb - \kab \ovla &= \kab^\S -\kab -\err_1(\kab,\kab^\S),
\end{split}
\eea
and
\bea
 \lab{GCMS-2S}
  -\div^\S\left(- \nab^\S \ovla
  -\left(\omb +\frac 1 4 \kab \right) f +\left(\om +\frac 1 4 \ka\right) \fb\right)&=&\mu^\S-\mu  -\err_2(\mu, \mu^\S),
\eea
with error $ \err_1, \err_2$ error terms  as in Definition \ref{Definition:errorterms-prime} and which we repeat below.

\begin{definition}
\lab{Definition:errorterms}
We denote by $\err_k$, $k=1,2$,   error terms\footnote{Note  however that  the precise error terms differ in each particular case and that  here  we only emphasize  their general structure.} which can be  written schematically in the form,
\bea
\bsplit
r\err_1&=  F\c  (r\Ga_b)+ F\c  (r \nab^\S)^{\le 1} F+r^{-1} F,\\
r^2 \err_2 &= ( r\nab^\S )^{\le 1}( r\err_1)+( F+\Ga_b )\c  r \dk \Ga_b. 
\end{split}
\eea
\end{definition}

    Using these conventions we   rewrite equation \eqref{GCMS-1S} \eqref{GCMS-2S} in the following form.
    \begin{lemma} 
    \lab{Lemma-adaptedGCM-equations}
    The following relations hold true for any adapted  frame $(e_1^\S, e_2^\S, e_3^\S,  e_4^\S)$ to a given sphere $\S$  connected to the reference frame $(e_1, e_2, e_3, e_4)$ by the transition coefficients $(f, \fb, \la)$, with $\ovla=\la-1$,    
     \bea
     \lab{GCMS-4S}
     \bsplit
\curl ^\S f  &=\aka^\S-\err_1[\curl^\S  f ],\\
\curl^\S \fb  &=\akab^\S-\err_1[\curl^\S  \fb ],\\
\div^\S f + \ka \ovla &= \ka^\S -\ka -\err_1[\div^\S f ],\\
\div^\S\fb - \kab \ovla &= \kab^\S -\kab -\err_1[\div^\S \fb ],\\
\lap^\S\ovla + V\ovla &=\mu^\S-\mu -\left(\omb +\frac 1 4 \kab \right) \big(\ka^\S-\ka \big)+\left(\om +\frac 1 4 \ka \right) \big(\kab^\S-\kab \big)+\err_2[ \lap^\S\ovla],
\end{split}
\eea  
where  the error terms  $\err_1[ \curl^\S f],  \err_1[  \div^\S  f ]$,  $\err_1[ \curl^\S_1 \fb],  \err_1[  \div^\S  \fb ]$  and $\err_2[ \lap^\S\ovla] $   are consistent with Definition
\ref{Definition:errorterms} but their  exact expressions   differ, of course, for each equation, and where $V$ is given by
\bea
V &:=& -\left(\frac 1 2 \ka \kab +\ka \omb+\kab \om\right).
\eea 
    \end{lemma}

\begin{proof}    
We   rewrite equation \eqref{GCMS-2S} in the form
\beaa
\Delta^\S\ovla&=&-\left(\omb +\frac 1 4 \kab \right) \div^\S f +\left(\om +\frac 1 4 \ka\right) \div^\S \fb+\mu^\S-\mu+  \err_2\\
&-&\nab^\S \left(\omb +\frac 1 4 \kab \right) \c  f +\nab^\S \left(\om +\frac 1 4 \ka\right) \c \fb
\eeaa
In view of the transformation formulas \eqref{eq:Generalframetransf},  for  every scalar Ricci coefficient $\Ga$,
\beaa
e_a^\S ( \Ga )&=&   \left(\de_{ab} +\frac{1}{2}\fb_af_b\right) e_b ( \Ga)  
+\frac 1 2  \fb_a\nab_4\Ga +\left(\frac 1 2 f_a +\frac{1}{8}|f|^2\fb_a\right) \nab_3\Ga.
\eeaa
Thus, we can easily check  that $- \nab^\S (\omb +\frac 1 4 \kab) \c  f +\nab^\S(\om +\frac 1 4 \ka) \c \fb$
is an $\err_2$ term. Making use of the $ \div^\S$ equations in \eqref{GCMS-1S} we  deduce,
\beaa
 \Delta^\S\ovla&=&-\left(\omb +\frac 1 4 \kab \right) \big(-\ka \ovla +\ka^\S-\ka +\err_1\big)
+\big(\om +\frac 1 4 \ka)\big(\kab \ovla +\kab^\S-\kab +\err_1\big)\\
&+& \mu^\S-\mu+\err_2\\
&=&- V\ovla + \mu^\S-\mu -\left(\omb +\frac 1 4 \kab \right) \big(\ka^\S-\ka \big)+\left(\om +\frac 1 4 \ka \right) \big(\kab^\S-\kab \big)+\err_2
\eeaa
where,
\beaa
V&=&-\left(\frac 1 2 \ka \kab +\ka \omb+\kab \om\right)
\eeaa
i.e.,
\beaa
\Delta^\S\ovla + V\ovla &=&\mu^\S-\mu -\left(\omb +\frac 1 4 \kab \right) \big(\ka^\S-\ka \big)+\left(\om +\frac 1 4 \ka \right) \big(\kab^\S-\kab \big)+\err_2
\eeaa
as stated.
\end{proof}

In \eqref{GCMS-4S}, the terms $\ka^\S -\ka$ and $\kab^\S -\kab$ on the right-hand side of the equations for $\div^\S f$ and $\div^\S\fb$ contain in fact implicitly a linear term, proportional to the scalar  $r-r^\S$, which will be denoted by the auxiliary function $\ovb$ below.  This term should be put on the left-hand side which is the purpose of the following reformulation of \eqref{GCMS-4S}.
    \begin{lemma} 
    \lab{Lemma-adaptedGCM-equations:bis}
    Under the assumptions of Lemma \ref{Lemma-adaptedGCM-equations}, the following relations hold true   
     \bea
     \lab{GCMS-4Sbis}
     \bsplit
\curl ^\S f  &=\aka^\S-\err_1[\curl^\S  f ],\\
\curl^\S \fb  &=\akab^\S-\err_1[\curl^\S  \fb ],\\
\div^\S f + \ka \ovla -\frac{2}{(r^\S)^2}\ovb &= \ka^\S-\frac{2}{r^\S} -\left(\ka-\frac{2}{r}\right) -\err_1[\div^\S f ] -\frac{2(r-r^\S)^2}{r(r^\S)^2},\\
\div^\S\fb - \kab \ovla+\frac{2}{(r^\S)^2}\ovb &= \kab^\S+\frac{2}{r^\S} -\left(\kab+\frac{2}{r}\right) -\err_1[\div^\S \fb ]+\frac{2(r-r^\S)^2}{r(r^\S)^2},\\
\Delta^\S\ovla + V\ovla &=\mu^\S-\mu -\left(\omb +\frac 1 4 \kab \right) \big(\ka^\S-\ka \big)\\
&+\left(\om +\frac 1 4 \ka \right) \big(\kab^\S-\kab \big)+\err_2[ \lap^\S\ovla],\\
\Delta^\S\ovb &= \div^\S\left(\frac{e_4(r)}{2}\fb   +\frac{e_3(r)}{2}\left(f +\frac{1}{4}|f|^2\fb\right)\right) , \quad \ov{\ovb}^\S=\ov{r}^\S-r^\S.
\end{split}
\eea  
    \end{lemma}

\begin{remark}
Though      $\ovb=r-r^\S$,   in  the treatment of the  system \eqref{GCMS-4Sbis}  we  will consider  it simply as a 
 solution its corresponding  elliptic equation.
\end{remark}

 \begin{proof}
Recall \eqref{GCMS-4S}
     \beaa
     \bsplit
\curl ^\S f  &= -\err_1[\curl^\S  f ],\\
\curl^\S \fb  &= -\err_1[\curl^\S  \fb ],\\
\div^\S f + \ka \ovla &= \ka^\S -\ka -\err_1[\div^\S f ],\\
\div^\S\fb - \kab \ovla &= \kab^\S -\kab -\err_1[\div^\S \fb ],\\
\Delta^\S\ovla + V\ovla &=\mu^\S-\mu -\left(\omb +\frac 1 4 \kab \right) \big(\ka^\S-\ka \big)+\left(\om +\frac 1 4 \ka \right) \big(\kab^\S-\kab \big)+\err_2[ \lap^\S\ovla].
\end{split}
\eeaa  
We rewrite this system as
     \beaa
     \bsplit
\curl ^\S f  &= -\err_1[\curl^\S  f ],\\
\curl^\S \fb  &= -\err_1[\curl^\S  \fb ],\\
\div^\S f + \ka \ovla -\left(\frac{2}{r^\S}-\frac{2}{r}\right) &= \ka^\S-\frac{2}{r^\S} -\left(\ka-\frac{2}{r}\right) -\err_1[\div^\S f ],\\
\div^\S\fb - \kab \ovla+\left(\frac{2}{r^\S}-\frac{2}{r}\right) &= \kab^\S+\frac{2}{r^\S} -\left(\kab+\frac{2}{r}\right) -\err_1[\div^\S \fb ],\\
\Delta^\S\ovla + V\ovla &=\mu^\S-\mu -\left(\omb +\frac 1 4 \kab \right) \big(\ka^\S-\ka \big)+\left(\om +\frac 1 4 \ka \right) \big(\kab^\S-\kab \big)+\err_2[ \lap^\S\ovla].
\end{split}
\eeaa  
Now, we have
\beaa
\frac{2}{r^\S}-\frac{2}{r} &=& \frac{2(r-r^\S)}{rr^\S}=\frac{2(r-r^\S)}{(r^\S)^2} -\frac{2(r-r^\S)^2}{r(r^\S)^2}\\
&=&\frac{2\ovb}{(r^\S)^2} -\frac{2(r-r^\S)^2}{r(r^\S)^2}
\eeaa
where
\beaa
\ovb=r-r^\S,
\eeaa 
which allows us to rewrite the system in the desired form \eqref{GCMS-4Sbis}
      \beaa
     \bsplit
\curl ^\S f  &= -\err_1[\curl^\S  f ],\\
\curl^\S \fb  &= -\err_1[\curl^\S  \fb ],\\
\div^\S f + \ka \ovla  -\frac{2}{(r^\S)^2}\ovb &= \ka^\S-\frac{2}{r^\S} -\left(\ka-\frac{2}{r}\right) -\err_1[\div^\S f ] -\frac{2(r-r^\S)^2}{r(r^\S)^2},\\
\div^\S\fb - \kab \ovla+\frac{2}{(r^\S)^2}\ovb &= \kab^\S+\frac{2}{r^\S} -\left(\kab+\frac{2}{r}\right) -\err_1[\div^\S \fb ] +\frac{2(r-r^\S)^2}{r(r^\S)^2},
\end{split}
\eeaa
\beaa
\Delta^\S\ovla + V\ovla &=&\mu^\S-\mu -\left(\omb +\frac 1 4 \kab \right) \big(\ka^\S-\ka \big)+\left(\om +\frac 1 4 \ka \right) \big(\kab^\S-\kab \big)+\err_2[ \lap^\S\ovla].
\eeaa

Next, we derive an equation for $\ovb$. Recall that 
\beaa
\ovb=r-r^\S.
\eeaa
In particular,  the scalar $\ovb$ is uniquely defined by 
\beaa
\Delta^\S\ovb = \Delta^\S\left(r-r^\S\right), \qquad \ov{\ovb}^\S=\ov{r}^\S-r^\S.
\eeaa
Note also that we have, using the null frame transformation from the background frame to the frame of $\S$, 
\beaa
 \nab^\S_a(r) &=&  \left(\left(\de_{ab} +\frac{1}{2}\fb_af_b\right) e_b +\frac 1 2  \fb_a  e_4 +\left(\frac 1 2 f_a +\frac{1}{8}|f|^2\fb_a\right)   e_3\right)r\\
 &=&  \left(\frac 1 2  \fb_a  e_4 +\left(\frac 1 2 f_a +\frac{1}{8}|f|^2\fb_a\right)   e_3\right)r
\eeaa
and hence
\beaa
 \nab^\S\left(r-r^\S\right) &=& \nab^\S(r)\\
 &=&   \frac{e_4(r)}{2}\fb   +\frac{e_3(r)}{2}\left(f +\frac{1}{4}|f|^2\fb\right). 
  \eeaa
Thus,   the scalar $\ovb$ is uniquely defined by
\beaa
\Delta^\S\ovb=\div^\S \nab^\S\left(r-r^\S\right) = \div^\S\left(\frac{e_4(r)}{2}\fb  +\frac{e_3(r)}{2}\left(f +\frac{1}{4}|f|^2\fb\right)\right) , \qquad \ov{\ovb}^\S=\ov{r}^\S-r^\S,
\eeaa
as desired.
\end{proof}

Finally, we rewrite \eqref{GCMS-4Sbis} as follows.
    \begin{corollary} 
    \lab{Lemma-adaptedGCM-equations:ter}
    Under the assumptions of Lemma \ref{Lemma-adaptedGCM-equations}, the following relations hold true   
     \bea
     \lab{GCMS-4Ster}
     \bsplit
\curl ^\S f  &=\aka^\S-\err_1[\curl^\S  f ],\\
\curl^\S \fb  &=\akab^\S-\err_1[\curl^\S  \fb ],\\
\div^\S f + \ka \ovla -\frac{2}{(r^\S)^2}\ovb &= \ka^\S-\frac{2}{r^\S} -\left(\ka-\frac{2}{r}\right) -\err_1[\div^\S f ] -\frac{2(r-r^\S)^2}{r(r^\S)^2},\\
\div^\S\fb - \kab \ovla +\frac{2}{(r^\S)^2}\ovb &= \kab^\S+\frac{2}{r^\S} -\left(\kab+\frac{2}{r}\right) -\err_1[\div^\S \fb ] +\frac{2(r-r^\S)^2}{r(r^\S)^2},\\
\Delta^\S\ovla + V\ovla &=\mu^\S-\mu -\left(\omb +\frac 1 4 \kab \right) \big(\ka^\S-\ka \big)\\
&+\left(\om +\frac 1 4 \ka \right) \big(\kab^\S-\kab \big)+\err_2[ \lap^\S\ovla],\\
\Delta^\S\ovb &= \frac{1}{2}\div^\S\left(\fb - \Up f +\err_1[\Delta^\S \ovb ] \right) , \quad \ov{\ovb}^\S=\ov{r}^\S-r^\S,
\end{split}
\eea  
where  the error term  $\err_1[ \Delta^\S\ovb]$  is consistent with Definition
\ref{Definition:errorterms}.
    \end{corollary}

\begin{proof}
In view of \eqref{GCMS-4Sbis}, we only need to focus on the equation for $\ovb$.  
We have 
\beaa
&&\frac{e_4(r)}{2}\fb   +\frac{e_3(r)}{2}\left(f +\frac{1}{4}|f|^2\fb\right) \\
&=& \frac{1}{2}\left(\fb - \Up f +(e_4(r)-1)\fb   +(e_3(r)+\Up)f +\frac{e_3(r)}{4}|f|^2\fb\right)\\
&=& \frac{1}{2}\Big(\fb - \Up f +\err_1[\Delta^\S \ovb ]\Big)  
\eeaa
where
\beaa
\err_1[\Delta^\S \ovb ] &:=& (e_4(r)-1)\fb   +(e_3(r)+\Up)f +\frac{e_3(r)}{4}|f|^2\fb\\
&=& r\Ga_b F   + F\c F
\eeaa
so that $\err_1[\Delta^\S \ovb ]$ is indeed consistent with Definition \ref{Definition:errorterms}. 
\end{proof}

%%%%%%%%%%%%%%%%%%%%%%%%%%%%%%%%%%%%%%%

\subsection{Definition of GCM spheres}

%%%%%%%%%%%%%%%%%%%%%%%%%%%%%%%%%%%%%%%

\begin{definition}
\lab{definition:GCMS}
We say that $\S\subset \RR$, endowed with an adapted  frame\footnote{i.e.$ (e^\S_1, e_2 ^\S)$ are tangent to $\S$.}   $(e_1^\S, e_2^\S, e_3^\S, e_4^\S)$,   is a general  covariant modulated (GCM) sphere    if  the following hold true:
\bea
\lab{def:GCMC:00}
\bsplit
\ka^\S&=\frac{2}{r^\S},\\
\kab^\S &=-\frac{2}{r^\S}\Up^\S+  \Cb^\S_0+\sum_p \CbpS \JpS,\\
\mu^\S&= \frac{2m^\S}{(r^\S)^3} +   M^\S_0+\sum _p\MpS \JpS,
\end{split}
\eea
for some constants $\Cb^\S_0,\,  \CbpS, \, M^\S_0, \, \MpS, \, p\in\{-,0, +\}$.
 In addition, since the $\S$- frame  is  automatically integrable, we also have
 \bea
 \label{Conditions:GCMS-automatic}
\aka^\S=\akab^\S=0.
\eea
\end{definition}

We will construct our GCM spheres in Theorem \ref{Theorem:ExistenceGCMS1}  under the following assumptions,
\bea
\lab{assumptions:oldGCMS1}
\bsplit
\ka&=\frac{2}{r}+\dot{\ka},\\
\kab&=-\frac{2\Up}{r} +  \Cb_0+\sum_p \Cbp \Jp+\dot{\kab},\\
\mu&= \frac{2m^\S}{r^3} + M_0+\sum _p\Mp \Jp+\dot{\mu},
\end{split}
\eea
where
\bea
\lab{assumptions:oldGCMS2}
\bsplit
|\Cb_0, \Cbp| &\les r^{-2} \epg, \qquad  |M_0, \Mp| \les r^{-3} \epg,
\\
\big\| \kadot, \kabdot\|_{\hk_{s_{max}}(\S)}&\les r^{-1}\dg,\qquad 
\big\|\mudot\| _{\hk_{s_{max}}(\S)}\les r^{-2}\dg.
\end{split}
\eea
  In view of the GCM conditions we deduce,
   \beaa
  \ka^\S-\ka&=&\frac{2}{r^\S} -\frac 2  r -\kadot,\\
   \kab^\S -\kab&=& -\frac{2}{r^\S}\Up^\S+\frac 2 r \Up  +\big(\Cb^\S_0-\Cb_0\big)+ \sum_p\big( \CbpS -\Cbp\big) \JpS\\
   &+&\Cbp \big(\JpS -\Jp)-\kabdot,\\
   \mu^\S-\mu&=&\frac{2m^\S}{(r^\S)^3}-\frac{2m}{r^3} +   \big(M^\S_0-M_0) + \sum _p\big(\MpS- \Mp\big) \JpS\\
    &+&\Mp \big(\JpS -\Jp)-\mudot,
   \eeaa
   or, introducing the notation,
      \bea\lab{eq:notationforCbdotMdotandCbpdot}
   \bsplit
   \Cbdot_0:&=  \Cb^\S_0-\Cb_0, \qquad\qquad\,\,
   \Mdot_0:=  M^\S_0-  M_0,\\
   \Cbpdot:&=  \CbpS - \Cbp, \qquad  \Mpdot:= \MpS-   \Mp,
   \end{split}
   \eea
   we write,
   \bea
   \lab{eq:differencesofGCMS}
   \bsplit
  \ka^\S-\ka&=\frac{2}{r^\S} -\frac 2  r -\kadot,\\
   \kab^\S -\kab&=  \Cbdot_0+\sum_p \Cbpdot \JpS  -\frac{2}{r^\S}\Up^\S+\frac 2 r \Up +\Cbp \big(\JpS -\Jp)-\kabdot,\\
   \mu^\S-\mu&= \Mdot_0+\sum _p\Mpdot \JpS +\frac{2m^\S}{(r^\S)^3}-\frac{2m}{r^3} 
    +\Mp \big(\JpS -\Jp)-\mudot.
    \end{split}
   \eea

 %%%%%%%%%%%%%%%%%%%%%%%%%%%%%%%%%%%%%%
 
   \subsection{Linearized GCM equations}
   
%%%%%%%%%%%%%%%%%%%%%%%%%%%%%%%%%%%%%%%   

\begin{definition}
\lab{definition:GCMSgen-equations}
Let   $\S\subset\RR$  a  smooth  $O(\epg)$-sphere. We say that $F=(f, \fb, \ovla)$  verifies
 the linearized GCM system on $\S$ if the following holds true,
 \bea
 \lab{GeneralizedGCMsystem}
\bsplit
\curl ^\S f &= h_1 -\ov{h_1}^\S,\\
\curl^\S \fb&= \underline{h}_1 - \ov{\underline{h}_1}^\S,\\
\div^\S f + \frac{2}{r^\S} \ovla  -\frac{2}{(r^\S)^2}\ovb &=  h_2,
\\
\div^\S\fb + \frac{2}{r^\S} \ovla +\frac{2}{(r^\S)^2}\ovb   
&=   \Cbdot_0+\sum_p \Cbpdot \JpS +\underline{h}_2,\\
\left(\Delta^\S+\frac{2}{(r^\S)^2}\right)\ovla  &=  \Mdot_0+\sum _p\Mpdot \JpS+\frac{1}{2r^\S}\left(\Cbdot_0+\sum_p \Cbpdot \JpS\right) +h_3,\\
\Delta^\S\ovb-\frac{1}{2}\div^\S\Big(\fb - f\Big) &= h_4 -\ov{h_4}^\S , \qquad \ov{\ovb}^\S=b_0,
\end{split}
\eea
 for  some choice of  constants $\Cbdot_0, \Mdot_0, \Cbpdot, \Mpdot$, $b_0$, and scalar functions $h_1$, $h_2$, $h_3$, $h_4$, $\underline{h}_1$, $\underline{h}_2$.
 \end{definition}
 
 \begin{remark}
 The system  \eqref{GeneralizedGCMsystem} is naturally connected  to the system  \eqref{GCMS-4Ster} 
 and  the notation  introduced in   \eqref{assumptions:oldGCMS1} and \eqref{eq:notationforCbdotMdotandCbpdot}
  with the following choices of  terms $h_1, \hb_1, h_2, \hb_2, h_3, h_4$.
  \beaa
h_1 =  -\err_1[\curl^\S  f ], \qquad\qquad \underline{h}_1=-\err_1[\curl^\S  \fb ],
\eeaa
\beaa
h_2 &=& -\left(\ka-\frac{2}{r^\S}\right) \ovla  + \ka^\S-\frac{2}{r^\S} -\kadot -\err_1[\div^\S f ] -\frac{2(r-r^\S)^2}{r(r^\S)^2}, \\ 
\underline{h}_2 &=& \left(\kab+\frac{2}{r^\S}\right) \ovla+ \left(\kab^\S+\frac{2\Up^\S}{r^\S} - \Cb^\S_0-\sum_p \CbpS \JpS\right)  -\kabdot + \frac{4m^\S}{(r^\S)^2} - \frac{4m}{r^2}\\
&& -\err_1[\div^\S \fb ]+\frac{2(r-r^\S)^2}{r(r^\S)^2},
\eeaa
\beaa
h_3 &=& -\left(V-\frac{2}{(r^\S)^2}\right)\ovla+ \left(\mu^\S -\frac{2m^\S}{(r^\S)^3}-M^\S_0-\sum _p\MpS \JpS\right) \\
&&-\mudot  +\frac{2m^\S}{(r^\S)^3}-\frac{2m}{r^3} -\left(\omb +\frac 1 4 \kab \right) \big(\ka^\S-\ka \big)+\left(\om +\frac 1 4 \ka \right) \big(\kab^\S-\kab \big)\\
&&-\frac{1}{2r^\S}\left(\Cbdot_0+\sum_p \Cbpdot \JpS\right)+\err_2[ \lap^\S\ovla], 
\eeaa
\beaa
h_4 = \frac{1}{2}\div^\S\left(\frac{2m}{r}f+\err_1[\Delta^\S \ovb ]\right),\qquad \qquad b_0 = \ov{r}^\S-r^\S.
\eeaa
In fact, with these choices,  the system  \eqref{GeneralizedGCMsystem} corresponds  precisely  to    \eqref{GCMS-4Ster} provided that
 we also have 
 \bea
 \lab{eq:cancellationofaverageerrakaakabneededfor} 
 \ov{\err_1[\curl^\S  f ]}^\S= \ov{\err_1[\curl^\S  \fb ]}^\S=0.
 \eea
 \end{remark}

 \begin{remark}
 \lab{remark-important-curls}
The following remarks motivate the introduction of the system  \eqref{GeneralizedGCMsystem}.
\begin{enumerate} 
\item The cancellation \eqref{eq:cancellationofaverageerrakaakabneededfor} holds true if the frame generated by $(f, \fb, \la)$ is adapted to $\S$. In particular, if the frame generated by $(f, \fb, \la)$ is adapted to $\S$, and if $(f, \fb, \la, \ovb)$ solves \eqref{GeneralizedGCMsystem} with the above particular choice for $h_1$, $h_2$, $h_3$, $h_4$, $\underline{h}_1$, $\underline{h}_2$, $b_0$, then $\S\subset \RR$  is a GCM sphere. 

\item The above particular choice for $h_1$, $h_2$, $h_3$, $h_4$, $\underline{h}_1$, $\underline{h}_2$, corresponds to the terms in  \eqref{GCMS-4Ster} which 
\begin{itemize}
\item either depend on $\ka-2/r$, $\kab+2\Up/r$, and $\mu-2m/r^3$, 

\item or contain an additional power of $r^{-1}$ compared to the other terms,

\item or are nonlinear.
\end{itemize}

\item The reason for subtracting   the averages $\ov{h_1}^\S$  and  $\ov{\underline{h}_1}^\S$ in the two first equations of \eqref{GeneralizedGCMsystem} is to ensure solvability of the system.
\end{enumerate}
 \end{remark}

 The following proposition provides existence, uniqueness and control of solutions to the linearized GCM system  \eqref{GeneralizedGCMsystem}.
 
 \begin{proposition}
 \lab{Thm.GCMSequations-fixedS}
  Assume  $\S$ is a given  $O(\epg)$-sphere  in $\RR$.  Then,   for every  triplets  $\La, \Lab\in \RRR^3$  and contant $b_0$,       there  exist unique  constants  
   $\Cbdot_0, \Mdot_0, \Cbpdot, \Mpdot$ such that the system \eqref{GeneralizedGCMsystem} has a unique solution $(\ovla, f, \fb)$   with prescribed  $\ell=1$ modes for $\div^\S f, \div^\S\fb$,
    \bea
    \lab{eq:badmodesforffb}
    (\div^\S f)_{\ell=1}=\La, \qquad (\div^\S \fb)_{\ell=1} =\Lab.
    \eea
Moreover, 
\bea
&&\|(f,\fb, \widecheck{\ovla}^\S)\|_{\hk_{s_{max}+1}(\S)}  +\sum_p\Big(r^2|\Cbpdot|+r^3|\Mpdot|\Big)\\
\nn&\les& r\|(\widecheck{h_1}^\S, \,\widecheck{\underline{h}_1}^\S, \,\widecheck{h_2}^\S,\,\widecheck{\underline{h}_2}^\S)\|_{\hk_{s_{max}}(\S)} +r^2\|\widecheck{h_3}^\S\|_{\hk_{s_{max}-1}(\S)}+r\|\widecheck{h_4}^\S\|_{\hk_{s_{max}-2}(\S)} +|\La|+|\Lab|,
\eea
and
\bea
\nn r^2|\Cbdot_0|+r^3|\Mdot_0|+r\Big|\ov{\ovla}^\S\Big|  &\les& r\|(\widecheck{h_1}^\S, \,\widecheck{\underline{h}_1}^\S,\, h_2, \,\underline{h}_2)\|_{L^2(\S)} +r^2\|h_3\|_{L^2(\S)}+r\|\widecheck{h_4}^\S\|_{L^2(\S)} \\
&&+|\La|+|\Lab|+|b_0|,
\eea   
where we have used the notation $\widecheck{h}^\S=h-\ov{h}^\S$  for a scalar function $h$ on $\S$.
 \end{proposition}

The proof of  Proposition \ref{Thm.GCMSequations-fixedS} is postponed to section \ref{sec:proofofThm.GCMSequations-fixedS}. 

The following proposition provides a priori estimates for the linearized GCM system \eqref{GeneralizedGCMsystem}.

\begin{proposition}
\lab{Thm.GCMSequations-fixedS:contraction}
Assume given a solution     $(f, \fb, \ovla, \Cbdot_0, \Mdot_0, \Cbpdot, \Mpdot, \ovb)$  of the system     \eqref{GeneralizedGCMsystem},      \eqref{eq:badmodesforffb} on $\S$. 
 Then, the following a priori estimates are verified 
\bea
\nn&&\|(f,\fb, \widecheck{\ovla}^\S)\|_{\hk_3(\S)} +\sum_p\Big(r^2|\Cbpdot|+r^3|\Mpdot|\Big)\\  &\les& r\|(\widecheck{h_1}^\S, \,\widecheck{\underline{h}_1}^\S, \,\widecheck{h_2}^\S,\,\widecheck{\underline{h}_2}^\S)\|_{\hk_2(\S)} +r^2\|\widecheck{h_3}^\S\|_{\hk_1(\S)}+r\|\widecheck{h_4}^\S\|_{L^2(\S)} +|\La|+|\Lab|,
\eea
and
\bea
\nn r^2|\Cbdot_0|+r^3|\Mdot_0|+r\Big|\ov{\ovla}^\S\Big|  &\les& r\|(\widecheck{h_1}^\S, \,\widecheck{\underline{h}_1}^\S,\, h_2, \,\underline{h}_2)\|_{L^2(\S)} +r^2\|h_3\|_{L^2(\S)}+r\|\widecheck{h_4}^\S\|_{L^2(\S)} \\
&&+|\La|+|\Lab|+|b_0|.
\eea   
 \end{proposition} 
 
 The proof of  Proposition \ref{Thm.GCMSequations-fixedS:contraction} is postponed to section \ref{sec:proofofThm.GCMSequations-fixedS:contraction}.
 
\begin{remark}
Note that the constants $\Cbdot_0$, $\Mdot_0$, $\Cbpdot$, $\Mpdot$ are given in Proposition \ref{Thm.GCMSequations-fixedS:contraction}, while there are chosen in  Proposition \ref{Thm.GCMSequations-fixedS}. Both propositions will be applied in the context of an iterative scheme. Proposition \ref{Thm.GCMSequations-fixedS} will be used for the existence of the iterates and their boundedness, see Proposition \ref{prop:Estimates:Fn+1boundednessiterativescheme}, while Proposition \ref{Thm.GCMSequations-fixedS:contraction} will be used to prove contraction of the iterative scheme, see Proposition \ref{Prop:contractionforNN}. 
\end{remark}

%%%%%%%%%%%%%%%%%%%%%%%%%%%%%%%%%%%%%%% 
 
 \subsection{Proof of  Proposition \ref{Thm.GCMSequations-fixedS}}
\lab{sec:proofofThm.GCMSequations-fixedS} 
 
%%%%%%%%%%%%%%%%%%%%%%%%%%%%%%%%%%%%

\noindent{\bf Step 1.} We start with the solvability  for $\ovla$. Recall that we have
\beaa
\left(\Delta^\S+\frac{2}{(r^\S)^2}\right)\ovla  &=&  \Mdot_0+\sum _p\Mpdot \JpS +\frac{1}{2r^\S}\left(\Cbdot_0+\sum_p \Cbpdot \JpS\right)+h_3.
\eeaa 
Let $\la_p$, $p=0,+,-$, three constants which will be chosen later. We apply Lemma \ref{lemma:sovabilityoftheoperatorDeltaSplus2overrsquare} with 
\beaa
\la=\ovla, \qquad c_p= \Mpdot, \qquad h=\Mdot_0+\frac{1}{2r^\S}\left(\Cbdot_0+\sum_p \Cbpdot \JpS\right)+h_3,
\eeaa
which yields the existence and uniqueness of the constants $\Mpdot$ and of a scalar function $(\ovla)^\perp$ such that the solution $\ovla$ is given by 
\bea\lab{eq:controloflinearizedGCMequations:0}
\ovla &=& (\ovla)^\perp+\sum_p\la_pj^{(p)}, \qquad \int_\S (\ovla)^\perp j^{(q)}=0, \quad q=0,+,-,
\eea
with $\Mpdot$ and $(\ovla)^\perp$ verifying 
\bea\lab{eq:controloflinearizedGCMequations:1}
\sum_p|\Mpdot|  &\les& r^{-1}\|\widecheck{h_3}^\S\|_{L^2(\S)}+\frac{1}{r}\sum_p|\Cbpdot|+\frac{\epg}{r^2}\sum_p|\la_p|,
\eea 
\bea\lab{eq:controloflinearizedGCMequations:2}
\Big|\ov{(\ovla)^\perp}^\S\Big| &\les& r^2|\Mdot_0|+r\left(|\Cbdot_0|+\epg\sum_p |\Cbpdot|\right)+r^2|\ov{h_3}^\S|+\epg \sum_p|\la_p|,
\eea
 and 
\bea\lab{eq:controloflinearizedGCMequations:3}
r^{-3}\Big\|\widecheck{(\ovla)^\perp}^\S\Big\|_{\hk_{s_{max}+1}(\S)} &\les& r^{-1}\|\widecheck{h_3}^\S\|_{\hk_{s_{max}-1}(\S)}+\frac{1}{r}\sum_p |\Cbpdot|+\frac{\epg}{r^2}\sum_p|\la_p|.
\eea

\noindent{\bf Step 2.} Taking the average of the equation for $\ovla$, we have
\beaa
\frac{2}{(r^\S)^2}\ov{\ovla}^\S  &=&  \Mdot_0+\sum _p\Mpdot \ov{\JpS}^\S+\frac{1}{2r^\S}\left(\Cbdot_0+\sum _p\Cbpdot \ov{\JpS}^\S\right) +\ov{h_3}^\S.
\eeaa
In view of the average of $\ovb$, we infer
\beaa
&&\frac{1}{|\S|}\int_\S\left(-\frac{2}{r^\S} \ovla  +\frac{2}{(r^\S)^2}\ovb +  h_2\right)\\
 &=& -\frac{2}{r^\S}\ov{\ovla}^\S+\frac{2}{(r^\S)^2}b_0+\ov{h_2}^\S\\
 &=&  - r^\S\Mdot_0 -r^\S\sum _p\Mpdot \frac{1}{|\S|}\int_\S\JpS -\frac{1}{2}\left(\Cbdot_0+\sum _p\Cbpdot \frac{1}{|\S|}\int_\S\JpS\right)  -r^\S\ov{h_3}^\S \\
 && +\frac{2}{(r^\S)^2}b_0+\ov{h_2}^\S
\eeaa
and
\beaa
&&\frac{1}{|\S|}\int_\S\left( -\frac{2}{r^\S} \ovla -\frac{2}{(r^\S)^2}\ovb   +   \Cbdot_0+\sum_p \Cbpdot \JpS +\underline{h}_2\right) \\
&=& -\frac{2}{r^\S}\ov{\ovla}^\S-\frac{2}{(r^\S)^2}b_0+\Cbdot_0+\left(\sum_p\Cbpdot\right)\frac{1}{|\S|}\int_\S \JpS+\ov{\underline{h}_2}^\S\\
&=&   - r^\S\Mdot_0 -r^\S\sum _p\Mpdot \frac{1}{|\S|}\int_\S\JpS  -r^\S\ov{h_3}^\S   \\
&& -\frac{2}{(r^\S)^2}b_0+\frac{1}{2}\left(\Cbdot_0+\left(\sum_p\Cbpdot\right)\frac{1}{|\S|}\int_\S \JpS\right)+\ov{\underline{h}_2}^\S.
\eeaa
From now on, we  choose $\Mdot_0$ and $\Cbdot_0$ as follows  
\beaa
  r^\S\Mdot_0+\frac{1}{2}\Cbdot_0 &=&   -r^\S\sum _p\Mpdot \frac{1}{|\S|}\int_\S\JpS -\frac{1}{2}\sum _p\Cbpdot \frac{1}{|\S|}\int_\S\JpS  -r^\S\ov{h_3}^\S \\
 && +\frac{2}{(r^\S)^2}b_0+\ov{h_2}^\S,\\
    r^\S\Mdot_0 -\frac{1}{2}\Cbdot_0 &=&  -r^\S\sum _p\Mpdot \frac{1}{|\S|}\int_\S\JpS  -r^\S\ov{h_3}^\S   \\
    &&  -\frac{2}{(r^\S)^2}b_0+\frac{1}{2}\left(\sum_p\Cbpdot\right)\frac{1}{|\S|}\int_\S \JpS+\ov{\underline{h}_2}^\S.
\eeaa
With this choice, we have 
\bea
\bsplit
\int_\S\left(-\frac{2}{r^\S} \ovla  +\frac{2}{(r^\S)^2}\ovb +  h_2\right) &= 0,\\
\int_\S\left( \frac{2}{r^\S} \ovla -\frac{2}{(r^\S)^2}\ovb   +   \Cbdot_0+\sum_p \Cbpdot \JpS +\underline{h}_2\right)  &= 0.
\end{split}
\eea
Furthermore, $\Mdot_0$ and $\Cbdot_0$  satisfy
\bea\lab{eq:controloflinearizedGCMequations:5}
\nn r^2|\Cbdot_0|+r^3|\Mdot_0| &\les& r^2\epg\left(\sum_p|\Cbpdot|\right)+r^3\epg\left(\sum_p|\Mpdot|\right)+|b_0|\\
&&+r^2|\ov{h_2}^\S|+r^2|\ov{\underline{h}_2}^\S|+r^3|\ov{h_3}^\S|.
\eea

\noindent{\bf Step 3.} $f+\fb$ satisfies 
\beaa
\ddd_1^\S(f+\fb) &=& \left(-\frac{4}{r^\S} \ovla+h_2+\underline{h}_2 +\Cbdot_0+\sum_p \Cbpdot \JpS, h_1 -\ov{h_1}^\S +\underline{h}_1 - \ov{\underline{h}_1}^\S\right).
\eeaa
The choice of $\Mdot_0$ and $\Cbdot_0$ in Step 4 is such that the right-hand side of the equation is average free. Thus, we may solve for $f+\fb$, and we have
\beaa
\ddd_1^\S(f+\fb) &=& \left( -\frac{4}{r^\S} \widecheck{
\ovla}^\S+\widecheck{h_2}^\S+\widecheck{\underline{h}_2}^\S +\sum_p \Cbpdot \JpS-\sum_p\Cbpdot \frac{1}{|\S|}\int_\S\JpS, \widecheck{h_1}^\S  +\widecheck{\underline{h}_1}^\S\right).
\eeaa
We infer the estimates 
\bea\lab{eq:controloflinearizedGCMequations:6}
\|f+\fb\|_{\hk_{s_{max}+1}(\S)} \les r\|(\widecheck{h_1}^\S, \widecheck{\underline{h}_1}^\S, \widecheck{h_2}^\S, \widecheck{\underline{h}_2}^\S)\|_{\hk_{s_{max}}(\S)}  +\| \widecheck{\ovla}\|_{\hk_{s_{max}}(\S)}+r^2\sum_p|\Cbpdot|.
\eea

\noindent{\bf Step 4.} $f-\fb$ satisfies
\beaa
\curl ^\S(f -\fb) &=& \widecheck{h_1}^\S - \widecheck{\underline{h}_1}^\S,\\
\div^\S(f-\fb) -\frac{4}{(r^\S)^2}\ovb &=& h_2 -\Cbdot_0-\sum_p \Cbpdot \JpS -\underline{h}_2.
\eeaa
The choice of $\Mdot_0$ and $\Cbdot_0$ in Step 4 is such that the right-hand side of the second equation is average free. Hence, we may rewrite it as
\beaa
\div^\S(f-\fb) -\frac{4}{(r^\S)^2}\check{\ovb}^\S &=& \widecheck{h_2}^\S -\sum_p \Cbpdot \JpS +\sum_p\Cbpdot \frac{1}{|\S|}\int_\S\JpS -\widecheck{\underline{h}_2}^\S.
\eeaa
Since all terms have average 0, this is equivalent to solving 
\beaa
\Delta^\S\div^\S(f-\fb) -\frac{4}{(r^\S)^2}\Delta^\S\check{\ovb}^\S &=& \Delta^\S\left(\widecheck{h_2}^\S -\sum_p \Cbpdot \JpS  -\widecheck{\underline{h}_2}^\S\right).
\eeaa
In view of the definition of $\ovb$, this is equivalent to  
\bea\lab{eq:closedequationforlappdivfminusfbtobeusefullagainlater}
\left(\Delta^\S+\frac{2}{(r^\S)^2}\right)\div^\S(f-\fb) = \Delta^\S\left(\widecheck{h_2}^\S -\sum_p \Cbpdot \JpS  -\widecheck{\underline{h}_2}^\S\right)  +\frac{4}{(r^\S)^2}\widecheck{h_4}^\S.
\eea 

\noindent{\bf Step 5.} In view of Step 4, we consider the solution $\varpi$ to
\beaa
\left(\Delta^\S+\frac{2}{(r^\S)^2}\right)\varpi  &=&  \Delta^\S\left(\widecheck{h_2}^\S -\sum_p \Cbpdot \JpS  -\widecheck{\underline{h}_2}^\S\right)  +\frac{4}{(r^\S)^2}\widecheck{h_4}^\S.
\eeaa 
Let $\varpi_p$, $p=0,+,-$, three constants which will be chosen later. We apply Lemma \ref{lemma:sovabilityoftheoperatorDeltaSplus2overrsquare} with 
\beaa
&& \la=\varpi, \qquad c_p= \frac{2}{(r^\S)^2}\Cbpdot,\\ 
&& h=\Delta^\S\left(\widecheck{h_2}^\S   -\widecheck{\underline{h}_2}^\S\right)  +\frac{4}{(r^\S)^2}\widecheck{h_4}^\S -\sum_p \Cbpdot\left( \Delta^\S+\frac{2}{(r^\S)^2}\right)\JpS,
\eeaa
which yields the existence and uniqueness of the constants $\Cbpdot$ and of a scalar function $\varpi^\perp$ such that the solution $\varpi$ is given by 
\bea\lab{eq:controloflinearizedGCMequations:7}
\varpi &=& \varpi^\perp+\sum_p\varpi_pj^{(p)}, \qquad \int_\S\varpi^\perp j^{(q)}=0, \quad q=0,+,-,
\eea
with $\Cbpdot$ and $\varpi^\perp$ verifying 
\bea\lab{eq:controloflinearizedGCMequations:8}
\sum_p|\Cbpdot|  &\les& r^{-1}\|\widecheck{h_2}^\S\|_{L^2(\S)}+r^{-1}\|\widecheck{\underline{h}_2}^\S\|_{L^2(\S)}+r^{-1}\|\widecheck{h_4}^\S\|_{L^2(\S)}+\epg\sum_p|\varpi_p|,
\eea 
 and 
\bea\lab{eq:controloflinearizedGCMequations:9}
\nn r^{-1}\|\widecheck{\varpi^\perp}^\S\|_{\hk_{s_{max}}(\S)} &\les& r^{-1}\|\widecheck{h_2}^\S\|_{\hk_{s_{max}}(\S)}+r^{-1}\|\widecheck{\underline{h}_2}^\S\|_{\hk_{s_{max}}(\S)}\\
&&+r^{-1}\|\widecheck{h_4}^\S\|_{\hk_{s_{max}-2}(\S)}+\epg\sum_p|\varpi_p|.
\eea
Also, taking the average of the equation for $\varpi$, and in view of \eqref{eq:theeigenvectorsjphavezeroaverageonS}, we infer
\bea\lab{eq:controloflinearizedGCMequations:10}
\ov{\varpi}^\S =0, \qquad \ov{\varpi^\perp}^\S=0.
\eea

\noindent{\bf Step 6.} In view of Step 4 and the definition of $\varpi$ in Step 5, we have
\beaa
\div(f-\fb) &=& \varpi.
\eeaa 
Since $\ov{\varpi}^\S =0$ in view of Step 5, this is equivalent to
\beaa
\div(f-\fb) &=& \widecheck{\varpi}^\S
\eeaa 
and hence
\beaa
\ddd_1^\S(f-\fb) &=& \Big(\widecheck{\varpi}^\S, \widecheck{h_1}^\S - \widecheck{\underline{h}_1}^\S\Big).
\eeaa
Since the right-hand side has average 0, this system is solvable, and we obtain a unique $f-\fb$ satisfying 
\bea\lab{eq:controloflinearizedGCMequations:11}
\nn\|f-\fb\|_{\hk_{s_{max}+1}(\S)} &\les& r\|\widecheck{h_1}^\S\|_{\hk_{s_{max}}(\S)} +r\|\widecheck{\underline{h}_1}^\S\|_{\hk_{s_{max}}(\S)}\\
&& +r\| \widecheck{\varpi^\perp}^\S\|_{\hk_{s_{max}}(\S)}+r^2\sum_p|\varpi_p|.
\eea

\noindent{\bf Step 7.} It remains to ensure the conditions $(\div^\S f)_{\ell=1}=\La$ and $(\div^\S \fb)_{\ell=1} =\Lab$. 
Recall that we have
\beaa
\div^\S(f+\fb) &=& -\frac{4}{r^\S} \widecheck{
\ovla}^\S+\widecheck{h_2}^\S+\widecheck{\underline{h}_2}^\S +\sum_p \Cbpdot \JpS-\sum_p\Cbpdot \frac{1}{|\S|}\int_\S\JpS
\eeaa
which we rewrite
\beaa
\div^\S(f+\fb) &=& -\frac{4}{r^\S} \widecheck{
\ovla^\perp}^\S-\frac{4}{r^\S} \sum_p\la_p j^{(p)}+\widecheck{h_2}^\S+\widecheck{\underline{h}_2}^\S +\sum_p \Cbpdot \JpS\\
&&-\sum_p\Cbpdot \frac{1}{|\S|}\int_\S\JpS.
\eeaa
It is at this stage that we choose the constants $\la_p$ such that 
\beaa
\frac{4}{r^\S} \sum_p\la_p \int_\S j^{(p)}J^{(\S, q)} &=& -\La^{(q)}-\Lab^{(q)}+ \int_\S J^{(\S, q)}\Bigg[-\frac{4}{r^\S} \widecheck{
\ovla^\perp}^\S+\widecheck{h_2}^\S+\widecheck{\underline{h}_2}^\S\\
&& +\sum_p \Cbpdot \JpS -\sum_p\Cbpdot \frac{1}{|\S|}\int_\S\JpS\Bigg], \quad q=0,+,-.
\eeaa
This immediately yields 
\beaa
(\div^\S(f+\fb))_{\ell=1}=\La+\Lab
\eeaa
as well as the estimate
\bea\lab{eq:controloflinearizedGCMequations:12}
\nn\sum_p|\la_p| &\les& \frac{1}{r}(|\La|+|\Lab|)+r^{-1}\ep\|\widecheck{\ovla^\perp}^\S\|_{L^2(\S)}\\
&&+\|\widecheck{h_2}^\S\|_{L^2(\S)}+\|\widecheck{\underline{h}_2}^\S\|_{L^2(\S)}+r\sum_p|\Cbdot_p|,
\eea
where we used the fact that by the  properties of $\ovla^\perp$ and of $j^{(q)}$, we have
\beaa
\int_\S J^{(\S, q)}\widecheck{\ovla^\perp}^\S &=& \int_\S (J^{(\S, q)}-j^{(q)})\widecheck{\ovla^\perp}^\S=O(\epg)r\|\widecheck{\ovla^\perp}^\S\|_{L^2(\S)}.
\eeaa

Also, recall that we have
\beaa
\div^\S(f-\fb) &=& \widecheck{\varpi}^\S
\eeaa
which we rewrite 
\beaa
\div^\S(f-\fb) &=& \widecheck{\varpi^\perp}^\S+\sum_p\varpi_pj^{(p)}.
\eeaa
It is at this stage that we choose the constants $\varpi_p$ such that 
\beaa
\sum_p\varpi_p\int_\S j^{(p)}J^{(\S, q)} &=& \La^{(q)}-\Lab^{(q)} - \int_\S \widecheck{\varpi^\perp}^\S J^{(\S, q)}, \qquad q=0,+,-.
\eeaa
This immediately yields 
\beaa
(\div^\S(f-\fb))_{\ell=1}=\La-\Lab
\eeaa
as well as the estimate
\bea\lab{eq:controloflinearizedGCMequations:13}
\sum_p|\varpi_p| &\les& \frac{1}{r^2}(|\La|+|\Lab|)+r^{-1}\ep\|\widecheck{\varpi^\perp}^\S\|_{L^2(\S)},
\eea
where we used the fact that by the  properties of $\varpi^\perp$ and of $j^{(q)}$, we have
\beaa
\int_\S J^{(\S, q)}\widecheck{\varpi^\perp}^\S &=& \int_\S (J^{(\S, q)}-j^{(q)})\widecheck{\varpi^\perp}^\S=O(\epg)r\|\widecheck{\varpi^\perp}^\S\|_{L^2(\S)}.
\eeaa

\noindent{\bf Step 8.} We now gather the estimates \eqref{eq:controloflinearizedGCMequations:0}-\eqref{eq:controloflinearizedGCMequations:13}, closing the estimates in the following order
\begin{enumerate}
\item estimate  $\Cbpdot$,  $\varpi_p$ and $\widecheck{\varpi^\perp}^\S$ using \eqref{eq:controloflinearizedGCMequations:8}, \eqref{eq:controloflinearizedGCMequations:9} and \eqref{eq:controloflinearizedGCMequations:13},

\item estimate $f-\fb$ using  \eqref{eq:controloflinearizedGCMequations:11},

\item estimate $\Mpdot$,  $\la_p$ and $\widecheck{\ovla^\perp}^\S$ using \eqref{eq:controloflinearizedGCMequations:1}, 
\eqref{eq:controloflinearizedGCMequations:3} and \eqref{eq:controloflinearizedGCMequations:12},

\item estimate $f+\fb$ using \eqref{eq:controloflinearizedGCMequations:6},

\item estimate $\Cbdot_0$ and $\Mdot_0$ using \eqref{eq:controloflinearizedGCMequations:5},

\item estimate $\ov{\ovla^\perp}^\S$ using \eqref{eq:controloflinearizedGCMequations:2},
\end{enumerate}
which finally yields
\beaa
&&\|(f,\fb, \widecheck{\ovla}^\S)\|_{\hk_{s_{max}+1}(\S)}  +\sum_p\Big(r^2|\Cbpdot|+r^3|\Mpdot|\Big)\\
 &\les& r\|(\widecheck{h_1}^\S, \widecheck{\underline{h}_1}^\S, \widecheck{h_2}^\S, \widecheck{\underline{h}_2}^\S)\|_{\hk_{s_{max}}(\S)} +r^2\|\widecheck{h_3}^\S\|_{\hk_{s_{max}-1}(\S)}+r\|\widecheck{h_4}^\S\|_{\hk_{s_{max}-2}(\S)} +|\La|+|\Lab|,
\eeaa
and
\beaa
r^2|\Cbdot_0|+r^3|\Mdot_0|+r|\ov{\ovla}^\S|  &\les& r\|(\widecheck{h_1}^\S, \widecheck{\underline{h}_1}^\S, h_2, \underline{h}_2)\|_{L^2(\S)} \\
&&+r^2\|h_3\|_{L^2(\S)}+r\|\widecheck{h_4}^\S\|_{L^2(\S)} +|\La|+|\Lab|+|b_0|
\eeaa   
as desired. This concludes the proof of Proposition \ref{Thm.GCMSequations-fixedS}.

%%%%%%%%%%%%%%%%%%%%%%%%%%%%%%%%%%%%%%% 
 
 \subsection{Proof of  Proposition \ref{Thm.GCMSequations-fixedS:contraction}}
\lab{sec:proofofThm.GCMSequations-fixedS:contraction} 
 
%%%%%%%%%%%%%%%%%%%%%%%%%%%%%%%%%%%%

The proof is similar to the one of Proposition \ref{Thm.GCMSequations-fixedS}, and simpler since one does not need to prove existence and uniqueness of the system, but only a priori estimates. 

\noindent{\bf Step 1.} Recall that $\div^\S(f-\fb)$ satisfies equation \eqref{eq:closedequationforlappdivfminusfbtobeusefullagainlater}
\beaa
\left(\Delta^\S+\frac{2}{(r^\S)^2}\right)\div^\S(f-\fb) &=& \Delta^\S\left(\widecheck{h_2}^\S -\sum_p \Cbpdot \JpS  -\widecheck{\underline{h}_2}^\S\right)  +\frac{4}{(r^\S)^2}\widecheck{h_4}^\S
\eeaa
which we rewrite as
\beaa
\left(\Delta^\S+\frac{2}{(r^\S)^2}\right)\div^\S(f-\fb) &=& \frac{2}{(r^\S)^2}\sum_p \Cbpdot\JpS -\sum_p \Cbpdot\left(\Delta^\S+\frac{2}{(r^\S)^2}\right)\JpS  \\
&& +\Delta^\S\left(\widecheck{h_2}^\S   -\widecheck{\underline{h}_2}^\S\right)  +\frac{4}{(r^\S)^2}\widecheck{h_4}^\S.
\eeaa
Multiplying by $J^{(\S,q)}$ and integrating on $\S$, and using \eqref{eq:basicpropertiesJpSonOepsphere} and integration by parts, we infer,
\beaa
&&\frac{2}{(r^\S)^2}\sum_p \Cbpdot\int_\S\JpS J^{(\S,q)} \\
&=& \int_\S \div^\S(f-\fb) \left(\Delta^\S+\frac{2}{(r^\S)^2}\right)J^{(\S,q)} +O(\epg)\sum_p |\Cbpdot|  \\
&& +O\Big(r^{-1}\|(\widecheck{h_2}^\S, \widecheck{\underline{h}_2}^\S)\|_{L^2(\S)} +r^{-1}\|\widecheck{h_4}^\S\|_{L^2(\S)}\Big)\\
&=& O(r^{-2}\epg)\|(f, \fb)\|_{\hk_1(\S)}  +O(\epg)\sum_p |\Cbpdot|   +O\Big(r^{-1}\|(\widecheck{h_2}^\S, \widecheck{\underline{h}_2}^\S)\|_{L^2(\S)} +r^{-1}\|\widecheck{h_4}^\S\|_{L^2(\S)}\Big).
 \eeaa
Using again \eqref{eq:basicpropertiesJpSonOepsphere}, we deduce for $\epg>0$ small enough 
\beaa
r^2\sum_p |\Cbpdot| &\les& \epg\|(f, \fb)\|_{\hk_1(\S)}    +r\|(\widecheck{h_2}^\S, \widecheck{\underline{h}_2}^\S)\|_{L^2(\S)} +r\|\widecheck{h_4}^\S\|_{L^2(\S)}.
 \eeaa

\noindent{\bf Step 2.} Recall the equation for $\ovla$
\beaa
\left(\Delta^\S+\frac{2}{(r^\S)^2}\right)\ovla  &=&  \Mdot_0+\sum _p\Mpdot \JpS+\frac{1}{2r^\S}\left(\Cbdot_0+\sum_p \Cbpdot \JpS\right) +h_3.
\eeaa
Subtracting the average, we infer
\bea\lab{eq:ellipticeequationforovlaoncetheaveragehasbeensubtracted}
\left(\Delta^\S+\frac{2}{(r^\S)^2}\right)\widecheck{\ovla}^\S  &=&  \sum _p\left(\Mpdot+\frac{1}{2r^\S}\Cbpdot \right) \Big(\JpS-\ov{\JpS}^\S\Big) +\widecheck{h_3}^\S.
\eea 
Multiplying by $J^{(\S,q)}$ and integrating on $\S$, and using \eqref{eq:basicpropertiesJpSonOepsphere} and integration by parts, we infer,
\beaa
&&\sum _p\left(\Mpdot+\frac{1}{2r^\S}\Cbpdot \right)\int_\S\JpS J^{(\S,q)}\\
&=& \int_\S\widecheck{\ovla}^\S\left(\Delta^\S+\frac{2}{(r^\S)^2}\right)J^{(\S,q)}+O(r^2\epg)\left(\Mpdot+\frac{1}{2r^\S}\Cbpdot \right)+r\|\widecheck{h_3}^\S\|_{L^2(\S)}\\
&=& O(r^{-1}\epg)\|\widecheck{\ovla}^\S\|_{L^2(S)}+O(r^2\epg)\left(\Mpdot+\frac{1}{2r^\S}\Cbpdot \right)+r\|\widecheck{h_3}^\S\|_{L^2(\S)}.
\eeaa
Using again \eqref{eq:basicpropertiesJpSonOepsphere}, we deduce for $\epg>0$ small enough 
\beaa
 r^3\sum _p|\Mpdot| &\les& r^2\sum _p|\Cbpdot|+\epg\|\widecheck{\ovla}^\S\|_{L^2(S)}+r^2\|\widecheck{h_3}^\S\|_{L^2(\S)}.
\eeaa

\noindent{\bf Step 3.} $\div^\S(f+\fb)$ satisfies 
\beaa
\div^\S(f+\fb) &=&  -\frac{4}{r^\S} \widecheck{\ovla}^\S+\widecheck{h_2}^\S+\widecheck{\underline{h}_2}^\S +\sum_p \Cbpdot \JpS-\sum_p\Cbpdot \frac{1}{|\S|}\int_\S\JpS.
\eeaa
Multiplying by $J^{(\S,q)}$, integrating on $\S$, and using \eqref{eq:basicpropertiesJpSonOepsphere}, we infer,
\beaa
\frac{4}{r^\S} \int_\S\widecheck{\ovla}^\S J^{(\S,q)} &=& -\int_\S\div^\S(f+\fb) J^{(\S,q)} +O(r^2)\sum_p |\Cbpdot|  +O\Big(r\|(\widecheck{h_2}^\S, \widecheck{\underline{h}_2}^\S)\|_{L^2(\S)} \Big).
\eeaa
Using  \eqref{eq:badmodesforffb}, we deduce
\beaa
\frac{1}{r}\sum_p\left|\int_\S\widecheck{\ovla}^\S\JpS\right| &\les& |\La|+|\Lab| +r^2\sum_p |\Cbpdot|  +r\|(\widecheck{h_2}^\S, \widecheck{\underline{h}_2}^\S)\|_{L^2(\S)}.
\eeaa
In view of the properties of $j^{(p)}$ introduced in Lemma \ref{lemma:sovabilityoftheoperatorDeltaSplus2overrsquare}, we obtain
\beaa
\frac{1}{r}\sum_p\left|\int_\S\widecheck{\ovla}^\S j^{(p)}\right| \les |\La|+|\Lab| +r^2\sum_p |\Cbpdot|  +r\|(\widecheck{h_2}^\S, \widecheck{\underline{h}_2}^\S)\|_{L^2(\S)}+\epg\|\widecheck{\ovla}^\S\|_{L^2(\S)}.
\eeaa

\noindent{\bf Step 4.} Recall \eqref{eq:ellipticeequationforovlaoncetheaveragehasbeensubtracted}
\beaa
\left(\Delta^\S+\frac{2}{(r^\S)^2}\right)\widecheck{\ovla}^\S  &=&  \sum _p\left(\Mpdot+\frac{1}{2r^\S}\Cbpdot \right) \Big(\JpS-\ov{\JpS}^\S\Big) +\widecheck{h_3}^\S.
\eeaa
In view of the  definition and properties of $j^{(p)}$ introduced in Lemma \ref{lemma:sovabilityoftheoperatorDeltaSplus2overrsquare},  we deduce 
\beaa
\|\widecheck{\ovla}^\S\|_{\hk_2(\S)}  &\les&  r^2\sum _p\left(r|\Mpdot|+|\Cbpdot| \right) +r^2\|\widecheck{h_3}^\S\|_{L^2(\S)}+\frac{1}{r}\sum_p\left|\int_\S\widecheck{\ovla}^\S j^{(p)}\right|.
\eeaa
Coming back to \eqref{eq:ellipticeequationforovlaoncetheaveragehasbeensubtracted}, we infer from standard elliptic regularity 
\beaa
\|\widecheck{\ovla}^\S\|_{\hk_3(\S)}  &\les&  r^2\sum _p\left(r|\Mpdot|+|\Cbpdot| \right) +r^2\|\widecheck{h_3}^\S\|_{\hk_1(\S)}+\frac{1}{r}\sum_p\left|\int_\S\widecheck{\ovla}^\S j^{(p)}\right|.
\eeaa

\noindent{\bf Step 5.}  $f+\fb$ satisfies 
\beaa
\ddd_1^\S(f+\fb) &=& \left(-\frac{4}{r^\S} \ovla+h_2+\underline{h}_2 +\Cbdot_0+\sum_p \Cbpdot \JpS, h_1 -\ov{h_1}^\S +\underline{h}_1 - \ov{\underline{h}_1}^\S\right).
\eeaa
Subtracting the average, we infer
\beaa
\ddd_1^\S(f+\fb) &=& \left( -\frac{4}{r^\S} \widecheck{
\ovla}^\S+\widecheck{h_2}^\S+\widecheck{\underline{h}_2}^\S +\sum_p \Cbpdot \JpS-\sum_p\Cbpdot \frac{1}{|\S|}\int_\S\JpS, \widecheck{h_1}^\S  +\widecheck{\underline{h}_1}^\S\right).
\eeaa
We deduce 
\beaa
\|f+\fb\|_{\hk_3(\S)}  &\les& \|\widecheck{\ovla}^\S\|_{\hk_2(\S)}+ r\|(\widecheck{h_1}^\S, \widecheck{\underline{h}_1}^\S, \widecheck{h_2}^\S, \widecheck{\underline{h}_2}^\S)\|_{\hk_{2}(\S)}+ r^2\sum _p|\Cbpdot|.
\eeaa

\noindent{\bf Step 6.} Recall  \eqref{eq:closedequationforlappdivfminusfbtobeusefullagainlater}
\beaa
\left(\Delta^\S+\frac{2}{(r^\S)^2}\right)\div^\S(f-\fb) &=& \Delta^\S\left(\widecheck{h_2}^\S -\sum_p \Cbpdot \JpS  -\widecheck{\underline{h}_2}^\S\right)  +\frac{4}{(r^\S)^2}\widecheck{h_4}^\S.
\eeaa
In view of the  definition and properties of $j^{(p)}$ introduced in Lemma \ref{lemma:sovabilityoftheoperatorDeltaSplus2overrsquare},  we infer 
\beaa
\|\div^\S(f-\fb)\|_{\hk_2(\S)} &\les&   \|( \widecheck{h_2}^\S, \widecheck{\underline{h}_2}^\S)\|_{\hk_{2}(\S)}+\|\widecheck{h_4}^\S\|_{L^2(\S)}+ r\sum _p|\Cbpdot|\\
&&+\frac{1}{r}\sum_p\left|\int_\S \div^\S(f-\fb)j^{(p)}\right|.
\eeaa
Using again the properties of  $j^{(p)}$, together with \eqref{eq:badmodesforffb}, we have
\beaa
\frac{1}{r}\sum_p\left|\int_\S \div^\S(f-\fb)j^{(p)}\right| &\les& \frac{1}{r}\sum_p\left|\int_\S \div^\S(f-\fb)\JpS\right|+\epg\|\div^\S(f-\fb)\|_{L^2(\S)}\\
&\les& \frac{1}{r}(|\La|+|\Lab|)+\epg\|\div^\S(f-\fb)\|_{L^2(\S)}.
\eeaa
Using the smallness of $\epg$, we infer
\beaa
\|\div^\S(f-\fb)\|_{\hk_2(\S)} \les   \|( \widecheck{h_2}^\S, \widecheck{\underline{h}_2}^\S)\|_{\hk_{2}(\S)}+\|\widecheck{h_4}^\S\|_{L^2(\S)}+ r\sum _p|\Cbpdot|+\frac{1}{r}(|\La|+|\Lab|).
\eeaa

\noindent{\bf Step 7.} Since 
\beaa
\ddd_1^\S(f-\fb) &=& \Big(\div^\S(f-\fb), \widecheck{h_1}^\S - \widecheck{\underline{h}_1}^\S\Big),
\eeaa
we infer
\beaa
\|f-\fb\|_{\hk_3(\S)}  &\les& r\|\div^\S(f-\fb)\|_{\hk_2(\S)}+ r\|(\widecheck{h_1}^\S, \widecheck{\underline{h}_1}^\S)\|_{\hk_{2}(\S)}.
\eeaa
Together with Step 1 to Step 6, we deduce
\beaa
\nn&&\|(f, \fb, \widecheck{\ovla}^\S)\|_{\hk_3(\S)}+r^2\sum_p |\Cbpdot| +  r^3\sum _p|\Mpdot| \\
&\les&  r\|(\widecheck{h_1}^\S, \widecheck{\underline{h}_1}^\S, \widecheck{h_2}^\S, \widecheck{\underline{h}_2}^\S)\|_{\hk_{2}(\S)} +r^2\|\widecheck{h_3}^\S\|_{\hk_1(\S)}+r\|\widecheck{h_4}^\S\|_{L^2(\S)}+|\La|+|\Lab|.
 \eeaa

\noindent{\bf Step 8.} It remains to control $\Cbdot_0$, $\Mdot_0$ and $\ov{\ovla}^\S$. Taking the average of 
\beaa
\div^\S(f-\fb) -\frac{4}{(r^\S)^2}\ovb &=& h_2 -\Cbdot_0-\sum_p \Cbpdot \JpS -\underline{h}_2,
\eeaa
we infer, using also $\ov{\ovb}^\S=b_0$, 
\beaa
r^2|\Cbdot_0| &\les& |b_0|+r\|(h_2, \underline{h}_2)\|_{L^2(\S)}+r^2\sum_p|\Cbpdot|.
\eeaa
Also, taking the average of 
\beaa
\div^\S(f+\fb) &=& -\frac{4}{r^\S} \ovla+h_2+\underline{h}_2 +\Cbdot_0+\sum_p \Cbpdot \JpS.
\eeaa
we infer
\beaa
r|\ov{\ovla}^\S| &\les& r^2|\Cbdot_0|+r\|(h_2, \underline{h}_2)\|_{L^2(\S)}+r^2\sum_p|\Cbpdot|.
\eeaa
Finally,  taking the average of 
\beaa
\left(\Delta^\S+\frac{2}{(r^\S)^2}\right)\ovla  &=&  \Mdot_0+\sum _p\Mpdot \JpS+\frac{1}{2r^\S}\left(\Cbdot_0+\sum_p \Cbpdot \JpS\right) +h_3,
\eeaa
we infer
\beaa
 r^3|M_0| &\les& r^2|\Cbdot_0|+r|\ov{\ovla}^\S|+r^2\|h_3\|_{L^2(\S)}+r^2\sum_p|\Cbpdot|+r^3\sum_p|\Mpdot|.
\eeaa
Gathering the three above estimates, we obtain
\beaa
r^2|\Cbdot_0| +r^3|M_0|+r|\ov{\ovla}^\S| &\les& r\|(h_2, \underline{h}_2)\|_{L^2(\S)}+r^2\|h_3\|_{L^2(\S)}\\
&&+r^2\sum_p|\Cbpdot|+r^3\sum_p|\Mpdot|+|b_0|.
\eeaa
Together with Step 7, we deduce
\beaa
r^2|\Cbdot_0| +r^3|M_0|+r|\ov{\ovla}^\S| &\les& r\|(\widecheck{h_1}^\S, \widecheck{\underline{h}_1}^\S, h_2, \underline{h}_2)\|_{L^2(\S)}+r^2\|h_3\|_{L^2(\S)}+r\|\widecheck{h_4}^\S\|_{L^2(\S)}\\
&&+|b_0|+|\La|+|\Lab|.
\eeaa
This concludes the proof of Proposition \ref{Thm.GCMSequations-fixedS:contraction}.

%%%%%%%%%%%%%%%%%%%%%%%%%%
 
\section{Deformations of surfaces}
\lab{sec:deformationofsurfaces}
 
%%%%%%%%%%%%%%%%%%%%%%%%

%%%%%%%%%%%%%%%%%%%%%%%%%%%%%%%%%%
  
\subsection{Deformations}
 
%%%%%%%%%%%%%%%%%%%%%%%%%%%%%%%%%%

 We recall  that the region $\RR=\RR_N\cup \RR_S $  is covered by  the  coordinate systems  denoted  $(u, s, y_N^1, y_N^2)$ and   $(u, s, y_S^1, y_S^2)$. The passage  from  the South coordinate system  $S$ to the North one  in the equatorial region $\RR_{Eq}=\RR_N\cap\RR_S$ is given by the transition functions $\vphi_{SN}$ and $\vphi_{NS}$.  Recall also that  $\ovS=S(\ovu, \ovs)$  is a fixed    sphere of the  $(u, s)$  foliation of   $\RR$.   
 
  \begin{definition}
 \label{definition:Deformations}
 We say that    $\S$ is an $O(\epg)$\,  deformation of $ \ovS$ if there exist  smooth  scalar functions $U, S$ defined on $\ovS$ and a  map 
  a map $\Psi:\ovS\longrightarrow \S $  verifying, on either coordinate  chart  $(y^1, y^2) $ of $\ovS$,  
   \bea
 \Psi(\ovu, \ovs,  y^1, y^2)=\left( \ovu+ U(y^1, y^2 ), \, \ovs+S(y^1, y^2 ), y^1, y^2  \right).
 \eea
 \end{definition}

%%%%%%%%%%%%%%%
 
  \subsection{Pull-back map}  
 
%%%%%%%%%%%%%%%
 
 Consider a fixed  deformation. 
 We recall that given a scalar function $f$ on $\S$ 
 one defines its pull-back on $\ovS$ to be  the function,
  \beaa
 f^\#:=  \Psi^\# f =f\circ \Psi.
  \eeaa 
On the other hand, given a vectorfield  $X$ on $\ovS$ one defines its push-forward   $\Psi_\# X$ to be the vectorfield 
on $\S$  defined by,
\beaa
\Psi_\# X(f)=X(\Psi^\# f)= X( f\circ \Psi).
\eeaa
Given a  covariant tensor $U$ on $\S$,  one defines  its pull back  to $\ovS$ to be  the tensor 
\beaa
\Psi^\# U ( X_1, \ldots, X_k)=  U(\Psi_\#  X_1, \ldots,   \Psi_\#X_k).
\eeaa

In what follows we restrict ourselves  to  a  fixed  chart $(y^1, y^2)$, either North or South,  on $\ovS$ relative to which the spacetime metric takes the form 
\eqref{spacetimemetric-y-coordinates}
\beaa
\g &=& - 2\vsi du ds + \vsi^2\Omb  du^2 +g_{ab}\big( dy^a- \vsi \undB^a du\big) \big( dy^b-\vsi \undB^b du\big),
\eeaa
where
\beaa
\Omb=e_3(s), \qquad \undB^a =\frac{1}{2} e_3(y^a), \qquad g_{ab}=\g(\pr_{y^a}, \pr_{y^b}).
\eeaa

    \begin{lemma}
    \lab{Lemma:deformation1}
    Let $\ovS=S(\ovu, \ovs)$  be a fixed  sphere of the background foliation of $\RR$ 
    and consider a deformation $\Psi:\ovS\longrightarrow \S$  of the form
     \beaa
 \Psi(\ovu, \ovs,  y^1, y^2)=\left( \ovu+ U(y^1, y^2 ), \, \ovs+S(y^1, y^2 ), y^1, y^2  \right)
 \eeaa
 with $(y^1, y^2)$ representing any one  of the two charts of $\RR$. 
 \begin{enumerate}
 \item The  push-forward vectorfields $  \YY_{(a)} =\Psi_\# (\pr_{y^a} )$   on $\S$  have the form
 \bea
   \YY_{(a)}&=  \YY_{(a)}^4 e_4+  \YY_{(a)}^3 e_3+ \YY_{(a)}^c e_c 
   \eea
   with coefficients
 \bea
 \lab{coefficients-YYa-US}
 \bsplit
   \YY_{(a)}^4&= \pr_{y^a} S-\frac 1 2 (\vsi \Omb)^\#\,  \pr_{y^a}  U,\\
    \YY_{(a)}^3&=\frac 1 2 \vsi^\# \pr_{y^a}  U, \\
     \YY_{(a)}^c&=  (Y_{(a)}^c)^\# -  (\vsi Z^c)^\# \, \pr_{y^a}  U.
   \end{split}
 \eea
 
 \item The pull back metric     $g^{\S,\#} :=\Psi^\# (g^\S)$  on $\ovS$  is given, in the coordinates $y^1, y^2$,  by 
 \bea
 \lab{Relation-YY-gSab}
  g^{\S, \#}_{ab}\big|_p&=&\Big(-  2 \YY_{(a)}^4 \YY_{(b)}^3- 2 \YY_{(b)}^4 \YY_{(a)}^3  +\sum_{c=1,2} \YY^c_{(a)} \YY^c_{(b)}\Big)\Big|_{\Psi(p)}.
  \eea
    
 \item  The $L^2$ norm of  $f^\#=\psi^\# f$ with respect to the metric $\gS^{\S, \#}$ is the same as 
 as the $L^2$ norm of  $f$ with respect to the metric $\gS^\S$, i.e.,
 \beaa
 \int_{\ovS}    |f^\#|^2    da_{\gS^{\S, \#}}&=& \int_\S |f|^2  da_{\gS^\S}.
 \eeaa
  \end{enumerate}     
    \end{lemma}

\begin{proof}
   Let   $\YY_{(a)}$, $a=1,2$,  denote  the push forwards   to $\S$ of the coordinate vectorfields $\pr_{y^a} $  on $\ovS$.  More precisely
     at every point $\Psi(p)$, $p\in \ovS$,
   \bea
   \YY_{(a)} &=&\Psi_\# (\pr_{y^a} )|_{\Psi(p)}=(\pr_{y^a}  U) \pr_u|_{\Psi(p)}  +(\pr_{y^a}  S) \pr_s|_{\Psi(p)} +\pr_{y^a} |_{\Psi(p)}.
   \eea 
   In view of Lemma  \ref{Lemma:geodesic-coordinates}   we have    at every point in $\RR$
   \beaa
\pr_s &=& e_4\\
\pr_u &=& \vsi \left(\frac{1}{2}e_3-\frac{1}{2}\Omb e_4 -  \sum_{c=1,2}  Z^c  e_c\right)\\
\pr_{y^a}&=& \sum_{c=1,2}  Y_{(a)}^c e_c, \qquad   a=1,2,\\
Z^c &=& \Bb^aY^c_{(a)}.
\eeaa
Denoting, at every point $p\in \ovS$,
\beaa
\vsi^\# (p)=\vsi(\Psi(p)), \quad \Omb^\#(p)=\Omb(\Psi(p) ), \quad  (Z^c)^\#(p)=Z^c(\Psi(p)), \quad
 (Y_{(a)}^c)^\#(p)=Y_{(a)}^c(\Psi(p)),
\eeaa
we deduce,
\beaa
\YY_{(a)}&=& \vsi^\#  (\pr_{y^a}  U) \left(\frac{1}{2}e_3-\frac{1}{2}\Omb^\# e_4 -  \sum_{c=1,2} (  Z^c)^\# e_c\right) +(\pr_{y^a}  S) e_4+\sum_{c=1,2}  (Y_{(a)}^c)^\# e_c\\
&=&\Big(\pr_{y^a} S-\frac 1 2 (\vsi\Omb)^\#\,  \pr_{y^a}  U\Big) e_4+\frac 1 2 \vsi^\# \pr_{y^a}  U \,  e_3 +\sum_{c=1,2} \Big( (Y_{(a)}^c)^\# -  (\vsi Z^c)^\# \, \pr_{y^a}  U\Big)e_c.
\eeaa
   We write in the form
   \beaa
   \YY_{(a)}&=&  \YY_{(a)}^4 e_4+  \YY_{(a)}^3 e_3+ \YY_{(a)}^c e_c 
   \eeaa
   with,
   \beaa
    \YY_{(a)}^4&=& \pr_{y^a} S-\frac 1 2 \vsi^\# \Omb^\#\,  \pr_{y^a}  U,\\
    \YY_{(a)}^3&=&\frac 1 2 \vsi^\# \pr_{y^a}  U, \\
     \YY_{(a)}^c&=&  (Y_{(a)}^c)^\# -  (\vsi Z^c)^\# \, \pr_{y^a}  U.
   \eeaa
    We denote by $g^{ \S, \#}=\Psi^\#(g^\S ) $  the  pull back to $\ovS$  of the metric   $g^\S$  on $\S$, i.e. at any point $p\in \ovS$,
    \beaa
    g^{\S, \#}(\pr_{y^a}, \pr_{y^b})&=&  g^\S(\YY_{(a)}, \YY_{(b)})=\g(\YY_{(a)}, \YY_{(b)})\\
    &=&\g\left(\YY_{(a)}^4 e_4+\YY_{(a)}^3 e_3+\sum_{c=1,2} \YY^c_{(a)} e_c, \YY_{(b)}^4 e_4+\YY_{(b)}^3 e_3+\sum_{d=1,2} \YY^d_{(b)} e_d\right)\\
    &=&- 2 \YY_{(a)}^4 \YY_{(b)}^3- 2 \YY_{(b)}^4 \YY_{(a)}^3  +\sum_{c=1,2} \YY^c_{(a)} \YY^c_{(b)}.
    \eeaa
    Hence
    \beaa
     g^{\S, \#}_{ab}&=& - 2 \YY_{(a)}^4 \YY_{(b)}^3- 2 \YY_{(b)}^4 \YY_{(a)}^3  +\sum_{c=1,2} \YY^c_{(a)} \YY^c_{(b)}
    \eeaa
    as desired.
 \end{proof}

    \begin{definition} Given a deformation $\Psi:\ovS\longrightarrow \S$  as above we denote:
    \begin{itemize}
 \item    At points $\Psi(p)$ in $\S$,
    \bea
      g^\S_{ab}(\Psi(p) )&:=&  g^\S\Big|_{\Psi(p)} \big(\YY_{(a)},\YY_{(b)} \big)= g^{\S, \#}_{ab}(p ).
    \eea
    With this definition,
    \bea
     g^{\S, \#}_{ab}&=& \Big(  g^\S_{ab}\Big)^\#.
    \eea    
    
    \item We denote by $\nab^\S$ the covariant derivative  operator on $\S$ induced by the metric  $g^\S$  and  by $\nab^{\S, \#}$   the covariant derivative  operator on $\ovS $ induced by the pull back metric  metric  $g^{\S, \#}.$ 
    \end{itemize}
    \end{definition}
    
    \begin{remark}
     Any geometric calculation   with respect to the $g^\S$ metric   can be reduced to a geometric calculation on $\ovS$  with respect to the metric $  \gS^{\S, \#}$.  More precisely, if  $U$ is a $k$ covariant tensor on $\S$ and $X_0, X_1, \ldots, X_k$
   vectorfields on $\ovS$,
   \beaa
\Big(  \nab^{\S, \#}  U^\#\Big)( X_0, X_1, \ldots, X_k) &=&\Bigg(\nab^\S U\Big(\Psi_{\#} X_0, \Psi_{\# }X_1, \ldots, \Psi_{\# }X_k\Big)\Bigg)^\#.
  \eeaa
  In particular, with respect to the coordinate vectorfields $\pr_{y^1}, \pr_{y^2}$ on $\ovS$,
  \beaa
  \nab^{\S, \#} _{a_0} U^\#_{a_1\ldots a_k}&=& \Big(\nab^\S U \big(\YY_{(a_0)}, \YY_{(a_1)},\ldots, \YY_{(a_n)}\big)\Big)^\#.
  \eeaa
    \end{remark}
    
    As a consequence of the remark we immediately deduce the following,
    \begin{lemma}\lab{lemma:pullbacksofcovderivatives}
    Let $\Psi:\ovS\longrightarrow \S$  be a deformation as above.  If $\U\in \SS_k(\S)$,  for   $k=0,1,2$ we have, for the  corresponding Hodge operators
    \bea
    \dddSdiez_k U^\#= \big(\dddS_k U\big)^\#, \qquad   \ddsSdiez_k\,  U^\#= \big(\ddsS_k U\big)^\#.
    \eea
    
    Also, if  $h$ is a scalar on $\S$ we have,
  \beaa
  \Big(\Delta^\S h\Big)^\# =\Delta^{\S, \#} \big(h^\#\big).
  \eeaa
    \end{lemma}
    
    \begin{corollary}
    If  $f\in \hk_k(\S)$ and $f^\#$ is its pull-back by $\Psi$ then,
 \beaa
 \|f^\#\|_{\hk_k(\ovS,\, \gS^{\S,\#})} = \| f\|_{\hk_k(\S)}.
 \eeaa
    \end{corollary}

%%%%%%%%%%%%%%%%%%%%%%%%%%%%%%%%%%%

\subsection{Comparison results}

%%%%%%%%%%%%%%%%%%%%%%%%%%%%%%%%%%%

We start with the following lemma.
\begin{lemma}
\lab{Le:Transportcomparison}
Let     $\Psi:\ovS\longrightarrow \S $   be  a  deformation in $\RR$ as in Definition \ref{definition:Deformations},    $F$ a scalar function on $\RR$ and $F^\# $  its pull back to $\ovS$ by $\Psi$. We have
\bea
\big\| F^\#-F\big\|_{L^\infty(\ovS)}&\les&  \big \| ( U, S) \big\|_{L^\infty(\ovS)}  \sup_{\RR}\left(\big|  e_3  F \big|+ r^{-1} \big|\dk F\big|\right).
\eea
\end{lemma}

\begin{proof}
We have, for $y=(y^1, y^2)$,
\beaa
 F\big(\ovu+U(y),  \ovs+S(y), y\big)-   F\big(\ovu,  \ovs, y\big) =\int_0^1  \frac{d}{d\la} F\big(\ovu+\la U(y),  \ovs+\la S(y), y\big)
\eeaa
 \beaa
 \Big|  F\big(\ovu+U(y),  \ovs+S(y), y\big)-   F\big(\ovu,  \ovs, y\big) \Big| &\les&\int_0^1 \Big|  \frac{d}{d\la} F\big(\ovu+\la U(y),  \ovs+\la S(y), y\big)\Big|\\
 &\les&  \big| U(y)\big| \int_0^1\Big|\pr_u  F\big( \ovu+\la U(y),  \ovs+\la S(y), y\big)\Big|\\
 &+& \big| S(y)\big| \int_0^1\Big|\pr_s  F\big(\ovu+\la U(y),  \ovs+\la S(y), y\big)\Big|.
 \eeaa
Recalling
  \beaa
\pr_s &=& e_4,\qquad 
\pr_u = \vsi \left(\frac{1}{2}e_3-\frac{1}{2}\Omb e_4 -  \sum_{c=1,2}  Z^c  e_c\right), 
\eeaa
 using our assumptions   on $\Omb, \vsi, Z$ and the definition of $\dk$ we easily derive
\beaa
 \Big|  F\big( \ovu+U(y),  \ovs+S(y), y\big)-   F\big(\ovu,  \ovs, y\big) \Big| &\les& \big \| ( U, S) \big\|_{L^\infty(\ovS)}  \sup_{\RR}\left(\big|  e_3  F \big|+ r^{-1} \big|\dk F\big|\right)
\eeaa
as desired.
\end{proof}

 \begin{lemma}\lab{lemma:comparison-gaS-ga}
 Let $\ovS \subset \RR$.    Let  $\Psi:\ovS\longrightarrow \S $  be   a  deformation generated by the  functions $(U, S)$ as in Definition \ref{definition:Deformations}. Assume the bound
 \bea
 \label{assumption-UV-dg}
   \| (U, S)\|_{L^\infty(\ovS)} +r  \| \nabzero(U, S)\|_{L^\infty(\ovS)}  &\les&  \dg.
 \eea
  Then
 \begin{enumerate}
 \item   We have
 \bea
 \lab{equation:difference-ofgas}
\sum_{a,b=1}^2  \big| g^{\S, \#}_{ab} -\ovg_{ab} \big|&\les r\dg.
 \eea 

 \item  For every $f\in\mathcal{S}_k(\S)$ we have,
  \bea
 \label{eq:lemma:int-gaS-ga3}
 \|f^\#\|_{L^2(\ovS, g^{\S, \#})}   &=&   \|f^\#\|_{L^2(\ovS, \ovg)}  \Big(1+ O(r^{-1} \dg)  \Big).
 \eea
\item  As a corollary  of \eqref{eq:lemma:int-gaS-ga3} (choosing $f=1$) we deduce, 
  \bea
 \frac{r^\S}{\ovr}= 1 + O(r ^{-1}  \dg )
 \eea
 where $r^\S$ is the area radius of $\S$ and $\ovr$ that of $\ovS$.
\item If in addition  to \eqref{assumption-UV-dg}  we have
\bea
\big\|(U, S)\big\|_{\hk_{s_{max}+1}(\ovS)}  &\les r \dg ,
\eea
then
\bea
\sum_{a,b,c=1,2}\Big\|(\Ga^{\S, \#})_{ab}^c-(\ovGa)_{ab}^c\Big\|_{\hk_{s_{max}}(\ovS)} &\les r\dg
\eea
where $ \Ga^{\S, \#}, \, \ovGa$ denote the Christoffel symbols  of the metrics $g^{\S, \#}, \, \ovg$ relative to the coordinates $y^1, y^2$
on $\ovS$.
\end{enumerate}
\end{lemma}

\begin{proof}
 Recall, see Lemma  \ref{Lemma:deformation1},   that the coefficients of the pull-back metric $  g^{\S, \#}_{ab}$ in the coordinates $y^1, y^2$ is given by
   \beaa
  g^{\S, \#}_{ab}\big|_p&=& g^{\S, \#}(\pr_{y^a}, \pr_{y^b})= \Big(-  2 \YY_{(a)}^4 \YY_{(b)}^3- 2 \YY_{(b)}^4 \YY_{(a)}^3  +\sum_{c=1,2} \YY^c_{(a)} \YY^c_{(b)}\Big)\Big|_{\Psi(p)}
  \eeaa
  where,
 \beaa
 \bsplit
   \YY_{(a)}^4&= \pr_{y^a} S-\frac 1 2 (\vsi \Omb)^\#\,  \pr_{y^a}  U,\\
    \YY_{(a)}^3&=\frac 1 2 \vsi^\# \pr_{y^a}  U, \\
     \YY_{(a)}^c&=  (Y_{(a)}^c)^\# -  (\vsi Z^c)^\# \, \pr_{y^a}  U.
   \end{split}
 \eeaa
 On the other hand, the metric   $\ovg$, induced by  the spacetime metric  on $\ovS$,  is given by
 \beaa
 \ovg_{ab}&=&  \ovga \big(\pr_{y^a}, \pr_{y^b} \big) =  \ovg\left( \sum_cY^c_{(a)} e_c,  \sum_dY^d_{(a)} e_d\right)=\sum_{c=1,2} \YY^c_{(a)} \YY^c_{(b)}.
 \eeaa
 Hence, at every point $p$,
 \beaa
 g^{\S, \#}_{ab}-  \ovg_{ab}&=&  \Big(-  2 \YY_{(a)}^4 \YY_{(b)}^3- 2 \YY_{(b)}^4 \YY_{(a)}^3\Big)(\Psi(p)) +\sum_{c=1,2}\Big( \YY^c_{(a)} \YY^c_{(b)} (\Psi(p))-  \YY^c_{(a)} \YY^c_{(b)}(p)\Big).
 \eeaa
 Note that
 \beaa
\sup_{\ovS}  \Big|\big( \YY_{(a)}^4 \YY_{(b)}^3- 2 \YY_{(b)}^4 \YY_{(a)}^3\big)(\Psi(p))\Big| &\les & 
 r^2 \| \nabzero(U, S))\|_{L^\infty(\ovS)}^2\les (\dg)^2.
 \eeaa

 For the remaining term $ \YY^c_{(a)} \YY^c_{(b)} (\Psi(p))-  \YY^c_{(a)} \YY^c_{(b)}(p)$ we make use  of Lemma \ref{Le:Transportcomparison} and estimate  \eqref{eq:assumptionY_a^b}  to derive
 \bea
 \lab{eq.defferencesofYY}
\Big| \YY^c_{(a)} \YY^c_{(b)} (\Psi(p))-  \YY^c_{(a)} \YY^c_{(b)}(p)\Big|&\les& r\dg.
 \eea
 Indeed 
  \beaa
 \YY^c_{(a)} \YY^c_{(b)} (\Psi(p))-  \YY^c_{(a)} \YY^c_{(b)}(p)&=&  \Big(\YY^c_{(a)}(\Psi(p)) - \YY^c_{(a)}(p) \Big) \YY^c_{(b)} (\Psi(p))\\
 &+&\YY^c_{(a)} (p)\Big( \YY^c_{(b)} (\Psi(p))-  \YY^c_{(b)}(p)\Big).
 \eeaa 
 We deduce,
 \beaa
 \Big| \YY^c_{(a)} \YY^c_{(b)} (\Psi(p))-  \YY^c_{(a)} \YY^c_{(b)}(p)\Big|&\les& r\sum_{a,c=1,2}\left|\YY^c_{(a)}(\Psi(p)) - \YY^c_{(a)}(p) \right|.
 \eeaa
 On the other hand, since we have from Lemma \ref{Le:Transportcomparison} 
 \beaa
  \left| Y^c_{(a)}(\Psi(p)) - Y^c_{(a)}(p)  \right|\les  \|(U, S)\|_{L^\infty} \sup_{\RR} \Big(\big| e_3(Y\big)| + r^{-1} \big|\dk Y \big|\Big)\les \dg,
\eeaa
we infer
 \beaa
 \left| \YY^c_{(a)}(\Psi(p)) - \YY^c_{(a)}(p)  \right|&\les& \left| Y^c_{(a)}(\Psi(p)) - Y^c_{(a)}(p)  \right|+\dg\les \dg.
 \eeaa
 We deduce that estimate \eqref{eq.defferencesofYY} holds true and therefore
 \beaa
 \sum_{a,b=1}^2  \big| g^{\S, \#}_{ab} -\ovg_{ab} \big|&\les r\dg 
 \eeaa
 as stated.
 
To prove the second part of the Lemma  we write,
 \beaa
 \int_{\ovS}    |f^\#|^2    da_{g^{\S, \#}} &=&  \int_{\ovS}  
   |f^\#|^2 \frac{\sqrt{\det g^{\S, \#}}}{\sqrt{\det\ovg}} da_{\ovg}\\
   &=& \int_{\ovS}     |f^\#|^2  da_{\ovg}+ \int_{\ovS}     |f^\#|^2\left( \frac{\sqrt{\det g^{\S, \#}}}{\sqrt{\det\ovg}}-1\right)  da_{\ovg}
 \eeaa
which yields, in view of  the first part,
 \beaa
 \int_{\ovS}    |f^\#|^2    da_{g^{\S, \#}} &=&  \int_{\ovS}     |f^\#|^2  da_{\ovg} \Big(1+ O(\ovr^{-1}\dg)  \Big).
 \eeaa
 The proof of the last statement follows easily from the form of the Christoffel symbols  of the two metrics in the  coordinates $y^1, y^2$ by following the calculations made for the first statement. 
\end{proof}

\begin{proposition}
\lab{Prop:comparison-gaS-ga:highersobolevregularity}
We assume
 \bea\lab{eq:boundonUSin h_snorms}
   \|( U, S)\|_{\hk_{s_{max}+1}(\ovS)}  &\les& r\dg.
 \eea
Then
\begin{enumerate}

 \item 
 If $V\in \hk_s(\S)$ and $V^\#$ is its pull-back by $\Psi$, we have for all $0\leq s\leq s_{max}+1$,
 \bea
 \lab{eq:Prop:comparison1}
 \|V\|_{\hk_s(\S)}= \|V^\#\|_{\hk_s(\ovS,\, g^{\S,\#})} = \| V^\#\|_{\hk_s(\ovS, \ovg)}\big(1+O(r^{-1} \dg)\big).
 \eea
 where, recall, $g^{\S,\#} $  denotes the pull-back by $\Psi$ of the metric  $g^\S$ on $\S$.
 \item  For any  scalar $h$  on $\RR$
 \bea
 \lab{eq:Prop:comparison2}
 \|h\|_{\hk_s(\S)} \les  r \sup_{\RR}\big|\dk^{\leq s}h\big|, \qquad 0\leq s \leq s_{max}.
 \eea
\end{enumerate}
\end{proposition}

\begin{proof}
The second statement is a consequence of the first. Indeed according to \eqref{eq:Prop:comparison1} it suffices to estimate
the norm $ \|h^\#\|_{\hk_s(\ovS)} $.  Clearly, by chain rule   and  assumption \eqref{eq:boundonUSin h_snorms},
 \beaa 
 \|h^\#\|_{\hk_s(\ovS)} &\les&r \big\|\dkb^{\leq s }h^\#\big\|_{L^\infty(\ovS)}\les  r \sup_{\RR}\big|\dk^{\leq s}h\big|.
 \eeaa

 To prove \eqref{eq:Prop:comparison1},   we  assume for simplicity that  $V$  is a one form on $\ovS$. Also, for a covariant $k$-tensor $H$ on $\ovS$, we denote 
  \bea
  |H|_{\ovg} &:=& \left((\ovg)^{a_1b_1}\cdots(\ovg)^{a_kb_k}H_{a_1\cdots a_k}H_{b_1\cdots b_k}\right)^{\frac{1}{2}}.
  \eea
  In particular, the properties of $\ovg$ in {\bf A3}  imply for any covariant $k$-tensor $H$ on $\ovS$
  \bea\lab{eq:comparisionbetweentensornormsonovS}
  r^{k}|H|_{\ovg}\les \sum_{a_1,\cdots, a_k=1,2}|H_{a_1\cdots a_k}|\les r^{k}|H|_{\ovg}.
  \eea
  
  {\bf Step 1.} We compare the first covariant derivatives of $V^\#$  with respect to the two connections
  \beaa
  \nab_a^\# V^\#_b&=& \pr_a V^\#_b-(\Ga^\#)_{ab}^c V^\#_c,\\
   \nabzero_a V^\#_b&=& \pr_a V^\#_b-(\ovGa)_{ab}^c V^\#_c.
  \eeaa
  Hence,
  \beaa
   \nab_a^\# V^\#_b-   \nabzero_a V^\#_b&=& -\big( (\Ga^\#)_{ab}^c-(\ovGa)_{ab}^c\big) V^\#_c
  \eeaa
  and thus, using \eqref{eq:comparisionbetweentensornormsonovS} and Lemma \ref{lemma:comparison-gaS-ga}, 
  \beaa
  \left|\nab^\# V^\# -   \nabzero V^\#\right|_{\ovg} &\les& r^{-2}\sum_{a,b=1,2}\left| \nab_a^\# V^\#_b-   \nabzero_a V^\#_b\right|\\
  &=& r^{-2}\sum_{a,b=1,2}\left| \big( (\Ga^\#)_{ab}^c-(\ovGa)_{ab}^c\big) V^\#_c\right|\\
  &\les& r^{-1}\left(\sum_{a,b,c=1,2}\left|(\Ga^\#)_{ab}^c-(\ovGa)_{ab}^c\right|\right)|V^\#|_{\ovg} \\
  &\les& r^{-1}\dg|V^\#|_{\ovg}.
  \eeaa
  Hence, 
  \beaa
  \big\|  \nab^\# V^\# \|_{L^2(\ovS, g^{\S, \#})} &=& \big\|  \nab^\# V^\# \|_{L^2(\ovS, \ovg)}  \big(1+O(r^{-1} \dg)\big)\\
  &\les&   \Big( \big\|  \nabzero V^\# \|_{L^2(\ovS, \ovg)}  + r^{-1} \dg  \|  V^\#   \|_{L^2(\ovS, \ovg)} \Big) \big(1+O(r^{-1} \dg)\big).
  \eeaa
  Thus, recalling the definition of the spaces   $\hk_s(\ovS)$,
  \bea
   \big\| \nab^\S V \big \|_{L^2(\S)} &\les & r^{-1}  \big\|  V^\# \big \|_{\hk_1(\ovS)}  \big(1+O(r^{-1} \dg)\big).
  \eea
  {\bf Step 2.} We assume by iteration  for $1\leq k\leq s_{max}$ 
  \beaa
  \big\| (\nab^\S )^k  V \big \|_{L^2(\S)} &\les & r^{-k}  \big\|  V^\# \big \|_{\hk_k(\ovS)}  \big(1+O(r^{-1} \dg)\big).
  \eeaa
Note that the iteration assumption holds for $k=1$ by Step 1, and our goal is to prove the analog estimate for $k+1$ derivatives. Writing $(\nab^\S )^{k+1}  V=(\nab^\S )^k (\nab^\S V)$, we have, using the iteration assumption 
  \beaa
  \big\| (\nab^\S )^{k+1}  V \big \|_{L^2(\S)} &\les & r^{-k}  \big\|  \nab^\# V^\# \big \|_{\hk_k(\ovS)}  \big(1+O(r^{-1} \dg)\big).
  \eeaa
  Then, decomposing $\nab^\# V^\#$ analogously to Step 1, using \eqref{eq:comparisionbetweentensornormsonovS} and Lemma \ref{lemma:comparison-gaS-ga}, and since $\hk_k(\ovS)$ is an algebra for $k\geq 2$ by the Sobolev embedding, we have
  \beaa
  \big\|  \nab^\# V^\# \big \|_{\hk_k(\ovS)} &\les&  \big\|  \nabzero V^\# \big \|_{\hk_k(\ovS)}+r^{-2}\sum_{a,b,c=1,2}\big\| \big( (\Ga^\#)_{ab}^c-(\ovGa)_{ab}^c\big) V^\#_c \big \|_{\hk_k(\ovS)}\\
  &\les&    \big\|  \nabzero V^\# \big \|_{\hk_k(\ovS)}+r^{-1}\sum_{a,b,c=1,2}\Big(\big\| (\Ga^\#)_{ab}^c-(\ovGa)_{ab}^c \big\|_{\hk_k(\ovS)}\\
  &&+\big\|  (\Ga^\#)_{ab}^c-(\ovGa)_{ab}^c \big\|_{\hk_1^\infty(\ovS)} \Big) \|V^\#\|_{\hk_k(\ovS)}\\
    &\les&  \big\|  \nabzero V^\# \big \|_{\hk_k(\ovS)}+r^{-1}\dg\big\| V^\# \big \|_{\hk_k(\ovS)}
  \eeaa
  and hence
   \beaa
  \big\| (\nab^\S )^{k+1}  V \big \|_{L^2(\S)} &\les & r^{-k}\Big( \big\|  \nabzero V^\# \big \|_{\hk_k(\ovS)}+r^{-1}\dg\big\| V^\# \big \|_{\hk_k(\ovS)}\Big)  \big(1+O(r^{-1} \dg)\big)\\
  &\les& r^{-k-1}  \big\|  V^\# \big \|_{\hk_{k+1}(\ovS)}  \big(1+O(r^{-1} \dg)\big)
  \eeaa
  which is the iteration assumption for $k+1$. Hence, we deduce that we have for all $1\leq k\leq s_{max}+1$
   \beaa
  \big\| (\nab^\S )^k  V \big \|_{L^2(\S)} &\les & r^{-k}  \big\|  V^\# \big \|_{\hk_k(\ovS)}  \big(1+O(r^{-1} \dg)\big).
  \eeaa
  Therefore,
  \beaa
  \big\|   V \big \|_{\hk_k(\S)} &\les &(r^\S)^k r^{-k}  \big\|  V^\# \big \|_{\hk_k(\ovS)}  \big(1+O(r^{-1} \dg)\big)
  \les  \big\|  V^\# \big \|_{\hk_k(\ovS)}  \big(1+O(r^{-1} \dg)\big)
  \eeaa
  as stated.
 \end{proof}
 
We have  the following corollary of Lemma \ref{lemma:comparison-gaS-ga} and Proposition \ref{Prop:comparison-gaS-ga:highersobolevregularity}.
 
  \begin{corollary}\lab{cor:comparison-gaS-ga-Oepgsphere}
 Let $\ovS \subset \RR$.    Let  $\Psi:\ovS\longrightarrow \S $  be   a  deformation generated by the  functions $(U, S)$ as in Definition \ref{definition:Deformations}. Assume the bound
 \beaa
   \| (U, S)\|_{L^\infty(\ovS)} +r  \| \nabzero(U, S)\|_{L^\infty(\ovS)}  &\les&  \dg.
 \eeaa
 Then, we have 
  \beaa
\sup_{\S}|r-r^\S| \les \dg,\qquad \sup_{\S}|m-\ov{m}^\S| \les \epg\dg.
 \eeaa
If we assume in addition that
\beaa
\big\|(U, S)\big\|_{\hk_{s_{max}+1}(\ovS)}  &\les r \dg 
\eeaa
then
 \beaa
K^\S = \frac{1+O(\epg)}{(r ^{\S})^2}, \qquad\quad \left\|K^\S -\frac{1}{(r^\S)^2}\right\|_{\hk_{s_{max}-1}(\S)}\les (r^\S)^{-1}\epg.
\eeaa
In particular, since $\dg\leq \epg$ by assumption, we infer that $\S$ is an $O(\epg)$-sphere.
\end{corollary}

\begin{proof}
We start with the estimate for $r-r^\S$ on $\S$. Consider a point $(y, \ug, \sg)$ on $\ovS$, and the corresponding point $(y, \ug+U(y), \sg+S(y))$ on $\S$. Then, we have
\beaa
r(y, \ug+U(y), \sg+S(y)) - \ovr &=& r(y, \ug+U(y), \sg+S(y)) - r(y, \ug, \sg)\\
&=& U(y)\int_0^1\pr_ur(y, \ug+\la U(y), \sg+\la S(y))d\la\\
&&+S(y)\int_0^1\pr_sr(y, \ug+\la U(y), \sg+\la S(y))d\la
\eeaa
and hence, we have 
\beaa
\sup_{\S}|r-\rg| &\les& \sup_{\RR}(|\pr_ur|+|\pr_sr|)\|U, S\|_{L^\infty(\ovS)}\les\dg
\eeaa
where we used  $|\pr_ur|+|\pr_sr|\les 1$ from the expression for $\pr_u$ and $\pr_s$ and from the control of the background foliation on $\RR$ given by assumptions {\bf A1}-{\bf A3}. Together with Lemma \ref{lemma:comparison-gaS-ga}, we infer 
\beaa
\sup_{\S}|r - r^\S| &\les& \dg
\eeaa
as desired.

Similarly, we have 
\beaa
\sup_{\S}|m-\mg| &\les& \sup_{\RR}(|\pr_um|+|\pr_sm|)\|U, S\|_{L^\infty(\ovS)}\les \epg\dg
\eeaa
where we used $|\pr_um|+|\pr_sm|\les \epg$ from the expression for $\pr_u$ and $\pr_s$ and from the control of the background foliation on $\RR$ given by assumptions {\bf A1}-{\bf A3}. Hence, we infer 
\beaa
\sup_{\S}|m-\ov{m}^\S| &\les& \epg\dg
\eeaa
as desired.

Also, using again Lemma \ref{lemma:comparison-gaS-ga}, we have
\beaa
\sum_{a,b,c=1,2}\Big\|(\Ga^{\S, \#})_{ab}^c-(\ovGa)_{ab}^c\Big\|_{\hk_{s_{max}}(\ovS)} &\les r\dg.
\eeaa
We deduce 
\beaa
\Big\|K^{\S, \#}-\overset{\circ}{K}\Big\|_{\hk_{s_{max}-1}(\ovS)} &\les r^{-1}\dg
\eeaa
and hence, using assumptions {\bf A1} on the sphere $\ovS=S(\ug, \sg)$, 
\beaa
\left\|K^{\S, \#}-\frac{1}{(\rg)^2}\right\|_{\hk_{s_{max}-1}(\ovS)} &\les r^{-1}\epg.
\eeaa
Together with Proposition \ref{Prop:comparison-gaS-ga:highersobolevregularity}, and the fact that $|r^\S-\rg|\les\dg$, we infer
\beaa
\left\|K^\S -\frac{1}{(r^\S)^2}\right\|_{\hk_{s_{max}-1}(\S)}\les (r^\S)^{-1}\epg
\eeaa
as desired. The $L^\infty$ estimate then follows using the Sobolev embedding and the fact that $s_{max}\geq 3$. 
\end{proof}

%%%%%%%%%%%%%%%%%%%%%%%%%%%%%%%%%%%%%

\subsection{Adapted $\ell=1$ modes}

%%%%%%%%%%%%%%%%%%%%%%%%%%%%%%%%%%%%%%

Consider a deformation $\Psi:\ovS \longrightarrow \S$ and recall the existence of the family of scalar functions $\Jp$, $p\in\big\{0, +, -\big\}$, on $\RR$  introduced in assumption {\bf A4}, see  \eqref{eq:Jpsphericalharmonics}, which form a basis of the $\ell=1$ modes on the spheres $S(u,s)$ of $\RR$, and hence in particular on $\ovS$.

 \begin{definition} 
 \lab{def:ell=1sphharmonicsonS}
 We define the basis of adapted $\ell=1$ modes  $\JpS$   on $\S$ by
 \beaa
\JpS = \Jp\circ\Psi^{-1}, \qquad p\in\big\{ -, 0, +\big\}.
 \eeaa
 \end{definition}
 
 \begin{proposition}
 Assume  the deformation verifies the bounds  \eqref{eq:boundonUSin h_snorms}.
If $\Jp$ is an admissible triplet of $\ell=1$ modes on $\RR$ (and hence on $\ovS$), i.e. satisfying \eqref{eq:Jpsphericalharmonics}, then  $\JpS$ is  an admissible triplet of $\ell=1$ modes on $\S$, i.e.
\bea
\lab{eq:admissibleJpS}
\bsplit
 \Big((r^\S)^2\lap^\S +2\Big) \JpS  &= O(\epg),\\
 \frac{1}{|\S|} \int_{\S}  \JpS J^{(\S, q)}  &=  \frac{1}{3}\de_{pq} +O(\epg),\\
 \frac{1}{|\S|}  \int_{\S}   \JpS   &=O(\epg).
\end{split}
\eea

Moreover at all point  of $\S$  we have
\bea
\lab{eq:admissibleJpS2}
\Big|\JpS-\Jp\Big|&\les \epg.
\eea
 \end{proposition} 
 
 \begin{proof}
 According to Lemma \ref{lemma:pullbacksofcovderivatives}  and the definition of $\JpS$
 \beaa
\Big( \Delta^\S \JpS\Big)^\#&=&  \Delta^{\S, \#} \Jp  =  \lapzero \Jp  +\Big(\Delta^{\S, \#} -\lapzero\Big) \Jp  \\
 &=&-\frac{2}{(\ovr)^2}  \Jp +         O(\epg \ovr^{-2}    )+\Big(\Delta^{\S, \#} -\lapzero\Big) \Jp. 
 \eeaa
 Now,
 \beaa
 \Big(\Delta^{\S, \#} -\lapzero\Big) \Jp &=& \Big((g^{\S,\#})^{ab}-(\ovg)^{ab}\Big) \Big(\pr_{y^a}\pr_{y^b} \Jp+(\Ga^{\S, \#})^c_{ab}  \pr_{y^c} \Jp\Big)\\
 &&+ (\ovg)^{ab}\Big((\Ga^{\S, \#})^c_{ab} -(\ovGa)^c_{ab} \Big) \pr_{y^c} \Jp. 
 \eeaa
 In view of the estimates   of Lemma \ref{lemma:comparison-gaS-ga}, 
 we deduce
 \beaa
 \Big| \big(\Delta^{\S, \#} -\lapzero\big) \Jp  \Big|&\les& r^{-2} \dg \left(\sum_{a=1,2}\Big|\pr_{y^a}\Jp  \Big|+\sum_{a,b=1,2}\Big|\pr_{y^a}\pr_{y^b}\Jp  \Big|\right) \les r^{-2} \dg. 
 \eeaa
 Therefore
 \beaa
 \Big( \Delta^\S \JpS\Big)^\#&=& -\frac{2}{(\ovr)^2}( \JpS)^\# +  O(\epg \ovr^{-2}    )
 \eeaa
 from which we deduce,
 \beaa
  \Delta^\S \JpS&=&  -\frac{2}{(r^\S)^2}\JpS+  O(\epg \ovr^{-2}    )
 \eeaa
 as stated.
 
 Also,
 \beaa
  \int_{\S}  \JpS J^{(\S, q)} da_{g^\S} &=& \int_{\ovS}  \Jp  \ovJq da_{g^{\S, \#}}= \int_{\ovS}  \Jp  \ovJq \sqrt{\det g^{\S, \#}}\\
   &=&\int_{\ovS}  \Jp  \ovJq \frac{\sqrt{\det g^{\S, \#}}}{\sqrt{\det\ovg}} da_{\ovg}\\
   &=&\int_{\ovS}  \Jp  \ovJq  da_{\ovg}+\int_{\ovS}  \Jp  \ovJq \left( \frac{\sqrt{\det g^{\S, \#}}}{\sqrt{\det\ovg}}-1\right)  da_{\ovg}\\
   &=&|\ovS| \left( \frac{1}{3}\de_{pq} +\epg\right)+ O( r^{-1} \dg) |\ovS|.
 \eeaa
 We infer that
 \beaa
 \frac{1}{|\S|}  \int_{\S}  \JpS J^{(\S, q)}&=& \frac{1}{3}\de_{pq} +O(\epg)
 \eeaa
 as stated.
 The last statement in \eqref{eq:admissibleJpS}  is proved in the same manner. Finally, 
   property \eqref{eq:admissibleJpS2} follows from an application of Lemma \ref{Le:Transportcomparison} using 
   the fact that $\JpS = \Jp\circ\Psi^{-1}$ together with the bounds  \eqref{eq:boundonUSin h_snorms} for the deformation.
      \end{proof}

%%%%%%%%%%%%%%%%%%%%%%%%%%%%%%%%%%%%%%%%%%%%%%%%%%%%%%%%

 \subsection{A corollary to Proposition \ref{Thm.GCMSequations-fixedS:contraction}}
 
 %%%%%%%%%%%%%%%%%%%%%%%%%%%%%%%%%%%%%%%%%%%%%%%%%%%%%%%%

The following corollary to Proposition \ref{Thm.GCMSequations-fixedS:contraction} will be used to prove contraction in an iterative scheme, see Proposition \ref{Prop:contractionforNN}. 
 \begin{corollary}
 \lab{Thm.GCMSequations-fixedS:contraction:deformationsphereversion}
 Let $\ovS \subset \RR$.    Let  $\Psi:\ovS\longrightarrow \S $  be   a  deformation generated by the  functions $(U, S)$ as in Definition \ref{definition:Deformations}. Assume the bound
\beaa
\big\|(U, S)\big\|_{\hk_{s_{max}+1}(\ovS)}  &\les r \dg. 
\eeaa
Let  $\La$, $\Lab$ in $\mathbb{R}^3$  and let $b_0$ a constant. Also, let $h_1$, $h_2$, $h_3$, $h_4$, $\underline{h}_1$ and $\underline{h}_2$ scalar functions on $\ovS$.    Assume given a solution     $(f, \fb, \ovla, \Cbdot_0, \Mdot_0, \Cbpdot, \Mpdot, \ovb)$ of the following system  on $\ovS$
  \bea
 \lab{GeneralizedGCMsystem:deformationsphereversion}
\bsplit
\curl^{\S,\#} f &= h_1 -\ov{h_1}^{\S,\#},\\
\curl^{\S,\#} \fb&= \underline{h}_1 - \ov{\underline{h}_1}^{\S,\#},\\
\div^{\S,\#} f + \frac{2}{r^{\S}} \ovla  -\frac{2}{(r^{\S})^2}\ovb &=  h_2,
\\
\div^{\S,\#}\fb + \frac{2}{r^{\S}} \ovla +\frac{2}{(r^{\S})^2}\ovb   
&=   \Cbdot_0+\sum_p \Cbpdot \Jp+\underline{h}_2,\\
\left(\Delta^{\S,\#}+\frac{2}{(r^{\S})^2}\right)\ovla  &=  \Mdot_0+\sum _p\Mpdot \Jp+\frac{1}{2r^{\S}}\left(\Cbdot_0+\sum_p \Cbpdot \Jp\right) +h_3,\\
\Delta^{\S,\#}\ovb-\frac{1}{2}\div^{\S,\#}\Big(\fb - f\Big) &= h_4 -\ov{h_4}^{\S,\#} , \qquad \ov{\ovb}^{\S,\#}=b_0,
\end{split}
\eea
and
     \bea
    \lab{eq:badmodesforffb:deformationsphereversion}
    (\div^{\S,\#} f)_{\ell=1}=\La, \qquad (\div^{\S,\#} \fb)_{\ell=1} =\Lab,
    \eea
 where  $g^{\S,\#} $  denotes the pull-back by $\Psi$ of the metric  $g^\S$ on $\S$,  where $\div^{\S,\#}$, $\curl^{\S,\#}$ and $\Delta^{\S,\#}$ are operator on $\ovS$ induced by the pull back metric, 
 and where the $\ell=1$ modes on $\ovS$ in \eqref{eq:badmodesforffb:deformationsphereversion} are defined with respect to $\Jp$.

   Then, the following a priori estimates are verified  
\bea
&&\|(f,\fb, \widecheck{\ovla}^{\S,\#})\|_{\hk_3(\ovS)} +\sum_p\Big(r^2|\Cbpdot|+r^3|\Mpdot|\Big)\\  
\nn&\les& r\|(\widecheck{h_1}^{\S,\#}, \,\widecheck{\underline{h}_1}^{\S,\#}, \,\widecheck{h_2}^{\S,\#},\,\widecheck{\underline{h}_2}^{\S,\#})\|_{\hk_2(\ovS)} +r^2\|\widecheck{h_3}^{\S,\#}\|_{\hk_1(\ovS)}+r\|\widecheck{h_4}^{\S,\#}\|_{L^2(\ovS)} +|\La|+|\Lab|,
\eea
and
\bea
\nn r^2|\Cbdot_0|+r^3|\Mdot_0|+r\Big|\ov{\ovla}^{\S,\#}\Big|  &\les& r\|(\widecheck{h_1}^{\S,\#}, \,\widecheck{\underline{h}_1}^{\S,\#},\, h_2, \,\underline{h}_2)\|_{L^2(\ovS)} +r^2\|h_3\|_{L^2(\ovS)}\\
&&+r\|\widecheck{h_4}^{\S,\#}\|_{L^2(\ovS)} +|\La|+|\Lab|+|b_0|.
\eea   
 \end{corollary}

\begin{proof}
The proof follows by pulling back on $\ovS$ by the map $\Psi$  the statement of Proposition \ref{Thm.GCMSequations-fixedS:contraction} holding on $\S$, and by using Proposition \ref{Prop:comparison-gaS-ga:highersobolevregularity} to compare the norms $\hk_s(\ovS, g^{\S,\#})$ and $\hk_s(\ovS, \ovg)=\hk_s(\ovS)$ for $s=0,1,2,3$.
\end{proof}

%%%%%%%%%%%%%%%%%%%%%%%%%%%

 \subsection{Adapted frame transformations}
 
 %%%%%%%%%%%%%%%%%%%%%%%%%%%%

\begin{definition}
Given a deformation $\Psi:\ovS\longrightarrow \S$ we  say that 
 a new frame   $(e_3', e_4',  e_1', e_2')$ on $\S$, obtained from the standard frame $(e_3, e_4, e_1, e_2)$  via the transformation  \eqref{eq:Generalframetransf},  is  $\S$-adapted  if   the horizontal  vectorfields $e'_1, e'_2$   are tangent to $\S$.
\end{definition}

\begin{proposition}\lab{prop:mainpropequationsUSforadaptedspheres}
Consider a  fixed deformation $\Psi:\ovS\longrightarrow \S$  in $\RR$  generated by the functions $U, S:\ovS\longrightarrow \RRR$.   A new frame $e_4', e_3', e_1', e_2'$  on $\S$  generated by  $(f, \fb, \la)$   from the  reference frame  $e_4,  e_3,  e_1, e_2$    according to the transformation formulas  \eqref{eq:Generalframetransf} is $\S$-adapted if and only if   the following relations are satisfied
 \bea
 \lab{Compatibility-Deformation2}
 \bsplit
 \pr_{y^a} S&=  \Big( \SS(f, \fb, \Ga)_bY^b_{(a)} \Big)^\#,\\
 \pr_{y^a}  U&=\Big(\UU(f, \fb, \Ga)_bY^b_{(a)}\Big)^\#,
 \end{split}
 \eea
 where we have introduced the 1-forms $\SS(f, \fb, \Ga)$ and $\UU(f, \fb, \Ga)$ on $\S$ given by 
 \bea\lab{eq:defintionofmathcsalUandmathcalSforcompatibilitydeformation}
 \bsplit
 \SS(f, \fb, \Ga) &:= \frac{a_{22}}{a_{11}a_{22}-a_{12}a_{21}}f -\frac{a_{12}}{a_{11}a_{22}-a_{12}a_{21}}\left(\fb+\frac 1 4 |\fb|^2f\right), \\ 
 \UU(f, \fb, \Ga) &:=-\frac{a_{21}}{a_{11}a_{22}-a_{12}a_{21}}f +\frac{a_{11}}{a_{11}a_{22}-a_{12}a_{21}}\left(\fb+\frac 1 4 |\fb|^2f\right),
 \end{split}
 \eea
 with the scalars $a_{11}$, $a_{12}$, $a_{21}$, $a_{22}$ on $\S$  defined by
  \bea\lab{eq:defintionofmathcsalUandmathcalSforcompatibilitydeformation:defscalarsa11a12a21a22}
 \bsplit
 a_{11} &:= \vsi +\vsi Z\c f  -\frac{1}{4}|f|^2\vsi \Omb,\\
 a_{12} &:= \frac{1}{2}|f|^2,\\
 a_{21} &:= -\left(1+\frac{1}{2}f\c\fb  +\frac{1}{16} |f|^2  |\fb|^2\right)\vsi \Omb  +\vsi Z\c\left(\fb+\frac 1 4 |\fb|^2f\right)+\frac{1}{4}|\fb|^2\vsi,\\
  a_{22} &:= 2+f\c\fb  +\frac{1}{8} |f|^2  |\fb|^2.
  \end{split}
 \eea
\end{proposition}

\begin{remark}\lab{rem:formofmathcalUandmathcalSatthemainorder}
Note that \eqref{eq:defintionofmathcsalUandmathcalSforcompatibilitydeformation} and \eqref{eq:defintionofmathcsalUandmathcalSforcompatibilitydeformation:defscalarsa11a12a21a22} imply in particular in view of 
{\bf A1}, {\bf A3} and \eqref{eq:assumptionY_a^b}
 \bea\lab{eq:formofmathcalUandmathcalSatthemainorder}
 \bsplit
 \SS(f, \fb, \Ga) &= f +O\Big(\epg |f|+|f|^2+|\fb|^2\Big), \\ 
 \UU(f, \fb, \Ga) &=\frac{1}{2}\Big(-\Up f +\fb\Big)+O\Big(\epg |f|+|f|^2+|\fb|^2\Big).
 \end{split}
 \eea
\end{remark}

\begin{proof}
The frame $(e_4', e_3', e_1', e_2')$ is adapted to $\S$   if   the horizontal  vectorfields $e'_1, e'_2$   are tangent to $\S$, i.e. 
if and only if 
\beaa
\g( \YY_{(a)}, \la^{-1}e_4')=0, \qquad \g( \YY_{(a)}, \la e_3')=0.
\eeaa
Since
\beaa
\la^{-1}e_4' &=&  e_4 + f^b  e_b +\frac 1 4 |f|^2  e_3,\\
\la e_3' &=&\left(1+\frac{1}{2}f\c\fb  +\frac{1}{16} |f|^2  |\fb|^2\right) e_3 + \left(\fb^b+\frac 1 4 |\fb|^2f^b\right) e_b  + \frac 1 4 |\fb|^2 e_4,
\eeaa
this is equivalent to 
\beaa
\YY_{(a)}^3-\frac{1}{2}f_b\YY_{(a)}^b+\frac{1}{4}|f|^2\YY_{(a)}^4 &=& 0,\\
\left(1+\frac{1}{2}f\c\fb  +\frac{1}{16} |f|^2  |\fb|^2\right)\YY_{(a)}^4-\frac{1}{2}\left(\fb^b+\frac 1 4 |\fb|^2f^b\right)\YY_{(a)}^b+\frac{1}{4}|\fb|^2\YY_{(a)}^3 &=& 0.
\eeaa
Now, recall \eqref{coefficients-YYa-US}, 
\beaa
 \bsplit
   \YY_{(a)}^4&= \pr_{y^a} S-\frac 1 2 (\vsi \Omb)^\#\,  \pr_{y^a}  U,\\
    \YY_{(a)}^3&=\frac 1 2 \vsi^\# \pr_{y^a}  U, \\
     \YY_{(a)}^c&=  (Y_{(a)}^c)^\# -  (\vsi Z^c)^\# \, \pr_{y^a}  U.
   \end{split}
 \eeaa
 We infer
 \beaa
\frac 1 2 \vsi^\# \pr_{y^a}  U -\frac{1}{2}(f_b)^\#\left( (Y_{(a)}^b)^\# -  (\vsi Z^b)^\# \, \pr_{y^a}  U\right)+\frac{1}{4}(|f|^2)^\#\left(\pr_{y^a} S-\frac 1 2 (\vsi \Omb)^\#\,  \pr_{y^a}  U\right) &=& 0,\\
\left(1+\frac{1}{2}f\c\fb  +\frac{1}{16} |f|^2  |\fb|^2\right)^\#\left(\pr_{y^a} S-\frac 1 2 (\vsi \Omb)^\#\,  \pr_{y^a}  U\right)\\
-\frac{1}{2}\left(\fb_b+\frac 1 4 |\fb|^2f_b\right)^\#\left( (Y_{(a)}^b)^\# -  (\vsi Z^b)^\# \, \pr_{y^a}  U\right)+\frac{1}{4}(|\fb|^2)^\#\frac 1 2 \vsi^\# \pr_{y^a}  U &=& 0.
\eeaa
We rewrite this system as
 \beaa
\left( \vsi +   \vsi f\c Z  -\frac{1}{4}|f|^2\vsi \Omb\right)^\#\pr_{y^a}  U +\frac{1}{2}(|f|^2)^\#\pr_{y^a} S &=& (f\c Y_{(a)})^\# ,
\eeaa
and
\beaa
&&\left(-\left(1+\frac{1}{2}f\c\fb  +\frac{1}{16} |f|^2  |\fb|^2\right)\vsi\Omb  +\vsi Z\c\left(\fb+\frac 1 4 |\fb|^2f\right)+\frac{1}{4}|\fb|^2\vsi \right)^\#\pr_{y^a}  U\\
&&+\left(2+f\c\fb  +\frac{1}{8} |f|^2  |\fb|^2\right)^\#\pr_{y^a} S\\
& =& \left(\left(\fb+\frac 1 4 |\fb|^2f\right)\c Y_{(a)}\right)^\#.
\eeaa
We infer
 \beaa
 \bsplit
 \pr_{y^a} S&=  \Big( \SS(f, \fb, \Ga)_bY^b_{(a)} \Big)^{\#},\\
 \pr_{y^a}  U&=\Big(\UU(f, \fb, \Ga)_bY^b_{(a)}\Big)^{\#},
 \end{split}
 \eeaa
 where we have introduced the notation
  \beaa
   \bsplit
 \SS(f, \fb, \Ga) &= \frac{a_{22}}{a_{11}a_{22}-a_{12}a_{21}}f -\frac{a_{12}}{a_{11}a_{22}-a_{12}a_{21}}\left(\fb+\frac 1 4 |\fb|^2f\right), \\ 
 \UU(f, \fb, \Ga) &=-\frac{a_{21}}{a_{11}a_{22}-a_{12}a_{21}}f +\frac{a_{11}}{a_{11}a_{22}-a_{12}a_{21}}\left(\fb+\frac 1 4 |\fb|^2f\right),
 \end{split}
 \eeaa
 with the scalars $a_{11}$, $a_{12}$, $a_{21}$, $a_{22}$ on $\S$  defined by
 \beaa
 a_{11} &=& \vsi +\vsi Z\c f  -\frac{1}{4}|f|^2\vsi \Omb,\\
 a_{12} &=& \frac{1}{2}|f|^2,\\
 a_{21} &=& -\left(1+\frac{1}{2}f\c\fb  +\frac{1}{16} |f|^2  |\fb|^2\right)\vsi \Omb  +\vsi Z\c\left(\fb+\frac 1 4 |\fb|^2f\right)+\frac{1}{4}|\fb|^2\vsi,\\
  a_{22} &=& 2+f\c\fb  +\frac{1}{8} |f|^2  |\fb|^2.
 \eeaa
This concludes the proof of the proposition.
\end{proof}

  \begin{corollary}\lab{cor:comparison-gaS-ga-Oepgsphere:mSminusm}
 Let $\ovS \subset \RR$    Let  $\Psi:\ovS\longrightarrow \S $  be   a  deformation generated by the  functions $(U, S)$ as in Definition \ref{definition:Deformations}. Assume the bound
 \beaa
   \| (U, S)\|_{L^\infty(\ovS)} +r  \| \nabzero(U, S)\|_{L^\infty(\ovS)}+r^2  \| \nabzero^2(U, S)\|_{L^\infty(\ovS)}  &\les&  \dg.
 \eeaa
 Then, we have 
  \beaa
\sup_{\S}|m-m^\S| \les \dg.
 \eeaa
\end{corollary}

\begin{proof}
In view of the proof of Corollary \ref{cor:comparison-gaS-ga-Oepgsphere}, we have
\beaa
\sup_{\S}|m-\mg| &\les& \epg\dg.
\eeaa
Thus, from now on, we focus on proving 
\beaa
|m^\S-\mg| &\les& \dg.
\eeaa
We have 
\beaa
m^\S-\mg &=& \frac{r^\S}{2}+\frac{r^\S}{32\pi}\int_{\S}\ka^\S\kab^\S -\frac{\rg}{2}- \frac{\rg}{32\pi}\int_{\ovS}\ka\kab\\
&=& \frac{r^\S}{32\pi}\int_{\S}\left(\ka^\S\kab^\S+\frac{4}{(r^\S)^2}\right) - \frac{\rg}{32\pi}\int_{\ovS}\left(\ka\kab+\frac{4}{(\rg)^2}\right).
\eeaa
In view of Lemma \ref{lemma:comparison-gaS-ga}, we infer
\beaa
|m^\S-\mg| &\les& \dg+r\left|\int_{\S}\left(\ka^\S\kab^\S+\frac{4}{(r^\S)^2}\right) - \int_{\ovS}\left(\ka\kab+\frac{4}{(\rg)^2}\right)\right|\\
&\les& \dg+r\left|\int_{\S}\left(\ka^\S\kab^\S+\frac{4}{(r^\S)^2} - \left(\ka\kab+\frac{4}{(\rg)^2}\right)\circ\Psi^{-1}\right)\right|\\
&\les& \dg+r\left|\int_{\S}\left(\ka^\S\kab^\S+\frac{4}{(r^\S)^2} - \left(\ka\kab+\frac{4}{(\rg)^2}\right)\right)\right|
\eeaa
and hence
\beaa
|m^\S-\mg| &\les& \dg+r\left|\int_{\S}\Big(\ka^\S\kab^\S -\ka\kab\Big)\right|.
\eeaa

We denote by $(f, \fb, \la)$ the frame coefficients between the background frame of $\RR$ and the frame $(e^\S_1, e^\S_2, e^\S_4, e^\S_3)$ adapted to $\S$. Using the following frame transformation formulas of Proposition \ref{Prop:transformation formulas-integrtogeneral}
\beaa
\la^{-1}\trch^\S &=& \trch  +  \div^\S f  +\err(\trch,\trch^\S),\\
\la\trchb^\S &=& \trchb +\div^\S\fb   +\err(\trchb, \trchb^\S),
\eeaa
we infer
\beaa
\ka^\S\kab^\S &=& \ka\kab +\ka\div^\S\fb+\kab\div^\S f+(\ka+ \div^\S f  +\err(\trch,\trch^\S))\err(\kab, \kab^\S)\\
&&+(\kab+\div^\S\fb)\err(\ka,\ka^\S)
\eeaa
and hence
\beaa
|\ka^\S\kab^\S -\ka\kab| &\les& r^{-2}\Big(|f|+|\dkb^\S f|+|\fb|+|\dkb^\S\fb|\Big).
\eeaa
This implies
\beaa
|m^\S-\mg| &\les& \dg+\|f\|_{\hk_1(\S)}+\|\fb\|_{\hk_1(\S)}.
\eeaa
Now, since $(f, \fb, \la)$ are the frame coefficients between the background frame of $\RR$ and the frame $(e^\S_1, e^\S_2, e^\S_4, e^\S_3)$ adapted to $\S$, we have in view of \eqref{eq:formofmathcalUandmathcalSatthemainorder}, 
 \beaa
\|f\|_{\hk_1(\S)}+\|\fb\|_{\hk_1(\S)} &\les&  r\| \UU(f, \fb, \Ga)\|_{L^\infty(\S)}+  r\| \SS(f, \fb, \Ga)\|_{L^\infty(\S)}\\
&& +r^2\|\nab^\S \UU(f, \fb, \Ga)\|_{L^\infty(\S)}+r^2  \| \nab^\S\SS(f, \fb, \Ga)\|_{L^\infty(\S)} 
\eeaa
 which together with \eqref{Compatibility-Deformation2} and  the control for $Y_{(a)}$ provided by \eqref{eq:assumptionY_a^b}   yields
\beaa
\|f\|_{\hk_1(\S)}+\|\fb\|_{\hk_1(\S)} &\les& r \| \nabzero(U, S)\|_{L^\infty(\ovS)}+r^2 \| \nabzero^2(U, S)\|_{L^\infty(\ovS)}\les\dg 
\eeaa
and hence
\beaa
|m^\S-\mg| &\les& \dg.
\eeaa
This concludes the proof of the corollary.
\end{proof}

%%%%%%%%%%%%%%%%%%%%%%

\section{Existence of GCM spheres}
\lab{sec:exitstenceofGCMspheres}

%%%%%%%%%%%%%%%%%%%%%%

%%%%%%%%%%%%%%%%%%%%%%%

\subsection{Statement of the main theorem}

%%%%%%%%%%%%%%%%%%%%%%%

In what follows we  consider  deformations $\Psi:\ovS\longrightarrow \S$  endowed  with adapted frames  $(e_1^\S, e_2^\S, e_3^\S, e^\S_4)$ on $\S$.
 As in section \ref{section:GCMspheres}  we denote by   $\nab^\S$ the induced covariant derivative on $\S$ and by      $\Ga^\S$, $R^\S$ the corresponding Ricci  and curvature coefficients   associated to the  frame.

  The following theorem is the main result of this  paper.
  \begin{theorem}[Existence of GCM spheres]
\lab{Theorem:ExistenceGCMS1}
Let $m_0>0$ a constant.   Let $0<\dg\leq \epg $   two sufficiently   small   constants, and let  $(\ug, \sg, \rg)$ three real numbers with $\rg$ sufficiently large so that
\beaa
\epg\ll m_0, \qquad\qquad  \rg\gg m_0.
\eeaa
Let a fixed  spacetime region $\RR$, as in Definition \ref{defintion:regionRRovr}, together with a $(u, s)$ outgoing geodesic   foliation verifying the  assumptions ${\bf A1-A4}$, see section \ref{sec:defintionspacetimeregionRR}. Let  $\ovS=S(\ovu, \ovs)$   be  a fixed    sphere from this foliation, and let $\rg$ and $\mg$ denoting respectively its area radius and its Hawking mass.  Assume that  the GCM quantities   $\ka, \kab, \mu$  of the background foliation verify          the  following:
\bea
\bsplit
\ka&=\frac{2}{r}+\dot{\ka},\\
\kab&=-\frac{2\Up}{r} +  \Cb_0+\sum_p \Cbp \Jp+\dot{\kab},\\
\mu&= \frac{2m}{r^3} + M_0+\sum _p\Mp \Jp+\dot{\mu},
\end{split}
\eea
where
\bea
\lab{eq:GCM-improved estimate1-again}
|\Cb_0, \Cbp| &\les& r^{-2} \epg, \qquad  |M_0, \Mp| \les r^{-3} \epg,
\eea
and
\bea
\lab{eq:GCM-improved estimate2-again}
\big\| \kadot, \kabdot\|_{\hk_{s_{max} }(\S) }&\les& r^{-1}\dg,\qquad 
\big\|\mudot\| _{\hk_{s_{max} }(\S) }\les r^{-2}\dg.
\eea
 Then
for any fixed pair of triplets   $\La, \Lab \in \RRR^3$  verifying
\bea\lab{eq:assumptionsonLambdaabdLambdabforGCMexistence}
|\La|,\,  |\Lab|  &\les & \dg,
\eea
 there 
exists a unique  GCM sphere $\S=\S^{(\La, \Lab)}$, which is a deformation of $\ovS$, 
such that  the GCM conditions  of Definition \ref{definition:GCMS} are verified,
 i.e. there exist constants $\Cb^\S_0,\,  \CbpS$, \, $ M^\S_0$, \, $\MpS, \, p\in\{-,0, +\}$  for which
 \bea
\lab{def:GCMC}
\bsplit
\ka^\S&=\frac{2}{r^\S},\\
\kab^\S &=-\frac{2}{r^\S}\Up^\S+  \Cb^\S_0+\sum_p \CbpS \JpS,\\
\mu^\S&= \frac{2m^\S}{(r^\S)^3} +   M^\S_0+\sum _p\MpS \JpS,
\end{split}
\eea
where we recall that $\JpS=\Jp\circ\Psi^{-1}$, see Definition \ref{def:ell=1sphharmonicsonS}. Moreover, 
 \bea
\lab{GCMS:l=1modesforffb}
(\div^\S f)_{\ell=1}&=\La, \qquad   (\div^\S\fb)_{\ell=1}=\Lab,
\eea
where we recall that the $\ell=1$ modes for scalars  on $\S$ are defined by \eqref{eq:defell=1forscalarsonOofepgspheres}.

The resulting  deformation has the following additional properties:
\begin{enumerate}
\item The triplet $(f,\fb,\ovla)$ verifies
\bea
\lab{eq:ThmGCMS1}
\|(f,\fb, \ovla)\|_{\hk_{s_{max}+1}(\S)} &\les &\dg. 
\eea
\item The  GCM constants  $\Cb^\S_0,\,  \CbpS$, \, $ M^\S_0$, \, $\MpS, \, p\in\{-,0, +\}$  verify
\bea
\lab{eq:ThmGCMS2}
\bsplit
\big| \Cb^\S_0-\Cb_0\big|+\big| \CbpS-\Cbp\big|&\les r^{-2}\dg,\\
\big| M^\S_0-M_0\big|+\big| \MpS-\Mp\big|&\les r^{-3}\dg.
\end{split}
\eea

\item The volume radius $r^\S$  verifies
\bea
\lab{eq:ThmGCMS3}
\left|\frac{r^\S}{\rg}-1\right|\les  r^{-1} \dg.
\eea
\item  The parameter  functions $U, S$  of the deformation verify
\bea
\lab{eq:ThmGCMS4}
 \|( U, S)\|_{\hk_{s_{max}+1}(\ovS)}  &\les&  r \dg.
 \eea

\item The Hawking mass  $m^\S$  of $\S$ verifies the estimate
\bea
\lab{eq:ThmGCMS5}
 \big|m^\S-\ovm\big|&\les &\dg. 
 \eea
 \item The well defined\footnote{See Remark \ref{remark:welldefinedGa}.}
  Ricci and curvature coefficients of $\S$  verify,
   \bea
   \lab{eq:ThmGCMS6}
\bsplit
\| \Ga^\S_g\|_{\hk_{s_{max} }(\S) }&\les  \epg  r^{-1},\\
\| \Ga^\S_b\|_{\hk_{s_{max} }(\S) }&\les  \epg.
\end{split}
\eea
\end{enumerate}
 \end{theorem}

%%%%%%%%%%%%%%%%%%%%%%%%%%%%%%%%%%%%%%%%%

 \subsection{Structure of the  proof of Theorem  \ref{Theorem:ExistenceGCMS1}}

%%%%%%%%%%%%%%%%%%%%%%%%%%%%%%%%%%%%%%%%%

In view of Corollary \ref{Lemma-adaptedGCM-equations:ter} and Proposition \ref{prop:mainpropequationsUSforadaptedspheres}, 
 $\S$ is a GCM sphere which is a deformation of $\ovS$ if and only if the corresponding $(U,S, f,\fb, \ovla)$  solve the following coupled system 
         \bea\lab{eq:qequivalentGCMsystemwhicisnotsolvabledirectly:1}
     \bsplit
\curl ^\S f  &= -\err_1[\curl^\S  f ],\\
\curl^\S \fb  &= -\err_1[\curl^\S  \fb ],
\end{split}
\eea
         \bea\lab{eq:qequivalentGCMsystemwhicisnotsolvabledirectly:2}
     \bsplit
\div^\S f + \ka \ovla -\frac{2}{(r^\S)^2}\ovb &= \ka^\S-\frac{2}{r^\S} -\left(\ka-\frac{2}{r}\right) -\err_1[\div^\S f ] -\frac{2(r-r^\S)^2}{r(r^\S)^2},\\
\div^\S\fb - \kab \ovla +\frac{2}{(r^\S)^2}\ovb &= \kab^\S+\frac{2}{r^\S} -\left(\kab+\frac{2}{r}\right) -\err_1[\div^\S \fb ] +\frac{2(r-r^\S)^2}{r(r^\S)^2},\\
\Delta^\S\ovla + V\ovla &=\mu^\S-\mu -\left(\omb +\frac 1 4 \kab \right) \big(\ka^\S-\ka \big)\\
&+\left(\om +\frac 1 4 \ka \right) \big(\kab^\S-\kab \big)+\err_2[ \lap^\S\ovla],\\
\Delta^\S\ovb &= \frac{1}{2}\div^\S\left(\fb - \Up f +\err_1[\Delta^\S \ovb ] \right) , \quad \ov{\ovb}^\S=\ov{r}^\S-r^\S,
\end{split}
\eea
 \bea\lab{eq:qequivalentGCMsystemwhicisnotsolvabledirectly:3}
 \bsplit
 \pr_{y^a} S&=    \Big( \SS(f, \fb, \Ga)_bY^b_{(a)} \Big)^\#,\\
 \pr_{y^a}  U&=\Big(\UU(f, \fb, \Ga)_bY^b_{(a)}\Big)^\#,
 \end{split}
 \eea
together with  the GCM conditions \eqref{def:GCMC} and the prescribed $\ell=1$ conditions 
      \eqref{GCMS:l=1modesforffb}.  
      
      Note however that \eqref{eq:qequivalentGCMsystemwhicisnotsolvabledirectly:1} and \eqref{eq:qequivalentGCMsystemwhicisnotsolvabledirectly:3} are a priori not solvable. This forces us to  solve instead  the modified system 
         \bea\lab{eq:qequivalentGCMsystemwhicisnotsolvabledirectly:1:bis}
     \bsplit
\curl ^\S f  &= -\err_1[\curl^\S  f ]+\ov{\err_1[\curl^\S  f ]}^\S,\\
\curl^\S \fb  &= -\err_1[\curl^\S  \fb ]+\ov{\err_1[\curl^\S  \fb ]}^\S,
\end{split}
\eea
         \bea\lab{eq:qequivalentGCMsystemwhicisnotsolvabledirectly:2:bis}
     \bsplit
\div^\S f + \ka \ovla -\frac{2}{(r^\S)^2}\ovb &= \ka^\S-\frac{2}{r^\S} -\left(\ka-\frac{2}{r}\right) -\err_1[\div^\S f ] -\frac{2(r-r^\S)^2}{r(r^\S)^2},\\
\div^\S\fb - \kab \ovla +\frac{2}{(r^\S)^2}\ovb &= \kab^\S+\frac{2}{r^\S} -\left(\kab+\frac{2}{r}\right) -\err_1[\div^\S \fb ] +\frac{2(r-r^\S)^2}{r(r^\S)^2},\\
\Delta^\S\ovla + V\ovla &=\mu^\S-\mu -\left(\omb +\frac 1 4 \kab \right) \big(\ka^\S-\ka \big)\\
&+\left(\om +\frac 1 4 \ka \right) \big(\kab^\S-\kab \big)+\err_2[ \lap^\S\ovla],\\
\Delta^\S\ovb &= \frac{1}{2}\div^\S\left(\fb - \Up f +\err_1[\Delta^\S \ovb ] \right) , \quad \ov{\ovb}^\S=\ov{r}^\S-r^\S,
\end{split}
\eea
 \bea\lab{systemUU-SS-derived}
 \bsplit
\lapzero U&=\divzero\left(\big(\UU(f, \fb, \Ga)\big)^\#\right),\\
\lapzero S&=\divzero\left(\big(\SS(f, \fb, \Ga)\big)^\#\right),
 \end{split}
 \eea
together with  the GCM conditions \eqref{def:GCMC} and the prescribed $\ell=1$ conditions 
      \eqref{GCMS:l=1modesforffb}, and with  the values of $U, S$ fixed at the South Pole of $\ovS$ by
\bea
\lab{eq:US-southpole}
U(South)=S(South)=0.
\eea

The proof of Theorem  \ref{Theorem:ExistenceGCMS1} then proceeds as follows.
\begin{enumerate}
\item We introduce an iterative scheme for the resolution of the nonlinear system   \eqref{eq:qequivalentGCMsystemwhicisnotsolvabledirectly:1:bis}-\eqref{systemUU-SS-derived}, and we prove its convergence in section \ref{sec:iterativeschemeforGCMconstuction}.

\item We analyse  the limit $(U^{(\infty)},S^{(\infty)}, f^{(\infty)},\fb^{(\infty)}, \ovla^{(\infty)})$ of the iterative scheme, solution to  the nonlinear system \eqref{eq:qequivalentGCMsystemwhicisnotsolvabledirectly:1:bis}-\eqref{systemUU-SS-derived}, 
 in section  \ref{sec:limitoftheiterativescheme}. In particular, we exhibit two frames on the limiting sphere $\S^{(\infty)}$, one associated to the frame transformation coefficients $( f^{(\infty)},\fb^{(\infty)}, \ovla^{(\infty)})$, and one adapted to $\S^{(\infty)}$.

\item We then show in section \ref{sec:wheretheproofofthemaintheoremisfinallyconcluded}  that the two frame on $\S^{(\infty)}$ in fact coincide. This  implies  that  $(U^{(\infty)},S^{(\infty)}, f^{(\infty)},\fb^{(\infty)}, \ovla^{(\infty)})$ not only solves \eqref{eq:qequivalentGCMsystemwhicisnotsolvabledirectly:1:bis}-\eqref{systemUU-SS-derived}, but also solves the original system of equations \eqref{eq:qequivalentGCMsystemwhicisnotsolvabledirectly:1}-\eqref{eq:qequivalentGCMsystemwhicisnotsolvabledirectly:3} hence concluding the proof of Theorem  \ref{Theorem:ExistenceGCMS1}.  
\end{enumerate}

%%%%%%%%%%%%%%%%%%%%%%%%%%%%%%%%%%
       
 \subsection{Definition and convergence of the iterative scheme}
 \lab{sec:iterativeschemeforGCMconstuction}

%%%%%%%%%%%%%%%%%%%%%%%%%%%%%%%%%%

Starting with   the  trivial quintet  
  \beaa
  \QQ^{(0)}:=(U^{(0)}, S^{(0)}, \ovla^{(0)},  f^{(0)},  \fb^{(0)})=(0,0,0,0,0), 
  \eeaa
  corresponding to the undeformed sphere $\ovS$,  we define iteratively  the quintet 
 \beaa
 \QQ^{(n+1)}= \Big(U^{(n+1)}, S^{(n+1)}, \ovla^{(n+1)},  f^{(n+1)},  \fb^{(n+1)}\Big)= \Big(U^{(n+1)}, S^{(n+1)},
F^{(n+1)}\Big)
 \eeaa
  from
  \beaa
     \QQ^{(n)}= \Big(U^{(n)}, S^{(n)}, \ovla^{(n)},  f^{(n)},  \fb^{(n)}\Big)=
      \Big(U^{(n)}, S^{(n)},  F^{(n)}\Big)
  \eeaa
   as follows.

  {\bf Step 1.} The pair  $(U^{(n)}, S^{(n)})$ defines the deformation  sphere $\S(n)$ and the corresponding pull back map  $\#_n$ given by the map $\Psi^{(n)} :\ovS\longrightarrow \S(n)$,
  \bea
  \label{definition:Psin-iteration}
  (\ovu, \ovs,y^1, y^2) \longrightarrow (\ovu+U^{(n)}(y^1, y^2) , \ovs+S^{(n)}(y^1, y^2) ,  y^1, y^2).
  \eea
By induction we may assume that the following estimates hold true  for $ U^{(n)}$, $S^{(n)}$, $\fn$, $\fbn$, $\ovlan$, and the constants $\dot{\underline{C}}_0^{(n)}$, $\dot{\underline{C}}^{(n),p}$,  
\bea
 \label{eq:ExistenceGCMS-Thm-US-steon}
   \|\left(U^{(n)}, S^{(n)}\right)\|_{\hk_{s_{max}+1} (\ovS)} &\les&  \dg r,
  \eea
\bea
 \label{eq:ExistenceGCMS-Thm-US-steon:fnfbnovlan}
   \|(\fn,\fbn, \widecheck{\ovlan}^{\S(n-1)})\|_{\hk_{s_{max}+1}(\S(n-1))}  +r^2\sum_p|\Cbpndot| &\les&  \dg,
  \eea
and
\bea \label{eq:ExistenceGCMS-Thm-US-steon:ovovlanandCbdot}
r^2|\Cbndot_0|+ r\Big|\ov{\ovlann}^{\S(n)}\Big|  &\les& \epg.
\eea

\begin{remark}\lab{rem:SofnisanOofepgsphere}
In view of Corollary \ref{cor:comparison-gaS-ga-Oepgsphere}, \eqref{eq:ExistenceGCMS-Thm-US-steon} implies in particular that $\S(n)$ is an $O(\epg)$-sphere. 
\end{remark}

  The surface  $\S(n)$ also comes equipped  with  the  triplet $\JpSn$,  $p\in\{-,0, +\}$, of adapted   $\ell=1$ modes, see Definition \ref{def:ell=1sphharmonicsonS}, 
   \beaa
\Big( \JpSn\Big)^{\#_n}= \Jp, \qquad p\in\big\{ -, 0, +\big\}.
 \eeaa

  The area radius of $\S(n)$ is denoted by  $ r^{(n)}:=r^{\S(n)} $. The Hawking mass of  $\S(n)$ is denoted by $ m^{(n)} :=m^{\S(n)}$.

{\bf Step 2.}    We  define the triplet $(\fnn, \fbnn, \ovlann)$ as the  solution of the following  linear  system
 of equations 
    \bea
 \lab{GeneralizedGCMsystem-n+1-MainThm1}
\bsplit
\curlSn \fnn &=h_1^{(n)}-\ov{h_1^{(n)}}^{\S(n)},\\
\curlSn \fbnn&=\underline{h}_1^{(n)}-\ov{\underline{h}_1^{(n)}}^{\S(n)},
\end{split}
\eea
\bea
 \lab{GeneralizedGCMsystem-n+1-MainThm2}
\bsplit
\divSn \fnn + \frac{2}{r^{\S(n)}} \ovlann -\frac{2}{(r^{\S(n)})^2}\ovb^{(n+1)} &=h_2^{(n)},\\
\divSn\fbnn +\frac{2}{r^{\S(n)}}  \ovlann   +\frac{2}{(r^{\S(n)})^2}\ovb^{(n+1)}
&= \Cbnndot_0+\sum_p \Cbpnndot \JpSn +\underline{h}_2^{(n)},
\end{split}
\eea
\bea
 \lab{GeneralizedGCMsystem-n+1-MainThm3}
\bsplit
\left(\lapSn+\frac{2}{(r^{\S(n)})^2}\right) \ovlann &=\Mnndot_0+\sum _p\Mpnndot \JpSn +h_3^{(n)}\\
&+\frac{1}{2r^{\S(n)}}\left(\Cbnndot_0+\sum_p \Cbpnndot \JpSn\right),\\
\lapSn\ovb^{(n+1)}-\frac{1}{2}\divSn\left(\fbnn -  \fnn\right) &= h_4^{(n)}-\ov{h_4^{(n)}}^{\S(n)}, \\ 
\ov{\ovb^{(n+1)}}^{\S(n)} &=\ov{r}^{\S(n)}-r^{\S(n)},
\end{split}
\eea

  where 
  \bea\lab{eq:definitionofh1nhb1nh2nhb2nh3nh4n}
  \bsplit
  h_1^{(n)} &:= - \err_1[\curl^{\S(n-1)}\fn]\circ(\Psi^{(n-1)}\circ(\Psi^{(n)})^{-1}),\\
  \underline{h}_1^{(n)} &:= - \err_1[\curl^{\S(n-1)} \fbn]\circ(\Psi^{(n-1)}\circ(\Psi^{(n)})^{-1}),\\
  h_2^{(n)} &:= -\left(\ka-\frac{2}{r^{\S(n)}}\right) \ovlan\circ(\Psi^{(n-1)}\circ(\Psi^{(n)})^{-1})  -\kadot\\
  & -\err_1[\div^{\S(n-1)} \fn]\circ(\Psi^{(n-1)}\circ(\Psi^{(n)})^{-1}) -\frac{2(r-r^{\S(n)})^2}{r(r^{\S(n)})^2},\\
 \underline{h}_2^{(n)} &:= \left(\kab+\frac{2}{r^{\S(n)}}\right) \ovlan\circ(\Psi^{(n-1)}\circ(\Psi^{(n)})^{-1})  -\kabdot + \frac{4m^{\S(n)}}{(r^{\S(n)})^2} - \frac{4m}{r^2}\\
 &  -\err_1[\div^{\S(n-1)}\fbn]\circ(\Psi^{(n-1)}\circ(\Psi^{(n)})^{-1})+\frac{2(r-r^{\S(n)})^2}{r(r^{\S(n)})^2},\\
  h_3^{(n)} &:= -\left(V-\frac{2}{(r^{\S(n)})^2}\right)\ovlan\circ(\Psi^{(n-1)}\circ(\Psi^{(n)})^{-1}) -\mudot  +\frac{2m^{\S(n)}}{(r^{\S(n)})^3}-\frac{2m}{r^3} \\
  &+\left(\om +\frac 1 4 \ka \right) \left(\frac{2\Up}{r} -\frac{2\Up^{\S(n)}}{r^{\S(n)} }+  \dot{\underline{C}}_0^{(n)}+\sum_p\dot{\underline{C}}^{(n),p} \JpSn  -\kabdot \right)\\
  &-\left(\omb +\frac 1 4 \kab \right) \left(\frac{2}{r^{\S(n)}} -\frac{2}{r} -\kadot \right)  -\frac{1}{2r^{\S(n)}}\left( \dot{\underline{C}}_0^{(n)}+\sum_p\dot{\underline{C}}^{(n),p} \JpSn\right)\\
  &  +\err_2[ \lap^{\S(n-1)}\ovlan]\circ(\Psi^{(n-1)}\circ(\Psi^{(n)})^{-1}),\\
   h_4^{(n)} &:= \div^{\S(n-1)}\left(\frac{2m}{r}\fn+\err_1[\Delta^{\S(n-1)} \ovb^{(n)} ]\right)\circ(\Psi^{(n-1)}\circ(\Psi^{(n)})^{-1}),
   \end{split}
  \eea
  with the notations 
  \beaa
\Cbnndot_0 &=& \Cbnn_0-\Cb_0, \quad \Cbpnndot= \Cbpnn-\Cbp, \\
 \Mnndot&=& \Mnn_0-M_0, \quad \Mpnndot=\Mpnn-\Mp,
\eeaa
and where  the  error terms  $\err_1[\curl^{\S(n-1)} \fn]$, $\err_1[\curl^{\S(n-1)} \fbn]$, $\err_1[\div^{\S(n-1)} \fn]$, $\err_1[\div^{\S(n-1)} \fbn]$,  $ \err_2[ \lap^{\S(n-1)}\ovlan]$ and $ \err_1[ \lap^{\S(n-1)}\ovb^{(n)}]$ depend only on the  previous  iterates $(\fn, \fbn, \ovlan)$ defined on $\S(n-1)$.  

The existence, uniqueness and control  of  $(\fnn, \fbnn, \ovlann)$ is  ensured by the following proposition.

\begin{proposition}\lab{prop:Estimates:Fn+1boundednessiterativescheme}
Under the induction assumptions \eqref{eq:ExistenceGCMS-Thm-US-steon} \eqref{eq:ExistenceGCMS-Thm-US-steon:fnfbnovlan} \eqref{eq:ExistenceGCMS-Thm-US-steon:ovovlanandCbdot},  there exists unique constants $\Cbnn_0$, $\Cbpnn$, $\Mnn_0$, $\Mpnn$ such that  the system \eqref{GeneralizedGCMsystem-n+1-MainThm1}--\eqref{GeneralizedGCMsystem-n+1-MainThm3} has a unique solution $(\fnn, \fbnn, \ovlann)$ with prescribed  $\ell=1$  modes
\bea
\lab{GeneralizedGCMsystem-n+1-LaLab-MainThm:00}
( \div^\S \fnn)_{\ell=1}=\La, \qquad   (\div^\S \fbnn)_{\ell=1}=\Lab,
\eea
relative to the  given triplet  $\JpSn$ of $\S(n)$. Moreover, we have
 \bea
\nn&&\|(\fnn,\fbnn)\|_{\hk_{s_{max}+1}(\S(n))} +\|\widecheck{\ovlann}^{\S(n)}\|_{\hk_{s_{max}+2}(\S(n))}\\
\nn&&+\sum_p\Big(r^2|\Cbpnndot|+r^3|\Mpnndot|\Big)\\
 &\les&  \left(\frac{1}{\rg}+\epg\right)\dg+ r\|\kadot, \kabdot\|_{\hk_{s_{max}}(\S(n))}+r^2\|\mudot\|_{\hk_{s_{max}}(\S(n))}+|\La|+|\Lab|
\eea
and
\bea
 \nn r^2|\Cbnndot_0|+r^3|\Mnndot_0|+r\Big|\ov{\ovlann}^{\S(n)}\Big|  &\les& \dg + r\|\kadot, \kabdot\|_{\hk_{s_{max}}(\S(n))}+r^2\|\mudot\|_{\hk_{s_{max}}(\S(n))} \\
 &&+|\La|+|\Lab|
\eea 
uniformly for all $n\in\NNN$.
\end{proposition}

\begin{proof}
Since $\S(n)$ is an $O(\epg)$-sphere, see Remark \ref{rem:SofnisanOofepgsphere}, we may apply Proposition  \ref{Thm.GCMSequations-fixedS}. We deduce the existence of unique constants $\Cbnn_0$, $\Cbpnn$, $\Mnn_0$, $\Mpnn$ such that the system  \eqref{GeneralizedGCMsystem-n+1-MainThm1} -\eqref{GeneralizedGCMsystem-n+1-MainThm3} has a unique solution $(\fnn, \fbnn, \ovlann)$ with prescribed  $\ell=1$  modes 
\beaa
( \div^\S \fnn)_{\ell=1}=\La, \qquad   (\div^\S \fbnn)_{\ell=1}=\Lab,
\eeaa
relative to the  given triplet  $\JpSn$ of $\S(n)$. Moreover, we have
\beaa
&&\|(\fnn,\fbnn, \widecheck{\ovlann}^{\S(n)})\|_{\hk_{s_{max}+1}(\S(n))} +\sum_p\Big(r^2|\Cbpnndot|+r^3|\Mpnndot|\Big)\\
 &\les& r\|(\widecheck{h_1^{(n)}}^{\S(n)}, \,\widecheck{\underline{h}_1^{(n)}}
^{\S(n)}, \,\widecheck{h_2^{(n)}}^{\S(n)},\,\widecheck{\underline{h}^{(n)}_2}^{\S(n)})\|_{\hk_{s_{max}}(\S(n))} \\
\nn&&+r^2\|\widecheck{h_3^{(n)}}^{\S(n)}\|_{\hk_{s_{max}-1}(\S(n))}+r\|\widecheck{h_4^{(n)}}^{\S(n)}\|_{\hk_{s_{max}-2}(\S(n))} +|\La|+|\Lab|,
\eeaa
and
\beaa
&& r^2|\Cbnndot_0|+r^3|\Mnndot_0|+r\Big|\ov{\ovlann}^{\S(n)}\Big| \\
 &\les& r\|(\widecheck{h_1^{(n)}}^{\S(n)}, \,\widecheck{\underline{h}_1^{(n)}}^{\S(n)}, \, h_2^{(n)},\,\underline{h}^{(n)}_2)\|_{L^2(\S(n))}   +r^2\|h_3^{(n)}\|_{L^2(\S(n))}+\|\widecheck{h_4^{(n)}}^{\S(n)}\|_{L^2(\S(n))} \\
&&+|\La|+|\Lab|+\sup_{\S(n)}|r-r^{\S(n)}|.
\eeaa 

Next,  we estimate $h_1^{(n)}, \cdots, h_4^{(n)}$. We have
\bea\lab{eq:simplyficationlongexpresssioninthedeifntionofh3}
\nn&& \left(\om +\frac 1 4 \ka \right) \left(\frac{2\Up}{r} -\frac{2\Up^{\S(n)}}{r^{\S(n)} }+  \dot{\underline{C}}_0^{(n)}+\sum_p\dot{\underline{C}}^{(n),p} \JpSn  -\kabdot \right)\\
 \nn &&-\left(\omb +\frac 1 4 \kab \right) \left(\frac{2}{r^{\S(n)}} -\frac{2}{r} -\kadot \right)  -\frac{1}{2r^{\S(n)}}\left( \dot{\underline{C}}_0^{(n)}+\sum_p\dot{\underline{C}}^{(n),p} \JpSn\right)\\
\nn  &=& \left(\frac{1}{2r} -\frac{1}{2r^{\S(n)}}\right)\left( \dot{\underline{C}}_0^{(n)}+\sum_p\dot{\underline{C}}^{(n),p} \JpSn\right)  +  \frac{2m^{\S(n)}}{r(r^{\S(n)})^2 }   -\frac{2m}{r^2r^{\S(n)}}  \\ 
 \nn && +  \left(\om +\frac 1 4 \left(\ka-\frac{2}{r}\right) \right) \left(\frac{2\Up}{r} -\frac{2\Up^{\S(n)}}{r^{\S(n)} }+  \dot{\underline{C}}_0^{(n)}+\sum_p\dot{\underline{C}}^{(n),p} \JpSn  \right)\\
  && -\left(\omb +\frac 1 4 \left(\kab+\frac{2\Up}{r}\right) \right) \left(\frac{2}{r^{\S(n)}} -\frac{2}{r} \right)   -\left(\om +\frac 1 4 \ka \right)\kabdot  +\left(\omb +\frac 1 4 \kab \right) \kadot  
\eea
and hence, in view of \eqref{eq:definitionofh1nhb1nh2nhb2nh3nh4n}, we may rewrite $h_3^{(n)}$ as
\bea\lab{eq:newformulafordefintionh3n}
\nn h_3^{(n)} &=& -\left(V-\frac{2}{(r^{\S(n)})^2}\right)\ovlan\circ(\Psi^{(n-1)}\circ(\Psi^{(n)})^{-1}) -\mudot  +\frac{m^{\S(n)}}{(r^{\S(n)})^3}-\frac{m}{r^3} \\
\nn && + \left(\frac{1}{2r} -\frac{1}{2r^{\S(n)}}\right)\left( \dot{\underline{C}}_0^{(n)}+\sum_p\dot{\underline{C}}^{(n),p} \JpSn\right)  +  \frac{2m^{\S(n)}}{r(r^{\S(n)})^2 }   -\frac{2m}{r^2r^{\S(n)}}  \\ 
\nn  && +  \left(\om +\frac 1 4 \left(\ka-\frac{2}{r}\right) \right) \left(\frac{2\Up}{r} -\frac{2\Up^{\S(n)}}{r^{\S(n)} }+  \dot{\underline{C}}_0^{(n)}+\sum_p\dot{\underline{C}}^{(n),p} \JpSn  \right)\\
\nn  && -\left(\omb +\frac 1 4 \left(\kab+\frac{2\Up}{r}\right) \right) \left(\frac{2}{r^{\S(n)}} -\frac{2}{r} \right)   -\left(\om +\frac 1 4 \ka \right)\kabdot  +\left(\omb +\frac 1 4 \kab \right) \kadot \\
  &&  +\err_2[ \lap^{\S(n-1)}\ovlan]\circ(\Psi^{(n-1)}\circ(\Psi^{(n)})^{-1}).
\eea
We also use,  for a scalar $\nu$ defined on $\S(n-1)$, the following consequence of \eqref{eq:Prop:comparison1}, which applies in view of \eqref{eq:ExistenceGCMS-Thm-US-steon},
\bea\lab{eq:inequalityusedforallscalarscomposedwtihpsinpsin-1}
\|\nu\circ(\Psi^{(n-1)}\circ(\Psi^{(n)})^{-1})\|_{L^2(\S^{(n)})} &\les& \|\nu\|_{L^2(\S^{(n-1)})}.
\eea
In view of the definition \eqref{eq:definitionofh1nhb1nh2nhb2nh3nh4n} of  $h_1^{(n)}, \cdots, \underline{h}_2^{(n)}$ and $h_4^{(n)}$,  using \eqref{eq:newformulafordefintionh3n} for $h_3^{(n)}$, and using \eqref{eq:inequalityusedforallscalarscomposedwtihpsinpsin-1} for the terms composed with $\Psi^{(n-1)}\circ(\Psi^{(n)})^{-1}$, 
we have
\beaa
&& r\|(\widecheck{h_1^{(n)}}^{\S(n)}, \,\widecheck{\underline{h}_1^{(n)}}
^{\S(n)}, \,\widecheck{h_2^{(n)}}^{\S(n)},\,\widecheck{\underline{h}^{(n)}_2}^{\S(n)})\|_{\hk_{s_{max}}(\S(n))} \\
&&+r^2\|\widecheck{h_3^{(n)}}^{\S(n)}\|_{\hk_{s_{max}-1}(\S(n))}+r\|\widecheck{h_4^{(n)}}^{\S(n)}\|_{\hk_{s_{max}-2}(\S(n))}\\
&\les& \left(\frac{1}{\rg}+\epg\right)\left(\dg+\sup_{\S(n)}|r-r^{\S(n)}|\right)+\sup_{\S(n)}|m-\ov{m}^{\S(n)}|+ r\|\kadot, \kabdot\|_{\hk_{s_{max}}(\S(n))}\\
&&+r^2\|\mudot\|_{\hk_{s_{max}}(\S(n))}
\eeaa
and
\beaa
&& r\|(\widecheck{h_1^{(n)}}^{\S(n)}, \,\widecheck{\underline{h}_1^{(n)}}
^{\S(n)},
 \, h_2^{(n)},\,\underline{h}^{(n)}_2)\|_{L^2(\S(n))}   +r^2\|h_3^{(n)}\|_{L^2(\S(n))}+\|\widecheck{h_4^{(n)}}^{\S(n)}\|_{L^2(\S(n))}\\
&\les& \left(\frac{1}{\rg}+\epg\right)\left(\dg+\sup_{\S(n)}|r-r^{\S(n)}|\right)+\sup_{\S(n)}|m-m^{\S(n)}|+ r\|\kadot, \kabdot\|_{\hk_{s_{max}}(\S(n))}\\
&&+r^2\|\mudot\|_{\hk_{s_{max}}(\S(n))}
\eeaa
where we have used \eqref{eq:ExistenceGCMS-Thm-US-steon}, \eqref{eq:ExistenceGCMS-Thm-US-steon:fnfbnovlan} and 
\eqref{eq:ExistenceGCMS-Thm-US-steon:ovovlanandCbdot}. 

We infer from the above 
 \beaa
&&\|(\fnn,\fbnn, \widecheck{\ovlann}^{\S(n)})\|_{\hk_{s_{max}+1}(\S(n))} +\sum_p\Big(r^2|\Cbpnndot|+r^3|\Mpnndot|\Big)\\
 &\les&  \left(\frac{1}{\rg}+\epg\right)\left(\dg+\sup_{\S(n)}|r-r^{\S(n)}|\right)+\sup_{\S(n)}|m-\ov{m}^{\S(n)}|\\
 &&+ r\|\kadot, \kabdot\|_{\hk_{s_{max}}(\S(n))}+r^2\|\mudot\|_{\hk_{s_{max}}(\S(n))}+|\La|+|\Lab|,
\eeaa
and
\beaa
&& r^2|\Cbnndot_0|+r^3|\Mnndot_0|+r\Big|\ov{\ovlann}^{\S(n)}\Big| \\
 &\les& \left(\frac{1}{\rg}+\epg\right)\dg +\sup_{\S(n)}|r-r^{\S(n)}|+\sup_{\S(n)}|m-m^{\S(n)}|+ r\|\kadot, \kabdot\|_{\hk_{s_{max}}(\S(n))}+r^2\|\mudot\|_{\hk_{s_{max}}(\S(n))} \\
 &&+|\La|+|\Lab|.
\eeaa 
Now, in view of \eqref{eq:ExistenceGCMS-Thm-US-steon} and Remark \ref{rem:SofnisanOofepgsphere}, we may apply \eqref{cor:comparison-gaS-ga-Oepgsphere} and \eqref{cor:comparison-gaS-ga-Oepgsphere:mSminusm} which yield 
\beaa
\sup_{\S(n)}|r-r^{\S(n)}| \les \dg,\qquad \sup_{\S(n)}|m-\ov{m}^{\S(n)}| \les \epg\dg, \qquad \sup_{\S(n)}|m-m^{\S(n)}| \les \dg.
 \eeaa
We deduce
 \beaa
&&\|(\fnn,\fbnn, \widecheck{\ovlann}^{\S(n)})\|_{\hk_{s_{max}+1}(\S(n))} +\sum_p\Big(r^2|\Cbpnndot|+r^3|\Mpnndot|\Big)\\
 &\les&  \left(\frac{1}{\rg}+\epg\right)\dg+ r\|\kadot, \kabdot\|_{\hk_{s_{max}}(\S(n))}+r^2\|\mudot\|_{\hk_{s_{max}}(\S(n))}+|\La|+|\Lab|,
\eeaa
and
\beaa
 r^2|\Cbnndot_0|+r^3|\Mnndot_0|+r\Big|\ov{\ovlann}^{\S(n)}\Big|  &\les& \dg + r\|\kadot, \kabdot\|_{\hk_{s_{max}}(\S(n))}+r^2\|\mudot\|_{\hk_{s_{max}}(\S(n))} \\
 &&+|\La|+|\Lab|
\eeaa 
as desired.
\end{proof}

{\bf Step 3.}   We use  the new pair     $( f^{(n+1)}, \fb^{(n+1)})$   to solve the equations on $\ovS$,
\bea
\lab{systemUU-SS-derivedn+1}
 \bsplit
\lapzero \Unn &=\divzero\left(\big(\UU(\fnn, \fbnn , \Ga)\big)^{\#_n}\right),\\
\lapzero \Snn &=\divzero\left(\big(\SS(\fnn, \fbnn , \Ga)\big)^{\#_n}\right),\\
\Unn(South)&=\Snn(South)=0,
 \end{split}
 \eea
where the pull back  $\#_{n}$ is defined with respect to the  map $\Psi^{(n)} :\ovS\longrightarrow \S(n)$.  The new  pair  $( U^{(n+1)} ,  S^{(n+1)}  )$ defines  the  new sphere $\S(n+1)$ and we can proceed with the next step  of the iteration. 
The  boundedness of  $(\Unn, \Snn)$ is  assured by the following proposition.

 \begin{proposition}
 \lab{Prop:estimatesUn+1Sn+1}
  The equation \eqref{systemUU-SS-derivedn+1} admits a unique solution $(U^{(1+n)}, S^{(1+n)})$, verifying the estimates
    \bea
 \label{eq:ExistenceGCMS-Thm-US-induction-n+1}
\nn  r^{-1}\Big \|\big(U^{(n+1)}, S^{(n+1)}\big)\Big\|_{\hk_{s_{max}+1}(\ovS) } &\les&  \left(\frac{1}{\rg}+\epg\right)\dg+ r\|\kadot, \kabdot\|_{\hk_{s_{max}}(\S(n))}\\
  &&+r^2\|\mudot\|_{\hk_{s_{max}}(\S(n))}+|\La|+|\Lab|
  \eea
uniformly for all  $n\in\NNN$. 
   \end{proposition}
   
   \begin{proof}
   The proof,   based on the    previously established  bounds          for  $f^{n+1}, \fb^{n+1} $  in Proposition \ref{prop:Estimates:Fn+1boundednessiterativescheme}, standard elliptic estimates for $\lapzero$ and the  comparison of norms estimates of 
   Proposition \ref{Prop:comparison-gaS-ga:highersobolevregularity}, is straightforward and thus  left to the reader. 
   \end{proof}

{\bf Step 4.} In view of Proposition \ref{prop:Estimates:Fn+1boundednessiterativescheme} and Proposition \ref{Prop:estimatesUn+1Sn+1}, and in view of the assumptions \eqref{eq:GCM-improved estimate2-again} on $\kadot$, $\kabdot$, $\mudot$, and \eqref{eq:assumptionsonLambdaabdLambdabforGCMexistence} on $\La$, $\Lab$, we obtain  the boundedness of  all quintets  $\QQ^{(n)}$. More precisely  we have, uniformly for all $n\in \NNN$,
 \bea
 \| \QQ^{(n)}\|_{s_{max}+1} &\les& \dg,
 \eea
 where
 \bea
\label{quintet-norm}
\bsplit
\| Q^{(n)}\|_k: &= r^{-1}   \Big\| \big(U^{(n)}, S^{(n)}\big)\Big \|_{\hk_{k}(\ovS)}  +\Big \|\big( f^{(n)},  \fb^{(n)}, \ovla^{(n)}\big) \Big\|_{\hk_{k}(\S)}.
\end{split}
\eea
To account for the constants $\Cbn_0, \Mn_0, \Cbpn, \Mpn$ we introduce the ninetets
\beaa
\NN^{(n)}:=\Big(U^{(n)}, S^{(n)}, \ovla^{(n)},  f^{(n)},  \fb^{(n)}; \, \Cbn_0, \Mn_0, \Cbpn, \Mpn\Big)
\eeaa
with norms,
\beaa
\big\| \NN^{(n)} \big\|_k= \| \QQ^{(n)}\|_k+ r^2 \left(\big| \Cbndot_0\big|+\sum_p\big| \Cbpndot\big| +r \big| \Mndot_0\big|+
r \sum_p\big| \Mpndot\big|\right)
\eeaa
where, recall,
\beaa
\Cbndot_0= \Cbn_0-\Cb_0, \,\,\,\, \Cbpndot= \Cbpn-\Cbp, \,\,\,\, \Mndot_0= \Mn_0-M_0, \,\,\,\, \Mpndot=\Mpn-\Mp.
\eeaa
According to  Proposition \ref{prop:Estimates:Fn+1boundednessiterativescheme},    we also have, uniformly in $n$,
\bea
\lab{eq:nintetbound}
\big\| \NN^{(n)} \big\|_{s_{max}+1} &\les&\dg.
\eea

{\bf Step 5.}  To insure convergence we also need  to establish  a contraction estimate.  We cannot  compare  directly
 the ninetets  $\NN^{(n)}$ so we compare instead the modified ninetets, well defined on $\ovS$,
 \bea
 \NN^{n,\#}:=\Big(U^{(n)}, S^{(n)}, \ovla^{n,\#},  f^{n,\#},  \fb^{n,\#}; \, \Cbn_0, \Mn_0, \Cbpn, \Mpn\Big)
 \eea
 where $\ovla^{n,\#},  f^{n,\#},  \fb^{n,\#}$ are the pull-backs by $\#_{n-1}$   of  the triplet $\ovla^{(n)}, f^{(n)}, \fb^{(n)} $
   defined on the sphere  $\S(n-1)$.
We also  introduce  the modified norms
\bea
\bsplit
\big\| \NN^{n,\#} \big\|_{k, \ovS} :&= r^{-1}   \Big\| \big(U^{(n)}, S^{(n)}\big)\Big \|_{\hk_{k}(\ovS)} +
\Big\|\big( f^{n,\#},  \fb^{n,\#}, \ovla^{n,\#}\big)\Big\| _{\hk_{k}(\ovS)}\\
&+ r^{2} \left(\big| \Cbndot_0\big|+\sum_p\big| \Cbpndot\big| +r \big| \Mndot_0\big|+
r \sum_p\big| \Mpndot\big|\right).
\end{split}
\eea
In view of the  Sobolev norm  comparison  of Proposition \ref{Prop:comparison-gaS-ga:highersobolevregularity},  we  deduce from
\eqref{eq:nintetbound}
 \bea
 \lab{eq:nintetboundmodified}
\big\| \NN^{n,\#} \big\|_{s_{max}+1, \ovS} &\les&\dg.
\eea
Contraction in this modified norms is established in the following.
\begin{proposition}
\lab{Prop:contractionforNN}
The following estimate holds true.
\bea
\nn\|\NN^{n+1,\#}-\NN^{n,\#}\|_{3,\ovS} &\les& (r^{-1}+\epg)\Big[\|\NN^{n,\#}-\NN^{n-1,\#}\|_{3,\ovS}+\|\NN^{n-1,\#}-\NN^{n-2,\#}\|_{3,\ovS}\\
&&+\|\NN^{n-2,\#}-\NN^{n-3,\#}\|_{3,\ovS}\Big].
\eea
\end{proposition}

\begin{proof}
See Appendix \ref{sec:proofofProp:contractionforNN}.
\end{proof}

%%%%%%%%%%%%%%%%%%%%%%%%%%%%%%%%%%

\subsection{Limit of the iterative scheme}
\lab{sec:limitoftheiterativescheme} 

%%%%%%%%%%%%%%%%%%%%%%%%%%%%%%%%%%

%%%%%%%%%%%%%%%%%%%%%%%%%%%%%%%%%%

\subsubsection{Limiting ninetet} 

%%%%%%%%%%%%%%%%%%%%%%%%%%%%%%%%%%

We infer the existence of a  ninetet $\NN ^{\infty, \#}$ on $\ovS$ such that 
\bea
\lab{eq:boundednessPinfty}
 \| \NN^{\infty, \#}\|_{s_{max}+1}&\les& \dg
\eea
and, using interpolation between $3$ and $s_{max}+1$, 
\bea\lab{eq:PnconvergestoPinftyinhknormasngoestoinfty}
\lim_{n\to +\infty}\|\NN^{n, \#}-\NN^{\infty,\#}\|_{s_{max}}=0,
\eea
where
\beaa
\NN^{(\infty, \#)}=\Big(U^{(\infty)}, S^{(\infty)}, \ovla^{\infty, \#}, f^{\infty, \#}, \fb^{\infty, \#},   \Cb^{(\infty)}_0, M^{(\infty)}_0, \Cb^{(\infty),\,p}, M^{(\infty),\, p} \Big).
\eeaa
 The functions  $(U^{(\infty)}, S^{(\infty)})$ defines a sphere $\S^{(\infty)}$  parametrized by  the map 
\beaa
\Psi^{(\infty)}(\ug, \sg, y^1, y^2) &=& \Big(\ug+U^{(\infty)}(y^1, y^2 ), \sg+S^{(\infty)}(y^1, y^2), y^1, y^2\Big).
\eeaa
We then define
\beaa
\ovla^{(\infty)}=\ovla^{\infty,\#}\circ(\Psi^{(\infty)})^{-1},\quad  f^{(\infty)}=f^{\infty,\#}\circ(\Psi^{(\infty)})^{-1}, \quad \fb^{(\infty)}=\fb^{\infty,\#}\circ(\Psi^{(\infty)})^{-1}
\eeaa 
so that $\ovla^{(\infty)}, f^{(\infty)}, \fb^{(\infty)}$ are defined on $\S^{(\infty)}$  and 
\beaa
\ovla^{\infty,\#}= (\ovla^{(\infty)})^{\#_{\infty}}, \quad f^{\infty,\#}= (f^{(\infty)})^{\#_{\infty}},\quad \fb^{{\infty,\#}}= (\fb^{(\infty)})^{\#_{\infty}}.
\eeaa
We also define
\beaa
\NN^{(\infty)}=\Big(U^{(\infty)}, S^{(\infty)}, \ovla^{(\infty)}, f^{(\infty)}, \fb^{(\infty)},   \Cb^{(\infty)}_0, M^{(\infty)}_0, \Cb^{(\infty),\,p}, M^{(\infty),\, p} \Big).
\eeaa
From these definitions, in view of  \eqref{eq:boundednessPinfty}     and the norm comparison estimates of Proposition   \ref{Prop:comparison-gaS-ga:highersobolevregularity},  we  deduce
\begin{enumerate}
\item  Uniform bounds
\beaa
\|\NN^{(\infty)}\|_{s_{max}+1} &\les &\dg,
\eeaa
i.e.
\bea\lab{eq:controlofUSaffbforthelimitoftheiterationscheme}
 r^{-1}\|(U^{(\infty)}, S^{(\infty)})\|_{\hk_{s_{max}+1}(\ovS)}+\|(f^{(\infty)}, \fb^{(\infty)}, \ovla^{(\infty)})\|_{\hk_{s_{max}+1}(\S^{(\infty)})}\\
\nn + r^2\big| \Cbdot^{(\infty)}_0\big|+r^2\sum_p\big| \Cbdot^{\infty,p}\big| +r^3 \big| \Mdot^{(\infty)}_0\big|+
r^3 \sum_p\big| \Mdot^{\infty, p}\big| &\les& \dg.
\eea
\item    The following sequences converge
\begin{itemize}
\item  The sequences of pairs  $ (U^{(n)}, S^{(n)})$ converges to  $ (U^{(\infty)}, S^{(\infty)})$ in the norm ${\hk_{s_{max}}(\ovS)}$.

\item The sequence  $(\fn, \fbn, \ovlan)$ converges to  $(f^{(\infty)}, \fb^{(\infty)}, \ovla^{(\infty)})$ in  ${\hk_{s_{max}}(\S^{(\infty)})}$.

\item The sequence of GCM  constants  $(\Cbn_0, \Mn_0, \Cbpn, \Mpn)$ converges to the constants 
 $(\Cb^{(\infty)}_0, M^{(\infty)}_0$,  $\Cb^{(\infty), p}, M^{(\infty), p})$.
 
\end{itemize}
\end{enumerate}

%%%%%%%%%%%%%%%%%%%%%%%%%%%%%%%%%%

\subsubsection{Limiting equations} 

%%%%%%%%%%%%%%%%%%%%%%%%%%%%%%%%%%

Taking $n\to \infty$ in the equations \eqref{systemUU-SS-derivedn+1},  \eqref{GeneralizedGCMsystem-n+1-MainThm1}-\eqref{GeneralizedGCMsystem-n+1-MainThm3}, \eqref{eq:definitionofh1nhb1nh2nhb2nh3nh4n} and \eqref{GeneralizedGCMsystem-n+1-LaLab-MainThm:00}, 
 we derive
 \begin{proposition}
 The   triplet   $\NN^{(\infty)}$ verifies the following  equations
  \bea
\lab{systemUU-SS-derivedinfty}
\bsplit
\lapzero \Uinfty &=\divzero\left(\big(\UU(\finfty, \fbinfty , \Ga)\big)^{\#_{\infty}}\right),\\
 \lapzero \Sinfty &=\divzero\left(\big(\SS(\finfty, \fbinfty , \Ga)\big)^{\#_{\infty}}\right),\\
\Uinfty(South)&=\Sinfty(South)=0,
\end{split}
\eea
\bea
\lab{GeneralizedGCMsystem-infty}
\bsplit
\curlSinfty \finfty &= -\err_1[\curl \finfty] +\overline{\err_1[\curl \finfty]}^{\S(\infty)},\\
\curlSinfty \fbinfty&= -\err_1[\curl \fbinfty]  +\overline{\err_1[\curl \fbinfty]}^{\S(\infty)},
\end{split}
\eea
\bea
\lab{GeneralizedGCMsystem-infty:111}
\bsplit
\divSinfty \finfty + \ka \ovlainfty  -\frac{2}{(r^{\S^{(\infty)}})^2}\left(\ovb^{(\infty)} -\big(r-r^{\S^{(\infty)}}\big)\right) &=\ka^{(\infty)} -\ka -\err_1[\divSinfty \finfty],
\\
\divSinfty\fbinfty - \kab  \ovlainfty   + \frac{2}{(r^{\S^{(\infty)}})^2}\left(\ovb^{(\infty)} -\big(r-r^{\S^{(\infty)}}\big)\right)
&= \kab^{(\infty)} -\kab -\err_1[\divSinfty \fbinfty ],
\end{split}
\eea
\bea
\nn\lapSinfty\ovlainfty + V \ovlainfty &=& \muinfty-\mu -\left(\omb +\frac 1 4 \kab \right) \big(\kainfty-\ka \big)\\
&&+\left(\om +\frac 1 4 \ka \right) \big(\kabinfty-\kab \big)+\err_2[ \lapSinfty\ovlainfty],
\eea
with
\bea
\lab{eq:binftyequationafterconvergenceiterativescheme}
\bsplit
\Delta^{\S(\infty)}\ovb^{(\infty)} &= \frac{1}{2}\div^{\S(\infty)}\left(\fb^{(\infty)} - \Up^{(\infty)} f^{(\infty)} +\err_1[\Delta^{\S(\infty)} \ovb^{(\infty)} ] \right), \\
 \ov{\ovb^{(\infty)}}^{\S(\infty)} &=\ov{r}^{\S(\infty)} -r^{\S(\infty)},
\end{split}
\eea
\bea
\lab{GeneralizedGCMsystem-infty-GCMSinfty}
\bsplit
\kainfty&=\frac{2}{r^{\S(\infty)}},\\
\kabinfty &=-\frac{2}{r^{\S(\infty)} }\Up^{(\infty)}+  \Cbinfty_0+\sum_p \Cbpinfty \Jpinfty,\\
\muinfty&= \frac{2m^{\S(\infty)} }{(r^{\S(\infty)})^3} +   \Minfty_0+\sum _p\Mpinfty \Jpinfty,
\end{split}
\eea
and $\ell=1$  conditions,
\bea
\lab{GeneralizedGCMsystem-infty-LaLab}
( \divSinfty \finfty)_{\ell=1}=\La, \qquad   (\divSinfty \fbinfty)_{\ell=1}=\Lab,
\eea
with respect to the  $\ell=1$  modes\footnote{According to Definition \ref{def:ell=1sphharmonicsonS}.} $J^{\S^{(\infty)}, p} $ on $\S^{(\infty)}$.
\end{proposition}

%%%%%%%%%%%%%%%%%%%%%%%%%%%%%%%%%%

\subsubsection{The limiting frame}

%%%%%%%%%%%%%%%%%%%%%%%%%%%%%%%%%%

Using  $(\ovlainfty, \finfty, \fbinfty )$ and transformation formula   \eqref{eq:Generalframetransf}  we define the corresponding  null frame $\einfty_1, \einfty_2, \einfty_3, \einfty_4$ and the associated  Ricci coefficients $\Ga^{(\infty)}, R^{(\infty)}$. Note that the frame   is a priori not  adapted to $\S^{(\infty)}$.  In fact $\aka^\infty, \akab^\infty$  do not necessarily vanish and thus     the distribution  generated by  $\einfty_1, \einfty_2$  may not even  be  integrable.

%%%%%%%%%%%%%%%%%%%%%%%%%%%%%%%%%%

\subsubsection{The adapted  frame on $\S^{(\infty)}$}

%%%%%%%%%%%%%%%%%%%%%%%%%%%%%%%%%%

We associate to the sphere $\S=\S^{(\infty)}$  a  second null  frame, which is adapted to $\S$, as follows. We  use the limiting  functions 
$U=U^{(\infty)}, S=S^{(\infty)}$  of the   deformation map $\Psi=\Psi^{(\infty)}:\ovS\longrightarrow \S$   to  define (see  Lemma \ref{Lemma:deformation1})  the tangent vectorfields  $\YY_{(a)} =\Psi_\# (\pr_{y^a} )$. Then, let $(f, \fb)$ denote the 1-forms such that, for $a=1,2$, we have
\beaa
\g\left(\YY_{(a)}, e_4 + f^b  e_b +\frac 1 4 |f|^2  e_3\right) &=& 0,\\
\g\left(\YY_{(a)},\left(1+\frac{1}{2}f\c\fb  +\frac{1}{16} |f|^2  |\fb|^2\right) e_3 + \left(\fb^b+\frac 1 4 |\fb|^2f^b\right) e_b  + \frac 1 4 |\fb|^2 e_4 \right) &=& 0.
\eeaa
With this choice of $(f, \fb)$, we then define the null frame $(e^\S_1, e^\S_2, e^\S_3, e^\S_4)$ as the one obtained from the background frame  $(e_1, e_2, e_3, e_4)$ using the frame transformation coefficients $(f, \fb, \la)$  with  $\la=1+\ovla$ chosen such that
 \bea
 \ovla=\ovla^{(\infty)}.
 \eea
In view of the choice of $(f, \fb)$, $e_4^\S$ and $e_3^\S$ are orthogonal to $\S$, and hence $(e^\S_1, e^\S_2, e^\S_3, e^\S_4)$  is adapted to $\S$ as desired. Furthermore, using \eqref{Compatibility-Deformation2}, \eqref{eq:assumptionY_a^b},  \eqref{eq:formofmathcalUandmathcalSatthemainorder} and the control of $U$ and $S$ to control  $(f, \fb)$,  it is straightforward to check that
 \beaa
\big \|(f, \fb)\big\|_{\hk_1(\S)}+\big \|(f, \fb)\big\|_{L^\infty(\S)} \les \dg.
 \eeaa

%%%%%%%%%%%%%%%%%%%%%%%%%%%%%%%%%%%%%%

\subsection{End of the proof of Theorem \ref{Theorem:ExistenceGCMS1}}
\lab{sec:wheretheproofofthemaintheoremisfinallyconcluded} 

%%%%%%%%%%%%%%%%%%%%%%%%%%%%%%%%%%%%%%

So far we have produced a sphere 
$\S=\S^{(\infty)}$, defined by  the functions $U=U^{(\infty)},  S=S^{(\infty)}$ and two frames 
\begin{itemize}
\item The frame  $\einfty_1, \einfty_2, \einfty_3, \einfty_4$ induced by
 the transition functions $(\ovlainfty, \finfty, \fbinfty )$.     The  functions $U, S$ and   transition functions   $(\ovlainfty, \finfty, \fbinfty )$
    verify  the  coupled system \eqref{systemUU-SS-derivedinfty}-\eqref{GeneralizedGCMsystem-infty-LaLab}.
    
\item  The geometric frame $e_1^\S, e_2^\S, e_3^\S, e_4^\S$,  induced by the deformation  map  defined by $U=U^{(\infty)},  S=S^{(\infty)}$, with corresponding transition functions $(\ovla=\ovla^{(\infty)},  f, \fb)$.
\end{itemize}
The  main remaining  hurdle in the proof of Theorem   \ref{Theorem:ExistenceGCMS1} is to show that the two frames coincide.

 {\bf Step 1.}  Since the frame $(e^\S_1, e^\S_2, e^\S_3, e^\S_4)$ is adapted to the sphere $\S$, we have on $\ovS$
\bea
\bsplit
\pr_{y^a} U&=\Big(\UU(f, \fb , \Ga)_bY^b_{(a)}\Big)^{\#},\\
\pr_{y^a} S&=\Big(\SS(f, \fb , \Ga)_bY^b_{(a)}\Big) ^{\#},\\
U(South)&=S(South)=0,
\end{split}
\eea
where $\#$ denotes the pull-back with respect to the  deformation map $\Psi$. 
We deduce
 \beaa
\lapzero U&=&\divzero\left(\big(\UU(f, \fb, \Ga)\big)^\#\right),\\
\lapzero S&=&\divzero\left(\big(\SS(f, \fb, \Ga)\big)^\#\right).
 \eeaa
On the other hand we have, see \eqref{systemUU-SS-derivedinfty},
 \beaa
\lapzero U&=& \divzero\left(\big(\UU(\finfty, \fbinfty, \Ga)\big)^\#\right),\\
\lapzero S&=& \divzero\left(\big(\SS(\finfty, \fbinfty, \Ga)\big)^\#\right).
 \eeaa
Subtracting the two equations we deduce
\beaa
\divzero\left(\big(\UU(f, \fb, \Ga)-\UU(\finfty, \fbinfty , \Ga)\big)^\#\right) &=& 0,\\
\divzero\left(\big(\SS(f, \fb, \Ga)-\SS(\finfty, \fbinfty , \Ga)\big)^\#\right) &=& 0,
\eeaa
or, introducing,
\beaa
\de\UU=\UU(f, \fb, \Ga)-\UU(\finfty, \fbinfty , \Ga), \qquad \de\SS=\SS(f, \fb, \Ga)-\SS(\finfty, \fbinfty, \Ga),
\eeaa
\bea
\lab{eqdivzerodeU-deS}
\divzero\left(\big(\de\UU\big)^\#\right) = 0,\qquad \qquad \divzero\left(\big(\de\SS\big)^\#\right) = 0.
\eea
Let  $g^{\S, \#}$ be the pull back of the  metric  of $\S$ to  $\ovS$  by  the map $\Psi:\ovS\longrightarrow  \S$. Denoting by
$\div^\#$  the divergence with respect  $g^{\S, \#}$ along $\ovS$  we  have, according to  the  connection   comparison  estimates
of Lemma \ref{lemma:comparison-gaS-ga},
\beaa
\left\|\divzero\left(\big(\de\UU\big)^\#\right) - \div^{\S, \#}\left(\big(\de\UU\big)^\#\right)\right\|_{L^2(\ovS)} &\les&\dg r^{-1} \big\| \de \UU \big\|_{\hk_{1}(\S)}, \\
\left\|\divzero\left(\big(\de\SS\big)^\#\right) - \div^{\S, \#}\left(\big(\de\SS\big)^\#\right)\right\|_{L^2(\ovS)} &\les&\dg r^{-1} \big\| \de \SS \big\|_{\hk_{1}(\S)}.
\eeaa
Thus, in view of \eqref{eqdivzerodeU-deS},
\beaa
\left\|\div^{\S, \#}\left(\big(\de\UU\big)^\#\right) \right\|_{L^2(\ovS)} &\les&\dg r^{-1}\big\| \de \UU \big\|_{\hk_{1}(\S)},\\
\left\|\div^{\S, \#}\left(\big(\de\SS\big)^\#\right) \right\|_{L^2(\ovS)} &\les&\dg  r^{-1}\big\| \de \SS \big\|_{\hk_{1}(\S)}.
\eeaa
In view of the norm comparison Proposition  \ref{Prop:comparison-gaS-ga:highersobolevregularity}, we deduce
\bea
\bsplit
\left\|\div^{\S}\big(\de\UU\big) \right\|_{L^2(\S)} &\les\dg r^{-1} \big\| \de \UU \big\|_{\hk_{1}(\S)},\\
\left\|\div^{\S}\big(\de\SS\big) \right\|_{L^2(\S)} &\les\dg r^{-1} \big\| \de \SS \big\|_{\hk_{1}(\S)}.
\end{split}
\eea

{\bf Step 2.} Recall from \eqref{eq:formofmathcalUandmathcalSatthemainorder} that we have
  \beaa
 \bsplit
 \SS(f, \fb, \Ga) &= f +O\Big(\epg |f|+|f|^2+|\fb|^2\Big), \\ 
 \UU(f, \fb, \Ga) &=\frac{1}{2}\Big(-\Up f +\fb\Big)+O\Big(\epg |f|+|f|^2+|\fb|^2\Big).
 \end{split}
 \eeaa
Hence
 \beaa
 \de\SS(f, \fb, \Ga) &=& f-\finfty +O\Big((r^{-1}+\epg)(|f-\finfty|+|\fb-\fbinfty|)\Big), \\ 
 \de\UU(f, \fb, \Ga) &=&\frac{1}{2}\Big(- (f-\finfty) +\fb-\fbinfty\Big)+O\Big((r^{-1}+\epg)(|f-\finfty|+|\fb-\fbinfty|)\Big).
 \eeaa
This yields
\beaa
\nn&&\left\|\de\SS(f, \fb, \Ga) - (f-\finfty)\right\|_{\hk^1(\S)}\\
 &\les& (r^{-1}+\epg)(\|f-\finfty\|_{\hk_{1}(\S)}+\|\fb-\fbinfty\|_{\hk_1(\S)}),\\
\nn&&\left\|\de\UU(f, \fb, \Ga) +\frac{1}{2}(f-\finfty) -\frac{1}{2}(\fb-\fbinfty)\right\|_{\hk_1(\S)} \\
&\les& (r^{-1}+\epg)(\|f-\finfty\|_{\hk_{1}(\S)}+\|\fb-\fbinfty\|_{\hk_1(\S)}).
\eeaa
Therefore, in view of Step 1, we infer
\bea
\lab{eq:div(f-finfty)}
\bsplit
\big\|\div^\S( f-\finfty)\big\|_{L^2(\S)}&\les (r^{-1}+\epg)r^{-1}\Big(\big\| f-\finfty \big\|_{\hk_{1}(\S)}+ \big\| \fb-\fbinfty \big\|_{\hk_{1}(\S)}\Big),\\
\big\|\div^\S( \fb-\fbinfty)\big\|_{L^2(\S)}&\les (r^{-1}+\epg)r^{-1}\Big(\big\| f-\finfty \big\|_{\hk_{1}(\S)}+ \big\| \fb-\fbinfty \big\|_{\hk_{1}(\S)}\Big).
\end{split}
\eea

{\bf Step 3.}  According to \eqref{GeneralizedGCMsystem-infty}  we  have,
\beaa
\curl^\S  \finfty   &=& \err_1[\curl^\S  \finfty ]-\ov{\err_1[\curl^\S \finfty]}^{\S},  \qquad \\
\curl^\S \fbinfty  &=& \err_1[\curl^\S  \fbinfty] -\ov{\err_1[\curl^\S \fbinfty]}^{\S}.\\
\eeaa
On the other hand, since the frame  $(e^\S_1, e^\S_2, e^\S_3, e^\S_4)$ is adapted to the sphere $\S$, we have $\atrchS=0$ and $\atrchbS=0$, i.e.  the transition functions $f, \fb$ must verify
\beaa
 \curl^\S(f)&=&- \err_1[\curl^\S f], \\ 
\curl^S(\fb)&=& - \err_1[\curl^\S\fb],
\eeaa
 with the same algebraic expressions for the errors  $\err_1[\curl^\S f], \err_1[\curl^\S \fb]$ as those  for $\err_1[\curl^\S  \finfty ], \err_1[\curl^\S  \fbinfty ]$. Moreover  
 \beaa
 \ov{\err_1[\curl^\S f]}^\S=  \ov{\err_1[\curl^\S \fb]}^\S=0.
 \eeaa
 Subtracting the two equations we derive
 \bea
 \lab{eq:curl(f-finfty)}
 \bsplit
\big\| \curl ^\S(f-\finfty)\big\|_{L^2(\S)}& \les  \dg  r^{-1}\Big(\big\| f-\finfty \big\|_{\hk_{1}(\S)}+ \big\| \fb-\fbinfty \big\|_{\hk_{1}(\S)}\Big),\\
\big\| \curl ^\S(\fb-\fbinfty)\big\|_{L^2(\S)}& \les  \dg r^{-1} \Big(\big\| f-\finfty \big\|_{\hk_{1}(\S)}+ \big\| \fb-\fbinfty \big\|_{\hk_{1}(\S)}\Big).
\end{split}
 \eea
 Combining  \eqref{eq:div(f-finfty)} with  \eqref{eq:curl(f-finfty)} we deduce,
 \bea
 \lab{eq:dddS(f-finfty)}
 \bsplit
\big\| \dddS_1 (f-\finfty)\big\|_{L^2(\S)}& \les  (r^{-1}+\epg) r^{-1}\Big(\big\| f-\finfty \big\|_{\hk_{1}(\S)}+ \big\| \fb-\fbinfty \big\|_{\hk_{1}(\S)}\Big),\\
\big\| \dddS_1 (\fb-\fbinfty)\big\|_{L^2(\S)}& \les  (r^{-1}+\epg)  r^{-1}\Big(\big\| f-\finfty \big\|_{\hk_{1}(\S)}+ \big\| \fb-\fbinfty \big\|_{\hk_{1}(\S)}\Big).
\end{split}
 \eea
Therefore, by elliptic estimates,
\beaa
\big\| f-\finfty \big\|_{\hk_{1}(\S)}+ \big\| \fb-\fbinfty \big\|_{\hk_{1}(\S)}&\les& (r^{-1}+\epg) \Big(\big\| f-\finfty \big\|_{\hk_{1}(\S)}+ \big\| \fb-\fbinfty \big\|_{\hk_{1}(\S)}\Big)
\eeaa
and thus, for $\epg$ small enough and $r$ large enough,
\bea
f=f^\infty, \qquad \fb=\fb^\infty.
\eea

 {\bf Step 4. } We have thus established that the limiting frame  $\einfty_1, \einfty_2, \einfty_3, \einfty_4$  is in fact adapted to $\S=\S^{(\infty)}$.  We now show that $\S$ endowed with this frame, and the  induced $\ell=1$  modes $\JpS$,
  is actually a GCM sphere. From now on, we denote $e_1^\S=\einfty_1$, $e_2^\S=\einfty_2$, $e_3^\S=\einfty_3$, $e_4^\S=\einfty_4$, 
 $\ovla=\ovlainfty, f= \finfty, \fb=\fbinfty$, $\ovb=\ovb^{(\infty)}$.    First, we prove that $\ovb=r-r^\S$. Indeed, we have, using the equality of the two frames, 
\beaa
&& \frac{e_4(r)}{2}\fb_a   +\frac{e_3(r)}{2}\left(f_a +\frac{1}{4}|f|^2\fb_a\right)\\
&=& \left(\left(\de_{ab} +\frac{1}{2}\fb_af_b\right) e_b +\frac 1 2  \fb_a  e_4 +\left(\frac 1 2 f_a +\frac{1}{8}|f|^2\fb_a\right)   e_3\right)r\\
&=& e^{\S}_a(r)\\
&=& e^{\S}_a\left(r-r^{\S}\right).
\eeaa
Now, recall that $\ovb$, taking into account the definition of $\err_1[ \Delta^\S\ovb]$,   is  uniquely  defined by
\beaa
\Delta^{\S}\ovb = \div^{\S}\left(\frac{e_4(r)}{2}\fb   +\frac{e_3(r)}{2}\left(f +\frac{1}{4}|f|^2\fb\right)\right),\qquad \ov{\ovb}^{\S}=\ov{r}^{\S} -r^{\S}.
\eeaa
We infer
\beaa
\Delta^{\S}\ovb = \div^{\S}\nab^{\S}\left(r-r^{\S}\right)=\Delta^{\S}\left(r-r^{\S}\right),\qquad  \ov{\ovb}^{\S}=\ov{r-r^\S}^{\S}.
\eeaa
The unique solution of the  above system of equation is provided by
\bea
\ovb=r-r^\S
\eea
as claimed.
  
Since $(f, \fb, \ovla)$       verifies    equations \eqref{GeneralizedGCMsystem-infty}-\eqref{GeneralizedGCMsystem-infty-LaLab}, and since $\ovb=r-r^\S$, we infer 
  \bea
\lab{GeneralizedGCMsystem-infty-again}
\bsplit
\div^\S f + \ka \ovla &=\ka^{(\infty)} -\ka -\err_1[\div^\S f ],
\\
\div^\S\fb - \kab  \ovla   
&= \kab^{(\infty)} -\kab -\err_1[\div^\S \fb ],\\
\lap^\S\ovla+ V \ovla &=\muinfty-\mu -\left(\omb +\frac 1 4 \kab \right) \big(\kainfty-\ka \big)+\left(\om +\frac 1 4 \ka \right) \big(\kabinfty-\kab \big)\\
&+\err_2[ \lap^\S\ovla],
\end{split}
\eea
with
\bea
\lab{GeneralizedGCMsystem-infty-GCMSinfty-again}
\bsplit
\kainfty&=\frac{2}{r^{\S}},\\
\kabinfty &=-\frac{2\Up^{\S}}{r^{\S}}+  \Cbinfty_0+\sum_p \Cbpinfty \JpS,\\
\muinfty&= \frac{2m^{\S} }{(r^{\S})^3} +   \Minfty_0+\sum _p\Mpinfty \JpS,
\end{split}
\eea
  and
  \bea
\lab{GeneralizedGCMsystem-infty-LaLab-again}
( \div^\S f )_{\ell=1}=\La, \qquad   (\div^\S\fb)_{\ell=1}=\Lab.
\eea

  On the other hand, according to 
Lemma \ref{Lemma-adaptedGCM-equations},     $\ovla, f, \fb$  verify the equations
     \bea
     \lab{GCMS-4Sagain}
     \bsplit
\div^\S f + \ka \ovla &= \ka^\S -\ka -\err_1[\div^\S f ],\\
\div^\S\fb - \kab \ovla &= \kab^\S -\kab -\err_1[\div^\S \fb ],\\
\Delta^\S\ovla + V\ovla &=\mu^\S-\mu -\left(\omb +\frac 1 4 \kab \right) \big(\ka^\S-\ka \big)+\left(\om +\frac 1 4 \ka \right) \big(\kab^\S-\kab \big)+\err_2[ \lap^\S\ovla].
\end{split}
\eea  
Subtracting we deduce
\beaa
\ka^\S &=& \kainfty=\frac{2}{r^{\S}},\\
 \kab^\S&=&\kabinfty =-\frac{2\Up^{\S}}{r^{\S} }+  \Cbinfty_0+\sum_p \Cbpinfty \Jpinfty,\\
 \mu^\S&=&\muinfty= \frac{2m^{\S} }{(r^{\S})^3} +   \Minfty_0+\sum _p\Mpinfty \Jpinfty.
\eeaa
Thus the GCM conditions \eqref{def:GCMC} are verified with  the constants 
\beaa
\Cb^\S_0:=\Cbinfty_0, \quad  \CbpS:=\Cbpinfty, \quad  M^\S_0:=  \Minfty_0, \quad \MpS=\Mpinfty.
\eeaa
Finally equation  \eqref{GCMS:l=1modesforffb} is   verified in view of 
 \eqref{GeneralizedGCMsystem-infty-LaLab-again}.
 
 {\bf Step 5.} The remaining  results of Theorem \ref{Theorem:ExistenceGCMS1} are now easy to  derive. \eqref{eq:ThmGCMS1}, \eqref{eq:ThmGCMS2} and \eqref{eq:ThmGCMS4} follow from \eqref{eq:controlofUSaffbforthelimitoftheiterationscheme}.  \eqref{eq:ThmGCMS3} follows  from Lemma  \ref{lemma:comparison-gaS-ga}, and \eqref{eq:ThmGCMS5} follows from Corollary \ref{cor:comparison-gaS-ga-Oepgsphere:mSminusm}. Finally,   the estimates \eqref{eq:ThmGCMS6}   follow from the transformation formulas of Proposition \ref{Prop:transformation formulas-integrtogeneral}, \eqref{eq:ThmGCMS1} and {\bf A1}. We note that the  transformation formulas for the well defined    
 quantities  $\ka^\S, \kab^\S,  \chih^\S, \chibh^\S, \ze^\S$, $ \a^\S, \b^\S, \rho^\S, \rhod^\S, \bb^\S, \aa^\S, \mu^\S$   (see Remark  \ref{remark:welldefinedGa}) involve only
 $\S$-tangential derivatives of $f, \fb, \ovla$ and can thus  indeed be estimated  using \eqref{eq:ThmGCMS1}.
   This concludes  the proof of Theorem  \ref{Theorem:ExistenceGCMS1}.

%%%%%%%%%%%%%%%%%%%%%%%%%%%%%%%%%%%%%%%
 
  \subsection{Differentiability with respect to the parameters $(\La, \Lab)$}
  
%%%%%%%%%%%%%%%%%%%%%%%%%%%%%%%%%%%%%%%  

The following proposition investigates the differentiability with respect to $(\La, \Lab)$ of the various 
quantities appearing in Theorem \ref{Theorem:ExistenceGCMS1}.
\begin{proposition}
Under the assumptions of Theorem \ref{Theorem:ExistenceGCMS1}, let $\S^{(\La, \Lab)}$ the deformed spheres constructed in Theorem \ref{Theorem:ExistenceGCMS1} for parameter    $\La, \Lab \in \RRR^3$  verifying
\bea\lab{eq:assumptionsonLambdaabdLambdabforGCMexistence:bis}
|\La|,\,  |\Lab|  &\les & \dg.
\eea
Then
\begin{enumerate}
\item  The   transition parameters $(f, \fb, \ovla)$ are continuous and differentiable with respect to $\La, \Lab $ and verify
\bea
\bsplit
\frac{\pr f }{\pr \La}&=O\big( r^{-1}\big), \quad  \frac{ \pr f }{\pr \Lab}=O\big(\dg r^{-1} \big), \quad 
\frac{\pr\fb }{\pr \La}&=O\big(\dg  r^{-1}\big), \quad  \frac{\pr\fb}{\pr \Lab}=O\big( r^{-1} \big),\\
\frac{\pr\ovla  }{\pr \La}&=O\big(\dg  r^{-1}\big), \quad \frac{\pr\ovla  }{\pr \Lab}=O\big(\dg  r^{-1}\big).
\end{split}
\eea
\item  The parameter  functions $U, S$  of the deformation are continuous and differentiable with respect to $\La, \Lab $ and verify
\bea
\frac{\pr U}{\pr \La}= O(1), \qquad \frac{\pr U}{\pr \Lab}= O(1), \qquad \frac{\pr S}{\pr \La}= O(1), \qquad \frac{\pr S}{\pr \Lab}= O(\dg).
\eea
\item Relative to the  coordinate system induced by $\Psi$, the metric $g^\S$  of $\S=\S^{\La, \Lab}$  is continuous with respect to the parameters $\La, \Lab$ and verifies 
\beaa
\big\| \pr_\La g^\S, \,  \pr_{\Lab} g^\S\|_{L^\infty(\S)} &\les O( r^2).
\eeaa
\end{enumerate}
\end{proposition}

\begin{proof}
The proof follows by differentiating the equations satisfied by $(f, \fb, \la)$ and $(U, S)$ with respect to $(\La, \Lab)$ and relying on the estimates derived for $(f, \fb, \la)$ and $(U, S)$ in Theorem \ref{Theorem:ExistenceGCMS1}. The details are cumbersome but straightforward, and left to the reader.
\end{proof}

%%%%%%%%%%%%%%%%%%%%%%%%%%%%%%%%%%%%%%%
 
  \subsection{Existence of GCM spheres in Kerr}
  
%%%%%%%%%%%%%%%%%%%%%%%%%%%%%%%%%%%%%%%  

The following corollary of Theorem \ref{Theorem:ExistenceGCMS1} shows the existence of GCM spheres in Kerr. 

\begin{corollary}[Existence of GCM spheres in Kerr]
\lab{cor:ExistenceGCMS1inKerr}
Let $\g_{a_0,m_0}$, with $|a_0|\leq m_0$, denote a member of the Kerr family of metrics. Let $0<\de_0=\ep_0$   two sufficiently   small   constants, and let  $(\ug, \sg, \rg)$ three real numbers with $\rg$ sufficiently large so that
\beaa
\epg=\frac{a_0m_0}{\rg}, \qquad \dg=\frac{a_0m_0}{\rg}, \qquad \rg\gg m_0.
\eeaa
Let a fixed  spacetime region $\RR$ of Kerr  together with a $(u, s)$ outgoing geodesic   foliation, as discussed in Lemma \ref{lemma:controlfarspacetimeregionKerrassumptionRR}.  Let  $\ovS=S(\ovu, \ovs)$   be  a fixed    sphere from this foliation, and let $\rg$ and $\mg$ denoting respectively its area radius and its Hawking mass.   Then
for any fixed pair of triplets   $\La, \Lab \in \RRR^3$  verifying
\beaa
|\La|,\,  |\Lab|  &\les & \dg,
\eeaa
 there exists a unique  GCM sphere $\S=\S_{Kerr}^{(\La, \Lab)}$, which is a deformation of $\ovS$, 
such that  the GCM conditions  \eqref{def:GCMC} are verified, and 
 \beaa
(\div^\S f)_{\ell=1}&=\La, \qquad   (\div^\S\fb)_{\ell=1}=\Lab.
\eeaa
Furthermore, the deformation satisfies the properties \eqref{eq:ThmGCMS1}-\eqref{eq:ThmGCMS6}.
\end{corollary}

\begin{proof}
Recall from Lemma \ref{lemma:controlfarspacetimeregionKerrassumptionRR} that the assumptions ${\bf A1}$-${\bf A4}$ are satisfied by the spacetime region $\RR=\{r\geq r_0\}$ of Kerr  provided  that $r_0=r_0(m_0)$ is sufficiently large, with smallness constants $\epg$ and $\dg$ given by 
\beaa
\epg=\frac{a_0m_0}{\rg}, \qquad \dg=\frac{a_0m_0}{\rg}.
\eeaa
Thus, Theorem \ref{Theorem:ExistenceGCMS1} applies, which concludes the proof of the corollary.
\end{proof}

%%%%%%

\appendix

%%%%%%

%%%%%%%%%%%%%%%%%%%%%%%%%%%%%%%%%%%%%%%%%%%%%%%%%%%%

\section{Proof of Proposition \ref{Prop:transformation-formulas-generalcasewithoutassumptions}}
\lab{sec:proofofProp:transformation-formulas-generalcasewithoutassumptions} 

%%%%%%%%%%%%%%%%%%%%%%%%%%%%%%%%%%%%%%%%%%%%%%%%%%%

%%%%%%%%%%%%%%%%%%%%%%%%%%%%%

\subsection{Transformation formula for $\xi$}

%%%%%%%%%%%%%%%%%%%%%%%%%%%%%

We have
\beaa
2\xi_a' &=&  \g(\D_{e_4'}e_4', e_a')= \la^2\g(\D_{\la^{-1}e_4'}(\la^{-1}e_4'), e_a')\\
&=& \la^2\g\left(\D_{\la^{-1}e_4'}(\la^{-1}e_4'), e_a  +\frac 1 2 f_a e_3\right) =\la^2\g\left(\D_{\la^{-1}e_4'}\left(e_4+f^be_b+\frac{1}{4}|f|^2e_3\right), e_a  +\frac 1 2 f_a e_3\right)\\
&=& \la^2\g\left(\D_{\la^{-1}e_4'}e_4, e_a  +\frac 1 2 f_a e_3\right)+\la^2\D_{\la^{-1}e_4'}f_a+\la^2f^b\g\left(\D_{\la^{-1}e_4'}e_b, e_a  +\frac 1 2 f_a e_3\right)\\
&&+\frac{1}{4}|f|^2\la^2\g\left(\D_{\la^{-1}e_4'}e_3, e_a\right).
\eeaa
We compute the terms on the right-hand side
\beaa
\g\left(\D_{\la^{-1}e_4'}e_4, e_a  +\frac 1 2 f_a e_3\right) &=& \g\left(\D_{e_4+f^be_b+\frac{1}{4}|f|^2e_3 }e_4, e_a  +\frac 1 2 f_a e_3\right)\\
&=& 2\xi+f^b\chi_{ba}+\frac{1}{2}|f|^2\eta_a+2\om f_a+f\c \ze\,f_a +\lot\\
&=& 2\xi+\frac{1}{2}(\trch f_a -\atrch\dual f_a)+2\om f_a+f^b\chih_{ba}+\frac{1}{2}|f|^2\eta_a+f\c \ze\,f_a +\lot,
\eeaa
\beaa
&&f^b\g\left(\D_{\la^{-1}e_4'}e_b, e_a  +\frac 1 2 f_a e_3\right)\\
 &=& f^b\g\left(\D_{\la^{-1}e_4'}\left(e_b+\frac 1 2 f_b e_3\right), e_a  +\frac 1 2 f_a e_3\right) - \frac 1 2 f^bf_b\g\left(\D_{\la^{-1}e_4'}e_3, e_a  +\frac 1 2 f_a e_3\right)\\
&=& -f^b\g\left(\D_{\la^{-1}e_4'}\left(e_a+\frac 1 2 f_a e_3\right), e_b  +\frac 1 2 f_b e_3\right) - \frac 1 2 |f|^2\g\left(\D_{\la^{-1}e_4'}e_3, e_a\right)\\
&=& -f^b\g\left(\D_{\la^{-1}e_4'}\left(e_a'-\frac 1 2  \fb_a \la^{-1} e_4'\right), e_b'-\frac 1 2  \fb_b \la^{-1} e_4'\right) - \frac 1 2 |f|^2\g\left(\D_{e_4}e_3, e_a\right)+\lot\\
&=&  -f^b\g\left(\D_{\la^{-1}e_4'}e_a', e_b'\right)+\la^{-2}\fb_af^b\xi'_b+\la^{-2}f^b\fb_b\xi'_a  -  f^bf_b\etab_a +\lot
\eeaa
and
\beaa
|f|^2\g\left(\D_{\la^{-1}e_4'}e_3, e_a\right) &=& |f|^2\g\left(\D_{e_4}e_3, e_a\right) +\lot =  2|f|^2\etab_a +\lot
\eeaa
We infer
\beaa
2\la^{-2}\xi_a' &=& \g\left(\D_{\la^{-1}e_4'}e_4, e_a  +\frac 1 2 f_a e_3\right)+\D_{\la^{-1}e_4'}f_a+f^b\g\left(\D_{\la^{-1}e_4'}e_b, e_a  +\frac 1 2 f_a e_3\right)\\
&&+\frac{1}{4}|f|^2\g\left(\D_{\la^{-1}e_4'}e_3, e_a\right)\\
&=& 2\xi +\nab_{\la^{-1}e_4'}f_a+\frac{1}{2}(\trch f_a -\atrch\dual f_a)+2\om f_a+f^b\chih_{ba}+\frac{1}{2}|f|^2\eta_a+f\c \ze\,f_a \\
&& +\la^{-2}\fb_af^b\xi'_b+\la^{-2}f^b\fb_b\xi'_a  -  f^bf_b\etab_a +\frac{1}{2}|f|^2\etab_a +\lot
\eeaa
and hence
\beaa
\la^{-2}\xi' &=& \xi +\frac{1}{2}\nab_{\la^{-1}e_4'}f+\frac{1}{4}(\trch f -\atrch\dual f)+\om f +\err(\xi,\xi'),\\
\err(\xi,\xi') &=& \frac{1}{2}f\c\chih+\frac{1}{4}|f|^2\eta+\frac{1}{2}(f\c \ze)\,f -\frac{1}{4}|f|^2\etab +  \frac{1}{2}\Big(\la^{-2}(f\c\xi')\,\fb+\la^{-2}(f\c\fb)\,\xi'   \Big)  +\lot
\eeaa
as desired.

%%%%%%%%%%%%%%%%%%%%%%%%%%%%%

\subsection{Transformation formula for $\xib$}

%%%%%%%%%%%%%%%%%%%%%%%%%%%%%

We have
\beaa
2\xib_a' &=&  \g(\D_{e_3'}e_3', e_a') = \la^{-2}\g(\D_{\la e_3'}(\la e_3'), e_a') =\la^{-2}\g\left(\D_{\la e_3'}\left(e_3 +  \fb^be_b' -\frac 1 4 |\fb|^2\la^{-1} e_4'\right), e_a'\right)\\
&=& \la^{-2}\g\left(\D_{\la e_3'}e_3, e_a'\right) + \la^{-1}e_3'(\fb_a')+ \la^{-2}\fb^b\g\left(\D_{\la e_3'}e_b', e_a'\right) -\frac 1 2\la^{-2} |\fb|^2\eta_a'\\
&=& \la^{-2}\g\left(\D_{\la e_3'}e_3,  \left(\de_a^b +\frac{1}{2}\fb_af^b\right) e_b +\frac 1 2  \fb_a  e_4 \right) + \la^{-1}\nab_3'\fb_a' -\frac 1 2\la^{-2} |\fb|^2\eta_a'\\
&=& \la^{-2}\g\left(\D_{ e_3 + \fb^b e_b  + \frac 1 4 |\fb|^2 e_4}e_3,  e_a +\frac 1 2  \fb_a  e_4 \right) + \la^{-1}\nab_3'\fb_a' -\frac 1 2\la^{-2} |\fb|^2\eta_a'+\lot\\
&=& 2\la^{-2}\xib_a +  2  \fb_a\la^{-2}\omb + \fb^b\la^{-2}\chib_{ba} - \fb^b\fb_a\la^{-2}\ze_b +  \frac 1 2 |\fb|^2\la^{-2}\etab_a + \la^{-1}\nab_3'\fb_a' -\frac 1 2\la^{-2} |\fb|^2\eta_a'+\lot
\eeaa
and hence
\beaa
\la^2\xib' &=& \xib + \frac{1}{2}\la\nab_3'\fb' +    \omb\,\fb + \frac{1}{4}\trchb\,\fb - \frac{1}{4}\atrchb\dual\fb +\err(\xib, \xib'),\\
\err(\xib, \xib') &=&   \frac{1}{2}\fb\c\chibh - \frac{1}{2}(\fb\c\ze)\fb +  \frac 1 4 |\fb|^2\etab  -\frac 1 4 |\fb|^2\eta'+\lot
\eeaa
as desired.

%%%%%%%%%%%%%%%%%%%%%%%%%%%%%%%%%%%%%%%%%%%%%

\subsection{Transformation formulas for $\chi $}

%%%%%%%%%%%%%%%%%%%%%%%%%%%%%%%%%%%%%%%%%%%%%

Next, we have
\beaa
\la^{-1}\chi_{ab}' &=& \g\left(\D_{e_a'}(\la^{-1}e_4'), e_b' \right)=  \g\left(\D_{e_a'}(\la^{-1}e_4'), e_b   +\frac 1 2 f_b e_3\right) \\
&=&  \g\left(\D_{e_a'}\left(e_4+f^ce_c+\frac{1}{4}|f|^2e_3\right), e_b   +\frac 1 2 f_b e_3\right)\\
&=&  \g\left(\D_{e_a'}e_4, e_b   +\frac 1 2 f_b e_3\right)+  e_a'(f_b)+  f^c\g\left(\D_{e_a'}e_c, e_b   +\frac 1 2 f_b e_3\right) +\frac{1}{4}|f|^2\g\left(\D_{e_a'}e_3, e_b\right)\\
&=&  \g\left(\D_{e_a'}e_4, e_b   +\frac 1 2 f_b e_3\right)+  e_a'(f_b)-  f^c\g\left(\D_{e_a'}\left(e_b   +\frac 1 2 f_b e_3\right), e_c\right) +\frac{1}{4}|f|^2\g\left(\D_{e_a'}e_3, e_b\right).
\eeaa
We compute the terms on the right-hand side
\beaa
\g\left(\D_{e_a'}e_4, e_b   +\frac 1 2 f_b e_3\right) &=& \g\left(\D_{e_a'}e_4, e_b  \right)+\frac 1 2 f_b\g\left(\D_{e_a'}e_4,  e_3\right)\\
&=& \g\left(\D_{\left(\de_a^c +\frac{1}{2}\fb_af^c\right) e_c +\frac 1 2  \fb_a  e_4 +\left(\frac 1 2 f_a +\frac{1}{8}|f|^2\fb_a\right)   e_3}e_4, e_b  \right)\\
&&+\frac 1 2 f_b\g\left(\D_{\left(\de_a^c +\frac{1}{2}\fb_af^c\right) e_c +\frac 1 2  \fb_a  e_4 +\left(\frac 1 2 f_a +\frac{1}{8}|f|^2\fb_a\right)   e_3}e_4,  e_3\right)\\
&=& \left(\de_a^c +\frac{1}{2}\fb_af^c\right)\chi_{cb}+\fb_a\xi_b +f_a\eta_b+ f_b\ze_a +\om f_b\fb_a  -\omb f_b f_a+\lot, 
\eeaa
\beaa
&&f^c\g\left(\D_{e_a'}\left(e_b   +\frac 1 2 f_b e_3\right), e_c\right)\\
 &=& f^c\g\left(\D_{e_a'}\left(e_b   +\frac 1 2 f_b e_3\right), e_c +\frac 1 2 f_c e_3\right) -\frac 1 2 f^cf_c\g\left(\D_{e_a'}\left(e_b   +\frac 1 2 f_b e_3\right),  e_3\right) \\
&=&  f^c\g\left(\D_{e_a'}\left(e_b' -\frac 1 2  \fb_b \la^{-1} e_4'\right), e_c' -\frac 1 2  \fb_c \la^{-1} e_4'\right)+\frac 1 2|f|^2\g\left(\D_{e_a'}e_3, e_b   +\frac 1 2 f_b e_3\right)\\
&=&  f^c\g\left(\D_{e_a'}e_b', e_c' -\frac 1 2  \fb_c \la^{-1} e_4'\right) -\frac 1 2 f^c \fb_b \la^{-1}\g\left(\D_{e_a'}e_4', e_c' \right)+\frac 1 2 |f|^2\g\left(\D_{e_a}e_3, e_b\right)+\lot\\
&=&  f^c\g\left(\D_{e_a'}e_b', e_c'\right) +\frac 1 2  f^c\fb_c \la^{-1} \chi_{ab}'-\frac 1 2 f^c\fb_b \la^{-1}\chi_{ac}' +\frac 1 2 |f|^2\chib_{ab} +\lot,
\eeaa
and 
\beaa
 |f|^2\g\left(\D_{e_a'}e_3, e_b\right) &=&  |f|^2\chib_{ab} +\lot
\eeaa
We infer
\beaa
\la^{-1}\chi_{ab}' &=&  \g\left(\D_{e_a'}e_4, e_b   +\frac 1 2 f_b e_3\right)+  e_a'(f_b)-  f^c\g\left(\D_{e_a'}\left(e_b   +\frac 1 2 f_b e_3\right), e_c\right) +\frac{1}{4}|f|^2\g\left(\D_{e_a'}e_3, e_b\right)\\
&=& \chi_{ab}  +  \nab_a'f_b + f_a \eta_b + f_b\ze_a+\fb_a\xi_b+\frac{1}{2}\fb_af^c\chi_{cb}  +\om f_b\fb_a  -\omb f_b f_a \\
&& - \frac 1 2  f^c\fb_c \la^{-1} \chi_{ab}'+\frac 1 2  f^c\fb_b \la^{-1}\chi_{ac}' -\frac 1 2 f^cf_c\chib_{ab}   +\frac{1}{4}|f|^2\chib_{ab} +\lot
\eeaa
Hence
\beaa
\la^{-1}\trch' &=& \trch  +  \div'f + f\c\eta + f\c\ze+\err(\trch,\trch')\\
\err(\trch,\trch') &=& \fb\c\xi+\frac{1}{4}\fb\c\left(f\trch -\dual f\atrch\right) +\om (f\c\fb)  -\omb |f|^2 -\frac{1}{4}|f|^2\trchb\\
&& -  \frac 1 4 ( f\c\fb) \la^{-1}\trch' +\frac 1 4  (\fb\wedge f) \la^{-1}\atrch'+\lot,
\eeaa
\beaa
\la^{-1}\atrch' &=& \atrch  +  \curl'f + f\wedge\eta + f\wedge\ze +\err(\atrch,\atrch'),\\
\err(\atrch,\atrch') &=& \fb\wedge\xi+\frac{1}{4}\left(\fb\wedge f\trch +(f\c\fb)\atrch\right) +\om f\wedge\fb   -\frac{1}{4}|f|^2\atrchb\\
&& -  \frac 1 4 ( f\c\fb) \la^{-1}\atrch' +\frac 1 4   \la^{-1}(f\wedge \fb)\trch'+\lot,
\eeaa
and
\beaa
\la^{-1}\chih' &=& \chih  +  \nab'\hot f + f\hot\eta + f\hot\ze+\err(\chih,\chih')\\
\err(\chih,\chih') &=&\fb\hot\xi+\frac{1}{4}\fb\hot\left(f\trch -\dual f\atrch\right) +\om f\hot\fb  -\omb f\hot f \\
&&-  \frac 1 2 ( f\c\fb) \la^{-1} \chih'  +\frac 1 4  (f\hot\fb) \la^{-1}\trch' +\frac 1 4  (\dual f\hot\fb) \la^{-1}\atrch' +\frac 1 2  \fb\hot (f\c\la^{-1}\chih')+\lot
\eeaa
as desired.

%%%%%%%%%%%%%%%%%%%%%%%%%%%%%

\subsection{Transformation formula for $\chib$}

%%%%%%%%%%%%%%%%%%%%%%%%%%%%%

Next, we have
\beaa
\la\chib_{ab}' &=& \g\left(\D_{e_a'}(\la e_3'), e_b' \right)= \g\left(\D_{e_a'}\left(  e_3 +  \fb^ce_c' -\frac 1 4 |\fb|^2\la^{-1} e_4'\right), e_b' \right)\\
&=& \g\left(\D_{e_a'}e_3, e_b' \right)+e_a'(\fb_b) +\fb^c\g\left(\D_{e_a'}e_c', e_b' \right)-\frac 1 4 |\fb|^2\la^{-1}\chi_{ab}'\\
&=& \g\left(\D_{e_a'}e_3, e_b' \right)+\nab_a'\fb_b -\frac 1 4 |\fb|^2\la^{-1}\chi_{ab}'.
 \eeaa
We compute
\beaa
 \g\left(\D_{e_a'}e_3, e_b' \right) &=& \g\left(\D_{\left(\de_a^d +\frac{1}{2}\fb_af^d\right) e_d +\frac 1 2  \fb_a  e_4 +\frac 1 2 f_a  e_3}e_3, \left(\de_b^c +\frac{1}{2}\fb_bf^c\right) e_c +\frac 1 2  \fb_b  e_4  \right)+\lot\\
 &=& \chib_{ab}+\frac{1}{2}\fb_af^d\chib_{db}+\frac{1}{2}\fb_bf^c\chib_{ac} +\fb_a\,\etab_b+f_a\xib_b -\fb_b\,\fb_a\om + \fb_bf_a\omb  -  \fb_b\ze_a+\lot
\eeaa
and hence
\beaa
\la\chib_{ab}' &=& \chib_{ab} +\nab_a'\fb_b +\frac{1}{2}\fb_af^d\chib_{db}+\frac{1}{2}\fb_bf^c\chib_{ac} +\fb_a\,\etab_b+f_a\xib_b -\fb_b\,\fb_a\om + \fb_bf_a\omb \\
&&  -  \fb_b\ze_a -\frac 1 4 |\fb|^2\la^{-1}\chi_{ab}'+\lot
 \eeaa
We deduce
\beaa
\la\trchb' &=& \trchb +\div'\fb +\fb\c\etab  -  \fb\c\ze +\err(\trchb, \trchb'),\\
\err(\trchb, \trchb') &=& \frac{1}{2}(f\c\fb)\trchb+f\c\xib -|\fb|^2\om + (f\c\fb)\omb   -\frac 1 4 |\fb|^2\la^{-1}\trch'+\lot,
 \eeaa
\beaa
\la\atrchb' &=& \atrchb +\curl'\fb +\fb\wedge\etab  -  \ze\wedge\fb+\err(\atrchb, \atrchb'),\\
\err(\atrchb, \atrchb') &=& \frac{1}{2}(f\c\fb)\atrchb+f\wedge\xib  + (f\wedge\fb)\omb   -\frac 1 4 |\fb|^2\la^{-1}\atrch'+\lot,
 \eeaa
and
\beaa
\la\chibh' &=& \chibh +\nab'\hot\fb +\fb\hot\etab  -  \fb\hot\ze +\err(\chibh, \chibh'),\\
\err(\chibh, \chibh') &=& \frac{1}{2}(f\hot\fb)\trchb  +f\hot\xib -(\fb\hot\fb)\om + (f\hot\fb)\omb   -\frac 1 4 |\fb|^2\la^{-1}\chih'+\lot
 \eeaa
as desired.

%%%%%%%%%%%%%%%%%%%%%%%%%%%%%%%%%%%%%%%%%%%%%

\subsection{Transformation formula for $\ze$}

%%%%%%%%%%%%%%%%%%%%%%%%%%%%%%%%%%%%%%%%%%%%%

Next, we have
\beaa
2\ze_a' &=&  \g(\D_{e_a'}e_4', e_3') = -2e_a'(\log\la)+\g(\D_{e_a'}(\la^{-1}e_4'), \la e_3') = -2e_a'(\log\la)+\g\left(\D_{e_a'}(\la^{-1}e_4'), e_3 +  \fb^be_b' \right)\\
&=& -2e_a'(\log\la)+\g\left(\D_{e_a'}(\la^{-1}e_4'), e_3  \right)+\la^{-1}\fb^b\chi_{ab}'.
\eeaa
We compute the term on the right-hand side
\beaa
\g\left(\D_{e_a'}(\la^{-1}e_4'), e_3\right) &=& \g\left(\D_{e_a'}\left(e_4+f^be_b+\frac{1}{4}|f|^2e_3\right), e_3\right)\\
&=& \g\left(\D_{e_a'}e_4, e_3\right)+ f^b\g\left(\D_{e_a'}e_b, e_3\right)\\
&=& \g\left(\D_{\left(\de_a^c +\frac{1}{2}\fb_af^c\right) e_c +\frac 1 2  \fb_a  e_4 +\left(\frac 1 2 f_a +\frac{1}{8}|f|^2\fb_a\right)   e_3}e_4, e_3\right) \\
&&+ f^b\g\left(\D_{\left(\de_a^c +\frac{1}{2}\fb_af^c\right) e_c +\frac 1 2  \fb_a  e_4 +\left(\frac 1 2 f_a +\frac{1}{8}|f|^2\fb_a\right)   e_3}e_b, e_3\right)\\
&=& 2\ze_a +2\om\fb_a-2\omb f_a -\chib_{ba}f^b + \fb_af^c\ze_c -  \fb_a f^b\etab_b  +\lot
\eeaa
We infer
\beaa
2\ze_a' &=& -2e_a'(\log\la)+\g\left(\D_{e_a'}(\la^{-1}e_4'), e_3  \right)+\la^{-1}\fb^b\chi_{ab}'\\
&=& 2\ze_a -2e_a'(\log\la)  -\frac{1}{2}\trchb f_a +\frac{1}{2}\atrchb \dual f_a +2\om\fb_a-2\omb f_a +\frac{1}{2}\la^{-1}\fb_a\trch' +\frac{1}{2}\la^{-1}\dual\fb_a\atrch'\\
&& -\chibh_{ab}f^b + \fb_af^c\ze_c -  \fb_a f^b\etab_b +\la^{-1}\fb^b\chih_{ab}'+\lot
\eeaa
and hence, 
\beaa
\ze' &=& \ze -\nab'(\log\la)  -\frac{1}{4}\trchb f +\frac{1}{4}\atrchb \dual f +\om\fb -\omb f +\frac{1}{4}\la^{-1}\fb\trch' +\frac{1}{4}\la^{-1}\dual\fb\atrch' \\
&& -\frac{1}{2}\chibh\c f + \frac{1}{2}(f\c\ze)\fb -  \frac{1}{2}(f\c\etab)\fb +\frac{1}{2}\la^{-1}\fb\c\chih'+\lot
\eeaa
Using also the above transformation formulas for $\trch'$ and $\atrch'$, we infer
\beaa
\ze' &=& \ze -\nab'(\log\la)  -\frac{1}{4}\trchb f +\frac{1}{4}\atrchb \dual f +\om\fb -\omb f +\frac{1}{4}\fb\trch+\frac{1}{4}\dual\fb\atrch+\err(\ze, \ze'),\\
\err(\ze, \ze') &=& -\frac{1}{2}\chibh\c f + \frac{1}{2}(f\c\ze)\fb -  \frac{1}{2}(f\c\etab)\fb +\frac{1}{4}\fb(f\c\eta) + \frac{1}{4}\fb(f\c\ze)   + \frac{1}{4}\dual\fb(f\wedge\eta) + \frac{1}{4}\dual\fb(f\wedge\ze)\\
&& +  \frac{1}{4}\fb\div'f  + \frac{1}{4}\dual\fb \curl'f +\frac{1}{2}\la^{-1}\fb\c\chih'   -\frac{1}{16}(f\c\fb)\fb\la^{-1}\trch' +\frac{1}{16}  (\fb\wedge f) \fb\la^{-1}\atrch'\\
 &&  -  \frac{1}{16}\dual\fb ( f\c\fb) \la^{-1}\atrch' +\frac{1}{16}\dual\fb \la^{-1}(f\wedge \fb)\trch' +\lot
\eeaa
as desired.

%%%%%%%%%%%%%%%%%%%%%%%%%%%%%%%%%%%%%%%%%%%%%

\subsection{Transformation formula for $\eta$}

%%%%%%%%%%%%%%%%%%%%%%%%%%%%%%%%%%%%%%%%%%%%%

Next, we have
\beaa
2\eta_a' &=&  \g(\D_{e_3'}e_4', e_a') = \g(\D_{\la e_3'}(\la^{-1} e_4'), e_a') = \g\left(\D_{\la e_3'}(\la^{-1} e_4'), e_a +\frac 1 2  \fb_a \la^{-1} e_4'  +\frac 1 2 f_a e_3\right)\\
&=& \g\left(\D_{\la e_3'}(\la^{-1} e_4'), e_a  +\frac 1 2 f_a e_3\right)= \g\left(\D_{\la e_3'}\left(e_4+f^be_b+\frac{1}{4}|f|^2e_3\right), e_a  +\frac 1 2 f_a e_3\right)\\
&=& \g\left(\D_{\la e_3'}e_4, e_a  +\frac 1 2 f_a e_3\right)+\la e_3'(f_a)+f^b\g\left(\D_{\la e_3'}e_b, e_a  +\frac 1 2 f_a e_3\right)+\frac{1}{4}|f|^2\g\left(\D_{\la e_3'}e_3, e_a\right).
\eeaa
We compute the term on the right-hand side
\beaa
\g\left(\D_{\la e_3'}e_4, e_a  +\frac 1 2 f_a e_3\right) &=& \g\left(\D_{\left(1+\frac{1}{2}f\c\fb  +\frac{1}{16} |f|^2  |\fb|^2\right) e_3 + \left(\fb^b+\frac 1 4 |\fb|^2f^b\right) e_b  + \frac 1 4 |\fb|^2 e_4}e_4, e_a  +\frac 1 2 f_a e_3\right)\\
&=& \left(1+\frac{1}{2}f\c\fb  \right)\g\left(\D_{ e_3}e_4 , e_a +\frac 1 2 f_a e_3\right)+\fb^b\g\left(\D_{  e_b }e_4, e_a  +\frac 1 2 f_a e_3\right)+\lot\\
&=& 2\left(1+\frac{1}{2}f\c\fb  \right)\eta_a -2\omb f_a+\fb^b\chi_{ba}+f_a\fb^b\ze_b+\lot,
\eeaa
\beaa
&&f^b\g\left(\D_{\la e_3'}e_b, e_a  +\frac 1 2 f_a e_3\right) = - f^b\g\left(\D_{\la e_3'}\left(e_a  +\frac 1 2 f_a e_3\right), e_b\right)\\
&=&- f^b\g\left(\D_{\la e_3'}\left(e_a  +\frac 1 2 f_a e_3\right), e_b+\frac 1 2 f_b e_3\right) +\frac 1 2 |f|^2\g\left(\D_{\la e_3'}\left(e_a  +\frac 1 2 f_a e_3\right), e_3\right)\\
&=&- f^b\g\left(\D_{\la e_3'}\left(e_a'-\frac 1 2  \fb_a \la^{-1} e_4' \right), e_b'-\frac 1 2  \fb_b \la^{-1} e_4' \right) -\frac 1 2 |f|^2\g\left(\D_{\la e_3'}e_3, e_a  +\frac 1 2 f_a e_3\right)\\
&=& - f^b\g\left(\D_{\la e_3'}e_a', e_b' \right) -2  \fb_b f^b\eta_a'+ \fb_a  f^b\eta_b' +\lot,
\eeaa
and
\beaa
\frac{1}{4}|f|^2\g\left(\D_{\la e_3'}e_3, e_a\right) &=& \lot
\eeaa
We infer
\beaa
2\eta_a' &=& 2\left(1+\frac{1}{2}f\c\fb  \right)\eta_a  -2\omb f_a+\fb^b\chi_{ba}+f_a\fb^b\ze_b +\la \nab_3'f_a  -2  \fb_b f^b\eta_a'+ \fb_a  f^b\eta_b' +\lot
\eeaa
and hence
\beaa
\eta' &=& \eta +\frac{1}{2}\la \nab_3'f +\frac{1}{4}\fb\trch -\frac{1}{4}\dual\fb\atrch -\omb\, f  +\err(\eta, \eta'),\\
\err(\eta, \eta') &=& \frac{1}{2}(f\c\fb)\eta +\frac{1}{2}\fb\c\chih
+\frac{1}{2}f(\fb\c\ze)  -  (\fb\c f)\eta'+ \frac{1}{2}\fb (f\c\eta') +\lot
\eeaa
as desired.

%%%%%%%%%%%%%%%%%%%%%%%%%%%%%%%%%%%%%%%%%%%%%

\subsection{Transformation formula for $\etab$}

%%%%%%%%%%%%%%%%%%%%%%%%%%%%%%%%%%%%%%%%%%%%%

Next, we have
\beaa
2\etab_a' &=&  \g(\D_{e_4'}e_3', e_a')= \g(\D_{\la^{-1}e_4'}(\la e_3'), e_a') = \g\left(\D_{\la^{-1}e_4'}\left(e_3 +  \fb^be_b' -\frac 1 4 |\fb|^2\la^{-1} e_4'\right), e_a'\right)\\
&=& \g\left(\D_{\la^{-1}e_4'}e_3 , e_a'\right) +\g\left(\D_{\la^{-1}e_4'}\left( \fb^be_b' \right), e_a'\right) -\frac 1 2 |\fb|^2\la^{-2}\xi'_a.
\eeaa
We compute the terms on the right-hand side
\beaa
\g\left(\D_{\la^{-1}e_4'}e_3 , e_a'\right) &=& \g\left(\D_{\la^{-1}e_4'}e_3 , \left(\de_a^b +\frac{1}{2}\fb_af^b\right) e_b +\frac 1 2  \fb_a  e_4 \right)\\
&=& \left(\de_a^b +\frac{1}{2}\fb_af^b\right)\g\left(\D_{\la^{-1}e_4'}e_3 ,  e_b  \right)+\frac 1 2  \fb_a\g\left(\D_{\la^{-1}e_4'}e_3 ,   e_4 \right)\\
&=& \left(\de_a^b +\frac{1}{2}\fb_af^b\right)\g\left(\D_{e_4+f^ce_c+\frac{1}{4}|f|^2e_3}e_3 ,  e_b  \right)+\frac 1 2  \fb_a\g\left(\D_{e_4+f^be_b+\frac{1}{4}|f|^2e_3}e_3 ,   e_4 \right)\\
&=& \left(\de_a^b +\frac{1}{2}\fb_af^b\right)\left(2\eta_b+f^c\chib_{cb}+\frac{1}{2}|f|^2\xib_b\right)+\frac 1 2  \fb_a\big(-4\om-f\c\ze \big)+\lot
\eeaa
and
\beaa
\g\left(\D_{\la^{-1}e_4'}\left( \fb^be_b' \right), e_a'\right) &=& \la^{-1}e_4'(\fb_a)+ \fb^b\g\left(\D_{\la^{-1}e_4'}e_b', e_a'\right) = \la^{-1}e_4'(\fb_a) - \fb^b\g\left(\D_{\la^{-1}e_4'}e_a', e_b'\right)\\
&=& \nab_{\la^{-1}e_4'}\fb_a.
\eeaa
We infer
\beaa
2\etab_a' &=& \g\left(\D_{\la^{-1}e_4'}e_3 , e_a'\right) +\g\left(\D_{\la^{-1}e_4'}\left( \fb^be_b' \right), e_a'\right) -\frac 1 2 |\fb|^2\la^{-2}\xi'_a\\
&=& 2\etab_a+\nab_{\la^{-1}e_4'}\fb_a +\frac{1}{2}\trchb f_a- \frac{1}{2}\atrchb\dual f_a -2\om\fb_a+f^c\chibh_{ca} + \fb_af^b\eta_b -\frac 1 2  \fb_a (f\c\ze) \\
&&-\frac 1 2 |\fb|^2\la^{-2}\xi'_a+\lot
\eeaa
and hence
\beaa
\etab' &=& \etab +\frac{1}{2}\nab_{\la^{-1}e_4'}\fb +\frac{1}{4}\trchb f - \frac{1}{4}\atrchb\dual f -\om\fb +\err(\etab, \etab'),\\
\err(\etab, \etab') &=&  \frac{1}{2}f\c\chibh + \frac{1}{2}(f\c\eta)\fb-\frac 1 4  (f\c\ze)\fb  -\frac 1 4 |\fb|^2\la^{-2}\xi'+\lot
\eeaa
as desired.

%%%%%%%%%%%%%%%%%%%%%%%%%%%%%%%%

\subsection{Transformation formula for $\om$}

%%%%%%%%%%%%%%%%%%%%%%%%%%%%%%%%

Next, we have
\beaa
4\om' &=&  \g(\D_{e_4'}e_4', e_3') = -2e_4'(\log\la)+\la\g(\D_{\la^{-1}e_4'}(\la^{-1}e_4'), \la e_3')\\
&=& -2e_4'(\log\la)+\la\g\left(\D_{\la^{-1}e_4'}(\la^{-1}e_4'), e_3 +  \fb^ae_a' -\frac 1 4 |\fb|^2\la^{-1} e_4'\right)\\
&=& -2e_4'(\log\la)+\la\g\left(\D_{\la^{-1}e_4'}(\la^{-1}e_4'), e_3\right)+2\la^{-1}\fb^a\xi'_a\\
&=& -2e_4'(\log\la)+\la\g\left(\D_{\la^{-1}e_4'}\left(e_4+f^be_b+\frac{1}{4}|f|^2e_3\right), e_3\right)+2\la^{-1}\fb^a\xi'_a\\
&=& -2e_4'(\log\la)+ \la\g\left(\D_{\la^{-1}e_4'}e_4, e_3\right)+ \la f^b\g\left(\D_{\la^{-1}e_4'}e_b, e_3\right)+2\la^{-1}\fb^a\xi'_a.
\eeaa
We compute the terms on the right-hand side
\beaa
\g\left(\D_{\la^{-1}e_4'}e_4, e_3\right) &=& \g\left(\D_{e_4+f^a e_a+\frac{1}{4}|f|^2e_3}e_4, e_3\right)\\
&=& 4\om+2f\c\ze-|f|^2\omb
\eeaa
and
\beaa
\g\left(\D_{\la^{-1}e_4'}e_b, e_3\right) &=& -\g\left(\D_{\la^{-1}e_4'}e_3, e_b\right) = -2\etab_b - f^c\chib_{cb} +\lot
\eeaa
We infer
\beaa
4\la^{-1}\om' &=& -2\la^{-1}e_4'(\log\la)+ \g\left(\D_{\la^{-1}e_4'}e_4, e_3\right)+  f^b\g\left(\D_{\la^{-1}e_4'}e_b, e_3\right)+2\la^{-2}\fb^a\xi'_a\\
&=& -2\la^{-1}e_4'(\log\la)+ 4\om+2 f\c(\ze-\etab) -|f|^2\omb+  f^b\left( - f^c\chib_{cb} \right)+2\la^{-2}\fb^a\xi'_a+\lot
\eeaa
and hence
\beaa
\la^{-1}\om' &=&  \om -\frac{1}{2}\la^{-1}e_4'(\log\la)+\frac{1}{2}f\c(\ze-\etab) +\err(\om, \om'),\\
\err(\om, \om') &=&   -\frac{1}{4}|f|^2\omb - \frac{1}{8}\trchb |f|^2+\frac{1}{2}\la^{-2}\fb\c\xi' +\lot
\eeaa
as desired.

%%%%%%%%%%%%%%%%%%%%%%%%%%%%%%%%%%%%%%%%%%%%%

\subsection{Transformation formula for $\omb$}

%%%%%%%%%%%%%%%%%%%%%%%%%%%%%%%%%%%%%%%%%%%%%

Next, we have
\beaa
4\omb' &=&  \g(\D_{e_3'}e_3', e_4') = 2e_3'(\log\la)+\la^{-1}\g(\D_{\la e_3'}(\la e_3'), \la^{-1} e_4')\\
&=& 2e_3'(\log\la) -\la^{-1}\g(\D_{\la e_3'}(\la^{-1}e_4'), \la e_3')\\
&=& 2e_3'(\log\la) -\la^{-1}\g\left(\D_{\la e_3'}(\la^{-1}e_4'), e_3 +  \fb^ae_a' -\frac 1 4 |\fb|^2\la^{-1} e_4'\right)\\
&=& 2e_3'(\log\la) -\la^{-1}\g\left(\D_{\la e_3'}\left(e_4+f^a e_a+\frac{1}{4}|f|^2e_3\right), e_3 \right) -2\fb^a\la^{-1}\eta_a'\\
&=& 2e_3'(\log\la) -\la^{-1}\g\left(\D_{\la e_3'}e_4, e_3 \right) - f^a\la^{-1}\g\left(\D_{\la e_3'}e_a, e_3 \right)  -2\fb^a\la^{-1}\eta_a'.
\eeaa
We compute
\beaa
\g\left(\D_{\la e_3'}e_4, e_3 \right) &=& \g\left(\D_{\left(1+\frac{1}{2}f\c\fb \right) e_3 + \fb^b e_b  + \frac 1 4 |\fb|^2 e_4}e_4, e_3 \right)+\lot\\
&=& -4\left(1+\frac{1}{2}f\c\fb \right)\omb+2\fb^b\ze_b+ |\fb|^2\om+\lot
\eeaa
and
\beaa
f^a\g\left(\D_{\la e_3'}e_a, e_3 \right) &=& f^a\g\left(\D_{e_3 + \fb^b e_b  }e_a, e_3 \right)+\lot = -2f^a\xib_a - f^a\fb^b\chib_{ba}+\lot
\eeaa
and hence
\beaa
\la\omb' &=& \frac{1}{2}\la e_3'(\log\la) - \frac{1}{4}\g\left(\D_{\la e_3'}e_4, e_3 \right) - \frac{1}{4}f^a\g\left(\D_{\la e_3'}e_a, e_3 \right)  -\frac{1}{2}\fb^a\eta_a'\\
&=& \frac{1}{2}\la e_3'(\log\la) - \frac{1}{4}\left(-4\left(1+\frac{1}{2}f\c\fb \right)\omb+2\fb^b\ze_b+ |\fb|^2\om \right)\\
&& -\frac{1}{4}\Big(-2f^a\xib_a - f^a\fb^b\chib_{ba}\Big)  -\frac{1}{2}\fb^a\eta_a'+\lot
\eeaa
Together with the above transformation formula for $\eta'$, we deduce
\beaa
\la\omb' &=& \omb+\frac{1}{2}\la e_3'(\log\la)  -\frac{1}{2}\fb\c\ze -\frac{1}{2}\fb\c\eta +\err(\omb,\omb')\\
\err(\omb,\omb') &=& f\c\fb\,\omb-\frac{1}{4} |\fb|^2\om  +\frac{1}{2} f\c\xib + \frac{1}{8}(f\c\fb)\trchb + \frac{1}{8}(\fb\wedge f)\atrchb \\
&& -\frac{1}{8}|\fb|^2\trch  -\frac{1}{4}\la \fb\c\nab_3'f    +\frac{1}{2}  (\fb\c f)(\fb\c\eta')- \frac{1}{4}|\fb|^2 (f\c\eta')+\lot
\eeaa
as desired.

%%%%%%%%%%%%%%%%%%%%%%%%%%%%%%%%%%%%%%%%%%%%%

\subsection{Transformation formula for $\a$}

%%%%%%%%%%%%%%%%%%%%%%%%%%%%%%%%%%%%%%%%%%%%%

Next, we have
\beaa
\la^{-2}\a_{ab}'&=&\R( e_a' , e_4', e_b', e_4')=  \R\Big( e_a' ,  e_4 + f^c  e_b +\frac 1 4 |f|^2  e_3  , e_b ',  e_4 + f^d  e_d +\frac 1 4 |f|^2  e_3 \Big)\\
&=&  \R( e_a' , e_4, e_b', e_4) + f ^c   \R\Big( e_a' ,   e_c , e_b',  e_4 \Big)+ f ^d   \R\Big( e_a' ,   e_4  , e_b',  e_d \Big)\\
&+&\frac 1 4 |f|^2   \R\Big( e_a' ,   e_3  , e_b ',  e_4  \Big)+\frac 1 4 |f|^2   \R\Big( e_a' ,   e_4 , e_b ',  e_3  \Big)+\lot\\
&=&  \R( e_a' , e_4, e_b', e_4) + f ^c   \R\Big( e_a' ,   e_c , e_b',  e_4 \Big)+ f ^d   \R\Big( e_a' ,   e_4  , e_b',  e_d \Big)\\
&+&\frac 1 4 |f|^2  \Bigg( \R\Big( e_a ,   e_3  , e_b ,  e_4  \Big)+\R\Big( e_a ,   e_4   , e_b ,  e_3  \Big)\Bigg) + f^c f^d \R\big(e_a, e_c, e_b, e_d\big)  +\lot\\
 &=& \R( e_a' , e_4, e_b', e_4) + f ^c   \R\Big( e_a' ,   e_c , e_b',  e_4 \Big)+ f ^d   \R\Big( e_a' ,   e_4  , e_b',  e_d \Big)\\
 &-&\frac 1 2  |f|^2 \rho\de_{ab}   -\in_{ac}\in_{bd} f^c f^d \rho +      \lot
\eeaa
We have
\beaa
 \R( e_a' , e_4, e_b', e_4) &=&  \R\Big( (e_a   +\frac{1}{2}\fb_af^c e_c) +\frac 1 2 f_a   e_3 , e_4, (e_b+ \frac{1}{2}\fb_b f^d e_d)  +\frac 1 2 f_b    e_3    , e_4\Big)+\lot
 \\
&=&  \R\Big( (e_a   +\frac{1}{2}\fb_af^c e_c) , e_4, (e_b+ \frac{1}{2}\fb_b f^d e_d )    , e_4\Big)
+\frac 1 2 f_a \R( e_3,  e_4, e_b, e_4\big) \\
&+   &\frac 1 2 f_b  \R( e_a ,  e_4, e_3 , e_4\big) +  \frac 1 4  f_a  f_b  \R( e_3 , e_4, e_3, e_4) +\lot\\
 &=&\a_{ab} + \big( f_a \b_b+f_b \b_a) +  f_a f_b \rho+\lot
\eeaa
Also,
\beaa
 f ^c   \R\Big( e_a' ,   e_c , e_b',  e_4 \Big)&=& f ^c   \R\Big( e_a+\frac 1 2 f_a e_3+\frac 1 2 \fb e_4,   e_c , e_b+\frac 1 2 f_b e_3+\frac 1 2 \fb_b e_4 ,  e_4 \Big)+\lot
 \\
 &=& f ^c   \R\Big( e_a,   e_c , e_b,  e_4 \Big)+\frac 1 2 f^c f_a \R\Big( e_3,   e_c , e_b,  e_4 \Big)+ \frac 1 2 f^c f_b \R\Big( e_a,   e_c , e_3,  e_4 \Big)\\
 &+&\lot\\
 &=&-\in_{ac} \dual \b_b f^c+\frac 1 2 f^c f_a\big( \rho\de_{cb} -\rhod \in_{cb}\big)+\frac 1 2 f^c f_b\big(  2 \in_{ac}\rhod \big)+\lot
 \\
 &=&- \dual \b_b  \dual f_a +\frac 1 2 f_a f_b \rho +\frac 1 2  f_a \dual f_b \rhod + \dual f_a\, f_b \,\rhod+\lot
 \eeaa
 and,
\beaa
 f ^d   \R\Big( e_a' ,   e_4  , e_b',  e_d \Big)&=&f ^c   \R\Big(  e_b',  e_c,  e_a' ,   e_4  \Big)=- \dual \b_a  \dual f_b +\frac 1 2 f_b f_a \rho +\frac 1 2  f_b \dual f_a  \rhod + \dual f_b  f_a  \rhod+\lot
\eeaa
Consequently,
\beaa
\la^{-2}\a_{ab}'&=& \R( e_a' , e_4, e_b', e_4) + f ^c   \R\Big( e_a' ,   e_c , e_b',  e_4 \Big)+ f ^d   \R\Big( e_a' ,   e_4  , e_b',  e_d \Big)-\frac 1 2  |f|^2 \rho  - \dual f_a  \dual f_b \rho \\
&=&\a_{ab} + \big( f_a \b_b+f_b \b_a) +  f_a f_b \rho - \dual \b_b  \dual f_a +\frac 1 2 f_a f_b \rho +\frac 1 2  f_a \dual f_b \rhod + \dual f_a f_b \rhod\\
& -&\dual \b_a  \dual f_b +\frac 1 2 f_b f_a \rho +\frac 1 2  f_b \dual f_a  \rhod + \dual f_b  f_a  \rhod -\frac 1 2  |f|^2 \rho\de_{ab}+\lot\\
&=&\a_{ab} + \big( f_a \b_b+f_b \b_a) - \big( \dual f_a\dual  \b_b+\dual f_b\dual  \b_a)+ \big( 2  f_a f_b-\dual f_a \dual f_b  -\frac 1 2 |f|^2 \de_{ab}\big) \rho
\\
&+&  \frac 3  2 \Big(  f_a\, \dual  f _b +f_b\,  \dual  f_a) \rhod +\lot
\eeaa
Since $\a$ is traceless, we infer
\beaa
\bsplit
\la^{-2} \a'&=\a +\err(\a, \a')\\
\err(\a, \a')&=  \big(  f\hot \b  -\dual f \hot \dual  \b )+ \big( f\hot f-\frac 1 2  \dual f \hot   \dual f \big) \rho
+  \frac 3  2 \big(  f \hot  \dual  f\big) \rhod +\lot
\end{split}
\eeaa
as desired.

%%%%%%%%%%%%%%%%%%%%%%%%%%%%%%%%

\subsection{Transformation formula for $\b$}

%%%%%%%%%%%%%%%%%%%%%%%%%%%%%%%%

 To derive the transformation formula for $\b$  we  write
\beaa
 2 \la^{-1}\b_a' &=& \R( e_a' , e_4', e_3', e_4')= \R\Big( e_a' ,  e_4 + f^b  e_b +\frac 1 4 |f|^2  e_3  , e_3',  e_4 + f^b  e_b +\frac 1 4 |f|^2  e_3 \Big)\\
 &=& \R( e_a' , e_4, e_3', e_4)+ f ^b  \R\Big( e_a' ,   e_b , e_3',  e_4 \Big)+ f ^b  \R\Big( e_a' ,   e_4  , e_3',  e_b \Big)\\
 &+&\frac 1 4 |f|^2   \R( e_a' , e_3 , e_3', e_4)+ \frac 1 4 |f|^2   \R( e_a' , e_4, e_3', e_3)+ f^b f^c    \R( e_a' , e_b , e_3', e_c)+\lot\\
 &=& \R( e_a' , e_4, e_3', e_4)+ f ^b  \R\Big( e_a' ,   e_b , e_3',  e_4 \Big)+ f ^b  \R\Big( e_a' ,   e_4  , e_3',  e_b \Big)+\lot 
\eeaa
We have,
\beaa
\R( e_a' , e_4, e_3', e_4)&=& \R\big(  (\de_{ab} +\frac{1}{2}\fb_af_b) e_b +\frac 1 2  \fb_a  e_4 +\frac 1 2 f_a   e_3 , e_4, e_3', e_4\big)+\lot\\
&=&\R(e_a, e_4, e_3', e_4)+\frac 1 2 f_a  \R(e_3 , e_4, e_3', e_4) +\lot\\
&=&2  \left(1+\frac{1}{2}f\c\fb \right) \b_a+\fb^b \a_{ab}+ 2 f_a \rho+ \lot \\
&=& 2 \b_a + 2 f_a\rho+ \fb^b \a_{ab}+\lot
\eeaa
Since $\R\Big( e_a ,   e_b , e_3,  e_4 \Big)= 2\in_{ab} \rhod$  and  $ \R\Big( e_a ,   e_4  , e_b,  e_3 \Big)= -\rho \de_{ab}-\in_{ab} \rhod$
\beaa
 f^b   \R\Big( e_a' ,   e_b , e_3',  e_4 \Big)&=&f^b  \R\Big( e_a  + \frac 1 2 \fb_a e_4+\frac 1 2 f_b e_3    ,   e_b , e'_3,  e_4 \Big)+\lot =f ^b \R\Big( e_a ,   e_b , e_3,  e_4 \Big)+\lot\\
 &=& 2 \dual  f_a \rhod+\lot\\
  f ^b  \R\Big( e_a' ,   e_4  , e_3',  e_b \Big)&=&- f^b \R(e_a, e_4, e_b, e_3)+\lot = f_a \rho +\dual f _a\rhod+\lot
  \eeaa
  Hence,
  \beaa
   2 \b_a' &=& \R( e_a' , e_4, e_3', e_4)+ f ^b  \R\Big( e_a' ,   e_b , e_3',  e_4 \Big)+ f ^b  \R\Big( e_a' ,   e_4  , e_3',  e_b \Big)+\lot \\
   &=& 2 \b_a + 2 f_a\rho+ \fb^b \a_{ab} +2 \dual  f_a \rhod + f_a \rho +\dual f _a\rhod+\lot
  \eeaa
  Therefore,
  \beaa
  \bsplit
  \b_a'&=\b_a +\frac 3 2\big(  f_a \rho+\dual  f _a \rhod\big)+\err_a(\b, \b') \\
  \err_a(\b, \b')&= \frac 1 2 \fb^b\a_{ab}+\lot 
  \end{split}
  \eeaa
  as stated.

%%%%%%%%%%%%%%%%%%%%%%%%%%%%%%%%

\subsection{Transformation formula for $\rho$}

%%%%%%%%%%%%%%%%%%%%%%%%%%%%%%%%
  
We start with $\rho$. We have
\beaa
4\rho' &=& \R(e_4', e_3', e_4', e_3')\\
&=& \R\left(e_4', \la^{-1}\left(  e_3 +  \fb^ae_a' -\frac 1 4 |\fb|^2\la^{-1} e_4'\right), e_4', e_3'\right)\\
&=&  \la^{-1}\R\left(e_4',e_3, e_4', e_3'\right) + \la^{-1}\fb^a\R\left(e_4', e_a', e_4', e_3'\right)\\
&=&  \la^{-1}\R\left(e_4',e_3, e_4', \la^{-1}\left(  e_3 +  \fb^ae_a' -\frac 1 4 |\fb|^2\la^{-1} e_4'\right)\right) + \fb^a\R\left(\la^{-1}e_4', e_a', \la^{-1}e_4', \la e_3'\right),
\eeaa
and hence
\beaa
4\rho' &=&  \R\left(\la^{-1}e_4',e_3, \la^{-1}e_4', e_3\right) + \fb^a\R\left(\la^{-1}e_4',e_3, \la^{-1}e_4', e_a +\frac 1 2  \fb_a \la^{-1} e_4'  +\frac 1 2 f_a e_3\right) \\
&&+  \fb^a\R\left(\la^{-1}e_4',  e_a +\frac 1 2  \fb_a \la^{-1} e_4'  +\frac 1 2 f_a e_3, \la^{-1}e_4', \la e_3'\right) +\lot\\
&=&  \R\left(\la^{-1}e_4',e_3, \la^{-1}e_4', e_3\right) + \fb^a\R\left(\la^{-1}e_4',e_3, \la^{-1}e_4', e_a   +\frac 1 2 f_a e_3\right) \\
&&+  \fb^a\R\left(\la^{-1}e_4',  e_a  +\frac 1 2 f_a e_3, \la^{-1}e_4', \la e_3'\right) +\lot\\
&=&  \R\left(\la^{-1}e_4',e_3, \la^{-1}e_4', e_3\right) + \fb^a\R\left(\la^{-1}e_4',e_3, \la^{-1}e_4', e_a  \right)\\
&& +  \fb^a\R\left(\la^{-1}e_4',  e_a, \la^{-1}e_4', \la e_3'\right)   + 4 (f\c\fb)\rho+\lot
\eeaa  
  
We compute
\beaa
\R\left(\la^{-1}e_4',e_3, \la^{-1}e_4', e_3\right) &=& \R\left(e_4+f^ae_a,e_3, e_4+f^be_b, e_3\right)\\
&=& 4\rho+ f^a\R\left(e_a,e_3, e_4, e_3\right) + f^b\R\left(e_4,e_3, e_b, e_3\right) +\lot \\
&=& 4\rho - 4f\c\bb  +\lot,
\eeaa  
\beaa
\fb^a\R\left(\la^{-1}e_4',e_3, \la^{-1}e_4', e_a  \right) &=& \fb^a\R\left(e_4+f^be_b, e_3, e_4+f^ce_c, e_a  \right)+\lot\\
&=& \fb^a\R\left(e_4, e_3, e_4, e_a  \right)+ \fb^af^b\R\left(e_b, e_3, e_4, e_a  \right)\\
&&+\fb^af^c\R\left(e_4, e_3, e_c, e_a  \right)+\lot\\
&=& 2\fb\c\b -\fb^af^b(-\rho\de_{ba}+\rhod\in_{ba})+2\fb^af^c\in_{ac}\rhod+\lot\\
&=& 2\fb\c\b+\rho(f\c\fb) -3\rhod (f\wedge\fb)+\lot
\eeaa 
  and 
 \beaa
 \fb^a\R\left(\la^{-1}e_4',  e_a, \la^{-1}e_4', \la e_3'\right) &=& \fb^a\R\left(e_4+f^be_b,  e_a, e_4+f^ce_c, e_3+\fb^de_d\right)\\
 &=& \fb^a\R\left(e_4,  e_a, e_4, e_3\right)+\fb^af^b\R\left(e_b,  e_a, e_4, e_3\right)\\
 &&+ \fb^af^c\R\left(e_4,  e_a, e_c, e_3\right) +\lot\\
 &=& 2\fb\c\b+ \rho(f\c\fb) -3\rhod (f\wedge\fb)+\lot
 \eeaa 
 We infer
  \beaa
4\rho' &=&  \R\left(\la^{-1}e_4',e_3, \la^{-1}e_4', e_3\right) + \fb^a\R\left(\la^{-1}e_4',e_3, \la^{-1}e_4', e_a  \right)\\
&& +  \fb^a\R\left(\la^{-1}e_4',  e_a, \la^{-1}e_4', \la e_3'\right)   + 4 (f\c\fb)\rho+\lot\\
&=& 4\rho +4\fb\c\b - 4f\c\bb +6\rho(f\c\fb) -6\rhod (f\wedge\fb) +\lot
\eeaa  
and hence
  \beaa
\rho' &=& \rho +\err(\rho, \rho'),\\
\err(\rho, \rho') &=& \fb\c\b - f\c\bb +\frac{3}{2}\rho(f\c\fb) -\frac{3}{2}\rhod (f\wedge\fb) +\lot
\eeaa   
 as desired. Finally, the transformation formulas for $\aa$, $\bb$ and $\dual\rho$ follow respectively from the ones for $\a$, $\b$ and $\rho$ by symmetry consideration. This concludes the proof of Proposition \ref{Prop:transformation-formulas-generalcasewithoutassumptions}.

%%%%%%%%%%%%%%%%%%%%%%%%%%%%%%%%%%%%%%% 
 
  \section{Proof of Proposition \ref{Prop:contractionforNN}}
  \lab{sec:proofofProp:contractionforNN} 
  
%%%%%%%%%%%%%%%%%%%%%%%%%%%%%%%%%%%%%%%

%%%%%%%%%%%%%%%%%%%%%%%%%%%%%%%%%%%%%%%
 
  \subsection{Notations for differences}
  
%%%%%%%%%%%%%%%%%%%%%%%%%%%%%%%%%%%%%%%  
 
  To compare the ninetets
    \beaa
 \NN^{n,\#}=\Big(U^{(n)}, S^{(n)}, \ovla^{n,\#},  f^{n,\#},  \fb^{n,\#}; \, \Cbn_0, \Mn_0, \Cbpn, \Mpn\Big)
 \eeaa
 we   start by introducing notations for differences.
  
 Recall the notations
 \bea
 \ovlanndiez&= \big(\ovlann\big)^{\#_n}, \quad  \fnndiez= \big(\fnn\big)^{\#_n}, \quad   \fbnndiez= \big(\fbnn\big)^{\#_n}.
 \eea
  We also introduce  the operators,
\bea
\curln:=\big(\curlSn\big)^{\#_n}, \quad  \divn:=\big(\divSn\big)^{\#_n}, \quad \lapn:=\big(\lapSn\big)^{\#_n},
\eea
defined with respect to the pull back  metric  
\bea
\gn:=g^{\S(n), \#_n}, 
\eea
i.e. the pull back by  $\Psi^{(n)}$  of  the metric $g^{\S(n)}$. We also introduce a notation for the area radius and the Hawking mass of $\S(n)$
\bea
r^{(n)}:=r^{\S(n)}, \qquad\qquad  m^{(n)}:=m^{\S(n)},
\eea
as well as 
\bea
F^{n+1,\#}:=( \ovlanndiez, \fnndiez, \fbnndiez), \qquad\qquad {\ovb}^{n+1, \#}  := \big({\ovb}\,^{(n+1)}\big)^{\#_n}.
\eea

We  define the differences
 \bea
 \lab{eq:def-differencesfnn-diez}
 \bsplit
    \fnndot&:=\fnndiez- \fndiez, \\   
    \fbnndot&:=\fbnndiez - \fbndiez, \, \\
     \ovlanndot&:=\ovlanndiez - \ovlandiez,\\
     (\de\ovb)^{(n+1)} &:= {\ovb}^{n+1, \#} - {\ovb}^{n, \#}, 
     \end{split}
   \eea
    and
 \bea
 \bsplit
  \Cbnnde_0&:=  \Cbnndot_0- \Cbndot_0,\\
  \Cbpnnde&:=\Cbpnndot-\Cbpndot,\\
  \Mnnde&:=\Mnndot_0-\Mndot_0,\\
  \Mpnnde&:=\Mpnndot-\Mpndot.
  \end{split}
  \eea

%%%%%%%%%%%%%%%%%%%%%%%%%%%%%%%%%%%%%%% 
 
 \subsection{Comparison results  for iterates}
 
%%%%%%%%%%%%%%%%%%%%%%%%%%%%%%%%%%%%%%%

  \begin{proposition}
  \lab{Prop:variouscomparisonsn}
 Recall that the sequence $U^{(n)}$ satisfies, in view of \eqref{eq:nintetbound},
 \beaa
 \big\|( \Un,  \Sn) \big\|_{\hk_{s_{max}+1}(\ovS)} &\les& r\dg
 \eeaa
 uniformly in $n$.    The following estimates hold true.
        \begin{enumerate}
        \item  We have, relative to the  coordinates $y^1, y^2$ on $\ovS$,
        \beaa
        \Big| \gn_{ab}- g^{(n-1)} _{ab}\Big| \les  \big\|( \de\Un,  \de\Sn) \big\|_{\hk_{3}(\ovS)}.
        \eeaa
         \item For every $f\in\mathcal{S}_k(\S)$ we have,
  \bea
 \label{eq:lemma:int-gaS-ga3-n}
 \|f^\#\|_{L^2(\ovS, \gn)}   &=&   \|f^\#\|_{L^2(\ovS, g^{(n-1)} )}  \Big(1+ O(r^{-2}\big\|( \de\Un,  \de\Sn) \big\|_{\hk_{3}(\ovS)})        \Big).
 \eea
\item  As a corollary  of \eqref{eq:lemma:int-gaS-ga3-n} (choosing $f=1$) we deduce
  \bea
  \lab{eq:lemma:int-gaS-ga4-n}
 \frac{r^{(n)}}{r^{(n-1)}}= 1 + O(r ^{-2}\big\|( \Un,  \Sn) \big\|_{\hk_{3}(\ovS)} ),  
 \eea
 where $r^{(n)} $ is the area radius of $\S(n)$ and $r^{(n-1)} $ that of $\S(n-1)$.
 \item We have
 \bea
 |m^{(n)}-m^{(n-1)}| &\les& r^{-1}\big\|( \Un,  \Sn) \big\|_{\hk_{3}(\ovS)}
 \eea
 where $m^{(n)}$ is the Hawking mass of $\S(n)$ and $m^{(n-1)}$  is the Hawking mass of $\S(n-1)$.
 \item  We have
 \bea
\sum_{a,b,c=1,2}\Big\|(\Ga^{(n)})_{ab}^c-(\Ga^{(n-1)})_{ab}^c\Big\|_{\hk_{2}(\ovS)} &\les \big\|( \de\Un,  \de\Sn) \big\|_{\hk_{3}(\ovS)}
\eea
where $ \Ga^{(n)}, \, \Ga^{(n-1)} $ denote the Christoffel symbols  of the metrics $g^{(n)}, \, g^{(n-1)} $ relative to the coordinates $y^1, y^2$
on $\ovS$ and $\hk_k(\ovS)$ the Sobolev  spaces w.r.t.  the metric $\ovg$.
\item  We have  for every $F\in\mathcal{S}_k(\ovS)$,   for all $k\le s_{max}$,
\bea
 \Big\|F\Big\|_{\hk_{k}(\ovS, \gn)} &=& \Big\|F \Big\|_{\hk_{k}(\ovS, g^{(n-1)})}\Big(1+ O(r^{-1}\dg)        \Big).
\eea

\item  We have  for every $F\in\mathcal{S}_k(\RR)$,   for all $k\le 2$,
\bea
 \Big\|F^{\#_n}-F^{\#_{n-1}}\Big\|_{\hk_{k}(\ovS)} &\les& r^{-1}\big\|( \de\Un,  \de\Sn) \big\|_{\hk_{3}(\ovS)}\sup_{\RR}|\dk^{\leq k+1}F|.
\eea
        \end{enumerate}
        \end{proposition}
        \begin{proof}
        The proof   follows by a simple adaptation of the proofs  of Lemma \ref{Le:Transportcomparison}, Lemma \ref{lemma:comparison-gaS-ga}, Proposition \ref{Prop:comparison-gaS-ga:highersobolevregularity} and Corollary \ref{cor:comparison-gaS-ga-Oepgsphere:mSminusm}.
        \end{proof}

%%%%%%%%%%%%%%%%%%%%%%%%%%%%%%%%%%%%%%%   
   
   \subsection{Equations  for $\fnndot, \fbnndot, \ovlanndot$}
   
%%%%%%%%%%%%%%%%%%%%%%%%%%%%%%%%%%%%%%%   

\begin{lemma}\lab{lemma:equationssatisfiedbythesuccessivedifferences:contraction}
The quantities $\fnndot, \fbnndot, \ovlanndot$ verify the following system
\bea
\bsplit
\curln  \fnndot&= (\de h_1)^{(n)},\\
\curln \fbnndot&=(\de \underline{h}_1)^{(n)},
\end{split}
\eea
\bea
\bsplit
\divn  \fnndot+\frac{2}{r^{(n)}}\ovlanndot - \frac{2}{(r^{(n)})^2}(\de\ovb)^{(n+1)}&=(\de h_2)^{(n)},\\
\divn  \fbnndot +\frac{2}{r^{(n)}}\ovlanndot + \frac{2}{(r^{(n)})^2}(\de\ovb)^{(n+1)}&=  \Cbnnde_0+\sum_p \Cbpnnde \Jp\\
&+(\de \underline{h}_2)^{(n)},
\end{split}
\eea
\bea
\bsplit
  \left(\lapn + \frac{2}{(r^{(n)})^2}\right)\ovlanndot  &=\Mnnde_0+\sum _p\Mpnnde\Jp+(\de h_3)^{(n)}\\
  &+\frac{1}{2r^{(n)}}\left(  \Cbnnde_0+\sum_p \Cbpnnde \Jp\right),
  \end{split}
\eea
\bea
\bsplit
\lapn (\de\ovb)^{(n+1)} -\frac{1}{2}\divn\Big(\fbnndot - \fnndot\Big) &= (\de h_4)^{(n)},\\
\ov{(\de\ovb)^{(n+1)}}^{\ovS, \#} &= (\de b_0)^{(n)},
\end{split}
\eea
   and,
\bea
(\divn\fnndot)_{\ell=1}=(\de\La)^{(n)}, \qquad (\divn\fbnndot)_{\ell=1}=(\de\Lab)^{(n)},
\eea
   where $(\de h_1)^{(n)}$, $(\de h_2)^{(n)}$, $(\de h_3)^{(n)}$, $(\de h_4)^{(n)}$, $(\de \underline{h}_1)^{(n)}$, $(\de \underline{h}_2)^{(n)}$ and $(\de b_0)^{(n)}$ are given by
\bea\lab{eq:definitionofh1nhb1nh2nhb2nh3nh4n:contraction:1}
\bsplit
(\de h_1)^{(n)} &= -\de\err_1[\Fndiez]-\big( \curln-\curlnmin\big)f^{n,\#},\\
(\de \underline{h}_1)^{(n)} &=-\de\err_1[\Fndiez]-\big( \curln-\curlnmin\big)\fb^{n,\#},
\end{split}
\eea
 \bea\lab{eq:definitionofh1nhb1nh2nhb2nh3nh4n:contraction:2}
   \bsplit
(\de h_2)^{(n)} &=-\left(\frac{2}{r^{(n)}}-\frac{2}{r^{(n-1)}}\right)\ovla^{n,\#} +\left(\frac{2}{(r^{(n)})^2}-\frac{2}{(r^{(n-1)})^2}\right)\ovb^{n,\#}\\
& -\left(\ka^{\#_{n-1}}-\frac{2}{r^{(n)}}\right) \ovlandot -\left(\ka^{\#_{n-1}}-\ka^{\#_{n-2}}-\frac{2}{r^{(n)}}+\frac{2}{r^{(n-1)}}\right) \ovla^{n-1,\#}\\
& - \left(\kadot^{\#_{n}}-\kadot^{\#_{n-1}}\right) -\big( \divn -\divnmin\big)f^{n,\#}\\
& - \de\err_1[\Fndiez] -\left(\frac{2(r^{\#_n}-r^{(n)})^2}{r^{\#_n}(r^{(n)})^2} - \frac{2(r^{\#_{n-1}}-r^{(n-1)})^2}{r^{\#_{n-1}}(r^{(n-1)})^2} \right),\\
(\de \underline{h}_2)^{(n)} &=-\left(\frac{2}{r^{(n)}}-\frac{2}{r^{(n-1)}}\right)\ovla^{n,\#} -\left(\frac{2}{(r^{(n)})^2}-\frac{2}{(r^{(n-1)})^2}\right)\ovb^{n,\#}\\
&+\left(\kab^{\#_{n-1}}+\frac{2}{r^{(n)}}\right) \ovlandot + \left(\kab^{\#_{n-1}}-\kab^{\#_{n-2}}+\frac{2}{r^{(n)}}-\frac{2}{r^{(n-1)}}\right) \ovla^{n-1,\#}\\
& - \left(\kabdot^{\#_{n}}-\kabdot^{\#_{n-1}}\right) +\left(\frac{4m^{(n)}}{(r^{(n)})^2} - \frac{4m^{(n-1)}}{(r^{(n-1)})^2}\right) -\left(\frac{4m^{\#_n}}{(r^{\#_n})^2} - \frac{4m^{\#_{n-1}}}{(r^{\#_{n-1}})^2}\right)\\
& -\big( \divn -\divnmin\big)\fb^{n,\#}- \de\err_1[\Fndiez] \\
&+\left(\frac{2(r^{\#_n}-r^{(n)})^2}{r^{\#_n}(r^{(n)})^2} - \frac{2(r^{\#_{n-1}}-r^{(n-1)})^2}{r^{\#_{n-1}}(r^{(n-1)})^2} \right),
\end{split}
\eea
 \bea\lab{eq:definitionofh1nhb1nh2nhb2nh3nh4n:contraction:3}
   \bsplit
(\de h_3)^{(n)}  &= -\left(\frac{2}{(r^{(n)})^2}-\frac{2}{(r^{(n-1)})^2}\right){\ovla}^{n,\#} -\left(V^{\#_n}-\frac{2}{(r^{(n)})^2}\right)\ovlanndot\\
&-\left(V^{\#_n}-V^{\#_{n-1}}-\frac{2}{(r^{(n)})^2}+-\frac{2}{(r^{(n-1)})^2}\right){\ovla}^{n-1,\#}\\
& - \left(\mudot^{\#_{n}}-\mudot^{\#_{n-1}}\right)  +\left(\frac{2m^{(n)}}{(r^{(n)})^3} - \frac{2m^{(n-1)}}{(r^{(n-1)})^3}\right) -\left(\frac{2m^{\#_n}}{(r^{\#_n})^3} - \frac{2m^{\#_{n-1}}}{(r^{\#_{n-1}})^3}\right)\\
&+\Bigg[\left(\om +\frac 1 4 \ka \right) \left(\frac{2\Up}{r} -\frac{2\Up^{(n)}}{r^{(n)} }+  \dot{\underline{C}}_0^{(n)}+\sum_p\dot{\underline{C}}^{(n),p} \JpSn  -\kabdot \right)\\
  &-\left(\omb +\frac 1 4 \kab \right) \left(\frac{2}{r^{(n)}} -\frac{2}{r} -\kadot \right)  -\frac{1}{2r^{(n)}}\left( \dot{\underline{C}}_0^{(n)}+\sum_p\dot{\underline{C}}^{(n),p} \JpSn\right)\Bigg]^{\#_n}\\
  &-\Bigg[\left(\om +\frac 1 4 \ka \right) \left(\frac{2\Up}{r} -\frac{2\Up^{(n-1)}}{r^{(n-1)} }+  \dot{\underline{C}}_0^{(n-1)}+\sum_p\dot{\underline{C}}^{(n-1),p} \JpSn  -\kabdot \right)\\
  &-\left(\omb +\frac 1 4 \kab \right) \left(\frac{2}{r^{(n-1)}} -\frac{2}{r} -\kadot \right)  -\frac{1}{2r^{(n-1)}}\left( \dot{\underline{C}}_0^{(n-1)}+\sum_p\dot{\underline{C}}^{(n-1),p} J^{\S(n-1),p}\right)\Bigg]^{\#_{n-1}}\\
&-\big( \lapn-\lapnmin\big){\ovla}^{n,\#} +\de\err_2[ \Fndiez],\\
\end{split}
\eea
 \bea\lab{eq:definitionofh1nhb1nh2nhb2nh3nh4n:contraction:4}
   \bsplit
(\de h_4)^{(n)} &=\div^{(n-1)}\left(\frac{2m^{\#_n}}{r^{\#_n}}\fndot+\left(\frac{2m^{\#_n}}{r^{\#_n}}-\frac{2m^{\#_{n-1}}}{r^{\#_{n-1}}}\right)f^{n-1,\#}+ \de\err_1[\Fndiez]\right)\\
&+(\div^{(n-1)}-\div^{(n-2)})\left(\frac{2m^{\#_{n-1}}}{r^{\#_{n-1}}}f^{n-1,\#}+ \err_1[\Fndiez]\right)\\
& -\big( \lapn-\lapnmin\big)\ovb^{n, \#}+\frac{1}{2}\big( \divn-\divn\big)\Big(\fb^{n,\#} - f^{n,\#}\Big),\\
(\de b_0)^{(n)} &=\Big(\ov{r}^{\S(n)}-\ov{r}^{\S(n-1)}\Big)-\Big(r^{(n)}-r^{(n-1)}\Big)+\ov{{\ovb}^{n,\#}}^{(\ovS, g^{(n-1)})}-\ov{{\ovb}^{n,\#}}^{(\ovS, g^{(n)})},
\end{split}
\eea
and where $(\de\La)^{(n)}$ and $(\de\Lab)^{(n)}$ are given by
\bea\lab{eq:definitionofh1nhb1nh2nhb2nh3nh4n:contraction:5}
\bsplit
(\de\La)^{(n)} &=\int_{\ovS}\Big[(\div^{(n)}-\div^{(n-1)})f^{n,\#}\Jp da_{\gn}\Big]\\
&+\int_{\ovS}\Big[\div^{(n-1)}f^{n,\#}\Jp \big( da_{\gn}-da_{g^{(n-1)} }\big)\Big], \\
(\de\Lab)^{(n)} &=\int_{\ovS}\Big[(\div^{(n)}-\div^{(n-1)})\fb^{n,\#}\Jp da_{\gn}\Big]\\
&+\int_{\ovS}\Big[\div^{(n-1)}\fb^{n,\#}\Jp \big( da_{\gn}-da_{g^{(n-1)} }\big)\Big].
\end{split}
\eea
\end{lemma}

\begin{proof}
The proof follows by pulling back the system  \eqref{GeneralizedGCMsystem-n+1-MainThm1}-\eqref{GeneralizedGCMsystem-n+1-MainThm3}, \eqref{eq:definitionofh1nhb1nh2nhb2nh3nh4n} on $\ovS$ and then taking differences on $\ovS$ between successive iterates.
\end{proof}

%%%%%%%%%%%%%%%%%%%%%%%%%%%%%%%%
 
 \subsection{Estimates for $\fnndot, \fbnndot, \ovlanndot$}
   
%%%%%%%%%%%%%%%%%%%%%%%%%%%%%%%%

In view of Lemma \ref{lemma:equationssatisfiedbythesuccessivedifferences:contraction},  $(\fnndot, \fbnndot, \ovlanndot)$ satisfy the assumptions of Corollary \ref{Thm.GCMSequations-fixedS:contraction:deformationsphereversion}. As a consequence, the following  a priori estimates are verified 
\bea\lab{eq:basicestimateofdifferencesofffbovla:contraction:1}
\nn&&\|(\fnndot ,\fbnndot , \widecheck{\ovlanndot }^{\ovS,g^{(n)}})\|_{\hk_3(\ovS)} +\sum_p\Big(r^2|\Cbpnnde|+r^3|\Mpnnde|\Big)\\  
\nn&\les& r\|(\widecheck{(\de h_1)^{(n)}}^{\ovS,g^{(n)}}, \,\widecheck{(\de \underline{h}_1)^{(n)}}^{\ovS,g^{(n)}}, \,\widecheck{(\de h_2)^{(n)}}^{\ovS,g^{(n)}},\,\widecheck{(\de \underline{h}_2)^{(n)}}^{\ovS,g^{(n)}})\|_{\hk_2(\ovS)} \\
&&+r^2\|\widecheck{(\de h_3)^{(n)}}^{\ovS,g^{(n)}}\|_{\hk_1(\ovS)}+r\|\widecheck{(\de h_4)^{(n)}}^{\ovS,g^{(n)}}\|_{L^2(\ovS)} +|(\de\La)^{(n)}|+|(\de\Lab)^{(n)}|,
\eea
and
\bea\lab{eq:basicestimateofdifferencesofffbovla:contraction:2}
&&\nn r^2|\Cbnnde_0|+r^3|\Mnnde_0|+r\Big|\ov{\ovlanndot }^{\ovS,g^{(n)}}\Big|  \\
\nn&\les&  r\|(\widecheck{(\de h_1)^{(n)}}^{\ovS,g^{(n)}}, \,\widecheck{(\de \underline{h}_1)^{(n)}}^{\ovS,g^{(n)}}, \,(\de h_2)^{(n)},\,(\de \underline{h}_2)^{(n)})\|_{L^2(\ovS)} +r^2\|(\de h_3)^{(n)}\|_{L^2(\ovS)}\\
&&+r\|\widecheck{(\de h_4)^{(n)}}^{\ovS,g^{(n)}}\|_{L^2(\ovS)} +|(\de\La)^{(n)}|+|(\de\Lab)^{(n)}|+|(\de b_0)^{(n)}|.
\eea   

Next, we estimate each term in the RHS of \eqref{eq:basicestimateofdifferencesofffbovla:contraction:1} and \eqref{eq:basicestimateofdifferencesofffbovla:contraction:2}. In view of  the definition \eqref{eq:definitionofh1nhb1nh2nhb2nh3nh4n:contraction:1} of $(\de h_1)^{(n)}$ and $(\de \underline{h}_1)^{(n)}$, Proposition \ref{Prop:variouscomparisonsn}, and the uniform in $n$ bound \eqref{eq:nintetbound}, we have
\bea
\nn&& r\|(\widecheck{(\de h_1)^{(n)}}^{\ovS,g^{(n)}}, \,\widecheck{(\de \underline{h}_1)^{(n)}}^{\ovS,g^{(n)}})\|_{\hk_2(\ovS)} \\
&\les& \dg\left(\|(\fndot,\fbndot, \ovlandot)\|_{\hk_3(\ovS)}+r^{-1}\big\|( \de\Un,  \de\Sn) \big\|_{\hk_{3}(\ovS)}\right).
\eea

Next, in view of  the definition \eqref{eq:definitionofh1nhb1nh2nhb2nh3nh4n:contraction:2} of $(\de h_2)^{(n)}$ and $(\de \underline{h}_2)^{(n)}$, Proposition \ref{Prop:variouscomparisonsn}, and the uniform in $n$ bound \eqref{eq:nintetbound}, we have
\beaa
\nn&& r\|(\widecheck{(\de h_2)^{(n)}}^{\ovS,g^{(n)}}, \,\widecheck{(\de \underline{h}_2)^{(n)}}^{\ovS,g^{(n)}})\|_{\hk_2(\ovS)} \\
\nn&\les& \dg|r^{(n)}-r^{(n-1)}|+\dg|r^{(n-1)}-r^{(n-2)}|+r^{-2}\|r^{\#_n}- r^{\#_{n-1}}\|_{\hk_2(\ovS)}\\
\nn&&+r^{-1}\|m^{\#_n}- m^{\#_{n-1}}\|_{\hk_2(\ovS)}+r\|\kadot^{\#_{n}}-\kadot^{\#_{n-1}}\|_{\hk_2(\ovS)}+r\|\kabdot^{\#_{n}}-\kabdot^{\#_{n-1}}\|_{\hk_2(\ovS)}\\
&&+ \epg\left(\|(\fndot,\fbndot, \ovlandot)\|_{\hk_3(\ovS)}+r^{-1}\big\|( \de\Un,  \de\Sn) \big\|_{\hk_{3}(\ovS)}\right)
\eeaa
and
\beaa
\nn&& r\|((\de h_2)^{(n)}, \,(\de \underline{h}_2)^{(n)})\|_{L^2(\ovS)} \\
\nn&\les& |m^{(n)}-m^{(n-1)}|+(r^{-1}+\dg)|r^{(n)}-r^{(n-1)}|+\dg|r^{(n-1)}-r^{(n-2)}|+r^{-2}\|r^{\#_n}- r^{\#_{n-1}}\|_{\hk_2(\ovS)}\\
\nn&&+r^{-1}\|m^{\#_n}- m^{\#_{n-1}}\|_{\hk_2(\ovS)}+r\|\kadot^{\#_{n}}-\kadot^{\#_{n-1}}\|_{\hk_2(\ovS)}+r\|\kabdot^{\#_{n}}-\kabdot^{\#_{n-1}}\|_{\hk_2(\ovS)}\\
&&+ \epg\left(\|(\fndot,\fbndot, \ovlandot)\|_{\hk_3(\ovS)}+r^{-1}\big\|( \de\Un,  \de\Sn) \big\|_{\hk_{3}(\ovS)}\right).
\eeaa
Using Proposition \ref{Prop:variouscomparisonsn} and ${\bf A1}$, we deduce
\bea
\nn&& r\|(\widecheck{(\de h_2)^{(n)}}^{\ovS,g^{(n)}}, \,\widecheck{(\de \underline{h}_2)^{(n)}}^{\ovS,g^{(n)}})\|_{\hk_2(\ovS)} \\
&\les&  (r^{-1}+\epg)\left(\|(\fndot,\fbndot, \ovlandot)\|_{\hk_3(\ovS)}+r^{-1}\big\|( \de\Un,  \de\Sn) \big\|_{\hk_{3}(\ovS)}\right),
\eea
and
\bea
 \nn&&r\|((\de h_2)^{(n)}, \,(\de \underline{h}_2)^{(n)})\|_{L^2(\ovS)} \\
 &\les& r^{-1}\big\|( \de\Un,  \de\Sn) \big\|_{\hk_{3}(\ovS)}+ (r^{-1}+\epg)\|(\fndot,\fbndot, \ovlandot)\|_{\hk_3(\ovS)}.
\eea

Next, we estimate $(\de h_3)^{(n)}$. First, we have in view of the definition \eqref{eq:definitionofh1nhb1nh2nhb2nh3nh4n:contraction:3} of $(\de h_3)^{(n)}$, 
\beaa
   \bsplit
(\de h_3)^{(n)}  &= -\left(\frac{2}{(r^{(n)})^2}-\frac{2}{(r^{(n-1)})^2}\right){\ovla}^{n,\#} -\left(V^{\#_n}-\frac{2}{(r^{(n)})^2}\right)\ovlanndot\\
&-\left(V^{\#_n}-V^{\#_{n-1}}-\frac{2}{(r^{(n)})^2}+-\frac{2}{(r^{(n-1)})^2}\right){\ovla}^{n-1,\#}\\
& - \left(\mudot^{\#_{n}}-\mudot^{\#_{n-1}}\right)  +\left(\frac{2m^{(n)}}{(r^{(n)})^3} - \frac{2m^{(n-1)}}{(r^{(n-1)})^3}\right) -\left(\frac{2m^{\#_n}}{(r^{\#_n})^3} - \frac{2m^{\#_{n-1}}}{(r^{\#_{n-1}})^3}\right)\\
&+(\de h_3)^{(n)}_0-\big( \lapn-\lapnmin\big){\ovla}^{n,\#} +\de\err_2[ \Fndiez],
\end{split}
\eeaa
where $(\de h_3)^{(n)}_0$ is given by
\beaa
(\de h_3)^{(n)}_0 &:=&  \Bigg[\left(\om +\frac 1 4 \ka \right) \left(\frac{2\Up}{r} -\frac{2\Up^{(n)}}{r^{(n)} }+  \dot{\underline{C}}_0^{(n)}+\sum_p\dot{\underline{C}}^{(n),p} \JpSn  -\kabdot \right)\\
  &&-\left(\omb +\frac 1 4 \kab \right) \left(\frac{2}{r^{(n)}} -\frac{2}{r} -\kadot \right)  -\frac{1}{2r^{(n)}}\left( \dot{\underline{C}}_0^{(n)}+\sum_p\dot{\underline{C}}^{(n),p} \JpSn\right)\Bigg]^{\#_n}\\
  &&-\Bigg[\left(\om +\frac 1 4 \ka \right) \left(\frac{2\Up}{r} -\frac{2\Up^{(n-1)}}{r^{(n-1)} }+  \dot{\underline{C}}_0^{(n-1)}+\sum_p\dot{\underline{C}}^{(n-1),p} \JpSn  -\kabdot \right)\\
  &&-\left(\omb +\frac 1 4 \kab \right) \left(\frac{2}{r^{(n-1)}} -\frac{2}{r} -\kadot \right)  -\frac{1}{2r^{(n-1)}}\left( \dot{\underline{C}}_0^{(n-1)}+\sum_p\dot{\underline{C}}^{(n-1),p} J^{\S(n-1),p}\right)\Bigg]^{\#_{n-1}}.
\eeaa
In view of \eqref{eq:simplyficationlongexpresssioninthedeifntionofh3}, we infer
\beaa
(\de h_3)^{(n)}_0  &=&   \Bigg[\left(\frac{1}{2r} -\frac{1}{2r^{(n)}}\right)\left( \dot{\underline{C}}_0^{(n)}+\sum_p\dot{\underline{C}}^{(n),p} \JpSn\right)  +  \frac{2m^{(n)}}{r(r^{(n)})^2 }   -\frac{2m}{r^2r^{(n)}}  \\ 
 \nn && +  \left(\om +\frac 1 4 \left(\ka-\frac{2}{r}\right) \right) \left(\frac{2\Up}{r} -\frac{2\Up^{(n)}}{r^{(n)} }+  \dot{\underline{C}}_0^{(n)}+\sum_p\dot{\underline{C}}^{(n),p} \JpSn  \right)\\
  && -\left(\omb +\frac 1 4 \left(\kab+\frac{2\Up}{r}\right) \right) \left(\frac{2}{r^{(n)}} -\frac{2}{r} \right)   -\left(\om +\frac 1 4 \ka \right)\kabdot  +\left(\omb +\frac 1 4 \kab \right) \kadot \Bigg]^{\#_n}\\ 
  &&-\Bigg[  \left(\frac{1}{2r} -\frac{1}{2r^{(n-1)}}\right)\left( \dot{\underline{C}}_0^{(n-1)}+\sum_p\dot{\underline{C}}^{(n-1),p} J^{\S(n-1),p}\right)  +  \frac{2m^{(n-1)}}{r(r^{(n-1)})^2 }   -\frac{2m}{r^2r^{(n-1)}}  \\ 
 \nn && +  \left(\om +\frac 1 4 \left(\ka-\frac{2}{r}\right) \right) \left(\frac{2\Up}{r} -\frac{2\Up^{(n-1)}}{r^{(n-1)} }+  \dot{\underline{C}}_0^{(n-1)}+\sum_p\dot{\underline{C}}^{(n-1),p} J^{\S(n-1),p}  \right)\\
  && -\left(\omb +\frac 1 4 \left(\kab+\frac{2\Up}{r}\right) \right) \left(\frac{2}{r^{(n-1)}} -\frac{2}{r} \right)   -\left(\om +\frac 1 4 \ka \right)\kabdot  +\left(\omb +\frac 1 4 \kab \right) \kadot  \Bigg]^{\#_{n-1}}.
  \eeaa
Coming back to $(\de h_3)^{(n)}$, using  Proposition \ref{Prop:variouscomparisonsn} and  the uniform in $n$ bound \eqref{eq:nintetbound}, we have
\beaa
\nn&& r^2\|\widecheck{(\de h_3)^{(n)}}^{\ovS,g^{(n)}}\|_{\hk_1(\ovS)} \\
\nn&\les& (r^{-1}+\epg)\Big[|r^{(n)}-r^{(n-1)}|+|r^{(n-1)}-r^{(n-2)}|+r^{-1}\|r^{\#_n}- r^{\#_{n-1}}\|_{\hk_2(\ovS)}\Big]\\
\nn&&+r^{-1}\|m^{\#_n}- m^{\#_{n-1}}\|_{\hk_2(\ovS)}+r\|\kadot^{\#_{n}}-\kadot^{\#_{n-1}}\|_{\hk_2(\ovS)}+r\|\kabdot^{\#_{n}}-\kabdot^{\#_{n-1}}\|_{\hk_2(\ovS)}\\
&&+r^2\|\mudot^{\#_{n}}-\mudot^{\#_{n-1}}\|_{\hk_2(\ovS)}+\|\omb^{\#_{n}}-\omb^{\#_{n-1}}\|_{\hk_2(\ovS)}\\
&&+ \epg\left(\|(\fndot,\fbndot, \ovlandot)\|_{\hk_3(\ovS)}+r^{-1}\big\|( \de\Un,  \de\Sn) \big\|_{\hk_{3}(\ovS)}\right)\\
&&+r^2\epg\left(|\de\Cbndot_0|+\sum_p|\de\Cbpndot|\right)
\eeaa
and
\beaa
\nn&& r^2\|(\de h_3)^{(n)}\|_{L^2(\ovS)} \\
\nn&\les& |m^{(n)}-m^{(n-1)}| +(r^{-1}+\epg)\Big[|r^{(n)}-r^{(n-1)}|+|r^{(n-1)}-r^{(n-2)}|+r^{-1}\|r^{\#_n}- r^{\#_{n-1}}\|_{\hk_2(\ovS)}\Big]\\
\nn&&+r^{-1}\|m^{\#_n}- m^{\#_{n-1}}\|_{\hk_2(\ovS)}+r\|\kadot^{\#_{n}}-\kadot^{\#_{n-1}}\|_{\hk_2(\ovS)}+r\|\kabdot^{\#_{n}}-\kabdot^{\#_{n-1}}\|_{\hk_2(\ovS)}\\
&&+r^2\|\mudot^{\#_{n}}-\mudot^{\#_{n-1}}\|_{\hk_2(\ovS)}+\|\omb^{\#_{n}}-\omb^{\#_{n-1}}\|_{\hk_2(\ovS)}\\
&&+ \epg\left(\|(\fndot,\fbndot, \ovlandot)\|_{\hk_3(\ovS)}+r^{-1}\big\|( \de\Un,  \de\Sn) \big\|_{\hk_{3}(\ovS)}\right)\\
&&+r^2\epg\left(|\de\Cbndot_0|+\sum_p|\de\Cbpndot|\right).
\eeaa
Using Proposition \ref{Prop:variouscomparisonsn} and ${\bf A1}$, we deduce
\bea
\nn&& r^2\|\widecheck{(\de h_3)^{(n)}}^{\ovS,g^{(n)}}\|_{\hk_1(\ovS)} \\
\nn&\les&  (r^{-1}+\epg)\left(\|(\fndot,\fbndot, \ovlandot)\|_{\hk_3(\ovS)}+r^{-1}\big\|( \de\Un,  \de\Sn) \big\|_{\hk_{3}(\ovS)}\right)\\
&& +r^2\epg\left(|\de\Cbndot_0|+\sum_p|\de\Cbpndot|\right),
\eea
and
\bea
 \nn&&r^2\|(\de h_3)^{(n)}\|_{L^2(\ovS)} \\
 \nn&\les& r^{-1}\big\|( \de\Un,  \de\Sn) \big\|_{\hk_{3}(\ovS)}+ (r^{-1}+\epg)\|(\fndot,\fbndot, \ovlandot)\|_{\hk_3(\ovS)}\\
 &&+r^2\epg\left(|\de\Cbndot_0|+\sum_p|\de\Cbpndot|\right).
\eea

The estimates for $(\de h_4)^{(n)}$ and $(\de b_0)^{(n)}$, $(\de\La)^{(n)}$ and $(\de\Lab)^{(n)}$ are similar and in fact easier. We obtain for those quantities 
\bea
\nn&& r\|\widecheck{(\de h_4)^{(n)}}^{\ovS,g^{(n)}}\|_{L^2(\ovS)} +|(\de\La)^{(n)}|+|(\de\Lab)^{(n)}|\\
&\les&  (r^{-1}+\epg)\left(\|(\fndot,\fbndot, \ovlandot)\|_{\hk_3(\ovS)}+r^{-1}\big\|( \de\Un,  \de\Sn) \big\|_{\hk_{3}(\ovS)}\right),
\eea
and
\bea
  |(\de b_0)^{(n)}| \les r^{-1}\big\|( \de\Un,  \de\Sn) \big\|_{\hk_{3}(\ovS)}+ (r^{-1}+\epg)\|(\fndot,\fbndot, \ovlandot)\|_{\hk_3(\ovS)}.
\eea

Gathering the above estimates for $(\de h_1)^{(n)}$, $(\de h_2)^{(n)}$, $(\de h_3)^{(n)}$, $(\de h_4)^{(n)}$, $(\de \underline{h}_1)^{(n)}$, $(\de \underline{h}_2)^{(n)}$,  $(\de b_0)^{(n)}$, $(\de\La)^{(n)}$ and $(\de\Lab)^{(n)}$, and plugging in \eqref{eq:basicestimateofdifferencesofffbovla:contraction:1} \eqref{eq:basicestimateofdifferencesofffbovla:contraction:2}, we infer
\bea\lab{eq:basicestimateofdifferencesofffbovla:contraction:1:bis}
\nn&&\|(\fnndot ,\fbnndot , \widecheck{\ovlanndot }^{\ovS,g^{(n)}})\|_{\hk_3(\ovS)} +\sum_p\Big(r^2|\Cbpnnde|+r^3|\Mpnnde|\Big)\\  
\nn&\les& (r^{-1}+\epg)\left(\|(\fndot,\fbndot, \ovlandot)\|_{\hk_3(\ovS)}+r^{-1}\big\|( \de\Un,  \de\Sn) \big\|_{\hk_{3}(\ovS)}\right)\\
&& +r^2\epg\left(|\de\Cbndot_0|+\sum_p|\de\Cbpndot|\right),
\eea
and
\bea\lab{eq:basicestimateofdifferencesofffbovla:contraction:2:bis}
&&\nn r^2|\Cbnnde_0|+r^3|\Mnnde_0|+r\Big|\ov{\ovlanndot }^{\ovS,g^{(n)}}\Big|  \\
\nn&\les&   r^{-1}\big\|( \de\Un,  \de\Sn) \big\|_{\hk_{3}(\ovS)}+ (r^{-1}+\epg)\|(\fndot,\fbndot, \ovlandot)\|_{\hk_3(\ovS)}\\
 &&+r^2\epg\left(|\de\Cbndot_0|+\sum_p|\de\Cbpndot|\right).
\eea

%%%%%%%%%%%%%%%%%%%%%%%%%%%

  \subsection{Equations for  $\de \Unn, \de \Snn$}

%%%%%%%%%%%%%%%%%%%%%%%%%%%
  
  According to \eqref{systemUU-SS-derivedn+1} we have
 \beaa
\bsplit
\lapzero \Unn&=\divzero \Big((\UU(\fnn, \fbnn , \Ga))^{\#_{n}}\Big),\\
\lapzero \Snn&=\divzero \Big((\SS(\fnn, \fbnn , \Ga))^{\#_{n}}\Big),\\
\Unn(South)&=\Snn(South)=0.
\end{split}
\eeaa
Introducing  
\beaa
\de \Unn=\Unn-\Un, \qquad  \de \Snn= \Snn-\Sn,
\eeaa 
we find,
\beaa
\lapzero \de\Unn&=&\divzero\Bigg( \Big(\UU(\fnn, \fbnn , \Ga)\Big)^{\#_{n}}-  \Big(\UU(\fn, \fbn , \Ga)\Big)^{\#_{n-1}}\Bigg),\\
\lapzero \de\Snn&=&\divzero\Bigg( \Big(\SS(\fnn, \fbnn , \Ga)\Big)^{\#_{n}}-  \Big(\SS(\fn, \fbn , \Ga)\Big)^{\#_{n-1}}\Bigg),
\eeaa
or, setting
\beaa
\UU^{(n+1)}&=& \UU(\fnn, \fbnn , \Ga), \qquad \UU^{(n)}= \UU(\fn, \fbn , \Ga),\\
\SS^{(n+1)}&=& \SS(\fnn, \fbnn , \Ga), \qquad \SS^{(n)}= \SS(\fn, \fbn , \Ga),
\eeaa
we write,
\beaa
\lapzero \de\Unn&=& \divzero\Big(  \big(\UU^{(n+1)} \big)^{\#_{n}} -\big(\UU^{(n)}\big)^{\#_{n-1}} \Big),\\
\lapzero \de\Snn&=& \divzero\Big(  \big(\SS^{(n+1)} \big)^{\#_{n}} -\big(\SS^{(n)}\big)^{\#_{n-1}} \Big).
\eeaa
By elliptic estimates we  deduce,
\beaa
&& r^{-1}\|(\de\Unn, \de\Snn\|_{\hk_{3}(\ovS)} \\
 &\les&\Big\| \big(\UU^{(n+1)} \big)^{\#_{n}} -\big(\UU^{(n)}\big)^{\#_{n-1}} \Big\|_{\hk_{2}(\ovS)}+\Big\| \big(\SS^{(n+1)} \big)^{\#_{n}} -\big(\SS^{(n)}\big)^{\#_{n-1}} \Big\|_{\hk_{2}(\ovS)}.
\eeaa
In view of the definition of $\UU(f,\fb, \Ga)$ and $\SS(f, \fb, \Ga)$, see  \eqref{eq:defintionofmathcsalUandmathcalSforcompatibilitydeformation} \eqref{eq:defintionofmathcsalUandmathcalSforcompatibilitydeformation:defscalarsa11a12a21a22}, we infer
\bea\lab{eq:basicestimateofdifferencesofffbovla:contraction:3:bis}
 r^{-1}\|(\de\Unn, \de\Snn\|_{\hk_{3}(\ovS)} 
 &\les&\|(\fndot,\fbndot)\|_{\hk_2(\ovS)}.
\eea

We are now in position to conclude the proof of Proposition \ref{Prop:contractionforNN}. We decompose $\NN^{n,\#}$ as
 \bea\lab{eq:NNassumofNN1NN2NN3}
 \NN^{n,\#}&=&  \NN^{n,\#}_1+\NN^{n,\#}_2+\NN^{n,\#}_3,
 \eea
 where 
 \beaa
  \NN^{n,\#}_1&:=&\Big(\widecheck{\ovla^{n,\#}}^{\ovS, g^{(n)}},  f^{n,\#},  \fb^{n,\#}, \Cbpn, \Mpn\Big),\\
  \NN^{n,\#}_2&:=&\Big(\ov{\ovla^{n,\#}}^{\ovS, g^{(n)}},  \, \Cbn_0, \Mn_0\Big),\\
  \NN^{n,\#}_3&:=&\Big(U^{(n)}, S^{(n)}\Big).
 \eeaa
Then, \eqref{eq:basicestimateofdifferencesofffbovla:contraction:1:bis}, \eqref{eq:basicestimateofdifferencesofffbovla:contraction:2:bis} and \eqref{eq:basicestimateofdifferencesofffbovla:contraction:3:bis} yield
\beaa
\|\NN^{n+1,\#}_1-\NN^{n,\#}_1\|_{3,\ovS}  &\les& (r^{-1}+\epg)\|\NN^{n,\#}-\NN^{n-1,\#}\|_{3,\ovS} ,\\
\|\NN^{n+1,\#}_2-\NN^{n,\#}_2\|_{3,\ovS}  &\les& \|\NN^{n,\#}_3-\NN^{n-1,\#}_3\|_{3,\ovS} +(r^{-1}+\epg)\|\NN^{n,\#}-\NN^{n-1,\#}\|_{3,\ovS} ,\\
\|\NN^{n+1,\#}_3-\NN^{n,\#}_3\|_{3,\ovS}  &\les& \|\NN^{n,\#}_1-\NN^{n-1,\#}_1\|_{3,\ovS} .
\eeaa
In view of \eqref{eq:NNassumofNN1NN2NN3}, we infer
\beaa
\nn\|\NN^{n+1,\#}-\NN^{n,\#}\|_{3,\ovS} &\les& (r^{-1}+\epg)\Big[\|\NN^{n,\#}-\NN^{n-1,\#}\|_{3,\ovS}+\|\NN^{n-1,\#}-\NN^{n-2,\#}\|_{3,\ovS}\\
&&+\|\NN^{n-2,\#}-\NN^{n-3,\#}\|_{3,\ovS}\Big]
\eeaa
as desired. This concludes the proof of Proposition \ref{Prop:contractionforNN}.


\begin{thebibliography}{99}
\bibitem{Ch-Kl} D. Christodoulou and S. Klainerman, \textit{The global nonlinear stability of the Minkowski space}, Princeton University Press (1993).

\bibitem{Daf}  M. Dafermos, \textit{The mathematical analysis of black holes in general relativity},  ICM, Seoul 2014. 

\bibitem{DafRod} M. Dafermos, I. Rodnianski, \textit{Lectures on black holes and linear waves, in Evolution equations}, Clay Mathematics Proceedings, Vol. 17. Amer. Math. Soc., Providence, RI, 2013.


\bibitem{GeKSz}  E. Giorgi, S. Klainerman, J. Szeftel, \textit{A general formalism for the stability of Kerr}, in preparation. 

\bibitem{HuYau} G. Huisken, S.-T. Yau, \textit{Definition of center of mass for isolated physical systems and unique foliations by stable spheres with constant mean curvature.} Invent. Math. 124 (1996).

\bibitem{KS} S. Klainerman, J. Szeftel, \textit{Global nonlinear stability of Schwarzschild spacetime under polarized perturbations}, arXiv:1711.07597.

\bibitem{KS:Kerr2}  S. Klainerman, J. Szeftel, \textit{Construction of GCM spheres in perturbations of Kerr II}, in preparation. 


\end{thebibliography}
    \end{document}